\documentclass[12pt, a4paper]{article}
\usepackage[top=1in, bottom=1.25in, left=1.25in, right=1.25in]{geometry}
\usepackage{amsfonts,amssymb,amsmath,amsthm} 
\usepackage{srcltx,tabularx}
\usepackage[utf8]{inputenc} 
\usepackage[cyr]{aeguill}
\usepackage{enumerate}
\usepackage[english]{babel}
\usepackage{authblk}
\usepackage[pagebackref=false,colorlinks=true,linktoc=page]{hyperref}
\usepackage{graphicx}
\usepackage{pstricks,pst-plot,pst-node}
\usepackage{pstricks-add}
\usepackage[bottom]{footmisc}
\usepackage[toc,page]{appendix}
\usepackage{breqn}
\usepackage{arydshln}
\usepackage{stfloats}
\usepackage{caption}
\usepackage{subcaption}
\usepackage{multirow}
\usepackage{pdflscape}
\usepackage{subfiles}

\usepackage{mathrsfs} 

\newcommand\Tstrut{\rule{0pt}{2.6ex}}         
\newcommand\Bstrut{\rule[-0.9ex]{0pt}{0pt}}   

\newtheorem{theorem}{Theorem}[section]
\newtheorem{lemma}[theorem]{Lemma}
\newtheorem{proposition}[theorem]{Proposition}
\newtheorem{corollary}[theorem]{Corollary}

\newtheorem*{claim*}{Claim}

\newenvironment{claimproof*}[1]{\par\noindent\textit{Proof of the claim:}\space#1}{}

\newtheorem{maintheorem}{Theorem}
\newtheorem{maincorollary}[maintheorem]{Corollary}

\newtheorem*{theorem*}{Theorem}

\newtheorem*{problem*}{Problem}
\newtheorem*{question*}{Question}

\newtheorem{thm}{Theorem}[section]

\newtheorem{prop}[thm]{Proposition}
\newtheorem{lem}[thm]{Lemma}
\newtheorem{cor}[thm]{Corollary}

\theoremstyle{definition}
\newtheorem{definition}[theorem]{Definition}
\newtheorem*{definition*}{Definition}
\newtheorem*{remark*}{Remark}

\newtheorem{rmk}[thm]{Remark}

\newcommand{\suchthat}{\;\ifnum\currentgrouptype=16 \middle\fi|\;}

\newcommand{\bigslant}[2]{{\raisebox{.2em}{$#1$}\left/\raisebox{-.2em}{$#2$}\right.}}
\DeclareMathOperator{\Aut}{\mathrm{Aut}}
\DeclareMathOperator{\PSL}{\mathrm{PSL}}
\DeclareMathOperator{\SL}{\mathrm{SL}}
\DeclareMathOperator{\Fix}{\mathrm{Fix}}

\DeclareMathOperator{\Sym}{\mathrm{Sym}}
\DeclareMathOperator{\Alt}{\mathrm{Alt}}
\DeclareMathOperator{\proj}{\mathrm{proj}}
\DeclareMathOperator{\Inn}{\mathrm{Inn}}

\DeclareMathOperator{\Rect}{\mathrm{Rect}}
\DeclareMathOperator{\C}{\mathbf{C}}

\DeclareMathOperator{\D}{\mathbf{D}}
\DeclareMathOperator{\N}{\mathbf{Z}_{\geq 0}}

\DeclareMathOperator{\sgn}{\mathrm{sgn}}
\DeclareMathOperator{\Sgn}{\mathrm{Sgn}}

\newcommand{\Mon}{\mathrm{Mon}}

\newcommand{\FC}{\mathrm{FC}}
\newcommand{\Comm}{\mathrm{Comm}}

\newcommand{\PGL}{\mathrm{PGL}}
\newcommand{\RR}{\mathbf{R}}
\newcommand{\QQ}{\mathbf{Q}}
\newcommand{\ZZ}{\mathbf{Z}}

\title{New simple lattices in products of\\ trees and their projections}
\author{Nicolas Radu\thanks{N.\ Radu is a F.R.S.-FNRS research fellow.}\\[0.3cm]with an appendix by Pierre-Emmanuel Caprace\thanks{P.-E.\ Caprace is a F.R.S.-FNRS senior research associate.}}

\affil{UCLouvain, 1348 Louvain-la-Neuve, Belgium}
\date{December 4, 2017}

\begin{document}


\maketitle


\begin{abstract}
Let $\Gamma \leq \Aut(T_{d_1}) \times \Aut(T_{d_2}) $ be a group acting freely and transitively on the product of two regular trees of degree $d_1$ and $d_2$. We develop an algorithm which computes the closure of the projection of $\Gamma$ on $\Aut(T_{d_t})$ under the hypothesis that $d_t \geq 6$ is even and that the local action of $\Gamma$ on $T_{d_t}$ contains $\Alt(d_t)$. We show that if $\Gamma$ is torsion-free and $d_1 = d_2 = 6$, exactly seven closed subgroups of $\Aut(T_6)$ arise in this way. We also construct two new infinite families of virtually simple lattices in $\Aut(T_{6}) \times \Aut(T_{4n})$ and in $\Aut(T_{2n}) \times \Aut(T_{2n+1})$ respectively, for all $n \geq 2$. In particular we provide an explicit presentation of a torsion-free infinite simple group on $5$ generators and $10$ relations, that splits as an amalgamated free product of two copies of $F_3$ over $F_{11}$. We include information arising from computer-assisted exhaustive searches of lattices in products of trees of small degrees. In an appendix by Pierre-Emmanuel Caprace, some of our results are used to show that abstract and relative commensurator groups of free groups are almost simple, providing partial answers to questions of Lubotzky and Lubotzky--Mozes--Zimmer.
\end{abstract}


\tableofcontents

\section{Introduction}

One of the starting points of this text is the following question, asked by Burger, Mozes and Zimmer in~\cite{BMZ}.

\begin{question*}[Burger--Mozes--Zimmer, 2009]
Which groups arise as closures of projections of cocompact lattices in $\Aut(T_1) \times \Aut(T_2)$, where $T_1$ and $T_2$ are locally finite regular trees?
\end{question*}

The cocompact lattices $\Gamma \leq \Aut(T_1) \times \Aut(T_2)$ of interest are those which are not commensurable to a product of lattices $\Gamma_t \leq \Aut(T_t)$, they are called \textbf{irreducible}. This irreducibility condition is equivalent to asking both projections on $\Aut(T_1)$ and $\Aut(T_2)$ to be non-discrete, see \cite[Proposition~1.2]{Burger2}.

Under this irreducibility assumption, the above question thus asks for which pairs of non-discrete closed subgroups $H_1 \leq \Aut(T_1)$ and $H_2 \leq \Aut(T_2)$ there exists a cocompact lattice $\Gamma \leq H_1 \times H_2$ whose projections on $H_1$ and $H_2$ are dense. A first remark is that $H_t$ must be locally topologically finitely generated for each $t \in \{1,2\}$ (see~\cite[Proposition~1.1.2]{BMZ}), which in particular excludes the full group $\Aut(T_t)$.

Most known irreducible cocompact lattices in products of trees come from the algebraic world: we call them arithmetic lattices. One can for instance construct a cocompact lattice $\Gamma \leq \mathrm{PGL}(2,\mathbf{Q}_p) \times \mathrm{PGL}(2,\mathbf{Q}_{p'})$ with dense projections for each distinct odd primes $p$ and $p'$, see~\cite[\S4, Theorem~1.1]{Vigneras} and \cite[Chapter~3]{Rattaggi}. Note that $\mathrm{PGL}(2,\mathbf{Q}_p)$ acts on its Bruhat--Tits tree $T$ (which is $(p+1)$-regular), so it can be seen as a closed subgroup of $\Aut(T)$.

The first non-algebraic groups that were shown to appear as closures of projections of cocompact lattices in a product of two trees are the universal groups $U(\Alt(d))$ defined and studied by Burger and Mozes in \cite{Burger} (for sufficiently large even values of $d$). Recall that, given a $d$-regular tree $T$ and a transitive finite permutation group $F \leq \Sym(d)$, the group $U(F)$ is the largest vertex-transitive subgroup of $\Aut(T)$ whose stabilizer of a vertex acts as $F$ on its $d$ neighbors (we write $\underline{G}(v) \cong F$ when the local action of a group $G \leq \Aut(T)$ at $v$ is given by $F$). In \cite{Burger2}, the same authors indeed constructed for each $m \geq 15$ and $n \geq 19$ a cocompact lattice $\Gamma \leq U(\Alt(2m)) \times U(\Alt(2n))$ with dense projections. Later, Rattaggi constructed in his thesis \cite{Rattaggi} such lattices for some smaller values of $m$ and $n$, for instance $m = n = 3$. He was also able to produce a cocompact lattice $\Gamma \leq U(\Alt(6)) \times U(M_{12})$ with dense projections, where $M_{12} \leq \Sym(12)$ is the Mathieu group of degree~$12$. All cocompact lattices $\Gamma \leq \Aut(T_1) \times \Aut(T_2)$ mentioned above are simply transitive on the vertices of $T_1 \times T_2$. We call such a lattice a \textbf{$(d_1,d_2)$-group}, where $d_1$ and $d_2$ are the degrees of the trees $T_1$ and $T_2$. 

The vertex-transitive non-discrete closed subgroups $H$ of $\Aut(T)$ such that $\underline{H}(v) \geq \Alt(d)$ for each $v \in V(T)$ were classified when $d \geq 6$, in \cite{Radu2} (see \S\ref{section:AutT} below). There are infinitely many isomorphism classes of such groups, and it is therefore natural to ask which of them can appear as the closure of a projection of a $(d_1,d_2)$-group. Our first aim is to develop tools enabling us to answer that question for small values of $d_1$ and $d_2$. Under some suitable hypotheses on the local action, we indeed develop algorithms that can be used to compute the closure of a projection.

\begin{maintheorem}\label{maintheorem:algorithms}
Let $\Gamma \leq \Aut(T_1) \times \Aut(T_2)$ be a $(d_1,d_2)$-group and let $H_1$ be the closure of the projection of $\Gamma$ on $\Aut(T_1)$. Suppose that $d_1 \geq 6$ and that $\underline{H_1}(v) \geq \Alt(d_1)$ for each $v \in V(T_1)$.
\begin{enumerate}[(i)]
\item There is an (efficient) algorithm that determines whether $\Gamma$ is irreducible.
\item If $\Gamma$ is irreducible and $d_1$ is even, then there is an (efficient) algorithm that computes the exact isomorphism class of $H_1$.
\end{enumerate}
\end{maintheorem}

Point (i) follows from results of Burger--Mozes \cite[Propositions~3.3.1 and~3.3.2]{Burger}, while (ii) requires a much more involved analysis.

In the following theorem we gather everything we can say about torsion-free $(6,6)$-groups, notably thanks to our algorithms above. This is a preview of what is done in~\S\ref{section:projections}. Two $(d_1,d_2)$-groups are called \textbf{equivalent} if they are conjugate in $\Aut(T_1 \times T_2)$. The systematic search of $(d_1,d_2)$-groups of small degree was undertaken in the torsion-free case by Kimberly-Robertson \cite{Kimberly} (for $d_1 = d_2 = 4$) and Rattaggi \cite{Rattaggi} (whose tables do not provide information on the number of equivalence classes). 

\begin{maintheorem}\label{maintheorem:66}
There are $32062$ torsion-free $(6,6)$-groups $\Gamma \leq \Aut(T_1) \times \Aut(T_2)$ up to equivalence. At least $18426$ of them are reducible, and at least $8227$ of them are irreducible. Moreover, there are exactly $7$ groups $H \leq \Aut(T_1)$ (up to conjugation) which are transitive on $V(T_1)$, satisfy $\underline{H}(v) \geq \Alt(6)$ for each $v \in V(T)$, and appear as the closure of the projection on $\Aut(T_1)$ of a torsion-free irreducible $(6,6)$-group $\Gamma \leq \Aut(T_1) \times \Aut(T_2)$.
\end{maintheorem}

We also show that, if $T_1$ is the $6$-regular tree, then infinitely many of the groups of \cite{Radu2} appear as the closure of the projection on $\Aut(T_1)$ of some $(6,d_2)$-group (see Theorem~\ref{maintheorem:64n} below). 

After our study of projections of $(d_1,d_2)$-groups, we go back to one of the original motivations of Burger--Mozes in~\cite{Burger}. The following problem was asked by Peter Neumann in~\cite{Neumann}.

\begin{problem*}[Neumann, 1973]
Let $G = F_m \ast_{F_k} F_n$ be a free amalgamated product of non-abelian free groups of finite rank over a subgroup of finite index. Can it happen that $G$ is simple?
\end{problem*}

Burger and Mozes answered this question in the positive by proving that for each $m \geq 109$ and $n \geq 150$, there exists a virtually simple $(2m,2n)$-group $\Gamma \leq U(\Alt(2m)) \times U(\Alt(2n))$ with dense projections (see Theorem~\cite[Theorems~5.5 and~6.4]{Burger}). The simple subgroup of finite index in $\Gamma$ is then isomorphic to its projection on $U(\Alt(2m))$, and this projection is edge-transitive. It can therefore be written as the free amalgamated product of two adjacent vertex stabilizers over an edge stabilizer (see~\cite{Serre}). Each vertex stabilizer in the first tree is a cocompact lattice in the second tree that is free, so we get a simple group of the form $F_m \ast_{F_k} F_n$ as wanted.

The values for $m$, $n$ and $k$ are however really huge in that case, so a presentation for such a group would be too large to manipulate. Rattaggi later found in \cite{Rattaggi} a $(8,12)$-group whose index~$4$ subgroup (preserving all types of vertices) is simple. With the same reasoning, he observed that this simple subgroup is a free amalgamated product $F_7 \ast_{F_{73}} F_7$. Even more recently, Bondarenko and Kivva found in~\cite{Bondarenko} a $(8,8)$-group whose index~$4$ subgroup is simple. This leads them to a simple group that decomposes as $F_7 \ast_{F_{49}} F_7$. In all these references, the explicit presentation for the simple groups was not computed.

In \S\ref{section:simple} we use the same techniques as the above authors to produce $(6,6)$-groups and $(4,5)$-groups whose index~$4$ subgroups are simple.

\begin{maintheorem}\label{maintheorem:simple66-45}
There exist (at least):
\begin{enumerate}[(i)]
\item $160$ pairwise non-commensurable virtually simple $(6,6)$-groups: two of them have a simple subgroup of index~$12$, and the other $158$ have a simple subgroup of index~$4$;
\item $60$ pairwise non-isomorphic virtually simple $(4,5)$-groups: $12$ of them have a simple subgroup of index~$8$, and the other $48$ have a simple subgroup of index~$4$.
\end{enumerate}
\end{maintheorem}

The $(6,6)$-groups obtained in Theorem~\ref{maintheorem:simple66-45}(i) have the advantage that they can have shorter presentations than the $(4,5)$-groups obtained in (ii). On the other hand, those $(4,5)$-groups lead to simple groups splitting as $F_3 \ast_{F_{11}} F_3$, which improves the results mentioned above. This is illustrated by the following down-to-earth statement.

\begin{maincorollary}\label{maincorollary:presentations}
\quad
\begin{enumerate}[(i)]
\item The following group, presented by $6$ generators and $10$ relators, is a $(6,6)$-group with a simple subgroup of index $4$:
$$\hspace{-1cm}\fontsize{10}{10}
\begin{array}{rl}
\langle a_1, a_2, a_3, b_1, b_2, b_3 \ \mid
& a_1b_1a_2^{-1}b_1,\
a_1b_2a_2b_2^{-1},\
a_1b_2^{-1}a_2^{-1}b_1^{-1},\
a_1b_1^{-1}a_2^{-1}b_2,\\
& a_1b_3a_1b_3^{-1},\
a_2b_3a_2b_3,\
a_2b_3^{-1}a_3b_3^{-1},\
a_3b_1a_3^{-1}b_1^{-1},\\
& a_3b_2a_3b_3,\
a_3b_2^{-1}a_3b_2^{-1}
\rangle
\end{array}$$
\item The following group is isomorphic to a free amalgamated product $F_3 \ast_{F_{11}} F_3$, is simple, and is an index~$4$ subgroup of a $(4,5)$-group:
$$\hspace{-1cm}\fontsize{10}{10}
\begin{array}{rl}
\langle x_1, x_2, x_3, y_1, y_2, y_3 \ \mid
& x_1 = y_1,\\
& x_2^2 = y_2y_1^{-1}y_2,\\
& x_3^2 = y_3^2,\\
& x_3^{-1}x_1x_3 = y_3^{-1}y_2y_3,\\
& x_3^{-1}x_2x_3 = y_3^{-1}y_1y_3,\\
& x_2^{-1}x_1x_2 = y_2^{-1}y_1^{-1}y_2,\\
& x_2^{-1}x_3^{-2}x_2 = y_2^{-1}y_1y_3^{-2}y_2,\\
& x_2^{-1}x_3^{-1}x_2^{-1}x_1x_3x_2 = y_2^{-1}y_1y_3^{-1}y_1^{-1}y_2y_3y_1^{-1}y_2,\\
&x_2^{-1}x_3^{-1}x_1x_2x_3x_2 = y_2^{-1}y_1y_3^{-1}y_2y_3^{-1}y_1y_3y_1^{-1}y_2,\\
&x_2^{-1}x_3^{-1}x_2^2x_3x_2 = y_2^{-1}y_1y_3^{-1}y_1y_3^{-1}y_1y_3y_1^{-1}y_2,\\
&x_2^{-1}x_3^{-1}x_2^{-1}x_3x_2x_3x_2 = y_2^{-1}y_1y_3^{-1}y_1^{-1}y_3^{-1}y_1y_3y_1^{-1}y_2 \rangle
\end{array}$$
\end{enumerate}
\end{maincorollary}

Observe that in view of the relation $x_1 = y_1$ in (ii), the simple group in question admits a presentation on $5$ generators and $10$ relators.

Remark that Burger--Mozes and their followers were only dealing with regular trees of even degrees. This is due to the fact that they were only considering torsion-free $(d_1,d_2)$-groups, which can only exist when $d_1$ and $d_2$ are even. In the following result, we show that any $d$-regular tree with $d \geq 4$ can appear as a factor of a product of two trees in which a simple cocompact lattice lives.

\begin{maintheorem}\label{maintheorem:2n2n+1}
For each $n \geq 2$, there exists a virtually simple $(2n,2n+1)$-group.
\end{maintheorem}

Theorem~\ref{maintheorem:2n2n+1} is used in Appendix~\ref{appendix:A}, which was written by Pierre-Emmanuel Caprace, to show that abstract and relative commensurator groups of free groups are almost simple, see Theorems~\ref{thm:RelCom} and~\ref{thm:AbsCom}.

Our next result provides a family of virtually simple $(6,d_2)$-groups with arbitrarily large $d_2$, so that the projection on the $6$-regular tree becomes larger and larger when $d_2 \to \infty$.

\begin{maintheorem}\label{maintheorem:64n}
There exists a virtually simple $(6,4n)$-group $\Gamma_{6,4n}$ for each $n \geq 2$, such that $\overline{\proj_1(\Gamma_{6,4n})} \to \Aut(T_1)$ in the Chabauty topology of $\Aut(T)$ when $n \to \infty$ (where $T_1$ is the $6$-regular tree).
\end{maintheorem}

This theorem can be used to prove the next statement. It was already established in much greater generality in~\cite[Corollary~4.25]{Bass-Kulkarni} and~\cite{Liu} with a completely different approach.

\begin{maincorollary}\label{maincorollary:commensurator}
Let $F_3$ be the free group on $3$ generators and let $T$ be the usual Cayley graph of $F_3$, i.e.\ the $6$-regular tree. Then the commensurator of $F_3$ in $\Aut(T)$ is dense in $\Aut(T)$.
\end{maincorollary}

We finally close this article with an experimental study of lattices in products of three trees. While our previous results show the existence of many irreducible $(d_1,d_2)$-groups, things are apparently different when a third tree pops up.

\begin{maintheorem}\label{maintheorem:3trees}
Let $T_1$, $T_2$, $T_3$ be $6$-regular trees and let $v_t \in V(T_t)$ for each $t \in \{1,2,3\}$. There is no subgroup $\Gamma \leq \Aut(T_1) \times \Aut(T_2) \times \Aut(T_3)$ acting simply transitively on the vertices of $T_1 \times T_2 \times T_3$ and such that the following conditions hold, where $H_t = \overline{\proj_t(\Gamma)} \leq \Aut(T_t)$:
\begin{itemize}
\item $H_1, H_2, H_3$ are non-discrete and $\underline{H_1}(v_1), \underline{H_2}(v_2), \underline{H_3}(v_3) \geq \Alt(6)$;
\item $\proj_{1,3}(\Gamma)$ is dense in $H_1 \times H_3$ and $\proj_{2,3}(\Gamma)$ is dense in $H_2 \times H_3$;
\item the stabilizers $\Gamma(v_1,v_3)$ and $\Gamma(v_2,v_3)$ are torsion-free.
\end{itemize}
\end{maintheorem}

\subsection*{Acknowledgements}

I thank the Isaac Newton Institute for Mathematical Sciences, Cambridge for
its hospitality during the programme \textit{Non-positive curvature group actions and cohomology} where part of the work on this paper was accomplished. I am also grateful to Pierre-Emmanuel Caprace for all his wise advice and his permanent interest in my work.

\section{The class \texorpdfstring{$\mathcal{G}'_{(i)}$}{G'_(i)} of automorphism groups of trees}\label{section:AutT}

In this section we recall a classification result from \cite{Radu2}. The groups involved in that result will play a great role in this paper; so we recall their definitions.

Let $T$ be the $d$-regular tree with $d \geq 3$. Let $V(T) = V_0(T) \sqcup V_1(T)$ be the canonical bipartition of the vertex set $V(T)$, so that each edge of $T$ has a vertex in $V_0(T)$ and a vertex in $V_1(T)$. We say that a vertex $v$ has \textbf{type} $t \in \{0,1\}$ if $v \in V_t(T)$. We denote by $\Aut(T)^+$ the index~$2$ subgroup of $\Aut(T)$ consisting of the type-preserving automorphisms.

A \textbf{legal coloring} $i$ of $T$ is a map $i \colon V(T) \to \{1, \ldots, d\}$ whose restriction $i|_{S(v,1)} \colon S(v,1) \to \{1,\ldots,d\}$ to $S(v,1)$ is a bijection for each $v \in V(T)$, where $S(v,r)$ denotes the set of vertices at distance from $r$ from $v$. Remark that vertices with different types do not interact in this definition, so that a legal coloring can be defined on $V_0(T)$ and on $V_1(T)$ independently.

Given $g \in \Aut(T)$ and $v \in V(T)$, the \textbf{local action} of $g$ at the vertex $v$ is defined by
$$\sigma_{(i)}(g,v) := i|_{S(g(v),1)} \circ g \circ i|_{S(v,1)}^{-1} \in \Sym(d).$$
This allows us to define the groups that will appear in the classification result.

\begin{definition}\label{definition:Radu}
Let $T$ be the $d$-regular tree with $d \geq 3$ and let $i$ be a legal coloring. Given a finite subset $A$ of $V(T)$ and $g \in \Aut(T)$, we write $\Sgn_{(i)}(g,A) := \prod_{w \in A} \sgn(\sigma_{(i)}(g,w))$. Given a subset $X$ of $\N$ and $v \in V(T)$, we set $S_X(v) := \bigcup_{r \in X}S(v,r)$. First define $G_{(i)}(\varnothing,\varnothing) = \Aut(T)$. Then, for $X$ a non-empty finite subset of $\N$, we define
$$G_{(i)}(X,X) := \left\{g \in \Aut(T) \suchthat \Sgn_{(i)}(g, S_X(v)) = 1 \text{ for each $v \in V(T)$}\right\},$$
$$G_{(i)}(X,X)^* := \left\{g \in \Aut(T) \suchthat 
\text{All } \Sgn_{(i)}(g, S_X(v))  \text{ with $v \in V(T)$ are equal}
\right\},$$
$$G_{(i)}'(X,X)^* := \left\{g \in \Aut(T) \suchthat \begin{array}{c}
\text{All } \Sgn_{(i)}(g, S_X(v)) \text{ with $v \in V_0(T)$} \\
\text{are equal to $p_0$, all } \Sgn_{(i)}(g, S_X(v)) \\ \text{with $v \in V_1(T)$ are equal to $p_1$, and} \\
\text{$p_0 = p_1$ if and only if $g \in \Aut(T)^+$}
\end{array}\right\},$$
$$G_{(i)}(X^*,X^*) := \left\{g \in \Aut(T) \suchthat \begin{array}{c}
\text{All } \Sgn_{(i)}(g, S_X(v)) \text{ with $v \in V_0(T)$} \\
\text{are equal and all } \Sgn_{(i)}(g, S_X(v)) \\ \text{ with $v \in V_1(T)$ are equal}
\end{array}\right\}.$$
We write $\mathcal{G}'_{(i)}$ for the collection of all these groups.
\end{definition}

We also set $G_{(i)}^+(X,X) := G_{(i)}(X,X) \cap \Aut(T)^+$ when $X$ is a (possibly empty) finite subset of $\N$. In the next theorem we collect some facts about those groups.

\begin{theorem}\label{theorem:Radusimple}
Let $T$ be the $d$-regular tree with $d \geq 3$, and let $i$ be a legal coloring of $T$. Let $X$ be a finite subset of $\N$. The group $G_{(i)}^+(X,X)$ is a normal subgroup of the four groups $G_{(i)}(X,X)$, $G_{(i)}(X,X)^*$, $G'_{(i)}(X,X)^*$ and $G_{(i)}(X^*,X^*)$, respectively of index~$2$, $4$, $4$ and $8$. If moreover $d \geq 4$, then $G_{(i)}^+(X,X)$ is simple.
\end{theorem}

\begin{proof}
See \cite[Theorem~A(ii)]{Radu2} for the simplicity of $G_{(i)}^+(X,X)$. The other statement is a consequence of the definitions of the groups.
\end{proof}

The classification result then states as follows.

\begin{theorem}\label{theorem:Raduclassification}
Let $T$ be the $d$-regular tree with $d \geq 4$, and let $i$ be a legal coloring of $T$.
\begin{enumerate}[(i)]
\item The members of $\mathcal{G}'_{(i)}$ are pairwise non-conjugate in $\Aut(T)$.
\item Suppose that $d \geq 6$. Let $H$ be a vertex-transitive non-discrete closed subgroup of $\Aut(T)$ such that $\underline{H}(v) \geq \Alt(d)$ for each $v \in V(T)$. Then $H$ is conjugate in $\Aut(T)$ to a group belonging to $\mathcal{G}'_{(i)}$.
\end{enumerate}
\end{theorem}

\begin{proof}
See~\cite[Corollary~C'(i) and Corollary~E']{Radu2}.
\end{proof}

\section{Structure theory of \texorpdfstring{$(d_1,d_2)$-groups}{(d_1,d_2)-groups}}

Given two integers $d_1, d_2 \geq 3$, we now consider the product $T_1 \times T_2$ of the $d_1$-regular tree $T_1$ and the $d_2$-regular tree $T_2$. It can be seen as a square-complex whose vertex set, edge set and square set are given by
\begin{align*}
V &= V(T_1) \times V(T_2),\\
E &= (E(T_1) \times V(T_2)) \cup (V(T_1) \times E(T_2)),\\
S &= E(T_1) \times E(T_2);
\end{align*}
with the natural incidence relations. We call $T_1$ the \textit{horizontal} tree and $T_2$ the \textit{vertical} tree, so that the edge set $E$ decomposes as $E = E_h \cup E_v$ with $E_h = E(T_1) \times V(T_2)$ being the \textit{horizontal edge set} and $E_v = V(T_1) \times E(T_2)$ the \textit{vertical edge set}.

Recall from the introduction that a \textbf{$(d_1,d_2)$-group} is a group $\Gamma \leq \Aut(T_1) \times \Aut(T_2)$ acting simply transitively on the set of vertices of $T_1 \times T_2$. In \cite[Chapter~6]{Burger2} and \cite{Rattaggi}, the authors studied those groups with the additional property that they are torsion-free. We do not make this assumption here; authorizing torsion will actually lead us to interesting examples. Recall also that a $(d_1,d_2)$-group $\Gamma$ is said to be \textbf{reducible} if it is commensurable to a product $\Gamma_1 \times \Gamma_2$ of lattices $\Gamma_t \leq \Aut(T_t)$. It is called \textbf{irreducible} if it is not reducible, which is equivalent to asking that $H_1 = \overline{\proj_1(\Gamma)} \leq \Aut(T_1)$ and $H_2 = \overline{\proj_2(\Gamma)} \leq \Aut(T_2)$ are both non-discrete \cite[Proposition~1.2]{Burger2}. (Note that $H_1$ and $H_2$ are either both discrete or both non-discrete.)

\subsection{\texorpdfstring{$(d_1,d_2)$-complexes}{(d_1,d_2)-complexes} and \texorpdfstring{$(d_1,d_2)$-data}{(d_1,d_2)-data}}
\label{subsection:complexesanddata}

Let $\Gamma$ be some $(d_1,d_2)$-group. Recall from \S\ref{section:AutT} that $V(T_1) = V_0(T_1) \sqcup V_1(T_1)$ and $V(T_2) = V_0(T_2) \sqcup V_1(T_2)$. Hence, the set $V$ can naturally be partitioned as $V = V_{00} \sqcup V_{01} \sqcup V_{10} \sqcup V_{11}$ where $V_{ij} = V_i(T_1) \times V_j(T_2)$ for each $i, j \in \{0,1\}$. We say that a vertex in $V_{ij}$ is of \textbf{type} $(i,j)$. Each element of $\Gamma$ can be type-preserving or not on each of the two trees $T_1$ and $T_2$, so that this induces a natural (surjective) homomorphism $\Gamma \to \C_2 \times \C_2$. The kernel of this homomorphism, which we denote by $\Gamma^+$, is an index~$4$ normal subgroup of $\Gamma$ and consists of elements of $\Gamma$ which preserve the types in $T_1 \times T_2$.

Let us focus for a moment on those groups $\Lambda \leq \Aut(T_1) \times \Aut(T_2)$ which preserve the types and act simply transitively on the vertices of each type, as $\Gamma^+$. In the next lemma, we show that those $\Lambda$ are always torsion-free.

\begin{lemma}\label{lemma:torsionfree}
Let $\Lambda \leq \Aut(T_1) \times \Aut(T_2)$ be a type-preserving group acting freely on the vertices of $T_1 \times T_2$. Then $\Lambda$ is torsion-free.
\end{lemma}

\begin{proof}
Let $g$ be a torsion element in $\Lambda$, i.e.\ such that $g^n = 1$ for some $n \geq 1$. For each $t \in \{1,2\}$, we deduce from \cite[Proposition~3.2]{Titsarbres} that $\proj_{\Aut(T_t)}(g)$ fixes a vertex of $T_t$ or inverses an edge of $T_t$. It cannot be an inversion since it is type-preserving, so it fixes a vertex of $T_t$. Hence, $g$ fixes a vertex of $T_1 \times T_2$. As $\Lambda$ acts freely on the vertices of $T_1 \times T_2$, this means that $g = 1$.
\end{proof}

When $\Lambda$ acts simply transitively on vertices of each type of $T_1 \times T_2$, the quotient square-complex $X_\Lambda = \Lambda \backslash (T_1 \times T_2)$ has four vertices. We denote them by $v_{00}, v_{10}, v_{11}$ and $v_{01}$, so that the projection of a vertex in $V_{ij}$ is $v_{ij}$. For each $j \in \{0,1\}$, there are $d_1$ edges between $v_{0j}$ and $v_{1j}$ (call them \textit{horizontal}), and for each $i \in \{0,1\}$, there are $d_2$ edges between $v_{i0}$ and $v_{i1}$ (call them \textit{vertical}). Also, there are exactly $d_1d_2$ squares in $X_\Lambda$, attached to the four vertices and such that the link of each vertex is a complete bipartite graph. We call such a finite square-complex a \textit{$(d_1,d_2)$-complex}, see Definition~\ref{definition:complex} below (and Figure~\ref{picture:complex}). Be aware that, in \cite{Rattaggi}, the author uses the same term for a similar (but different) square-complex.

\begin{definition}\label{definition:complex}
A \textbf{$(d_1,d_2)$-complex} is a square-complex with:
\begin{itemize}
\item four vertices $v_{00}, v_{10}, v_{11}$ and $v_{01}$;
\item $d_1$ edges between $v_{00}$ and $v_{10}$ and $d_1$ edges between $v_{11}$ and $v_{01}$;
\item $d_2$ edges between $v_{10}$ and $v_{11}$ and $d_2$ edges between $v_{01}$ and $v_{00}$;
\item $d_1d_2$ squares attached to the four vertices, such that for each pair $(e_h, e_v)$ of horizontal and vertical edges, there is exactly one square adjacent to both $e_h$ and $e_v$.
\end{itemize}
\end{definition}

\begin{figure}[t!]
\centering
\begin{pspicture*}(-1.7,-0.7)(3.7,2.7)
\fontsize{12pt}{12pt}\selectfont
\psset{unit=2cm}

\psline(0,0)(1,0)
\psbezier(0,0)(0.3,0.1)(0.7,0.1)(1,0)
\psbezier(0,0)(0.3,-0.1)(0.7,-0.1)(1,0)
\rput(0.5,-0.23){\small $d_1$ edges}

\psline(0,0)(0,1)
\psbezier(0,0)(0.1,0.3)(0.1,0.7)(0,1)
\psbezier(0,0)(-0.1,0.3)(-0.1,0.7)(0,1)
\rput(0.5,1.23){\small $d_1$ edges}

\psline(0,1)(1,1)
\psbezier(0,1)(0.3,0.9)(0.7,0.9)(1,1)
\psbezier(0,1)(0.3,1.1)(0.7,1.1)(1,1)
\rput(-0.5,0.5){\small $d_2$ edges}

\psline(1,0)(1,1)
\psbezier(1,0)(0.9,0.3)(0.9,0.7)(1,1)
\psbezier(1,0)(1.1,0.3)(1.1,0.7)(1,1)
\rput(1.5,0.5){\small $d_2$ edges}

\pscircle[fillstyle=solid,fillcolor=white](0,0){0.05}
\rput(-0.1,-0.15){$v_{00}$}
\pscircle[fillstyle=solid,fillcolor=white](1,0){0.05}
\rput(1.17,-0.15){$v_{10}$}
\pscircle[fillstyle=solid,fillcolor=white](1,1){0.05}
\rput(1.15,1.12){$v_{11}$}
\pscircle[fillstyle=solid,fillcolor=white](0,1){0.05}
\rput(-0.1,1.15){$v_{01}$}

\end{pspicture*}
\caption{The $1$-skeleton of a $(d_1,d_2)$-complex.}\label{picture:complex}
\end{figure}
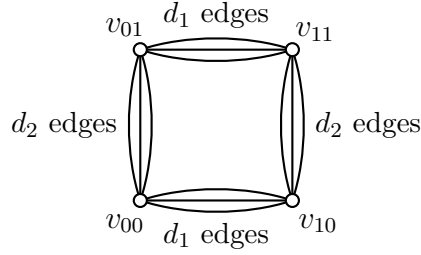

We saw above how to go from $\Lambda$ to a $(d_1,d_2)$-complex. Conversely, given a $(d_1,d_2)$-complex $X$, the universal cover $\tilde X$ of $X$ is the product of the $d_1$-regular tree and the $d_2$-regular tree, and the fundamental group $\pi_1(X)$ of $X$ acts simply transitively on the vertices of each type of $\tilde X$. One easily checks that this gives us a bijective correspondence between conjugacy classes of such groups $\Lambda$ (in $\Aut(T_1 \times T_2)$) and isomorphism classes of $(d_1,d_2)$-complexes.

Now if we come back to our $(d_1,d_2)$-group $\Gamma$, then we can consider the $(d_1,d_2)$-complex $X_{\Gamma^+} = \Gamma^+ \backslash (T_1 \times T_2)$. In addition, the action of $\Gamma$ on $T_1 \times T_2$ induces an action of $\C_2 \times \C_2$ on $X_{\Gamma^+}$. The three non-trivial elements of $\C_2 \times \C_2$ permute the vertices of $X_{\Gamma^+}$ with the three permutations $(v_{00}\ v_{10})(v_{01}\ v_{11})$, $(v_{00}\ v_{01})(v_{10}\ v_{11})$ and $(v_{00}\ v_{11})(v_{10}\ v_{01})$. We say that an action of $\C_2 \times \C_2$ on a $(d_1,d_2)$-complex is \textbf{good} if it induces those permutations on the vertices of that complex. Conversely, given a good action of $\C_2 \times \C_2$ on a $(d_1,d_2)$-complex $X$, we can consider the projection $\pi \colon T_1 \times T_2 \cong \tilde X \to X$ and lift $\C_2 \times \C_2$ to
$$\widetilde{\C_2 \times \C_2} = \{g \in \Aut(\tilde X) \mid \exists h \in \C_2 \times \C_2 : \pi \circ g = h \circ \pi\}.$$
This group $\widetilde{\C_2 \times \C_2}$ acts simply transitively on the vertices of $T_1 \times T_2$ and is thus a $(d_1,d_2)$-group. So we now have a bijective correspondence between $(d_1,d_2)$-groups (up to equivalence) and good $C_2 \times C_2$-actions on $(d_1,d_2)$-complexes (up to equivariant isomorphism).

A $(d_1,d_2)$-complex with a good $\C_2 \times \C_2$-action can be encoded via what we call a \textit{$(d_1,d_2)$-datum}. The following definition is inspired from the definition of a \textit{VH-datum} in \cite[Section~6]{Burger2}. The differences between the two notions come from the fact that we allow torsion. What the authors call a VH-datum will here be called a \textit{torsion-free $(d_1,d_2)$-datum}, see Definition~\ref{definition:tfdatum} below.

\begin{definition}\label{definition:datum}
A \textbf{$(d_1,d_2)$-datum} $(A,B,\varphi_A, \varphi_B, R)$ consists of two finite sets $A, B$ with $|A|\ = d_1$ and $|B|\ = d_2$, two involutions $\varphi_A \colon A \to A$, $\varphi_B \colon B \to B$ and a subset $R \subset A \times B \times A \times B$ satisfying conditions (1) and (2) below. Write $a^{-1} = \varphi_A(a)$ for $a \in A$ and $b^{-1} = \varphi_B(b)$ for $b \in B$. The two maps $\sigma, \rho \colon A \times B \times A \times B \to A \times B \times A \times B$ are defined by
\begin{align*}
\sigma(a,b,a',b') &= (a^{\prime-1}, b^{-1}, a^{-1}, b^{\prime-1}),\\
\rho(a,b,a',b') &= (a',b',a,b).
\end{align*}
\begin{enumerate}[(1)]
\item Each of the four projections of $R$ onto the subproducts of the form $A \times B$ or $B \times A$ are bijective;
\item $R$ is invariant under the action of the group $\langle \sigma, \rho \rangle \cong \C_2 \times \C_2$.
\end{enumerate}
\end{definition}

\begin{definition}
Given $(A,B,\varphi_A,\varphi_B,R)$ and $(A',B',\varphi_{A'},\varphi_{B'},R')$ two $(d_1,d_2)$-data, an object of one of the following forms is called an \textbf{equivalence} between the two data.
\begin{itemize}
\item A pair $(\alpha,\beta)$ of bijections $\alpha \colon A \to A'$ and $\beta \colon B \to B'$ such that $\varphi_{A'} = \alpha \varphi_A \alpha^{-1}$, $\varphi_{B'} = \beta \varphi_B \beta^{-1}$, and $R' = (\alpha \times \beta \times \alpha \times \beta)(R)$.
\item A pair $(\alpha, \beta)$ of bijections $\alpha \colon A \to B'$ and $\beta \colon B \to A'$ such that $\varphi_{A'} = \beta \varphi_B \beta^{-1}$, $\varphi_{B'} = \alpha \varphi_A \alpha^{-1}$, and
$$R' = \{(\beta(b), \alpha(a'), \beta(b'), \alpha(a')) \mid (a,b,a',b') \in R\}.$$
\end{itemize}
This defines an equivalence relation on the set of $(d_1,d_2)$-data.
\end{definition}

Given a $(d_1,d_2)$-datum $(A,B,\varphi_A, \varphi_B, R)$, one can build a $(d_1,d_2)$-complex with a good $\C_2 \times \C_2$-action as follows. Fix four vertices $v_{00}$, $v_{10}$, $v_{11}$ and $v_{01}$. For each $a \in A$ we draw an horizontal edge $e_a$ between $v_{00}$ and $v_{10}$ and an horizontal edge $e'_a$ between $v_{11}$ and $v_{01}$. Also, for each $b \in B$ we draw a vertical edge $f_b$ between $v_{10}$ and $v_{11}$ and a vertical edge $f'_b$ between $v_{01}$ and $v_{00}$. Then, for each $(a,b,a',b') \in R$ we glue a square to the edges $e_a, f_b, e'_{a'}$ and $f'_{b'}$. Condition (1) in Definition~\ref{definition:datum} ensures that this square complex $X$ is a $(d_1,d_2)$-datum, and (2) enables us to define a good $\C_2 \times \C_2$-action on it. Indeed, we can define $\tilde{\sigma} \in \Aut(X)$ by $\tilde{\sigma} \colon v_{00} \leftrightarrow v_{01}$, $v_{10} \leftrightarrow v_{11}$, $e_a \leftrightarrow e'_{a^{-1}}$, $f_b \leftrightarrow f_{b^{-1}}$, $f'_b \leftrightarrow f'_{b^{-1}}$ and $\tilde{\rho} \in \Aut(X)$ by $\tilde{\rho} \colon v_{00} \leftrightarrow v_{11}$, $v_{01} \leftrightarrow v_{10}$, $e_a \leftrightarrow e'_a$, $f_b \leftrightarrow f'_b$. (The actions on the squares are then clearly defined.) They are automorphisms of $X$ because $R$ is invariant under the action of $\langle \sigma, \rho \rangle$.

Conversely, from a good $\C_2 \times \C_2$-action on a $(d_1,d_2)$-complex we can come back to a $(d_1,d_2)$-datum. We can indeed consider two finite sets $A$ and $B$ with $|A|\ = d_1$ and $|B|\ = d_2$, denote the edges between $v_{00}$ and $v_{10}$ by $e_a$ with $a \in A$ (arbitrarily) and those between $v_{10}$ and $v_{11}$ by $f_b$ with $b \in B$ (arbitrarily). Then, if $\tilde{\rho}$ is the element of $\C_2 \times \C_2$ that exchanges $v_{00}$ and $v_{11}$, we write $e'_a = \tilde{\rho}(e_a)$ for each $a \in A$ and $f'_b = \tilde{\rho}(f_b)$ for each $b \in B$. If $\tilde{\sigma}$ is the element of $\C_2 \times \C_2$ that exchanges $v_{00}$ and $v_{01}$, then we define $\varphi_B \colon B \to B$ such that $\tilde{\sigma}(f_b) = f_{\varphi_B(b)}$ for each $b \in B$. Similarly, if $\tilde{\sigma}'$ exchanges $v_{00}$ and $v_{10}$ (i.e.\ $\tilde{\sigma}' = \tilde{\sigma}\tilde{\rho}$), then we define $\varphi_A \colon A \to A$ such that $\tilde{\sigma}'(e_a) = e_{\varphi_A(a)}$. Finally, we define $R \subset A \times B \times A \times B$ as the set of all $(a, b, a', b')$ such that $e_a, f_b, e'_{a'}$ and $f'_{b'}$ are the four edges of some square in the $(d_1,d_2)$-complex. One easily checks that $(A,B,\varphi_A,\varphi_B,R)$ is a $(d_1,d_2)$-datum.

Recalling that $(d_1,d_2)$-groups and good $\C_2 \times \C_2$-actions on $(d_1,d_2)$-complexes are in correspondence, we now have a bijective correspondence between equivalence classes of $(d_1,d_2)$-groups and equivalence classes of $(d_1,d_2)$-data. A presentation for the $(d_1,d_2)$-group $\Gamma$ corresponding to some $(d_1,d_2)$-datum $(A,B,\varphi_A,\varphi_B,R)$ is given by
$$\Gamma = \langle A \cup B \mid xx^{-1} = 1\ \forall x \in A \cup B,\ aba'b' = 1\ \forall (a,b,a',b') \in R\rangle.$$
Indeed, the group $\Gamma$ acts simply transitively on the vertices of $T_1 \times T_2$, so the $1$-skeleton of $T_1 \times T_2$ is a Cayley graph for $\Gamma$. Moreover, the edges of this Cayley graph are labelled by $A \cup B$ so that the four edges of each square in the graph correspond to elements of $R$.

\begin{definition}\label{definition:tfdatum}
A $(d_1,d_2)$-datum $(A,B,\varphi_A,\varphi_B,R)$ is \textbf{torsion-free} if $\varphi_A$ and $\varphi_B$ have no fixed points and the action of $\langle\sigma, \rho\rangle$ on $R$ is free.
\end{definition}

The above correspondence between equivalence classes of $(d_1,d_2)$-groups and equivalence classes of $(d_1,d_2)$-data then restricts to a correspondence between equivalence classes of torsion-free $(d_1,d_2)$-groups and equivalence classes of torsion-free $(d_1,d_2)$-data. Note that, since $\varphi_A$ and $\varphi_B$ cannot have any fixed point in Definition~\ref{definition:tfdatum}, there does not exist any torsion-free $(d_1,d_2)$-group when $d_1$ or $d_2$ is odd.

In the following, we will always consider $(d_1,d_2)$-data up to equivalence. We write $A = \{a_1,\ldots,a_{d_1}\}$ and $B = \{b_1,\ldots,b_{d_2}\}$. Also, we denote by $\tau_1$ (resp.\ $\tau_2$) the number of fixed points of $\varphi_A$ (resp.\ $\varphi_B$) in some $(d_1,d_2)$-data, and we assume (without losing any generality) that $\varphi_A(a_i) = a_{d_1+1-i}$ for each $i \in \{1, \ldots, \frac{d_1-\tau_1}{2}\}$ and $\varphi_A(a_i) = a_i$ for each $i \in \{\frac{d_1-\tau_1}{2}+1, \ldots, \frac{d_1+\tau_1}{2}\}$ (resp.\ $\varphi_B(b_i) = b_{d_2+1-i}$ for each $i \in \{1, \ldots, \frac{d_2-\tau_2}{2}\}$ and $\varphi_B(b_i) = b_i$ for each $i \in \{\frac{d_2-\tau_2}{2}+1, \ldots, \frac{d_2+\tau_2}{2}\}$). We will sometimes write $A_i$ instead of $a_i$ when $\varphi_A(a_i) = a_i$ (i.e.\ $a_i = a_i^{-1}$), and $B_i$ instead of $b_i$ when $\varphi_B(b_i) = b_i$ (i.e.\ $b_i = b_i^{-1}$). Note that, with these assumptions, $\tau_1$, $\tau_2$ and $R$ fully determine the $(d_1,d_2)$-datum.

\subsection{Geometric squares in a \texorpdfstring{$(d_1,d_2)$-datum}{(d_1,d_2)-datum}}
\label{subsection:geometricsquares}

Consider some $(d_1,d_2)$-datum with associated parameters $\tau_1$, $\tau_2$ and $R$. Given $(a,b,a',b') \in R$, we write $[a,b,a',b']$ for the set $$\{(a,b,a',b'),(a',b',a,b),(a^{\prime -1}, b^{-1}, a^{-1}, b^{\prime -1}), (a^{-1}, b^{\prime -1}, a^{\prime -1}, b^{-1})\}.$$
This is a subset of $R$, we call it a \textbf{geometric square}. If the $(d_1,d_2)$-datum is torsion-free, then the action of $\C_2 \times \C_2$ on $R$ is free so each geometric square contains exactly four elements. In this particular case, we have exactly $\frac{d_1d_2}{4}$ geometric squares (remember that $d_1$ and $d_2$ must be even). When allowing torsion, we can actually have up to $d_1d_2$ geometric squares (in the particular case where $\tau_1 = d_1$ and $\tau_2 = d_2$).

An easy way to define some particular $(d_1,d_2)$-datum is then to draw its geometric squares. We explain how the drawing works by giving an example. Consider the $(3,4)$-datum defined by the geodesic squares $[a_1, b_1, a_1, b_2^{-1}]$, $[a_1, b_2, a_1, b_2]$, $[a_1, b_1^{-1}, A_2, b_1^{-1}]$ and $[A_2, b_2, A_2, b_2^{-1}]$. Note that the values $\tau_1 = 1$ and $\tau_2 = 0$ can be understood from the squares. Then we draw this $(3,4)$-datum as in Figure~\ref{picture:example}. Each square can be read counterclockwise, starting from the bottom edge. The white symbols thus represent elements of $A$, while the black ones represent elements of $B$. A single arrow with the forward orientation means $a_1$ (or $b_1$), a double arrow with the forward orientation means $a_2$ (or $b_2$), etc. A single arrow with the backward orientation means $a_1^{-1}$ (or $b_1^{-1}$), a double arrow with the backward orientation means $a_2^{-1}$ (or $b_2^{-1}$), etc. Finally, a lozenge (that does not have any orientation) means $A_1$ (or $B_1$), two lozenges mean $A_2$ (or $B_2$), etc. Note that we can actually read each square from the bottom or from the top edge, and clockwise or counterclockwise. These four ways of reading give the four (possibly equal) elements of $R$ defined by the geometric square. In our example, the first and third geometric squares give four distinct elements of $R$, while the second and fourth geometric squares give two distinct elements of $R$. Note that $|R|\ = 4+2+4+2 = 12 = 3 \cdot 4$ as is needed for a $(3,4)$-datum.

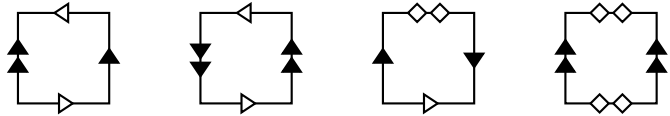
\begin{figure}[b]
\centering
\begin{pspicture*}(-0.2,-0.2)(8.6,1.4)
\fontsize{10pt}{10pt}\selectfont
\psset{unit=1.2cm}

\pspolygon(0,0)(1,0)(1,1)(0,1)

\pspolygon[fillstyle=solid,fillcolor=white](0.45,0.1)(0.45,-0.1)(0.6,0)
\pspolygon[fillstyle=solid,fillcolor=white](0.55,1.1)(0.55,0.9)(0.4,1)
\pspolygon[fillstyle=solid,fillcolor=black](0.9,0.45)(1.1,0.45)(1,0.6)
\pspolygon[fillstyle=solid,fillcolor=black](-0.1,0.55)(0.1,0.55)(0,0.7)
\pspolygon[fillstyle=solid,fillcolor=black](-0.1,0.35)(0.1,0.35)(0,0.5)

\pspolygon(2,0)(3,0)(3,1)(2,1)

\pspolygon[fillstyle=solid,fillcolor=white](2.45,0.1)(2.45,-0.1)(2.6,0)
\pspolygon[fillstyle=solid,fillcolor=white](2.55,1.1)(2.55,0.9)(2.4,1)
\pspolygon[fillstyle=solid,fillcolor=black](2.9,0.55)(3.1,0.55)(3,0.7)
\pspolygon[fillstyle=solid,fillcolor=black](2.9,0.35)(3.1,0.35)(3,0.5)
\pspolygon[fillstyle=solid,fillcolor=black](1.9,0.45)(2.1,0.45)(2,0.3)
\pspolygon[fillstyle=solid,fillcolor=black](1.9,0.65)(2.1,0.65)(2,0.5)

\pspolygon(4,0)(5,0)(5,1)(4,1)

\pspolygon[fillstyle=solid,fillcolor=white](4.45,0.1)(4.45,-0.1)(4.6,0)
\pspolygon[fillstyle=solid,fillcolor=white](4.525,1)(4.625,1.1)(4.725,1)(4.625,0.9)
\pspolygon[fillstyle=solid,fillcolor=white](4.475,1)(4.375,1.1)(4.275,1)(4.375,0.9)
\pspolygon[fillstyle=solid,fillcolor=black](4.9,0.55)(5.1,0.55)(5,0.4)
\pspolygon[fillstyle=solid,fillcolor=black](3.9,0.45)(4.1,0.45)(4,0.6)

\pspolygon(6,0)(7,0)(7,1)(6,1)

\pspolygon[fillstyle=solid,fillcolor=white](6.525,0)(6.625,0.1)(6.725,0)(6.625,-0.1)
\pspolygon[fillstyle=solid,fillcolor=white](6.475,0)(6.375,0.1)(6.275,0)(6.375,-0.1)
\pspolygon[fillstyle=solid,fillcolor=white](6.525,1)(6.625,1.1)(6.725,1)(6.625,0.9)
\pspolygon[fillstyle=solid,fillcolor=white](6.475,1)(6.375,1.1)(6.275,1)(6.375,0.9)
\pspolygon[fillstyle=solid,fillcolor=black](5.9,0.55)(6.1,0.55)(6,0.7)
\pspolygon[fillstyle=solid,fillcolor=black](5.9,0.35)(6.1,0.35)(6,0.5)
\pspolygon[fillstyle=solid,fillcolor=black](6.9,0.55)(7.1,0.55)(7,0.7)
\pspolygon[fillstyle=solid,fillcolor=black](6.9,0.35)(7.1,0.35)(7,0.5)
\end{pspicture*}
\caption{An example of a $(3,4)$-datum.}\label{picture:example}
\end{figure}

One can quickly check if some set of geometric squares satisfies the hypotheses for representing a $(d_1,d_2)$-datum. For a torsion-free $(d_1,d_2)$-datum, it suffices to verify that each possible \textit{corner} appears exactly once in the geometric squares. By a \textbf{corner}, we mean a vertex of a square together with the two labelled edges adjacent to it, where the orientation of the arrows matters. There are $d_1d_2$ possible corners, and they must all appear once (in one of the $\frac{d_1d_2}{4}$ geometric squares). Now the condition is almost the same for a general $(d_1,d_2)$-datum : we just need to take into account that some geometric squares can have non-trivial automorphisms. So the condition is now that each of the $d_1d_2$ possible corners must appear exactly once, up to these square automorphisms. Note that a geometric square as $[a_1, b_1, A_2, b_1^{-1}]$ cannot appear in a $(d_1,d_2)$-datum, since some corner is appearing twice and the square has no automorphism.

%
%

\section{Projections on each factor}
\label{section:projections}

Consider a $(d_1,d_2)$-group $\Gamma \leq \Aut(T_1) \times \Aut(T_2)$, associated to some set of geometric squares. Define $H_1 = \overline{\proj_1(\Gamma)} \leq \Aut(T_1)$ and $H_2 = \overline{\proj_2(\Gamma)} \leq \Aut(T_2)$ as the closures of the projections of $\Gamma$ on $\Aut(T_1)$ and $\Aut(T_2)$. The goal of this section is to analyze which pairs of groups $(H_1, H_2)$ can be obtained from a $(d_1,d_2)$-group $\Gamma$ which is irreducible.

\subsection{Action of \texorpdfstring{$\Gamma$}{Gamma} on \texorpdfstring{$T_1 \times T_2$}{T_1 x T_2}}
\label{subsection:action}

Let us see $a_1, \ldots, a_{d_1}, b_1, \ldots, b_{d_2}$ as the generators of $\Gamma$, as explained in \S\ref{subsection:complexesanddata}. In $T_1 \times T_2$, there is a vertex $(v_1,v_2) \in V(T_1) \times V(T_2)$ such that all generators $b_1, \ldots, b_{d_2}$ fix $v_1$ and all generators $a_1, \ldots, a_{d_1}$ fix $v_2$. Also, $b_1, \ldots, b_{d_2}$ send $v_2$ to its $d_2$ neighboring vertices and $a_1, \ldots, a_{d_1}$ send $v_1$ to its $d_1$ neighboring vertices. This means that $\langle b_1, \ldots, b_{d_2}\rangle = \Gamma(v_1)$ and that $\langle a_1, \ldots, a_{d_1}\rangle = \Gamma(v_2)$, where $\Gamma(v)$ denotes the fixator of $v$ in $\Gamma$. We now explain how the action of a particular element of $\Gamma$ on $T_1 \times T_2$ can be computed from the geodesic squares. These explanations are similar to those in~\cite[Section~1.4]{Rattaggi}, but we recall them here so as to remain self-contained.

The group $\Gamma$ acts simply transitively on the vertices of $T_1 \times T_2$, and the generators $a_1, \ldots, a_{d_1}, b_1, \ldots, b_{d_2}$ send $(v_1,v_2) \in V(T_1 \times T_2)$ to its $d_1+d_2$ neighbors in $T_1 \times T_2$. This means that the $1$-skeleton of the square-complex $T_1 \times T_2$ can be seen as the Cayley graph of $\Gamma$ with respect to the set of generators $\{a_1, \ldots, a_{d_1}, b_1, \ldots, b_{d_2}\}$. In particular, this gives a labelling of the edges of $T_1 \times T_2$ by the generators. As for the geometric squares, we can thus associate a (white or black) symbol to each edge of $T_1 \times T_2$. This labelling is actually such that each square in $T_1 \times T_2$ corresponds to one of the geometric squares associated to $\Gamma$. We also have natural embeddings $T_1 \hookrightarrow T_1 \times T_2 \colon w \in V(T_1) \mapsto (w,v_2)$ and $T_2 \hookrightarrow T_1 \times T_2 \colon w \in V(T_2) \mapsto (v_1,w)$ from which we get a labelling of $T_1$ with white symbols and a labelling of $T_2$ with black symbols. This is actually equivalent to seeing $T_1$ as the Cayley graph of $\langle a_1, \ldots, a_{d_1}\rangle$ and $T_2$ as the Cayley graph of $\langle b_1, \ldots, b_{d_2}\rangle$.

The image of $(v_1,v_2)$ by some element $g \in \Gamma$ is easy to get: it suffices to write $g$ as a product of the generators and then follow (from the vertex $(v_1,v_2)$) the sequence of symbols in $T_1 \times T_2$ corresponding to these generators. The vertex at the end of the path will be $g(v_1,v_2)$. Note however that this only works because $(v_1,v_2)$ is the vertex associated to the identity of $\Gamma$ in the Cayley graph $T_1 \times T_2$. Given some $g \in \Gamma$ and another vertex $(w_1,w_2) \in V(T_1 \times T_2)$, the way to obtain $g(w_1,w_2)$ is to first localize $g(v_1,v_2)$ with the above procedure, and then to recall that $\Gamma$ preserves the symbols in $T_1 \times T_2$. Hence, it suffices to look at the symbols on some path from $(v_1,v_2)$ to $(w_1,w_2)$ and to follow the same symbols from $g(v_1,v_2)$ so as to arrive at $g(w_1,w_2)$.

In particular, the action of an element $g \in \langle a_1, \ldots, a_{d_1}\rangle = \Gamma(v_2)$ on $T_2$ can be obtained by doing the following. In order to compute $g(w)$ for some $w \in V(T_2)$, we draw a rectangle whose bottom side is labelled by the sequence of white symbols corresponding to $g$ (from left to right) and whose right side is labelled by the sequence of black symbols on the path from $v_2$ to $w$ in $T_2$ (from bottom to top), see Figure~\ref{picture:computing}. Then we fill in the rectangle with the appropriate geometric squares (starting from the bottom-right corner). The rectangle that we obtain corresponds to a subcomplex of $T_1 \times T_2$: the bottom-left corner is $(v_1,v_2)$, the bottom-right corner is $g(v_1,v_2)$, and the top-right corner is $g(v_1,w)$. Now $T_2$ corresponds to $v_1 \times T_2$ in $T_1 \times T_2$, so we can read $g(w)$ by looking at the left side of our rectangle. Indeed, the symbols on the path from $v_2$ to $g(w)$ in $T_2$ are exactly those on the left side of the rectangle (from bottom to top). Another way to explain why this idea works is to write $h$ for the element of $\langle b_1, \ldots, b_{d_2}\rangle$ corresponding to the right side of the rectangle (i.e.\ $w = h(v_2)$ in $T_2$), $h'$ for the element of $\langle b_1, \ldots, b_{d_2}\rangle$ corresponding to the left side, and $g'$ for the element of $\langle a_1, \ldots, a_{d_1}\rangle$ corresponding to the top side. From $(v_1,v_2)$, following $g$ and then $h$ leads to the same vertex as following $h'$ and then $g'$, so $gh = h'g'$. In particular we have $gh(v_2) = h'g'(v_2)$, which reduces to $g(w) = h'(v_2)$ as wanted. Similarly, the action of an element $g \in \langle b_1, \ldots, b_{d_2}\rangle = \Gamma(v_1)$ on $T_1$ can be obtained with the same method.

This method is illustrated on Figure~\ref{picture:computing}, with the $(3,4)$-group $\Gamma$ defined by the squares of Figure~\ref{picture:example}. On this figure we computed the image of vertex $b_1b_2(v_2)$ by $a_1A_2a_1^{-1}A_2 \in \Gamma(v_2)$. The bottom side of the rectangle is indeed labelled by the symbols of $a_1A_2a_1^{-1}A_2$, and the right side by the symbols of $b_1b_2$. After filling in the rectangle with the squares of Figure~\ref{picture:example} (the numbers (1), (2), (3), (4) indicate which squares we used), it appears on the left side that the image of $b_1b_2(v_2)$ is $b_1^{-2}(v_2)$. Note that this rectangle also shows, for instance, that the image of $A_2 a_1^{-3}(v_1)$ by $b_1^{-2} \in \Gamma(v_1)$ is $a_1A_2a_1^{-1}A_2(v_1)$.

\begin{figure}
\centering
\begin{pspicture*}(-0.2,-0.2)(5,2.6)
\fontsize{10pt}{10pt}\selectfont
\psset{unit=1.2cm}


\psline(0,0)(4,0)
\psline(0,1)(4,1)
\psline(0,2)(4,2)
\psline(0,0)(0,2)
\psline(1,0)(1,2)
\psline(2,0)(2,2)
\psline(3,0)(3,2)
\psline(4,0)(4,2)

\rput(0.5,0.5){(1)}
\rput(1.5,0.5){(4)}
\rput(2.5,0.5){(1)}
\rput(3.5,0.5){(3)}

\rput(0.5,1.5){(3)}
\rput(1.5,1.5){(3)}
\rput(2.5,1.5){(1)}
\rput(3.5,1.5){(2)}


\pspolygon[fillstyle=solid,fillcolor=white](0.45,0.1)(0.45,-0.1)(0.6,0)
\pspolygon[fillstyle=solid,fillcolor=white](1.525,0)(1.625,0.1)(1.725,0)(1.625,-0.1)
\pspolygon[fillstyle=solid,fillcolor=white](1.475,0)(1.375,0.1)(1.275,0)(1.375,-0.1)
\pspolygon[fillstyle=solid,fillcolor=white](2.55,0.1)(2.55,-0.1)(2.4,0)
\pspolygon[fillstyle=solid,fillcolor=white](3.525,0)(3.625,0.1)(3.725,0)(3.625,-0.1)
\pspolygon[fillstyle=solid,fillcolor=white](3.475,0)(3.375,0.1)(3.275,0)(3.375,-0.1)


\pspolygon[fillstyle=solid,fillcolor=white](0.55,1.1)(0.55,0.9)(0.4,1)
\pspolygon[fillstyle=solid,fillcolor=white](1.525,1)(1.625,1.1)(1.725,1)(1.625,0.9)
\pspolygon[fillstyle=solid,fillcolor=white](1.475,1)(1.375,1.1)(1.275,1)(1.375,0.9)
\pspolygon[fillstyle=solid,fillcolor=white](2.45,1.1)(2.45,0.9)(2.6,1)
\pspolygon[fillstyle=solid,fillcolor=white](3.45,1.1)(3.45,0.9)(3.6,1)


\pspolygon[fillstyle=solid,fillcolor=white](0.525,2)(0.625,2.1)(0.725,2)(0.625,1.9)
\pspolygon[fillstyle=solid,fillcolor=white](0.475,2)(0.375,2.1)(0.275,2)(0.375,1.9)
\pspolygon[fillstyle=solid,fillcolor=white](1.55,2.1)(1.55,1.9)(1.4,2)
\pspolygon[fillstyle=solid,fillcolor=white](2.55,2.1)(2.55,1.9)(2.4,2)
\pspolygon[fillstyle=solid,fillcolor=white](3.55,2.1)(3.55,1.9)(3.4,2)


\pspolygon[fillstyle=solid,fillcolor=black](-0.1,0.55)(0.1,0.55)(0,0.4)
\pspolygon[fillstyle=solid,fillcolor=black](-0.1,1.55)(0.1,1.55)(0,1.4)


\pspolygon[fillstyle=solid,fillcolor=black](0.9,0.45)(1.1,0.45)(1,0.3)
\pspolygon[fillstyle=solid,fillcolor=black](0.9,0.65)(1.1,0.65)(1,0.5)
\pspolygon[fillstyle=solid,fillcolor=black](0.9,1.45)(1.1,1.45)(1,1.6)


\pspolygon[fillstyle=solid,fillcolor=black](1.9,0.45)(2.1,0.45)(2,0.3)
\pspolygon[fillstyle=solid,fillcolor=black](1.9,0.65)(2.1,0.65)(2,0.5)
\pspolygon[fillstyle=solid,fillcolor=black](1.9,1.55)(2.1,1.55)(2,1.4)


\pspolygon[fillstyle=solid,fillcolor=black](2.9,0.55)(3.1,0.55)(3,0.4)
\pspolygon[fillstyle=solid,fillcolor=black](2.9,1.45)(3.1,1.45)(3,1.3)
\pspolygon[fillstyle=solid,fillcolor=black](2.9,1.65)(3.1,1.65)(3,1.5)


\pspolygon[fillstyle=solid,fillcolor=black](3.9,0.45)(4.1,0.45)(4,0.6)
\pspolygon[fillstyle=solid,fillcolor=black](3.9,1.55)(4.1,1.55)(4,1.7)
\pspolygon[fillstyle=solid,fillcolor=black](3.9,1.35)(4.1,1.35)(4,1.5)
\end{pspicture*}
\caption{Example of a computation.}\label{picture:computing}
\end{figure}
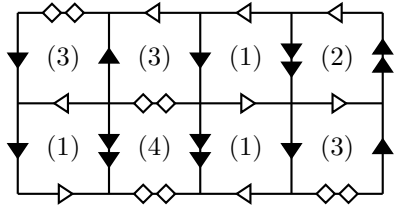

\subsection{Irreducibility}
\label{subsection:irred}

Recall that a $(d_1,d_2)$-group $\Gamma$ is said to be reducible if it is commensurable to a product $\Gamma_1 \times \Gamma_2$ of lattices $\Gamma_t \leq \Aut(T_t)$. By \cite[Proposition~1.2]{Burger2}, $\Gamma$ is reducible if and only if $\proj_1(\Gamma)$ or $\proj_2(\Gamma)$ is discrete. If $\Gamma$ is reducible, both projections are actually discrete. In some sense, the irreducible $(d_1,d_2)$-groups are the interesting ones. There is no known general algorithm deciding if a $(d_1,d_2)$-group is irreducible, but such an algorithm exists under suitable assumption on the local action.

We saw in the previous section that, for each $n \in \N$, the action of all $b_i$ (resp.\ $a_i$) on the ball $B(v_1,n)$ (resp.\ $B(v_2,n)$) of $T_1$ (resp.\ $T_2$) can be computed from the geometric squares. The group $\proj_t(\Gamma) \leq \Aut(T_t)$ is vertex-transitive, so it is discrete if and only if $\Fix_{\proj_t(\Gamma)}(B(v_t,n)) = \Fix_{\proj_t(\Gamma)}(B(v_t,n+1))$ for some $n \in \N$. A way to show that some $\Gamma$ is reducible thus consists in finding some $t \in \{1,2\}$ and some $n \in \N$ for which the latter equality is true.
 
Proving that a $(d_1,d_2)$-group $\Gamma$ is irreducible is not easy in general, but there is a case where a good criterion exists: when $d_t \geq 6$ and $\underline{H_t}(v_t) \geq \Alt(d_t)$ for some $t \in \{1,2\}$. Indeed, in this case it follows from \cite[Propositions~3.3.1 and~3.3.2]{Burger} that $H_t$ is non-discrete if and only if the image of $H_t(v_t)$ in $\Aut(B(v_t,2))$ has order 
$\geq \frac{d_j!}{2} \left(\frac{(d_j-1)!}{2}\right)^{d_j}$. (The result is actually more subtle but we only need this lower bound here.) Moreover, once we know that $H_t$ is non-discrete and that $\underline{H_t}(v_t) \geq \Alt(d_t)$ with $d_t \geq 6$, we actually get from Theorem~\ref{theorem:Raduclassification} that $H_t$ is a member of $\mathcal{G}'_{(i)}$ for some legal coloring $i$ of $T_t$ (as defined in Definition~\ref{definition:Radu}). In the next section, we investigate the question of determining from the geometric squares which group from $\mathcal{G}'_{(i)}$ it actually is.

\subsection{Recognizing a group in \texorpdfstring{$\mathcal{G}'_{(i)}$}{G'_(i)}}
\label{subsection:recognizing}

Let $T$ be the $d$-regular tree for some $d \geq 3$, let $i$ be a legal coloring of $T$ and let $H$ be an element of $\mathcal{G}'_{(i)}$ different from $\Aut(T)$. In particular $H$ is transitive on $V(T)$. We fix some vertex $v \in V(T)$ and denote by $K$ the smallest non-negative integer such that the homomorphism from $H(v)$ to $\Aut(B(v,K+1))$ is not surjective. Such an integer exists because $H \neq \Aut(T)$. We also define $\rho = \frac{|\Aut(B(v,K+1))|}{|\tilde{H}^K(v)|}$, where $\tilde{H}^K(v)$ is the image of $H(v)$ in $\Aut(B(v,K+1))$. Let us finally define the homomorphism $s \colon H(v) \to (\C_2)^{K+1}$ by $s(h) = (s_0(h), \ldots, s_K(h))$ where $s_k(h) = \Sgn_{(i)}(h, S(v,k)) = \prod_{w \in S(v,k)}\sgn(\sigma_{(i)}(h,w))$, as in Definition~\ref{definition:Radu}. As Lemma~\ref{lemma:nocoloring} below shows, the value of $s_k(h)$ does not depend on the coloring $i$.

\begin{lemma}\label{lemma:nocoloring}
Let $v$ be a vertex of $T$ and let $k \in \N$. Consider some $h \in \Aut(T)(v)$. For each $w \in S(v,k)$, fix a bijection $\iota_w \colon S(w,1) \to \{1, \ldots, d\}$. Then the value of
$$\prod_{w \in S(v,k)}\sgn(\iota_{h(w)} \circ h \circ \iota_w^{-1})$$
does not depend on the choices for the bijections $\iota_w$.
\end{lemma}

\begin{proof}
Fix some other bijections $\iota'_w \colon S(w,1) \to \{1, \ldots, d\}$ with $w \in S(v,k)$. Then we have
\begin{align*}
&\prod_{w \in S(v,k)}\sgn(\iota'_{h(w)} \circ h \circ \iota_w^{\prime-1})\\
=\ & \prod_{w \in S(v,k)}\sgn(\iota'_{h(w)} \circ \iota_{h(w)}^{-1} \circ \iota_{h(w)} \circ h \circ \iota_w^{-1} \circ \iota_w \circ \iota_w^{\prime-1})\\
=\ & \prod_{w \in S(v,k)} \sgn(\iota'_{h(w)} \circ \iota_{h(w)}^{-1}) \cdot \sgn(\iota_{h(w)} \circ h \circ \iota_w^{-1}) \cdot \sgn(\iota_w \circ \iota_w^{\prime-1})\\
=\ & \prod_{w \in S(v,k)} \sgn(\iota_{h(w)} \circ \iota_{h(w)}^{\prime-1}) \cdot \sgn(\iota_{h(w)} \circ h \circ \iota_w^{-1}) \cdot \sgn(\iota_w \circ \iota_w^{\prime-1})\\
=\ & \prod_{w \in S(v,k)} \sgn(\iota_{h(w)} \circ h \circ \iota_w^{-1}).
\end{align*}
The last equality holds because, for each $w \in S(v,k)$, the term $\sgn(\iota_w \circ \iota_w^{\prime-1})$ appears twice in the product.
\end{proof}

We now prove in the following results that, when $d$ is even, computing $K$, $\rho$ and $s(H(v)) \leq (\C_2)^K$ (almost) suffices to recognize which group of $\mathcal{G}'_{(i)}$ we actually have. These invariants do not depend on the coloring $i$ of $T$, which means that they can be computed without knowing for which coloring $i$ the group $H$ is contained in $\mathcal{G}'_{(i)}$.

\begin{lemma}\label{lemma:valueS}
Let $H = G_{(i)}(X,X)$ for some $X \subset_f \N$. Then $K = \max X$, $\rho = 2$ and $s(H(v)) = \{(s_0, \ldots, s_K) \in (\C_2)^{K+1} \mid \prod_{r \in X} s_r = 1\}$.
\end{lemma}

\begin{proof}
This follows immediately from the definition of $G_{(i)}(X,X)$.
\end{proof}

For each $X \subset_f \N$, we define $\alpha(X)$ as the subset of $\N$ such that $S_{\alpha(X)}(v)$ is the set of vertices of $T$ that appear in an odd number of sets $S_X(w_1), \ldots, S_X(w_d)$, where $w_1, \ldots, w_d$ are the $d$ neighbors of $v$ in $T$. In other words, $S_{\alpha(X)}(v)$ is the support of $\mathbf{1}_{S_X(w_1)} + \cdots + \mathbf{1}_{S_X(w_d)} \bmod 2$ where $\mathbf{1}$ denotes the characteristic function. It is clear from its definition that this set of vertices indeed takes the form $S_{\alpha(X)}(v)$ for some $\alpha(X)$. In the next lemma we give an explicit expression for $\alpha(X)$, depending on the parity of $d$.

\begin{lemma}\label{lemma:explicitalpha}
Let $X \subset_f \N$. We have the following expressions for $\alpha(X)$, where $\triangle$ denotes the symmetric difference.
$$\alpha(X) = \left\{\begin{array}{ll}
\{x+1 \mid x \in X\} \triangle \{x-1 \mid x \in X, x \geq 2\} & \text{if $d$ is even,}\\
\{x+1 \mid x \in X\} \cup (\{x-1 \mid x \in X\} \cap \{0\}) & \text{if $d$ is odd.}
\end{array}\right.$$

\end{lemma}

\begin{proof}
For each $j \in \{1, \ldots, d\}$, we write $S_X(w_j) = S_X^+(w_j) \sqcup S_X^-(w_j)$, where $S_X^+(w_j)$ is the set of vertices of $S_X(w_j)$ that are further from $v$ than from $w_j$, and $S_X^-(w_j) = S_X(w_j) \setminus S_X^+(w_j)$. Then, all the sets $S_X^+(w_j)$ with $j \in \{1, \ldots, d\}$ are disjoint and their union is $S_{\{x+1 \mid x \in X\}}(v)$. Now if we look at the sets $S_X^-(w_j)$, they only contain vertices that are at distance $x-1$ from $v$ for some $x \in X$ ($x \geq 1$). More precisely, if $x \in X$ and $x \geq 2$ then each vertex at distance $x-1$ from $v$ is contained in exactly $d-1$ of the sets $S_X^-(w_j)$, $j \in \{1, \ldots, d\}$. Also, if $x = 1 \in X$, then $v$ is contained in all $d$ sets $S_X^-(w_j)$, $j \in \{1, \ldots, d\}$. These affirmations directly lead to the expressions given in the statement.
\end{proof}

The next lemma then follows almost immediately.

\begin{lemma}\label{lemma:alpha}
We have $\alpha(X) \subset_f \N$ for each $X \subset_f \N$, and the map $\alpha \colon \{X \subset_f \N\} \to \{X \subset_f \N\}$ is injective. Moreover, we have
$$\alpha(\{X \subset_f \N\}) = \left\{\begin{array}{ll}
\{X \subset_f \N \mid 0 \not \in X\} & \text{if $d$ is even,}\\
\{X \subset_f \N \mid 0 \in X \Leftrightarrow 2 \in X\} & \text{if $d$ is odd.}
\end{array}\right.$$
\end{lemma}

\begin{proof}
From Lemma~\ref{lemma:explicitalpha} we know that $\alpha(X)$ is finite and non-empty (because $\max X + 1 \in \alpha(X)$), i.e.\ $\alpha(X) \subset_f \N$. Now remark from the definition of $\alpha$ that $\alpha(X \triangle X') = \alpha(X) \triangle \alpha(X')$ for each $X, X' \subset_f \N$, where we define $\alpha(\varnothing) = \varnothing$. Therefore, if $\alpha(X) = \alpha(X')$, then $\alpha(X \triangle X') = \varnothing$ and hence $X \triangle X' = \varnothing$, i.e.\ $X = X'$.

The expressions for $\alpha(\{X \subset_f \N\})$ can be found directly by examining Lemma~\ref{lemma:explicitalpha}.
\end{proof}

Lemma~\ref{lemma:stars} is the reason why $\alpha$ was defined above.

\begin{lemma}\label{lemma:stars}
Suppose that $d$ is even and let $H$ be one of $G_{(i)}(X^*, X^*)$, $G_{(i)}(X, X)^*$ and $G_{(i)}'(X, X)^*$ for some $X \subset_f \N$. Then $H$ is contained in $G_{(i)}(\alpha(X),\alpha(X))$.
\end{lemma}

\begin{proof}
For any $h \in H$ and $v \in V(T)$, we know that all $\Sgn_{(i)}(h, S(w_j,X))$ with $j \in \{1, \ldots, d\}$ are equal, where $w_1, \ldots, w_d$ are the $d$ neighbors of $v$ in $T$. Since $d$ is even, the product of these $d$ signatures is $1$. As this product is also equal to $\Sgn_{(i)}(h, S(v,\alpha(X)))$, we deduce that $h \in G_{(i)}(\alpha(X),\alpha(X))$.
\end{proof}

From the previous lemma we can now compute $s(H(v))$ for other groups $H$ in $\mathcal{G}'_{(i)}$ (when $d$ is even).

\begin{lemma}\label{lemma:valueS*}
Suppose that $d$ is even and let $H$ be one of $G_{(i)}(X^*, X^*)$, $G_{(i)}(X, X)^*$ and $G_{(i)}'(X, X)^*$ for some $X \subset_f \N$. Then $K = \max X + 1$ and $s(H(v)) = \{(s_0, \ldots, s_K) \in (\C_2)^{K+1} \mid \prod_{r \in \alpha(X)} s_r = 1\}$. Moreover, we have $\rho = 2^{d-1}$ if $H = G_{(i)}(X^*, X^*)$ and $\rho = 2^d$ if $H = G_{(i)}(X, X)^*$ or $G_{(i)}'(X, X)^*$.
\end{lemma}

\begin{proof}
The values of $K$ and $\rho$ can be directly deduced from the definitions of the groups. By definition of $K$, the homomorphism from $H(v)$ to $\Aut(B(v,K))$ is surjective. Hence, for each $(s_0, \ldots, s_{K-1}) \in (\C_2)^K$, there exists $s_K \in \C_2$ such that $(s_0, \ldots, s_K) \in s(H(v))$. If $H' = G_{(i)}(\alpha(X),\alpha(X))$, then Lemma~\ref{lemma:stars} states that $H \subseteq H'$, so $s(H(v)) \subseteq s'(H'(v))$ (where $s' \colon H'(v) \to (\C_2)^{K+1}$ is the map associated to $H'$). Note that $H$ and $H'$ share the same $K$ because $\max(\alpha(X)) = \max X + 1$. Now $s'(H'(v)) = \{(s_0, \ldots, s_K) \in (\C_2)^{K+1} \mid \prod_{r \in \alpha(X)} s_r = 1\}$ by Lemma~\ref{lemma:valueS}, and in particular for each $(s_0, \ldots, s_{K-1}) \in (\C_2)^K$ there is a unique $s_K \in \C_2$ such that $(s_0, \ldots, s_K) \in s'(H'(v))$. From all this information it follows that $s(H(v)) = s'(H'(v))$.
\end{proof}

When $d$ is even, we see from Lemmas~\ref{lemma:valueS} and~\ref{lemma:valueS*} that the groups in $\mathcal{G}'_{(i)}$ can be differentiated by computing $K$, $\rho$ and the image of the map $s$, with one exception: $G_{(i)}(X, X)^*$ and $G_{(i)}'(X, X)^*$ have the same invariants (for a fixed $X \subset_f \N$). This is due to the fact that their type-preserving subgroups are both equal to $G_{(i)}^+(X, X)^*$ and all invariants are computed from vertex stabilizers. We will however be able to differentiate these two groups later, see Proposition~\ref{proposition:primeornot}.

For $d$ odd, the task would be harder. For instance, $G_{(i)}(\{0,1\}^*, \{0,1\}^*)$ and $G_{(i)}(\{1\}^*, \{1\}^*)$ have the same invariants when $d$ is odd ($K = 2$, $\rho = 2^{d-1}$ and $s$ is surjective). This occurs because Lemma~\ref{lemma:stars} is no longer true in that case. We will not deal with the odd case.

\subsection{Labelled graphs associated to a \texorpdfstring{$(d_1,d_2)$-group}{(d_1,d_2)-group}}
\label{subsection:graph}

Let us now come back to a $(d_1,d_2)$-group $\Gamma$ associated to some $(d_1,d_2)$-datum $(A,B,\varphi_A,\varphi_B,R)$. We assume that $d_2 \geq 6$ is even, that $H_2 = \overline{\proj_2(\Gamma)}$ is non-discrete and that $\underline{H_2}(v_2) \geq \Alt(d_2)$. As explained in \S\ref{subsection:irred}, the non-discreteness of $H_2$ can be checked by computing the action of $H_2(v_2)$ on $B(v_2,2)$. We now would like an efficient algorithm to determine which group from $\mathcal{G}'_{(i)}$ is isomorphic to $H_2$. In this section, we give a way to compute $K^{(2)}$ (the $K$ associated to $H_2$) and $s^{(2)}(H_2(v_2))$ by associating a labelled graph $G^{(2)}_\Gamma$ to our $(d_1,d_2)$-group $\Gamma$. In view of the results of \S\ref{subsection:recognizing}, this will reduce to $4$ (or less) the number of groups in $\mathcal{G}'_{(i)}$ that could be isomorphic to $H_2$. Of course, everything we do here for $H_2$ can be translated for $H_1$.

Given $h \in H_2(v_2)$ and $k \in \N$, we write $s^{(2)}_k(h) = \Sgn_{(i)}(h,S(v,k))$ where $i$ is any legal coloring of $T_2$ as in \S\ref{subsection:recognizing}. The invariant $K^{(2)}$ can be characterized as the smallest non-negative integer such that the map
$$s^{(2)} \colon H_2(v_2) \to (\C_2)^{K^{(2)}+1} \colon h \mapsto (s^{(2)}_0(h), \ldots s^{(2)}_{K^{(2)}}(h))$$
is not surjective. This indeed follows from the definition of $K^{(2)}$ and from Lemmas~\ref{lemma:valueS} and~\ref{lemma:valueS*}. An efficient algorithm to compute $s^{(2)}_k(a_j)$ for each $j \in \{1, \ldots, d_1\}$ and each $k \in \N$ would thus be sufficient to determine $K^{(2)}$ as well as $s^{(2)}(H_2(v_2))$. (Note that $a_j$ should actually be read as $\proj_2(a_j)$ here, but we will omit the projection.) This is where the graph $G^{(2)}_\Gamma$, defined hereafter, becomes useful.

Let us define the \textbf{labelled graph} $G^{(2)}_\Gamma$ associated to $\Gamma$. First, the vertex set $V(G^{(2)}_\Gamma)$ is simply defined to be $A$. Then, we put an edge between $a \in A$ and $a' \in A$ if and only if $|R \cap (\{a\} \times B \times \{a^{\prime-1}\} \times B)|$ is odd. Note that $|R \cap (\{a\} \times B \times \{a^{\prime-1}\} \times B)|\ = |R \cap (\{a'\} \times B \times \{a^{-1}\} \times B)|$ because $R$ is invariant under the action of $\C_2 \times \C_2$, so the edge set is well-defined. We obtain an undirected graph that can possibly contain loops (edges from a vertex to itself). Finally, to each vertex $a$ of $G^{(2)}_\Gamma$, we associate a label $\sigma(a) = \pm 1$ whose value depends on the signature of the permutation that the generator $a \in \Gamma$ induces on $E(v_2)$ (the set of $d_2$ edges adjacent to $v_2$ in $T_2$). This labelled graph has a non-trivial automorphism defined by $a \mapsto a^{-1}$ for each $a \in A$. Indeed, we clearly have $\sigma(a) = \sigma(a^{-1})$ for each $a \in A$, and there is an edge between $a$ and $a'$ if and only if there is an edge between $a^{-1}$ and $a^{\prime-1}$. Once again this follows from the fact that $R$ is invariant under the action of $\C_2 \times \C_2$.

The labelled graph $G^{(2)}_\Gamma$ can easily be drawn from the geometric squares that define $\Gamma$. Indeed, the vertex set corresponds to the set of white symbols (with orientation) that the horizontal edges can have. For each vertex $a$, the permutation induced by $a$ on $E(v_2)$ can also be directly computed from the geometric squares. We thus obtain the labels associated to the vertices by taking the signatures. Then, given two vertices $a$ and $a'$, we can determine if there is an edge between $a$ and $a'$ by counting the number of $b \in B$ such that there is a geometric square that can be read as $(a, b, a^{\prime-1}, *)$. We put an edge if and only if there is an odd number of such $b \in B$.

In the graph $G^{(2)}_\Gamma$, a \textbf{non-repeating path} $p$ is a finite sequence of vertices $x_0, x_1, \ldots, x_n$ where $x_{i-1}$ and $x_i$ are connected by an edge for each $i \in \{1, \ldots, n\}$ and $x_{i-1} \neq x_{i+1}$ for each $i \in \{1, \ldots, n-1\}$. In other words, the path cannot use a same edge twice consecutively. Such a path has \textbf{length} $n$, \textbf{origin} $x_0$ and \textbf{destination} $d(p) = x_n$. We write $\mathcal{P}_n(x_0)$ for the set of all non-repeating paths of length $n$ whose origin is $x_0$. 

The next result now shows that the non-repeating paths in $G^{(2)}_\Gamma$ can be helpful in order to compute the values $s^{(2)}_k(a_j)$ defined above. Note that this proposition is also true for $d_2$ odd.

\begin{proposition}\label{proposition:algo}
Let $\Gamma$ be a $(d_1,d_2)$-group with $d_2 \geq 6$ even and suppose that $H_2 = \overline{\proj_2(\Gamma)}$ is non-discrete and satisfies $\underline{H_2}(v_2) \geq \Alt(d_2)$. Let $j \in \{1, \ldots, d_1\}$ and $k \in \N$. Then we have
$$s^{(2)}_k(a_j) = \prod_{p \in \mathcal{P}_k(a_j)} \sigma(d(p)).$$
\end{proposition}

\begin{proof}
Given a sequence of vertices $x_0, \ldots, x_n$ of $G^{(2)}_\Gamma$, we are first interested in rectangles $1 \times n$ made of $n$ geometric squares and whose $n+1$ white symbols from bottom to top exactly correspond to the $n+1$ vertices $x_0, \ldots, x_n$ (with the good orientation). Let us denote by $\Rect(x_0, \ldots, x_n)$ the set of all those rectangles.

\begin{claim*}
$|\Rect(x_0, \ldots, x_n)|$ is odd if and only if $(x_0, \ldots, x_n)$ is a non-repeating path in $G^{(2)}_\Gamma$.
\end{claim*}

\begin{claimproof*}
We prove the claim by induction. For $n = 1$ it follows from the definition of the edge set of $G^{(2)}_\Gamma$. Now let $n \geq 2$ and assume the claim is true for $n-1$. Observe that
$$|\Rect(x_0, \ldots, x_n)|\ = \left\{\begin{array}{ll}
|\Rect(x_0, \ldots, x_{n-1})| \cdot |\Rect(x_{n-1},x_n)| & \text{ if $x_{n-2} \neq x_n$,}\\
|\Rect(x_0, \ldots, x_{n-1})| \cdot (|\Rect(x_{n-1},x_n)|-1) & \text{ 
if $x_{n-2} = x_n$.}
\end{array}\right.$$
Indeed, a rectangle $1 \times n$ in $\Rect(x_0, \ldots, x_n)$ is made of a rectangle $1 \times (n-1)$ in $\Rect(x_0, \ldots, x_{n-1})$ and a square in $\Rect(x_{n-1}, x_n)$. The term $-1$ when $x_{n-2}=x_n$ appears because the square between $x_{n-1}$ and $x_n$ cannot be the same as the one between $x_{n-1}$ and $x_{n-2}$.

Assume now that $(x_0, \ldots, x_n)$ is a non-repeating. Then $(x_0, \ldots, x_{n-1})$ is a non-repeating path, $x_n \neq x_{n-2}$ and $(x_{n-1}, x_n)$ is a non-repeating path. By the induction hypothesis, $|\Rect(x_0, \ldots, x_{n-1})|$ and $|\Rect(x_{n-1},x_n)|$ are odd, hence $|\Rect(x_0, \ldots, x_n)|$ is odd by the above formula.

Conversely, assume that $|\Rect(x_0, \ldots, x_n)|$ is odd. By the formulas above, this already means that $|\Rect(x_0, \ldots, x_{n-1})|$ is odd and thus that $(x_0, \ldots, x_{n-1})$ is a non-repeating path. Then there are two possibilities:
\begin{itemize}
\item If $x_{n-2} \neq x_n$ then we also get that $|\Rect(x_{n-1},x_n)|$ is odd, so $(x_{n-1},x_n)$ is a non-repeating path. All together, these affirmations imply that $(x_0, \ldots, x_n)$ is a non-repeating path.
\item If $x_{n-2} = x_n$ then we get that $|\Rect(x_{n-1},x_n)|$ is even, so $(x_{n-1},x_n)$ is not a non-repeating path, i.e.\ there is no edge between $x_{n-1}$ and $x_n$. This situation is however impossible: we already know that there is an edge between $x_{n-2}$ and $x_{n-1}$, and $x_{n-2} = x_n$. \hfill $\blacksquare$
\end{itemize}
\end{claimproof*}

\medskip

We now prove the proposition. Recall from Lemma~\ref{lemma:nocoloring} that
$$s^{(2)}_k(a_j) = \prod_{w \in S(v_2,k)} \sgn(\iota_{a_j(w)} \circ a_j \circ \iota_w^{-1}),$$
where $\iota_w \colon S(w,1) \to \{1, \ldots, d\}$ is any bijection for each $w \in S(v,k)$. In our context, we have a canonical choice for the bijections $\iota_w$: the edges of $T_2$ are labelled by the black symbols, and the $d_2$ edges adjacent to any vertex carry the $d_2$ different black symbols (considered with their orientation). So the bijections $\iota_w$ can simply be defined by identifying $S(w,1)$ with $E(w)$ and the set $\{1, \ldots, d\}$ with the set of black symbols with orientation.

Take some $w \in S(v_2,k)$ and some $j \in \{1, \ldots, d\}$. Let $h \in \langle b_1, \ldots, b_{d_2}\rangle$ be the element such that $h(v_2) = w$. Let us draw, as in \S\ref{subsection:action}, the rectangle $1 \times k$ made of $k$ geometric squares such that the bottom symbol corresponds to $a_j$, and the $k$ symbols on the right-hand side correspond to $h$. After having filled in the rectangle with geometric squares, we obtain the equation $a_jh = h'a_{j'}$, where $a_{j'}$ is given by the top side and $h'$ by the left side of the rectangle. From the equality $a_jh(v_2) = h'a_{j'}(v_2)$ we obtain that $a_j(w) = h'(v_2)$. Moreover, since $h$ and $h'$ preserve the symbols in $T_2$, we have that $\iota_w^{-1} = h \iota_{v_2}^{-1}$ and $\iota_{a_j(w)}^{-1} = h' \iota_{v_2}^{-1}$. Using these equalities, we get
$$\iota_{a_j(w)} a_j \iota_w^{-1} = \iota_{v_2} h^{\prime-1} a_j h \iota_{v_2}^{-1} = \iota_{v_2} a_{j'} \iota_{v_2}^{-1}.$$
This implies that
\begin{align*}
s^{(2)}_k(a_j) &= \prod_{w \in S(v_2,k)} \sgn(\iota_{a_j(w)} \circ a_j \circ \iota_w^{-1})\\
&= \prod_{(a_j,x_1,\ldots,x_k) \in V(G_\Gamma)^{k+1}} \sigma(x_k)^{|\Rect(a_j,x_1,\ldots,x_k)|}.
\end{align*}
But $|\Rect(a_j,x_1,\ldots,x_k)|$ is odd if and only if $(a_j,x_1,\ldots,x_k)$ is a non-repeating path in $G^{(2)}_\Gamma$ (by the claim), so this leads to the formula
\[s^{(2)}_k(a_j) = \prod_{(a_j,x_1,\ldots,x_k) \in \mathcal{P}_k(a_j)} \sigma(x_k).\qedhere\]
\end{proof}

The graph $G^{(2)}_\Gamma$ has $d_1$ vertices and is somewhat redundant as it has a non-trivial automorphism defined by $a \mapsto a^{-1}$ for each $a \in A$. In the particular case where $a \neq a^{-1}$ for each $a \in A$, i.e.\ when $\tau_1 = 0$ (as defined in \S\ref{subsection:complexesanddata}), we can define the \textbf{simplified labelled graph} $\tilde{G}^{(2)}_\Gamma$ associated to $\Gamma$ as follows. The vertex set $V(\tilde{G}^{(2)}_\Gamma)$ corresponds to the set of all $\{a, a^{-1}\}$ with $a \in A$, so that there are $\frac{d_1}{2}$ vertices. Then, we put an edge between $\{a,a^{-1}\}$ and $\{a',a^{\prime-1}\}$ if and only if exactly one of $a'$ and $a^{\prime-1}$ is connected to $a$ by an edge in $G^{(2)}_\Gamma$. This amounts to saying that $|R\cap(\{a\}\times B \times \{a',a^{\prime-1}\} \times B)|$ is odd. The automorphism of $G^{(2)}_\Gamma$ ensures that this edge set is well-defined. A \textbf{non-repeating path} in $\tilde{G}^{(2)}_\Gamma$ is defined exactly as in $G^{(2)}_\Gamma$, and we write $\tilde{\mathcal{P}}_n(x)$ for the set of all non-repeating paths in $\tilde{G}^{(2)}_\Gamma$ with length $n$ and origin $x$. The next proposition then shows that the values $s_k^{(2)}(a_j)$ can be computed from the simplified labelled graph $\tilde{G}^{(2)}_\Gamma$ when $\tau_1 = 0$.

\begin{proposition}\label{proposition:algo2}
Let $\Gamma$ be a $(d_1,d_2)$-group with $d_2 \geq 6$ and $\tau_1 = 0$ and suppose that $H_2 = \overline{\proj_2(\Gamma)}$ is non-discrete and satisfies $\underline{H_2}(v_2) \geq \Alt(d_2)$. Let $j \in \{1, \ldots, d_1\}$ and $k \in \N$. Then we have
$$s^{(2)}_k(a_j) = \prod_{p \in \tilde{\mathcal{P}}_k(\{a_j,a_j^{-1}\})} \sigma(d(p)).$$
\end{proposition}

\begin{proof}
Recall that $\tau_1 = 0$ means that $a \neq a^{-1}$ for all $a \in A$. Let us first focus on the (not simplified) labelled graph $G^{(2)}_\Gamma$. Given an edge $(x,y)$ in $G^{(2)}_\Gamma$, we say that $(x,y)$ is \textbf{stylish} if $(x,y^{-1})$ is also an edge in  $G^{(2)}_\Gamma$. Note that this also means that $(x^{-1},y)$ and $(x^{-1},y^{-1})$ are edges in  $G^{(2)}_\Gamma$. Let us say that a non-repeating path $(x_0,\ldots,x_n)$ in $G^{(2)}_\Gamma$ is \textbf{redundant} if there exists $i \in \{0,\ldots,n-1\}$ such that $(x_i,x_{i+1})$ is stylish. Given such a redundant non-repeating path, we let $i \geq 1$ be the smallest number such that $(x_i,x_{i+1})$ is stylish, and $j \leq n$ be the greatest number such that all edges $(x_i,x_{i+1})$, $(x_{i+1},x_{i+2})$, $\ldots$, $(x_{j-1},x_j)$ are stylish. Then we define the \textbf{mirror} $m(p)$ of the path $p = (x_0,\ldots,x_n)$ to be
$$(x_0,\ldots,x_i,x_{i+1}^{-1},x_{i+2},x_{i+3}^{-1},\ldots,x_{j-1}^{-1},x_j,x_{j+1},\ldots,x_n)$$
if $j \equiv i \bmod 2$ and
$$(x_0,\ldots,x_i,x_{i+1}^{-1},x_{i+2},x_{i+3}^{-1},\ldots,x_{j-1},x_j^{-1},x_{j+1}^{-1},\ldots,x_n^{-1})$$
if $j \not\equiv i \bmod 2$. The mirror of a redundant non-repeating path is still a redundant non-repeating path, and taking the mirror is an involution. This means that, if we look at the formula
$$s^{(2)}_k(a_j) = \prod_{p \in \mathcal{P}_k(a_j)} \sigma(d(p)).$$
given by Proposition~\ref{proposition:algo}, we can simply compute the product over the non-redundant non-repeating path in $\mathcal{P}_k(a_j)$. Indeed, the redundant ones come by pairs $(p,m(p))$, and $d(p) = d(m(p))$ or $d(m(p))^{-1}$ so that $p$ and $m(p)$ give the same signs. In order to conclude, there remains to observe that the map
$$(a_j,x_1,\ldots,x_n) \mapsto (\{a_j,a_j^{-1}\},\{x_1,x_1^{-1}\},\ldots,\{x_n,x_n^{-1}\})$$
defines a bijection between the set of non-redundant non-repeating paths in $G^{(2)}_\Gamma$ starting in $a_j$ and the set of non-repeating paths in $\tilde{G}^{(2)}_\Gamma$ starting in $\{a_j,a_j^{-1}\}$. This is a simple exercise.
\end{proof}

Thanks to Propositions~\ref{proposition:algo} or~\ref{proposition:algo2}, the invariants $K^{(2)}$ and $s^{(2)} \colon H_2(v_2) \to (\C_2)^{K^{(2)}+1}$ can be easily computed from the geometric squares defining $\Gamma$. For small $d_1$ and $d_2$ this can be done by hand, as illustrated in \S\ref{subsection:example}. We also know in advance, from Lemmas~\ref{lemma:valueS} and~\ref{lemma:valueS*}, that $s^{(2)}(H_2(v_2))$ must take the form $\{(s_0, \ldots, s_{K^{(2)}}) \in (\C_2)^{K^{(2)}+1} \mid \prod_{r \in X} s_r = 1\}$ for some $X \subset_f \N$ with $\max X = K^{(2)}$. If $0 \in X$, then $X \not \in \alpha(\{Y \subset_f \N\})$ by Lemma~\ref{lemma:alpha} and thus the only possibility for $H_2$ is to be equal to $G_{(i)}(X,X)$ for some legal coloring $i$ of $T_2$. On the other hand, if there exists $Y \subset_f \N$ such that $\alpha(Y) = X$, then $Y$ is unique (once again by Lemma~\ref{lemma:alpha}) and we conclude that $H_2$ is equal to one of the four groups $G_{(i)}(X,X)$, $G_{(i)}(Y^*,Y^*)$, $G_{(i)}(Y,Y)^*$ and $G_{(i)}'(Y,Y)^*$ for some legal coloring $i$ of $T_2$. Then it is still possible to compute the invariant $\rho^{(2)}$: if $\rho^{(2)} = 2$ then $H_2 = G_{(i)}(X,X)$, if $\rho^{(2)} = 2^{d_2}$ then $H_2 = G_{(i)}(Y^*,Y^*)$, and if $\rho^{(2)} = 2^{d_2-1}$ then $H_2 = G_{(i)}(Y,Y)^*$ or $G_{(i)}'(Y,Y)^*$. However, computing $\rho^{(2)}$ in general requires a computer, and even a computer is too slow if $K^{(2)}$ is big. In the next subsection, we see how to identify which of the four groups is the good one.

\subsection{Choosing among the four possible groups}
\label{subsection:choosing}

Let us suppose we are in presence of a $(d_1,d_2)$-group as in \S\ref{subsection:graph} and such that $s^{(2)}(H_2(v_2)) = \{(s_0, \ldots, s_{K^{(2)}}) \in (\C_2)^{K^{(2)}+1} \mid \prod_{r \in X} s_r = 1\}$ for some $X \subset_f \N$ with $\max X = K^{(2)}$ and $0 \not \in X$. Let $Y \subset_f \N$ be such that $\alpha(Y) = X$. Our goal is now to give a method enabling us to determine which of the four groups $G_{(i)}(X,X)$, $G_{(i)}(Y^*,Y^*)$, $G_{(i)}(Y,Y)^*$ and $G_{(i)}'(Y,Y)^*$ is isomorphic to $H_2$.

We start with the following proposition which, in some sense, enables to compute the invariant $\rho^{(2)} \in \{1, 2^{d_2-1}, 2^{d_2}\}$.

\begin{proposition}\label{proposition:choose4}
Let $\Gamma$ be a $(d_1,d_2)$-group with $d_2 \geq 6$ and suppose that $H_2 = \overline{\proj_2(\Gamma)}$ is non-discrete and satisfies $\underline{H_2}(v_2) \geq \Alt(d_2)$. Let $X \subset_f \N$ be such that $\max X = K^{(2)}$ and
$$s^{(2)}(H_2(v_2)) = \left\{(s_0, \ldots, s_{K^{(2)}}) \in (\C_2)^{K^{(2)}+1} \suchthat \prod_{r \in X} s_r = 1\right\},$$
and assume that $0 \not \in X$. Let $Y \subset_f \N$ be such that $\alpha(Y) = X$.

For each $j \in \{1, \ldots, d_1\}$, define $\Sigma_j = \prod_{r \in Y} s^{(2)}_r(a_j) \in \{-1, 1\}$. Also, for each $j \in \{1, \ldots, d_1\}$ and each $k \in \{1, \ldots, d_2\}$, define $\mu_{j,k} \in \{1, \ldots, d_1\}$ and $\nu_{j,k} \in \{1, \ldots, d_2\}$ so that $a_j b_k = b_{\nu_{j,k}} a_{\mu_{j,k}}$.

Then exactly one of the following assertions holds.

\begin{enumerate}[(1)]
\item There exist $x_1, \ldots, x_{d_2} \in \{-1, 1\}$ such that
$$\hspace{-0.5cm}(*)\ \left\{\begin{array}{l}
x_1 x_{\nu_{1,1}} \Sigma_{\mu_{1,1}} = x_2 x_{\nu_{1,2}} \Sigma_{\mu_{1,2}} = \cdots = x_{d_2} x_{\nu_{1,d_2}} \Sigma_{\mu_{1,d_2}} = \Sigma_1 \\
x_1 x_{\nu_{2,1}} \Sigma_{\mu_{2,1}} = x_2 x_{\nu_{2,2}} \Sigma_{\mu_{2,2}} = \cdots = x_{d_2} x_{\nu_{2,d_2}} \Sigma_{\mu_{2,d_2}} = \Sigma_2 \\
\quad \vdots \\
x_1 x_{\nu_{d_1,1}} \Sigma_{\mu_{d_1,1}} = x_2 x_{\nu_{d_1,2}} \Sigma_{\mu_{d_1,2}} = \cdots = x_{d_2} x_{\nu_{d_1,d_2}} \Sigma_{\mu_{d_1,d_2}} = \Sigma_{d_2} \\
\end{array}\right.$$
and $H_2 = G_{(i)}(Y,Y)^*$ or $G_{(i)}'(Y,Y)^*$ for some legal coloring $i$ of $T_2$.
\item There exist no $x_1, \ldots, x_{d_2} \in \{-1, 1\}$ satisfying the system $(*)$ but there exist $x_1, \ldots, x_{d_2} \in \{-1, 1\}$ such that
$$\hspace{-0.5cm}(**)\ \left\{\begin{array}{l}
x_1 x_{\nu_{1,1}} \Sigma_{\mu_{1,1}} = x_2 x_{\nu_{1,2}} \Sigma_{\mu_{1,2}} = \cdots = x_{d_2} x_{\nu_{1,d_2}} \Sigma_{\mu_{1,d_2}} \\
x_1 x_{\nu_{2,1}} \Sigma_{\mu_{2,1}} = x_2 x_{\nu_{2,2}} \Sigma_{\mu_{2,2}} = \cdots = x_{d_2} x_{\nu_{2,d_2}} \Sigma_{\mu_{2,d_2}} \\
\quad \vdots \\
x_1 x_{\nu_{d_1,1}} \Sigma_{\mu_{d_1,1}} = x_2 x_{\nu_{d_1,2}} \Sigma_{\mu_{d_1,2}} = \cdots = x_{d_2} x_{\nu_{d_1,d_2}} \Sigma_{\mu_{d_1,d_2}} \\
\end{array}\right.$$
and $H_2 = G_{(i)}(Y^*,Y^*)$ for some legal coloring $i$ of $T_2$.
\item There exist no $x_1, \ldots, x_{d_2} \in \{-1, 1\}$ satisfying the system $(*)$ or $(**)$, and $H_2 = G_{(i)}(X,X)$ for some legal coloring $i$ of $T_2$.
\end{enumerate}
\end{proposition}

\begin{proof}
For each $w \in V(T_2)$, we define $\iota_w \colon E(w) \to \{1, \ldots, d_2\}$ as before, i.e.\ so that the edge $e \in E(w)$ is labelled by the black symbol corresponding to $b_{\iota_w(e)}$.

Given a legal coloring $i$ of $T_2$ and some $k \in \{1, \ldots, d_2\}$, we define $x_k^{(i)} = \prod_{z \in S(b_k(v_2), Y)} \sgn (i \circ \iota_z^{-1}) \in \{-1, 1\}$. It is clear that any element of $\{-1,1\}^{d_2}$ can be written as $(x_1^{(i)}, \ldots, x_{d_2}^{(i)})$ for some legal coloring $i$. Now for such a coloring, we write $i_z \colon S(z,1) \to \{1, \ldots, d_2\}$ for the restriction of $i$ to $S(z,1)$, and compute
\begin{align*}
&\prod_{z \in S(b_k(v_2), Y))} \sgn(i_{a_j(z)} \circ a_j \circ i_z^{-1})\\
=& \prod_{z \in S(b_k(v_2), Y))} \sgn(i_{a_j(z)} \circ \iota^{-1}_{a_j(z)}) \sgn(\iota_{a_j(z)} \circ a_j \circ \iota_z^{-1}) \sgn(\iota_z \circ i_z^{-1})\\
=&\ x^{(i)}_k x^{(i)}_{\nu_{j,k}} \prod_{z \in S(b_k(v_2), Y)} \sgn(\iota_{a_j(z)} \circ a_j \circ \iota_z^{-1})\\
=&\ x^{(i)}_k x^{(i)}_{\nu_{j,k}} \prod_{z \in S(v_2, Y)} \sgn(\iota_{a_jb_k(z)} \circ a_j \circ \iota_{b_k(z)}^{-1})\\
=&\ x^{(i)}_k x^{(i)}_{\nu_{j,k}} \prod_{z \in S(v_2, Y)} \sgn(\iota_{b_{\nu_{j,k}}a_{\mu_{j,k}}(z)} \circ b_{\nu_{j,k}}\circ a_{\mu_{j,k}} \circ b_k^{-1} \circ \iota_{b_k(z)}^{-1})\\
=&\ x^{(i)}_k x^{(i)}_{\nu_{j,k}} \prod_{z \in S(v_2, Y)} \sgn(\iota_{a_{\mu_{j,k}}(z)} \circ a_{\mu_{j,k}} \circ \iota_z^{-1})\\
=&\ x^{(i)}_k x^{(i)}_{\nu_{j,k}} \Sigma_{\mu_{j,k}}.
\end{align*}
This implies that, if $H_2 = G_{(i)}(Y,Y)^*$ or $G_{(i)}'(Y,Y)^*$, then $(x_1^{(i)}, \ldots, x_{d_2}^{(i)})$ is a solution of $(*)$. Conversely, if $(x_1^{(i)}, \ldots, x_{d_2}^{(i)})$ is a solution of $(*)$ for some coloring $i$, then the equalities defining $G_{(i)}(Y,Y)^*$ are true in $B(v_2,\max Y+2)$ and we can deduce in particular that $\rho^{(2)} \geq 2^{d_2}$. In view of Lemmas~\ref{lemma:valueS} and~\ref{lemma:valueS*}, the only options for $H_2$ are then $G_{(i)}(Y,Y)^*$ and $G_{(i)}'(Y,Y)^*$ (for some coloring $i$ that may be different).

Now if we assume that $(*)$ has no solution, $H_2$ is different from $G_{(i)}(Y,Y)^*$ and $G_{(i)}'(Y,Y)^*$. Then by the same argument we obtain that $(**)$ has a solution if and only if $H_2 = G_{(i)}(Y^*,Y^*)$. In the case where neither $(*)$ nor $(**)$ has a solution, the only remaining possibility is to have $H_2 = G_{(i)}(X,X)$.
\end{proof}

The next proposition then explains how $G_{(i)}(Y,Y)^*$ and $G'_{(i)}(Y,Y)^*$ can be distinguished. As explained earlier, this requires observing an element exchanging the two types of vertices.

\begin{proposition}\label{proposition:primeornot}
Let $\Gamma$ be a $(d_1,d_2)$-group as in Proposition~\ref{proposition:choose4}, and assume that $H_2 = G_{(i)}(Y,Y)^*$ or $G_{(i)}'(Y,Y)^*$ for some legal coloring $i$ of $T_2$. Let $k \in \{1, \ldots, d_2\}$, $m \geq 0$ and $j_1, \ldots, j_m, j'_1, \ldots, j'_m \in \{1, \ldots, d_1\}$ be such that
$$a_{j_1}\cdots a_{j_m} b_k = b_k^{-1} a'_{j'_1} \cdots a'_{j'_m}.$$
Then $H_2 = G_{(i)}(Y,Y)^*$ if and only if $\Sigma_{j_1} \cdots \Sigma_{j_m} \Sigma_{j'_1} \cdots \Sigma_{j'_m} = 1$, where $\Sigma_j = \prod_{r \in Y} s^{(2)}_r(a_j) \in \{-1, 1\}$ for each $j \in \{1, \ldots, d_1\}$.
\end{proposition}

\begin{proof}
The element $\gamma = \proj_2(a_{j_1}\cdots a_{j_m}b_k) \in H_2$ sends $v_2 \in V(T_2)$ to one of its neighbors, say $w$. In particular, $\gamma$ exchanges the types of vertices in $T_2$. Moreover, the hypothesis implies that
$$\gamma^2 = \proj_2(a_{j_1}\cdots a_{j_m} a'_{j'_1}\cdots a'_{j'_m}).$$
So $\gamma^2$ fixes $v_2$ (i.e.\ $\gamma$ exchanges $v_2$ and $w$) and
$$\Sgn_{(i)}(\gamma^2,S_Y(v_2)) = \Sigma_{j_1} \cdots \Sigma_{j_m} \Sigma_{j'_1} \cdots \Sigma_{j'_m}.$$
The conclusion then follows from the definitions of the groups $G_{(i)}(Y,Y)^*$ or $G_{(i)}'(Y,Y)^*$. Indeed, if $\gamma \in G_{(i)}(Y,Y)^*$ then $\Sgn_{(i)}(\gamma,S_Y(v_2)) = \Sgn_{(i)}(\gamma,S_Y(w))$ and hence $\Sgn_{(i)}(\gamma^2,S_Y(v_2)) = 1$. On the contrary, if $\gamma \in G_{(i)}'(Y,Y)^*$ then $\Sgn_{(i)}(\gamma,S_Y(v_2)) \neq \Sgn_{(i)}(\gamma,S_Y(w))$ and in that case $\Sgn_{(i)}(\gamma^2,S_Y(v_2)) = -1$.
\end{proof}

Note that there always exist elements $k \in \{1, \ldots, d_2\}$, $m \geq 0$ and $j_1, \ldots, j_m$, $j'_1, \ldots, j'_m \in \{1, \ldots, d_1\}$ as in Proposition~\ref{proposition:primeornot}: they simply correspond to a rectangle $m \times 1$ in the square-complex $T_1 \times T_2$, with the property that the left and right edges of the rectangle correspond to $b_k$ and $b_k^{-1}$ for some $k \in \{1, \ldots, d_2\}$. The existence of such a rectangle is a consequence of the transitivity of $\underline{H_2}(v_2) \geq \Alt(d_2)$ on $E(v_2)$.

All our previous considerations lead to Theorem~\ref{maintheorem:algorithms}.

\begin{proof}[Proof of Theorem~\ref{maintheorem:algorithms}]
As explained in \S\ref{subsection:irred}, it suffices to look at the image of $H_t(v_t)$ in $\Aut(B(v_t,2))$ to see if $\Gamma$ is reducible or irreducible, and this can be done with the method explained in \S\ref{subsection:action}. So (i) is clear. For (ii), we saw in  \S\ref{subsection:graph} that computing the labelled graph $G^{(t)}_\Gamma$ gives at most $4$ possibilities for $H_t$, and in Propositions~\ref{proposition:choose4} and~\ref{proposition:primeornot} that choosing among the four groups could be done by solving two systems whose unknowns belong to $\{-1, 1\}$ and constructing a suitable $m \times 1$ or $1 \times m$ rectangle. It is not hard to implement those algorithms on a computer and they have a pretty good complexity. We do not go into a detailed analysis of the complexity, but the slowest part of the algorithm is probably to check the irreducibility of $\Gamma$ by observing the action of $H_t(v_t)$ on $B(v_t,2)$.
\end{proof}

We end this section with a particular case where Propositions~\ref{proposition:choose4} and~\ref{proposition:primeornot} always give the same conclusion.

\begin{corollary}\label{corollary:choose4}
Let $\Gamma$ be a $(d_1,d_2)$-group as in Proposition~\ref{proposition:choose4}. If $\prod_{r \in Y} s^{(2)}_r(a_j) = -1$ for each $j \in \{1, \ldots, d_1\}$, then $H_2 = G_{(i)}(Y,Y)^*$ for some legal coloring $i$ of $T_2$.
\end{corollary}

\begin{proof}
In Proposition~\ref{proposition:choose4}, we have $\Sigma_j = -1$ for each $j \in \{1, \ldots, d_1\}$ and thus $x_1 = \cdots = x_{d_2} = 1$ is a solution of $(*)$. Moreover, in Proposition~\ref{proposition:primeornot} we directly get $\Sigma_{j_1} \cdots \Sigma_{j_m} \Sigma_{j'_1} \cdots \Sigma_{j'_m} = (-1)^{2m} = 1$ which ends the proof.
\end{proof}

\subsection{Illustration on an example}
\label{subsection:example}

Let us illustrate the previous ideas on a concrete example. Let $\Gamma$ be the torsion-free $(6,6)$-group corresponding to the $9$ geometric squares drawn in Figure~\ref{picture:9squares}. Our goal is to explain how $H_1$ and $H_2$ can be computed in this particular case.

\begin{figure}[t!]
\centering
\begin{pspicture*}(-0.2,2.8)(11,6.2)
\fontsize{10pt}{10pt}\selectfont
\psset{unit=1.2cm}

\pspolygon(0,4)(1,4)(1,5)(0,5)

\pspolygon[fillstyle=solid,fillcolor=white](0.45,4.1)(0.45,3.9)(0.6,4)
\pspolygon[fillstyle=solid,fillcolor=white](0.55,5.1)(0.55,4.9)(0.4,5)
\pspolygon[fillstyle=solid,fillcolor=black](0.9,4.45)(1.1,4.45)(1,4.6)
\pspolygon[fillstyle=solid,fillcolor=black](-0.1,4.45)(0.1,4.45)(0,4.6)

\pspolygon(2,4)(3,4)(3,5)(2,5)

\pspolygon[fillstyle=solid,fillcolor=white](2.45,4.1)(2.45,3.9)(2.6,4)
\pspolygon[fillstyle=solid,fillcolor=white](2.55,5.1)(2.55,4.9)(2.4,5)
\pspolygon[fillstyle=solid,fillcolor=black](2.9,4.55)(3.1,4.55)(3,4.7)
\pspolygon[fillstyle=solid,fillcolor=black](2.9,4.35)(3.1,4.35)(3,4.5)
\pspolygon[fillstyle=solid,fillcolor=black](1.9,4.55)(2.1,4.55)(2,4.7)
\pspolygon[fillstyle=solid,fillcolor=black](1.9,4.35)(2.1,4.35)(2,4.5)

\pspolygon(4,4)(5,4)(5,5)(4,5)

\pspolygon[fillstyle=solid,fillcolor=white](4.45,4.1)(4.45,3.9)(4.6,4)
\pspolygon[fillstyle=solid,fillcolor=white](4.65,5.1)(4.65,4.9)(4.5,5)
\pspolygon[fillstyle=solid,fillcolor=white](4.45,5.1)(4.45,4.9)(4.3,5)
\pspolygon[fillstyle=solid,fillcolor=black](4.9,4.65)(5.1,4.65)(5,4.8)
\pspolygon[fillstyle=solid,fillcolor=black](4.9,4.45)(5.1,4.45)(5,4.6)
\pspolygon[fillstyle=solid,fillcolor=black](4.9,4.25)(5.1,4.25)(5,4.4)
\pspolygon[fillstyle=solid,fillcolor=black](3.9,4.35)(4.1,4.35)(4,4.2)
\pspolygon[fillstyle=solid,fillcolor=black](3.9,4.55)(4.1,4.55)(4,4.4)
\pspolygon[fillstyle=solid,fillcolor=black](3.9,4.75)(4.1,4.75)(4,4.6)

\pspolygon(6,4)(7,4)(7,5)(6,5)

\pspolygon[fillstyle=solid,fillcolor=white](6.45,4.1)(6.45,3.9)(6.6,4)
\pspolygon[fillstyle=solid,fillcolor=white](6.35,5.1)(6.35,4.9)(6.2,5)
\pspolygon[fillstyle=solid,fillcolor=white](6.55,5.1)(6.55,4.9)(6.4,5)
\pspolygon[fillstyle=solid,fillcolor=white](6.75,5.1)(6.75,4.9)(6.6,5)
\pspolygon[fillstyle=solid,fillcolor=black](6.9,4.35)(7.1,4.35)(7,4.2)
\pspolygon[fillstyle=solid,fillcolor=black](6.9,4.55)(7.1,4.55)(7,4.4)
\pspolygon[fillstyle=solid,fillcolor=black](6.9,4.75)(7.1,4.75)(7,4.6)
\pspolygon[fillstyle=solid,fillcolor=black](5.9,4.65)(6.1,4.65)(6,4.8)
\pspolygon[fillstyle=solid,fillcolor=black](5.9,4.45)(6.1,4.45)(6,4.6)
\pspolygon[fillstyle=solid,fillcolor=black](5.9,4.25)(6.1,4.25)(6,4.4)

\pspolygon(8,4)(9,4)(9,5)(8,5)

\pspolygon[fillstyle=solid,fillcolor=white](8.35,4.1)(8.35,3.9)(8.5,4)
\pspolygon[fillstyle=solid,fillcolor=white](8.55,4.1)(8.55,3.9)(8.7,4) 
\pspolygon[fillstyle=solid,fillcolor=white](8.35,5.1)(8.35,4.9)(8.2,5)
\pspolygon[fillstyle=solid,fillcolor=white](8.55,5.1)(8.55,4.9)(8.4,5)
\pspolygon[fillstyle=solid,fillcolor=white](8.75,5.1)(8.75,4.9)(8.6,5)
\pspolygon[fillstyle=solid,fillcolor=black](8.9,4.45)(9.1,4.45)(9,4.6)
\pspolygon[fillstyle=solid,fillcolor=black](7.9,4.45)(8.1,4.45)(8,4.6)


\pspolygon(1,2.5)(2,2.5)(2,3.5)(1,3.5)

\pspolygon[fillstyle=solid,fillcolor=white](1.35,2.6)(1.35,2.4)(1.5,2.5)
\pspolygon[fillstyle=solid,fillcolor=white](1.55,2.6)(1.55,2.4)(1.7,2.5)
\pspolygon[fillstyle=solid,fillcolor=white](1.25,3.6)(1.25,3.4)(1.4,3.5)
\pspolygon[fillstyle=solid,fillcolor=white](1.45,3.6)(1.45,3.4)(1.6,3.5)
\pspolygon[fillstyle=solid,fillcolor=white](1.65,3.6)(1.65,3.4)(1.8,3.5)
\pspolygon[fillstyle=solid,fillcolor=black](1.9,2.85)(2.1,2.85)(2,3)
\pspolygon[fillstyle=solid,fillcolor=black](1.9,3.05)(2.1,3.05)(2,3.2)
\pspolygon[fillstyle=solid,fillcolor=black](0.9,2.75)(1.1,2.75)(1,2.9)
\pspolygon[fillstyle=solid,fillcolor=black](0.9,2.95)(1.1,2.95)(1,3.1)
\pspolygon[fillstyle=solid,fillcolor=black](0.9,3.15)(1.1,3.15)(1,3.3)

\pspolygon(3,2.5)(4,2.5)(4,3.5)(3,3.5)

\pspolygon[fillstyle=solid,fillcolor=white](3.35,2.6)(3.35,2.4)(3.5,2.5)
\pspolygon[fillstyle=solid,fillcolor=white](3.55,2.6)(3.55,2.4)(3.7,2.5)
\pspolygon[fillstyle=solid,fillcolor=white](3.25,3.6)(3.25,3.4)(3.4,3.5)
\pspolygon[fillstyle=solid,fillcolor=white](3.45,3.6)(3.45,3.4)(3.6,3.5)
\pspolygon[fillstyle=solid,fillcolor=white](3.65,3.6)(3.65,3.4)(3.8,3.5)
\pspolygon[fillstyle=solid,fillcolor=black](3.9,3.25)(4.1,3.25)(4,3.1)
\pspolygon[fillstyle=solid,fillcolor=black](3.9,3.05)(4.1,3.05)(4,2.9)
\pspolygon[fillstyle=solid,fillcolor=black](3.9,2.85)(4.1,2.85)(4,2.7)
\pspolygon[fillstyle=solid,fillcolor=black](2.9,3.15)(3.1,3.15)(3,3)
\pspolygon[fillstyle=solid,fillcolor=black](2.9,2.95)(3.1,2.95)(3,2.8)

\pspolygon(5,2.5)(6,2.5)(6,3.5)(5,3.5)

\pspolygon[fillstyle=solid,fillcolor=white](5.35,2.6)(5.35,2.4)(5.5,2.5)
\pspolygon[fillstyle=solid,fillcolor=white](5.55,2.6)(5.55,2.4)(5.7,2.5)
\pspolygon[fillstyle=solid,fillcolor=white](5.25,3.6)(5.25,3.4)(5.4,3.5)
\pspolygon[fillstyle=solid,fillcolor=white](5.45,3.6)(5.45,3.4)(5.6,3.5)
\pspolygon[fillstyle=solid,fillcolor=white](5.65,3.6)(5.65,3.4)(5.8,3.5)
\pspolygon[fillstyle=solid,fillcolor=black](5.9,3.15)(6.1,3.15)(6,3)
\pspolygon[fillstyle=solid,fillcolor=black](5.9,2.95)(6.1,2.95)(6,2.8)
\pspolygon[fillstyle=solid,fillcolor=black](4.9,3.05)(5.1,3.05)(5,2.9)

\pspolygon(7,2.5)(8,2.5)(8,3.5)(7,3.5)

\pspolygon[fillstyle=solid,fillcolor=white](7.35,2.6)(7.35,2.4)(7.5,2.5)
\pspolygon[fillstyle=solid,fillcolor=white](7.55,2.6)(7.55,2.4)(7.7,2.5)
\pspolygon[fillstyle=solid,fillcolor=white](7.25,3.6)(7.25,3.4)(7.4,3.5)
\pspolygon[fillstyle=solid,fillcolor=white](7.45,3.6)(7.45,3.4)(7.6,3.5)
\pspolygon[fillstyle=solid,fillcolor=white](7.65,3.6)(7.65,3.4)(7.8,3.5)
\pspolygon[fillstyle=solid,fillcolor=black](7.9,3.05)(8.1,3.05)(8,2.9)
\pspolygon[fillstyle=solid,fillcolor=black](6.9,2.85)(7.1,2.85)(7,3)
\pspolygon[fillstyle=solid,fillcolor=black](6.9,3.05)(7.1,3.05)(7,3.2)
\end{pspicture*}

\caption{The geometric squares of a torsion-free $(6,6)$-group.}\label{picture:9squares}
\end{figure}
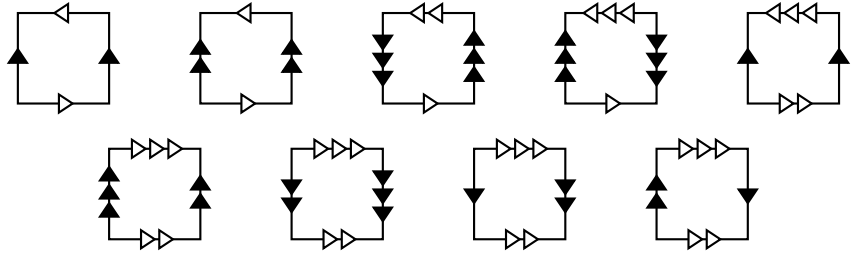

We start by computing the action of $H_2(v_2)$ on $B(v_2,1)$. The six vertices in $S(v_2,1)$ are $b_1(v_2)$, $b_2(v_2)$, $b_3(v_2)$, $b_3^{-1}(v_2)$, $b_2^{-1}(v_2)$ and $b_1^{-1}(v_2)$. The six edges from $v_2$ to these six vertices are labelled by the six black symbols (with orientation) corresponding to the six generators $b_1$, $b_2$, $b_3$, $b_4 = b_3^{-1}$, $b_5 = b_2^{-1}$ and $b_6 = b_1^{-1}$. From the geometric squares, we directly get the actions of $a_1$, $a_2$ and $a_3$ on these six edges. We denote them by the generator to which they correspond.
\begin{align*}
a_1\ &:\ (b_1) (b_1^{-1}) (b_2) (b_2^{-1}) (b_3 \ b_3^{-1})\\
a_2\ &:\ (b_1) (b_1^{-1}\ b_2\ b_3\ b_3^{-1}\ b_2^{-1})\\
a_3\ &:\ (b_1\ b_2^{-1}\ b_3^{-1}\ b_3\ b_2) (b_1^{-1})
\end{align*}
One can easily check that these three permutations generate $\Sym(6)$. As explained in~Section~\ref{subsection:irred}, it then suffices to compute the action of $H_1(v_1)$ on $B(v_1,2)$ and to use \cite[Proposition~3.3.2]{Burger} to conclude that $H_1$ is non-discrete and thus that $\Gamma$ is irreducible.

Let us now find out which group $H_2$ exactly is. As $\tau_1 = 0$, we can compute the simplified labelled graph $\tilde{G}^{(2)}_{\Gamma}$ and use Proposition~\ref{proposition:algo2}. We also compute $G^{(2)}_{\Gamma}$ so as to illustrate Proposition~\ref{proposition:algo} as well. The two graphs we obtain are given in Figure~\ref{picture:G^(2)}.

From $\tilde{G}^{(2)}_\Gamma$ (or $G^{(2)}_\Gamma$) we can compute the values of $s_k^{(2)}(a_j)$ for $j \in \{1,2,3\}$ and $k \in \N$. For $k \in \{0,1,2,3\}$ we obtain:
$$\begin{array}{c|cccc}
 & s^{(2)}_0 & s^{(2)}_1 & s^{(2)}_2 & s^{(2)}_3\\
 \hline
 a_1 & -1 & +1 & +1 & +1\\
 a_2 & +1 & -1 & -1 & +1\\
 a_3 & +1 & -1 & -1 & +1
\end{array}$$
The map $s^{(2)} \colon H_2(v_2) \to (\C_2)^3 \colon h \mapsto (s^{(2)}_0(h), s^{(2)}_1(h), s^{(2)}_2(h))$ is not surjective, so $K^{(2)} = 2$. Moreover, we see that
$$s^{(2)}(H_2(v_2)) = \left\{(s_0,s_1,s_2) \in (\C_2)^3 \suchthat \prod_{r \in \{1,2\}} s_r = 1\right\}.$$
Since $\alpha(\{0,1\}) = \{1,2\}$ (by Lemma~\ref{lemma:explicitalpha}), these values for $K^{(2)}$ and $s^{(2)}(H_2(v_2))$ imply that $H_2$ is one of $G_{(i)}(\{1,2\},\{1,2\})$, $G_{(i)}(\{0,1\}^*,\{0,1\}^*)$, $G_{(i)}(\{0,1\},\{0,1\})^*$ and $G_{(i)}'(\{0,1\},\{0,1\})^*$. We then remark that $\prod_{r \in \{0,1\}} s^{(2)}_r(a_j) = -1$ for each $j \in \{1,2,3\}$, which ensures via Corollary~\ref{corollary:choose4} that $H_2 = G_{(i_2)}(\{0,1\},\{0,1\})^*$ (for some legal coloring $i_2$ of $T_2$).

\begin{figure}[t!]
\begin{subfigure}{.49\textwidth}
\centering
\begin{pspicture*}(-0.2,-1)(4.2,3)
\fontsize{12pt}{12pt}\selectfont
\psset{unit=1cm}

\psline(0.58,0.616)(3.44,0.616)
\psline(1.08,-0.25)(2.94,-0.25)

\psline(0.58,0.616)(1.51,2.233)
\psline(1.08,-0.25)(2.51,2.233)

\psline(2.94,-0.25)(1.51,2.233)
\psline(3.44,0.616)(2.51,2.233)

\pscircle[fillstyle=solid,fillcolor=white](0.58,0.616){0.4}
\pspolygon[fillstyle=solid,fillcolor=white](0.52,0.716)(0.52,0.516)(0.67,0.616)
\rput(0.08,1.116){$-1$}
\pscircle[fillstyle=solid,fillcolor=white](1.08,-0.25){0.4}
\pspolygon[fillstyle=solid,fillcolor=white](1.14,-0.15)(1.14,-0.35)(0.99,-0.25)
\rput(0.58,-0.75){$-1$}

\pscircle[fillstyle=solid,fillcolor=white](2.94,-0.25){0.4}
\pspolygon[fillstyle=solid,fillcolor=white](2.78,-0.15)(2.78,-0.35)(2.93,-0.25)
\pspolygon[fillstyle=solid,fillcolor=white](2.98,-0.15)(2.98,-0.35)(3.13,-0.25)
\rput(3.44,-0.75){$+1$}
\pscircle[fillstyle=solid,fillcolor=white](3.44,0.616){0.4}
\pspolygon[fillstyle=solid,fillcolor=white](3.60,0.716)(3.60,0.516)(3.45,0.616)
\pspolygon[fillstyle=solid,fillcolor=white](3.40,0.716)(3.40,0.516)(3.25,0.616)
\rput(3.94,1.116){$+1$}

\pscircle[fillstyle=solid,fillcolor=white](2.51,2.233){0.4}
\pspolygon[fillstyle=solid,fillcolor=white](2.25,2.333)(2.25,2.133)(2.4,2.233)
\pspolygon[fillstyle=solid,fillcolor=white](2.45,2.333)(2.45,2.133)(2.6,2.233)
\pspolygon[fillstyle=solid,fillcolor=white](2.65,2.333)(2.65,2.133)(2.8,2.233)
\rput(3.01,2.733){$+1$}
\pscircle[fillstyle=solid,fillcolor=white](1.51,2.233){0.4}
\pspolygon[fillstyle=solid,fillcolor=white](1.77,2.333)(1.77,2.133)(1.62,2.233)
\pspolygon[fillstyle=solid,fillcolor=white](1.57,2.333)(1.57,2.133)(1.42,2.233)
\pspolygon[fillstyle=solid,fillcolor=white](1.37,2.333)(1.37,2.133)(1.22,2.233)
\rput(1.01,2.733){$+1$}
\end{pspicture*}
\end{subfigure}
\begin{subfigure}{.49\textwidth}
\centering
\begin{pspicture*}(-0.3,-0.45)(4.8,3.6)
\fontsize{12pt}{12pt}\selectfont
\psset{unit=1cm}

\psline(3.51,0)(2.01,2.6)

\psline(0.51,0)(3.51,0)

\psline(0.51,0)(2.01,2.6)

\pscircle[fillstyle=solid,fillcolor=white](0.51,0){0.4}
\pspolygon[fillstyle=solid,fillcolor=white](0.45,0.1)(0.45,-0.1)(0.6,0)
\rput(0.01,0.5){$-1$}

\pscircle[fillstyle=solid,fillcolor=white](3.51,0){0.4}
\pspolygon[fillstyle=solid,fillcolor=white](3.35,0.1)(3.35,-0.1)(3.5,0)
\pspolygon[fillstyle=solid,fillcolor=white](3.55,0.1)(3.55,-0.1)(3.7,0)
\rput(4.01,0.5){$+1$}

\pscircle[fillstyle=solid,fillcolor=white](2.01,2.6){0.4}
\pspolygon[fillstyle=solid,fillcolor=white](1.75,2.7)(1.75,2.5)(1.9,2.6)
\pspolygon[fillstyle=solid,fillcolor=white](1.95,2.7)(1.95,2.5)(2.1,2.6)
\pspolygon[fillstyle=solid,fillcolor=white](2.15,2.7)(2.15,2.5)(2.3,2.6)
\rput(2.51,3.1){$+1$}
\end{pspicture*}
\end{subfigure}
\caption{The labelled graphs $G^{(2)}_\Gamma$ and $\tilde{G}^{(2)}_\Gamma$.}\label{picture:G^(2)}
\end{figure}
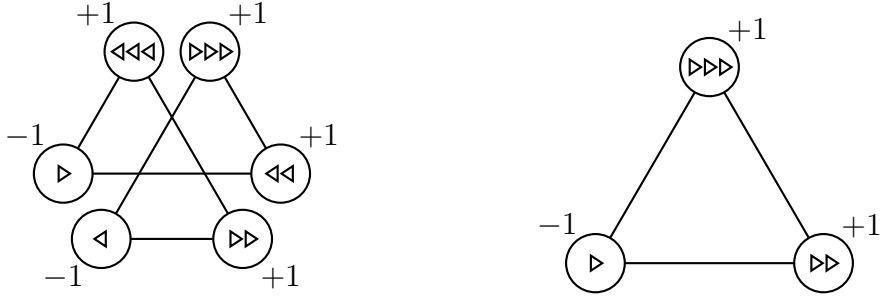

Let us do the same work for $H_1$. The action of $H_1(v_1)$ on $B(v_1,1)$ is given by the following permutations: they also generate $\Sym(6)$.
\begin{align*}
b_1\ &:\ (a_1\ a_1^{-1})(a_2\ a_3\ a_2^{-1}\ a_3^{-1})\\
b_2\ &:\ (a_1\ a_1^{-1}) (a_2\ a_3) (a_2^{-1}\ a_3^{-1})\\
b_3\ &:\ (a_1\ a_2^{-1}\ a_3^{-1})(a_1^{-1}\ a_3\ a_2)
\end{align*}

As $\tau_2 = 0$, we can compute the simplified labelled graph $\tilde{G}^{(1)}_{\Gamma}$. It has no edge, and exactly one of the three vertices is labelled by $-1$: the one corresponding to $\{b_2,b_2^{-1}\}$. The values of $s^{(1)}_k(b_j)$ for $j \in \{1,2,3\}$ and $k \in \{0,1,2,3\}$ are therefore:
$$\begin{array}{c|cccc}
 & s^{(1)}_0 & s^{(1)}_1 & s^{(1)}_2 & s^{(1)}_3\\
 \hline
 b_1 & +1 & +1 & +1 & +1\\
 b_2 & -1 & +1 & +1 & +1\\
 b_3 & +1 & +1 & +1 & +1
\end{array}$$
The map $s^{(1)} \colon H_1(v_1) \to (\C_2)^2 \colon h \mapsto (s^{(1)}_0(h), s^{(1)}_1(h))$ is not surjective, so $K^{(1)} = 1$. Also, $s^{(1)}(H_1(v_1)) = \left\{(s_0,s_1) \in (\C_2)^2 \suchthat \prod_{r \in \{1\}} s_r = 1\right\}$. Since $\alpha(\{0\}) = \{1\}$, we obtain that $H_1$ is one of the groups $G_{(i)}(\{1\},\{1\})$, $G_{(i)}(\{0\}^*,\{0\}^*)$, $G_{(i)}(\{0\},\{0\})^*$ and $G_{(i)}'(\{0\},\{0\})^*$. This time we do not have $\prod_{r \in \{0\}} s^{(1)}_r(b_j) = -1$ for each $j \in \{1,2,3\}$, so Corollary~\ref{corollary:choose4} cannot be used. We therefore need Proposition~\ref{proposition:choose4}. After looking carefully at the geometric squares, the system $(*)$ in Proposition~\ref{proposition:choose4} is
$$\left\{
\begin{array}{l}
x_1x_6 = -x_2x_3 = x_3x_5 = x_4x_2 = -x_5x_4 = x_6x_1 = 1\\
-x_1x_6 = x_2x_3 = x_3x_2 = x_4x_5 = x_5x_4 = -x_6x_1 = -1\\
x_1x_5 = x_2x_6 = -x_3x_2 = x_4x_1 = -x_5x_4 = x_6x_3 = 1
\end{array}
\right.$$
From $x_1x_5 = x_4x_1$ it follows that $x_4=x_5$, but this contradicts $x_4x_5 = -1$ so this system has no solution. Hence the groups $G_{(i)}(\{0\},\{0\})^*$ and $G_{(i)}'(\{0\},\{0\})^*$ can be excluded. The system $(**)$ in Proposition~\ref{proposition:choose4} is exactly the same, but without the last equality on each line:
$$\left\{
\begin{array}{l}
x_1x_6 = -x_2x_3 = x_3x_5 = x_4x_2 = -x_5x_4 = x_6x_1\\
-x_1x_6 = x_2x_3 = x_3x_2 = x_4x_5 = x_5x_4 = -x_6x_1\\
x_1x_5 = x_2x_6 = -x_3x_2 = x_4x_1 = -x_5x_4 = x_6x_3
\end{array}
\right.$$
A solution of this system is given by $(1,1,1,-1,-1,-1)$, so we conclude that $H_1 = G_{(i_1)}(\{0\}^*,\{0\}^*)$ for some legal coloring $i_1$ of $T_1$.

We summarize our computations in the following lemma.

\begin{lemma}
Let $\Gamma \leq \Aut(T_1) \times \Aut(T_2)$ be the torsion-free $(6,6)$-group defined by Figure~\ref{picture:9squares}. Then $H_1 = G_{(i_1)}(\{0\}^*,\{0\}^*)$ and $H_2 = G_{(i_2)}(\{0,1\},\{0,1\})^*$ for some legal colorings $i_1$ and $i_2$ of $T_1$ and $T_2$.
\end{lemma}

\begin{proof}
See above.
\end{proof}

\subsection{An inventory of possible projections}
\label{subsection:inventory}

As already mentioned, we do not have any tool to check the irreducibility of a $(d_1,d_2)$-group in full generality. With a computer it can be quickly seen if, for some $j \in \{1,2\}$ the fixator of $B(v_j,2)$ in $H_j$ is trivial. Indeed, since $H_j$ is transitive on vertices of $T_j$, it suffices to see whether the fixator of $B(v_j,2)$ also fixes $B(v_j,3)$. We will therefore say in this section that some $(d_1,d_2)$-group $\Gamma$ is \textbf{possibly irreducible} if the fixator of $B(v_j,2)$ in $H_j$ is non-trivial for each $j \in \{1,2\}$.

\begin{table}[b!]
\setlength{\tabcolsep}{0.3em}
\centering
\begin{tabular}{|c|ccc|c|ccc|c|c|}
\hline
\multirow{2}{*}{} & \multicolumn{4}{c|}{Torsion-free} & \multicolumn{4}{c|}{With torsion}\\
\cline{2-9}
& Irred. & ? & Red. & Total & Irred. & ? & Red. & Total\\
\hline
$(3,3)$ & - & - & - & - & 0 & 4 & 56 & 60\\
$(3,4)$ & - & - & - & - & 0 & 59 & 664 & 723\\
$(3,5)$ & - & - & - & - & 0 & 457 & 1986 & 2443\\
$(3,6)$ & - & - & - & - & 204 & 3018 & 10529 & 13751\\
\hline
$(4,4)$ & 0 & 2 & 50 & 52 & 0 & 686 & 2992 & 3678\\
$(4,5)$ & - & - & - & - & 0 & 23839 & 34700 & 58539\\
$(4,6)$ & 16 & 95 & 890 & 1001 & (111) & (433) & (1840) & (2384)\\
\hline
$(6,6)$ & 8227 & 5409 & 18426 & 32062 & (83581) & (33565) & (76037) & (193083)\\
\hline 
\end{tabular}
\caption{$(d_1,d_2)$-groups up to equivalence.}\label{table:numbers}
\end{table}

We always consider $(d_1,d_2)$-groups up to equivalence (i.e.\ up to conjugation in $\Aut(T_1 \times T_2)$). For some values of $d_1$ and $d_2$, we could compute the total number of torsion-free $(d_1,d_2)$-groups and the number of $(d_1,d_2)$-groups with torsion (up to equivalence) by enumerating them all (thanks to the GAP system). Some of these groups can be seen to be reducible by simply showing that they are not \textit{possibly irreducible}. Also, when $d_j = 6$ for some $j \in \{1,2\}$, some other $(d_1,d_2)$-groups can be proved to be irreducible, as explained previously, when $\underline{H_j}(v_j) \geq \Alt(d_j)$. The results we obtained are given in Table~\ref{table:numbers}. The numbers in parentheses correspond to $(d_1,d_2)$-groups with torsion but with $\tau_1 = \tau_2 = 0$ (i.e.\ without generators of order~$2$). Indeed, for $(d_1,d_2) \in \{(4,6), (6,6)\}$ the number of $(d_1,d_2)$-groups is so big that we could not count them up to equivalence. Note that we actually know a bit more that what is written in Table~\ref{table:numbers}. For instance, we will see in \S\ref{subsection:45} below that at least $60$ of the $23839$ possibly irreducible $(4,5)$-groups are irreducible.

In the remainder of this section, we give tables with the possible pairs of projections $(H_1,H_2)$ that a possibly irreducible $(d_1,d_2)$-group can have, for some values of $d_1$ and $d_2$. The idea is, for each (equivalence class of) $(d_1,d_2)$-group, to check that it is possibly irreducible, to compute $H_1$ and $H_2$ (if possible) and to write in the table that there is a group with these two projections. However, we only saw in the previous sections how to determine $H_j$ in the particular case where $d_j \geq 6$, $\underline{H_j}(v_j) \geq \Alt(d_j)$ and $H_j$ is non-discrete. In all other cases, we therefore only take note of the local action $\underline{H_j}(v_j) \leq \Sym(d_j)$.

Remark that, when $d_1$ and $d_2$ are fixed with $d_2 \geq 6$ even, we can in advance give a co-finite subset of groups of $\mathcal{G}'_{(i)}$ that cannot appear as a projection of a $(d_1,d_2)$-group on $T_2$. Indeed, the labelled graph $G^{(2)}_\Lambda$ has $d_1$ vertices, it has a non-trivial automorphism as explained in \S\ref{subsection:graph}, and we also know that the degree of each vertex is even (where a loop in a vertex only counts once in the degree of that vertex). Hence, one can simply go through all labelled graphs satisfying these three properties and compute all groups $\mathcal{G}'_{(i)}$ that correspond to them. This gives a finite list of groups that covers all possible projections on $T_2$. If we are only interested in torsion-free $(d_1,d_2)$-groups, then the task is even shorter as we can use the simplified labelled graphs which have $\frac{d_1}{2}$ vertices. Moreover, the degree of each vertex is also even and loops cannot appear in that case.

Let us for instance consider the torsion-free $(6,6)$-groups. We must observe all labelled graphs with $3$ vertices, without any loop and such that the degree of each vertex is even. If we do not consider the labels then there are only two such graphs: the one without any edge and the one with all three possible edges. For each of these two graphs we can put between zero and three labels $-1$, so at the end we get eight labelled graphs. The groups associated to those graphs are given in Figure~\ref{figure:projections66}. Note that we could exclude some groups by making use of Corollary~\ref{corollary:choose4}. In total, we obtain only seven groups of $\mathcal{G}'_{(i)}$ that could possibly appear as a projection of a torsion-free $(6,6)$-group. Our tables below (found with a computer) show that all these seven groups indeed arise, see Tables~\ref{table:66} and~\ref{table:66b}.

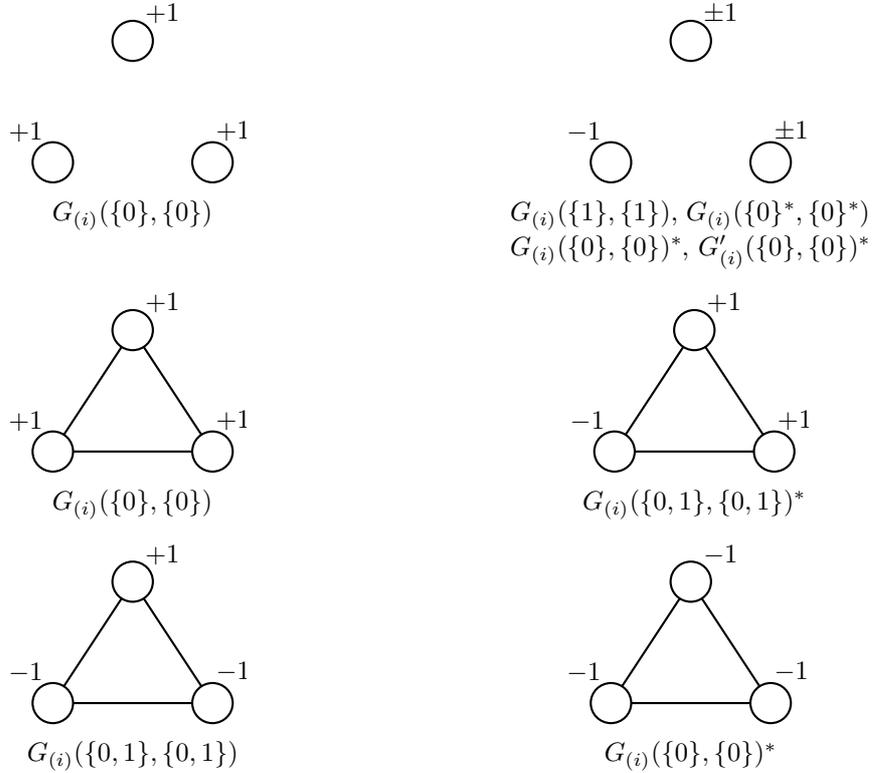
\begin{figure}[b!]
\begin{minipage}{.49\textwidth}
\centering
\begin{pspicture*}(-0.2,-1.6)(2.9,2.2)
\fontsize{10pt}{10pt}\selectfont
\psset{unit=0.7cm}

\pscircle[fillstyle=solid,fillcolor=white](0.51,0){0.4}
\rput(0,0.57){$+1$}

\pscircle[fillstyle=solid,fillcolor=white](3.51,0){0.4}
\rput(3.9,0.59){$+1$}

\pscircle[fillstyle=solid,fillcolor=white](2.01,2.3){0.4}
\rput(2.57,2.81){$+1$}

\rput(2,-1){$G_{(i)}(\{0\},\{0\})$}
\end{pspicture*}
\end{minipage}
\begin{minipage}{.5\textwidth}
\centering
\begin{pspicture*}(-1,-1.6)(3.8,2.2)
\fontsize{10pt}{10pt}\selectfont
\psset{unit=0.7cm}

\pscircle[fillstyle=solid,fillcolor=white](0.51,0){0.4}
\rput(0,0.57){$-1$}

\pscircle[fillstyle=solid,fillcolor=white](3.51,0){0.4}
\rput(3.9,0.59){$\pm1$}

\pscircle[fillstyle=solid,fillcolor=white](2.01,2.3){0.4}
\rput(2.57,2.81){$\pm1$}

\rput(2,-1){$G_{(i)}(\{1\},\{1\})$, $G_{(i)}(\{0\}^*,\{0\}^*)$}
\rput(2,-1.7){$G_{(i)}(\{0\},\{0\})^*$, $G'_{(i)}(\{0\},\{0\})^*$}
\end{pspicture*}
\end{minipage}

\begin{minipage}{.49\textwidth}
\centering
\begin{pspicture*}(-0.2,-1.1)(2.9,2.2)
\fontsize{10pt}{10pt}\selectfont
\psset{unit=0.7cm}

\psline(0.51,0)(3.51,0)

\psline(3.51,0)(2.01,2.3)

\psline(2.01,2.3)(0.51,0)

\pscircle[fillstyle=solid,fillcolor=white](0.51,0){0.4}
\rput(0,0.57){$+1$}

\pscircle[fillstyle=solid,fillcolor=white](3.51,0){0.4}
\rput(3.9,0.59){$+1$}

\pscircle[fillstyle=solid,fillcolor=white](2.01,2.3){0.4}
\rput(2.57,2.81){$+1$}

\rput(2,-1){$G_{(i)}(\{0\},\{0\})$}
\end{pspicture*}
\end{minipage}
\begin{minipage}{.5\textwidth}
\centering
\begin{pspicture*}(-0.2,-1.1)(2.9,2.2)
\fontsize{10pt}{10pt}\selectfont
\psset{unit=0.7cm}

\psline(0.51,0)(3.51,0)

\psline(3.51,0)(2.01,2.3)

\psline(2.01,2.3)(0.51,0)

\pscircle[fillstyle=solid,fillcolor=white](0.51,0){0.4}
\rput(0,0.57){$-1$}

\pscircle[fillstyle=solid,fillcolor=white](3.51,0){0.4}
\rput(3.9,0.59){$+1$}

\pscircle[fillstyle=solid,fillcolor=white](2.01,2.3){0.4}
\rput(2.57,2.81){$+1$}

\rput(2,-1){$G_{(i)}(\{0,1\},\{0,1\})^*$}
\end{pspicture*}
\end{minipage}

\begin{minipage}{.49\textwidth}
\centering
\begin{pspicture*}(-0.2,-1.0)(2.9,2.2)
\fontsize{10pt}{10pt}\selectfont
\psset{unit=0.7cm}

\psline(0.51,0)(3.51,0)

\psline(3.51,0)(2.01,2.3)

\psline(2.01,2.3)(0.51,0)

\pscircle[fillstyle=solid,fillcolor=white](0.51,0){0.4}
\rput(0,0.57){$-1$}

\pscircle[fillstyle=solid,fillcolor=white](3.51,0){0.4}
\rput(3.9,0.59){$-1$}

\pscircle[fillstyle=solid,fillcolor=white](2.01,2.3){0.4}
\rput(2.57,2.81){$+1$}

\rput(2,-1){$G_{(i)}(\{0,1\},\{0,1\})$}
\end{pspicture*}
\end{minipage}
\begin{minipage}{.5\textwidth}
\centering
\begin{pspicture*}(-1,-1.0)(3.8,2.2)
\fontsize{10pt}{10pt}\selectfont
\psset{unit=0.7cm}

\psline(0.51,0)(3.51,0)

\psline(3.51,0)(2.01,2.3)

\psline(2.01,2.3)(0.51,0)

\pscircle[fillstyle=solid,fillcolor=white](0.51,0){0.4}
\rput(0,0.57){$-1$}

\pscircle[fillstyle=solid,fillcolor=white](3.51,0){0.4}
\rput(3.9,0.59){$-1$}

\pscircle[fillstyle=solid,fillcolor=white](2.01,2.3){0.4}
\rput(2.57,2.81){$-1$}

\rput(2,-1){$G_{(i)}(\{0\},\{0\})^*$}
\end{pspicture*}
\end{minipage}

\caption{Possible projections for a torsion-free $(6,6)$-group.}\label{figure:projections66}
\end{figure}

So as to make the rendering of the tables better, let us introduce some notation. For groups in $\mathcal{G}'_{(i)}$, we use the following abbreviations:

\begin{table*}[h!]
\centering
\begin{tabular}{|c|c|}
\hline
Notation & Group\\
\hline
$X$ & $G_{(i)}(X,X)$\\
$X^*$ & $G_{(i)}(X,X)^*$\\
$X^{\prime*}$ & $G'_{(i)}(X,X)^*$\\
$X^{**}$ & $G_{(i)}(X^*,X^*)$\\
\hline
\end{tabular}
\end{table*}

Now for the local actions $\underline{H_j}(v_j) \leq \Sym(d_j)$, we need to give names to the conjugacy classes of subgroups of $\Sym(d_j)$. We will only give tables with $d_j \leq 6$, so we just need a notation for the conjugacy classes of subgroups of $\Sym(6)$. Indeed, each conjugacy class of subgroups of $\Sym(d)$ with $d < 6$ can also be seen as a conjugacy class in $\Sym(6)$ (by assuming that the $6-d$ other points are fixed). It can be computed that $\Sym(6)$ has exactly $56$ conjugacy classes of subgroups, and we give them names according to Table~\ref{table:classes}. The first part of the name of a conjugacy class is the order of a subgroup in that class, and we give for each one a set of generators of some representative subgroup. The classes of subgroups of $\Sym(3)$ (resp. $\Sym(4)$ and $\Sym(5)$) are marked with a $Y$ in the first (resp. second and third) column of the table.

The results of our computations are given in Tables~\ref{table:33T}--\ref{table:6600Tb}. Recall that, when $d_1 = d_2$, two $(d_1,d_2)$-groups can be conjugate by an element of $\Aut(T_1 \times T_2)$ exchanging $T_1$ and $T_2$. Hence, a group with projections $(H_1,H_2)$ is also conjugate to a group with projections $(H_2,H_1)$. For this reason, each equivalence class of $(d_1,d_2)$-groups appears twice or once in the table, depending on whether it is on the diagonal or not.

\begin{proof}[Proof of Theorem~\ref{maintheorem:66}]
See Tables~\ref{table:numbers}, \ref{table:66} and~\ref{table:66b}.
\end{proof}

\begin{table}
\setlength{\tabcolsep}{0.4em}
\renewcommand{\arraystretch}{0.75}
\footnotesize
\centering
\begin{tabular}{|c|c|c|c|c|c|}
\hline
3? & 4? & 5? & Name & Generators & Isomorphic to \Tstrut\Bstrut \\
\hline
 &  &  & 720.1 & (1,2,4,5)(3,6), (2,4), (1,2)(3,4) & $\Sym(6)$ \Tstrut \\
 &  &  & 360.1 & (1,2)(3,4), (1,2,4,5)(3,6) & $\Alt(6)$ \\
 &  &  & 120.2 & (1,3,2,4), (1,6,5,2,4) & $\mathrm{PGL}(2,5) \cong \Sym(5)$ \\
 &  & Y & 120.1 & (2,5), (1,2)(3,4,5) & $\Sym(5)$ \\
 &  &  & 72.1 & (2,3)(4,6,5), (1,4)(2,6,3,5) & $(\Sym(3) \times \Sym(3)) \rtimes \C_2$\\
 &  &  & 60.2 & (1,2)(3,4), (1,3,4)(2,5,6) & $\mathrm{PSL}(2,5) \cong \Sym(5)$ \\
 &  & Y & 60.1 & (1,2)(3,4), (2,3,5) & $\Alt(5)$ \\
 &  &  & 48.2 & (1,2)(3,6)(4,5), (1,3,4,2,5,6) & $\C_2 \times \Sym(4)$ \\
 &  &  & 48.1 & (1,2)(3,6,5), (1,2)(3,6,4,5) & $\C_2 \times \Sym(4)$ \\
 &  &  & 36.3 & (1,3)(5,6), (1,5,2,6)(3,4) & $(\C_3 \times \C_3) \rtimes \C_4$ \\
 &  &  & 36.2 & (2,3)(5,6), (1,3,2)(4,6) & $\Sym(3) \times \Sym(3)$ \\
 &  &  & 36.1 & (1,3)(4,6), (1,6,3,5,2,4) & $\Sym(3) \times \Sym(3)$ \\
 &  &  & 24.6 & (4,5,6), (1,2)(3,4,6,5) & $\Sym(4)$ \\
 &  &  & 24.5 & (1,3,6)(2,5,4), (1,2)(3,4,5,6) & $\Sym(4)$ \\
 & Y & Y & 24.4 & (1,4,2), (1,2), (3,4) & $\Sym(4)$ \\
 &  &  & 24.3 & (1,4,2,3), (1,6,2,5) & $\Sym(4)$ \\
 &  &  & 24.2 & (1,5,4,2,6,3), (1,5,4)(2,6,3) & $\C_2 \times \Alt(4)$ \\
 &  &  & 24.1 & (3,6,4), (1,2)(3,6,5) & $\C_2 \times \Alt(4)$ \\
 &  & Y & 20.1 & (2,4,5,3), (1,5)(2,4) & $\mathrm{GA}(1,5) \cong \C_5 \rtimes \C_4$ \\
 &  &  & 18.3 & (5,6), (1,2,3)(4,5) & $\C_3 \times \Sym(3)$ \\
 &  &  & 18.2 & (1,4,2,5,3,6), (1,3,2)(4,5,6) & $\C_3 \times \Sym(3)$ \\
 &  &  & 18.1 & (1,2,3), (1,3)(5,6), (4,5,6) & $(\C_3 \times \C_3) \rtimes \C_2$ \\
 &  &  & 16.1 & (5,6), (3,6)(4,5), (1,2)(3,4) & $\C_2 \times \D_8$ \\
 &  & Y & 12.4 & (1,2)(3,4), (1,2)(3,5,4) & $\C_2 \times \Sym(3) \cong \D_{12}$ \\
 &  &  & 12.3 & (1,3)(4,6), (1,4)(2,6)(3,5) & $\C_2 \times \Sym(3) \cong \D_{12}$ \\
 &  &  & 12.2 & (1,5,3)(2,6,4), (1,2)(3,4) & $\Alt(4)$ \\
 & Y & Y & 12.1 & (1,4,2), (1,2)(3,4) & $\Alt(4)$ \\
 &  & Y & 10.1 & (1,5)(2,4), (1,3)(4,5) & $\D_{10}$ \\
 &  &  & 9.1 & (1,2,3), (4,5,6) & $\C_3 \times \C_3$ \\
 &  &  & 8.7 & (3,6)(4,5), (1,2)(3,4) & $\D_8$ \\
 &  &  & 8.6 & (1,2), (3,4,5,6) & $\C_2 \times \C_4$ \\
 & Y & Y & 8.5 & (1,2), (1,4,2,3) & $\D_8$ \\
 &  &  & 8.4 & (1,2)(3,4,5,6), (4,6) & $\D_8$ \\
 &  &  & 8.3 & (3,5,4,6), (1,2)(3,4) & $\D_8$ \\
 &  &  & 8.2 & (3,4)(5,6), (3,6)(4,5), (1,2) & $\C_2 \times \C_2 \times \C_2$ \\
 &  &  & 8.1 & (1,2), (3,4), (5,6) & $\C_2 \times \C_2 \times \C_2$ \\
 &  &  & 6.6 & (1,3)(4,6), (1,2,3)(4,5,6) &  $\Sym(3)$ \\
 &  & Y & 6.5 & (1,2), (3,4,5) & $\C_6$ \\
 &  & Y & 6.4 & (3,4,5), (1,2)(3,4) & $\Sym(3)$ \\
 &  &  & 6.3 & (1,4,3,6,2,5) & $\C_6$ \\
Y & Y & Y & 6.2 & (1,2,3), (2,3) & $\Sym(3)$ \\
 &  &  & 6.1 & (1,2,3)(4,5,6), (1,6)(2,5)(3,4) & $\Sym(3)$ \\
 &  & Y & 5.1 & (1,3,5,2,4) & $\C_5$ \\
 &  &  & 4.7 & (1,2)(3,4,5,6) & $\C_4$ \\
 & Y & Y & 4.6 & (1,3,2,4) & $\C_4$ \\
 & Y & Y & 4.5 & (3,4), (1,2)(3,4) & $\C_2 \times \C_2$ \\
 &  &  & 4.4 & (3,4)(5,6), (1,2)(3,6)(4,5) & $\C_2 \times \C_2$ \\
 &  &  & 4.3 & (3,4)(5,6), (1,2) & $\C_2 \times \C_2$ \\
 & Y & Y & 4.2 & (1,4)(2,3), (1,2)(3,4) & $\C_2 \times \C_2$ \\
 &  &  & 4.1 & (3,4)(5,6), (1,2)(3,4) & $\C_2 \times \C_2$\\
 &  &  & 3.2 & (1,2,3)(4,5,6) & $\C_3$\\
Y & Y & Y & 3.1 & (1,2,3) & $\C_3$ \\
 & Y & Y & 2.3 & (1,2)(3,4) & $\C_2$ \\
 &  &  & 2.2 & (1,6)(2,5)(3,4) & $\C_2$ \\
Y & Y & Y & 2.1 & (1,2) & $\C_2$ \\
Y & Y & Y & 1.1 &  & 1\\
\hline
\end{tabular}
\caption{Conjugacy classes of subgroups of $\Sym(6)$.}\label{table:classes}
\end{table}

\begin{table}
\centering
\begin{tabular}{c|cc}
 & 6.2 & 2.1\\
\hline
6.2 & 3 & 1\Tstrut\\
2.1 & 1 & -\\
\end{tabular}
\caption{Possibly irreducible $(3,3)$-groups.}\label{table:33T}
\end{table}

\begin{table}
\centering
\begin{tabular}{c|cc}
 & 6.2 & 2.1\\
\hline
24.4 & 9 & 9\Tstrut\\
12.1 & 4 & 2\\
8.5 & 2 & 1\\
6.2 & 14 & 2\\
4.5 & 10 & -\\
2.1 & 6 & -\\
\end{tabular}
\caption{Possibly irreducible $(3,4)$-groups.}
\end{table}

\begin{table}
\centering
\begin{tabular}{c|cc}
 & 6.2 & 2.1\\
\hline
120.1 & 30 & 39\Tstrut\\
60.1 & 10 & 12\\
24.4 & 35 & 18\\
20.1 & 2 & 1\\
12.4 & 85 & 19\\
12.1 & 8 & 4\\
10.1 & 4 & 1\\
8.5 & 22 & 2\\
6.4 & 28 & 2\\
6.2 & 54 & 5\\
4.5 & 50 & -\\
4.2 & 2 & -\\
2.3 & 6 & -\\
2.1 & 18 & -\\
\end{tabular}
\caption{Possibly irreducible $(3,5)$-groups.}
\end{table}

\renewcommand{\arraystretch}{0.9}

\begin{table}
\centering
\small
\begin{tabular}{c|cc}
 & 6.2 & 2.1\\
\hline
$\{0,2\}^*$ & 7 & -\Tstrut\\
$\{2\}$ & 4 & -\\
$\{0,1\}^*$ & 22 & -\\
$\{0,1\}$ & 3 & -\\
$\{1\}^*$ & 4 & 22\\
$\{1\}$ & 1 & 7\\
$\{0\}^{**}$ & 2 & -\\
$\{0\}^*$ & 19 & 45\\
$\{0\}$ & 35 & 33\\
\hline
120.2 & 41 & 31\Tstrut\\
120.1 & 154 & 78\\
72.1 & 9 & 30\\
60.2 & 30 & 19\\
60.1 & 37 & 24\\
48.2 & 33 & 32\\
48.1 & 215 & 121\\
36.3 & - & 4\\
36.2 & 58 & 14\\
36.1 & 4 & 8\\
24.6 & 60 & 36\\
24.5 & 2 & 5\\
24.4 & 141 & 45\\
24.3 & 12 & 7\\
24.2 & 1 & 1\\
24.1 & 64 & 26\\
20.1 & 4 & 2\\
18.3 & 4 & -\\
18.2 & 4 & 4\\
18.1 & 9 & -\\
16.1 & 93 & 12\\
12.4 & 364 & 38\\
12.3 & 10 & 7\\
12.1 & 20 & 10\\
10.1 & 18 & 2\\
8.7 & 23 & 2\\
8.6 & 3 & -\\
8.5 & 108 & 5\\
8.4 & 13 & 5\\
8.3 & 54 & 2\\
8.2 & 15 & -\\
8.1 & 124 & -\\
6.6 & 23 & 3\\
6.4 & 88 & 4\\
6.2 & 136 & 8\\
4.5 & 210 & -\\
4.4 & 9 & -\\
4.3 & 106 & -\\
4.2 & 12 & -\\
4.1 & 36 & -\\
2.3 & 36 & -\\
2.1 & 50 & -\\
\end{tabular}
\caption{Possibly irreducible $(3,6)$-groups.}
\end{table}

\renewcommand{\arraystretch}{1}

\begin{table}
\centering
\begin{tabular}{c|ccc}
 & 24.4 & 12.1 & 8.5\\
\hline
24.4 & 1 & - & -\Tstrut\\
12.1 & - & - & 1\\
8.5 & - & 1 & -\\
\end{tabular}
\caption{Possibly irreducible torsion-free $(4,4)$-groups.}\label{table:44}
\end{table}

\begin{table}
\centering
\begin{tabular}{c|ccccccc}
 & 24.4 & 12.1 & 8.5 & 6.2 & 4.5 & 3.1 & 2.1\\
\hline
24.4 & 57 & 17 & 13 & 168 & 126 & 4 & 95\Tstrut\\
12.1 & 17 & 1 & 5 & 10 & 19 & - & 8\\
8.5 & 13 & 5 & 1 & 41 & 9 & 4 & 8\\
6.2 & 168 & 10 & 41 & 38 & 49 & - & 12\\
4.5 & 126 & 19 & 9 & 49 & 1 & - & -\\
3.1 & 4 & - & 4 & - & - & - & -\\
2.1 & 95 & 8 & 8 & 12 & - & - & -\\
\end{tabular}
\caption{Possibly irreducible $(4,4)$-groups with torsion.}
\end{table}

\begin{table}
\centering
\begin{tabular}{c|ccccccc}
 & 24.4 & 12.1 & 8.5 & 6.2 & 4.5 & 3.1 & 2.1\\
\hline
120.1 & 1833 & 64 & 81 & 1897 & 1436 & - & 1215\Tstrut\\
60.1 & 180 & 16 & 28 & 217 & 275 & - & 166\\
24.4 & 1739 & 49 & 301 & 781 & 679 & 4 & 190\\
20.1 & 28 & 2 & 2 & 13 & 21 & - & 9\\
12.4 & 2439 & 123 & 241 & 1431 & 613 & - & 192\\
12.1 & 122 & 2 & 48 & 20 & 84 & - & 16\\
10.1 & 36 & 2 & 2 & 30 & 18 & - & 7\\
8.5 & 480 & 38 & 109 & 269 & 50 & 4 & 16\\
6.5 & 66 & 1 & 16 & - & 4 & - & -\\
6.4 & 291 & 12 & 38 & 109 & 48 & - & 12\\
6.2 & 1496 & 36 & 277 & 300 & 181 & - & 30\\
4.6 & 4 & - & 6 & - & - & - & -\\
4.5 & 1879 & 80 & 117 & 255 & 3 & - & -\\
4.2 & 31 & 5 & 10 & 14 & - & - & -\\
3.1 & 30 & - & 32 & - & - & - & -\\
2.3 & 128 & 9 & 6 & 12 & - & - & -\\
2.1 & 592 & 20 & 35 & 36 & - & - & -\\
\end{tabular}
\caption{Possibly irreducible $(4,5)$-groups.}
\end{table}

\begin{table}
\centering
\begin{tabular}{c|cccc}
 & 24.4 & 12.1 & 8.5 & 4.5\\
\hline
$\{1\}$ & 3 & 2 & - & 1\Tstrut\\
$\{0\}^*$ & 2 & - & 3 & 1\\
$\{0\}$ & 2 & - & 2 & -\\
\hline
120.2 & 5 & - & 2 & 1\Tstrut\\
72.1 & 1 & - & - & -\\
60.2 & 4 & - & 3 & 1\\
48.2 & 5 & 2 & 1 & 4\\
48.1 & 2 & - & - & -\\
36.3 & - & 1 & 2 & 1\\
36.2 & 3 & 1 & 3 & 1\\
24.6 & 2 & - & - & -\\
24.5 & - & - & - & 1\\
24.4 & 5 & - & - & -\\
24.3 & 3 & 1 & 1 & 2\\
24.1 & - & - & 5 & -\\
12.3 & - & 1 & - & 1\\
12.1 & - & - & 6 & -\\
9.1 & 2 & - & 2 & -\\
8.7 & 1 & 1 & - & -\\
8.5 & 3 & 3 & - & -\\
8.4 & - & 2 & - & -\\
8.3 & 1 & 1 & - & -\\
6.6 & 4 & 2 & 1 & 1\\
\end{tabular}
\caption{Possibly irreducible torsion-free $(4,6)$-groups.}
\end{table}

\begin{table}
\centering
\begin{tabular}{c|ccccc}
 & 24.4 & 12.1 & 8.5 & 4.5 & 2.1\\
\hline
$\{1\}^*$ & 18 & - & 8 & - & -\Tstrut\\
$\{1\}$ & 11 & 2 & 4 & - & -\\
$\{0\}^{**}$ & 1 & 1 & - & - & -\\
$\{0\}^*$ & 13 & 2 & 10 & 15 & 6\\
$\{0\}$ & 9 & 3 & 3 & 3 & 2\\
\hline
120.2 & 18 & 2 & 11 & 11 & 4\Tstrut\\
72.1 & 13 & 2 & 10 & 5 & 2\\
60.2 & 5 & - & 3 & - & -\\
48.2 & 11 & - & 5 & - & -\\
48.1 & 20 & 4 & 6 & - & -\\
36.3 & 15 & - & 7 & 6 & 2\\
36.1 & 2 & - & - & - & -\\
24.6 & 21 & 3 & 12 & - & -\\
24.5 & 2 & - & 1 & - & -\\
24.4 & 15 & 3 & 12 & - & -\\
24.3 & 3 & 1 & 2 & - & -\\
24.2 & - & - & 2 & - & -\\
24.1 & 17 & - & 11 & - & -\\
18.2 & 6 & 1 & 6 & - & -\\
12.3 & 8 & - & 2 & - & -\\
12.2 & - & - & 1 & - & -\\
12.1 & 8 & - & 4 & - & -\\
8.7 & 15 & 3 & 2 & - & -\\
8.6 & 7 & - & 1 & - & -\\
8.5 & 19 & 3 & 6 & - & -\\
8.4 & 24 & 4 & 6 & - & -\\
8.3 & 17 & 3 & 2 & - & -\\
4.7 & 4 & - & - & - & -\\
4.6 & 4 & - & 2 & - & -\\
4.4 & 2 & - & - & - & -\\
4.3 & 2 & - & - & - & -\\
4.1 & 2 & - & - & - & -\\
\end{tabular}
\captionsetup{justification=centering}
\caption{Possibly irreducible $(4,6)$-groups\\ with torsion and $\tau_1 = \tau_2 = 0$.}
\end{table}

\begin{landscape}
\renewcommand{\arraystretch}{0.87}
\begin{table}
\scriptsize
\centering
\setlength{\tabcolsep}{0.4em}
\begin{tabular}{c|ccccccc|cccccccccccccc}
 & $\{0,1\}^*$ & $\{0,1\}$ & $\{1\}$ & $\{0\}^{**}$ & $\{0\}^{\prime*}$ & $\{0\}^*$ & $\{0\}$ & 120.2 & 72.1 & 60.2 & 48.2 & 48.1 & 36.3 & 36.2 & 36.1 & 24.6 & 24.5 & 24.4 & 24.3 & 24.2 & 24.1\\
\hline
$\{0,1\}^*$ & 80 & 125 & 218 & 19 & 13 & 145 & 103 & 71 & 9 & 28 & 128 & - & - & 99 & - & - & 13 & - & 29 & - & -\Tstrut\\
$\{0,1\}$ & 125 & 68 & 193 & 16 & 15 & 127 & 95 & 67 & 42 & 22 & 102 & - & - & 99 & 5 & - & 8 & - & 15 & - & -\\
$\{1\}$ & 218 & 193 & 210 & 19 & 14 & 277 & 185 & 115 & 28 & 44 & 31 & 158 & 2 & 70 & - & 48 & 4 & 170 & 12 & - & 108\\
$\{0\}^{**}$ & 19 & 16 & 19 & - & 2 & 22 & 15 & 3 & - & 2 & 12 & 12 & - & 5 & - & 5 & 1 & 10 & 3 & - & -\\
$\{0\}^{\prime*}$ & 13 & 15 & 14 & 2 & 2 & 7 & 6 & 9 & 5 & 3 & 6 & - & 2 & 5 & - & - & - & - & 3 & - & -\\
$\{0\}^*$ & 145 & 127 & 277 & 22 & 7 & 90 & 122 & 65 & 24 & 32 & 125 & 297 & 6 & 97 & - & 61 & 10 & 291 & 28 & - & 35\\
$\{0\}$ & 103 & 95 & 185 & 15 & 6 & 122 & 52 & 68 & 54 & 28 & 79 & 94 & - & 91 & 4 & 27 & 9 & 85 & 24 & - & 13\\
\hline
120.2 & 71 & 67 & 115 & 3 & 9 & 65 & 68 & 18 & 29 & 12 & 60 & 103 & 2 & 49 & - & 21 & 8 & 86 & 18 & - & 3\Tstrut\\
72.1 & 9 & 42 & 28 & - & 5 & 24 & 54 & 29 & - & 21 & 13 & 16 & - & 29 & - & 4 & 5 & 20 & 9 & 13 & 4\\
60.2 & 28 & 22 & 44 & 2 & 3 & 32 & 28 & 12 & 21 & 2 & 43 & 45 & - & 19 & 1 & 13 & 2 & 22 & 9 & - & 6\\
48.2 & 128 & 102 & 31 & 12 & 6 & 125 & 79 & 60 & 13 & 43 & 101 & 203 & - & 106 & - & 42 & 17 & 212 & 46 & - & 125\\
48.1 & - & - & 158 & 12 & - & 297 & 94 & 103 & 16 & 45 & 203 & 37 & - & 26 & 2 & 19 & 23 & 70 & 36 & 2 & -\\
36.3 & - & - & 2 & - & 2 & 6 & - & 2 & - & - & - & - & - & - & - & - & - & - & - & - & 12\\
36.2 & 99 & 99 & 70 & 5 & 5 & 97 & 91 & 49 & 29 & 19 & 106 & 26 & - & 16 & 4 & 10 & 16 & 12 & 24 & 20 & 12\\
36.1 & - & 5 & - & - & - & - & 4 & - & - & 1 & - & 2 & - & 4 & - & - & - & 4 & 1 & - & 2\\
24.6 & - & - & 48 & 5 & - & 61 & 27 & 21 & 4 & 13 & 42 & 19 & - & 10 & - & 3 & 9 & 11 & 8 & 2 & -\\
24.5 & 13 & 8 & 4 & 1 & - & 10 & 9 & 8 & 5 & 2 & 17 & 23 & - & 16 & - & 9 & 2 & 26 & 7 & - & 16\\
24.4 & - & - & 170 & 10 & - & 291 & 85 & 86 & 20 & 22 & 212 & 70 & - & 12 & 4 & 11 & 26 & 19 & 32 & - & -\\
24.3 & 29 & 15 & 12 & 3 & 3 & 28 & 24 & 18 & 9 & 9 & 46 & 36 & - & 24 & 1 & 8 & 7 & 32 & 8 & - & 15\\
24.2 & - & - & - & - & - & - & - & - & 13 & - & - & 2 & - & 20 & - & 2 & - & - & - & - & 2\\
24.1 & - & - & 108 & - & - & 35 & 13 & 3 & 4 & 6 & 125 & - & 12 & 12 & 2 & - & 16 & - & 15 & 2 & -\\
18.3 & - & - & 10 & 2 & 2 & 1 & 1 & 1 & - & 1 & - & - & 1 & - & - & - & - & - & - & - & -\\
18.2 & - & 4 & 9 & - & 3 & - & 9 & 2 & - & - & 7 & 4 & - & 6 & - & - & 2 & 8 & 2 & 3 & -\\
18.1 & 21 & 18 & 54 & 2 & 2 & 25 & 21 & 6 & 5 & 2 & 32 & 16 & - & 6 & - & 4 & 3 & 8 & 6 & - & 4\\
16.1 & - & - & 67 & - & - & 283 & 78 & 35 & 11 & 20 & 145 & 29 & 8 & 31 & 1 & 5 & 10 & 30 & 25 & 8 & 54\\
12.3 & - & 6 & 2 & - & 2 & 1 & 12 & 5 & - & 6 & 11 & 6 & - & 9 & - & - & 2 & 12 & 5 & - & 16\\
12.2 & - & - & - & - & - & - & - & - & 4 & - & - & - & - & 2 & - & - & - & - & - & - & -\\
12.1 & - & - & 68 & - & - & 28 & 12 & 2 & 8 & 2 & 130 & - & 4 & 4 & 2 & - & 14 & - & 16 & - & -\\
9.1 & 10 & 10 & 6 & 1 & 1 & 11 & 9 & 3 & 4 & 3 & 26 & 10 & - & 4 & - & 2 & 5 & 4 & 7 & 9 & -\\
8.7 & - & - & 14 & - & - & 81 & 24 & 7 & 5 & 7 & 39 & 9 & 4 & 11 & - & 4 & 7 & 6 & 5 & 2 & 13\\
8.6 & - & - & 18 & - & - & 36 & 6 & 11 & 4 & 4 & 44 & - & 8 & 8 & - & - & 7 & - & 6 & 3 & -\\
8.5 & - & - & 68 & - & - & 429 & 139 & 55 & 33 & 19 & 254 & 30 & 8 & 25 & 2 & 18 & 35 & 18 & 41 & - & 48\\
8.4 & - & - & - & - & - & 54 & 57 & 19 & - & 15 & 7 & - & 8 & 12 & - & - & 2 & - & 4 & - & 22\\
8.3 & - & - & 35 & - & - & 96 & 30 & 12 & 2 & 6 & 43 & 7 & 4 & 11 & - & 3 & 8 & 6 & 5 & 2 & 10\\
8.2 & - & - & 2 & - & - & 19 & 8 & 3 & 2 & - & 5 & - & 2 & 2 & - & - & - & - & - & 1 & -\\
8.1 & - & - & 68 & - & - & 55 & 22 & 11 & - & 11 & 93 & - & 8 & 8 & 1 & - & 17 & - & 14 & 1 & -\\
6.6 & 39 & 40 & 10 & 2 & 2 & 40 & 38 & 17 & 10 & 7 & 66 & 30 & - & 9 & 1 & 10 & 11 & 12 & 16 & - & 20\\
4.7 & - & - & - & - & - & 4 & 4 & - & - & - & - & - & - & - & - & - & - & - & - & - & -\\
4.6 & - & - & 14 & - & - & 46 & 14 & 10 & 8 & 2 & 57 & - & 4 & 4 & - & - & 8 & - & 9 & - & -\\
4.5 & - & - & 44 & - & - & 72 & 16 & 12 & 8 & 4 & 136 & - & 4 & 4 & - & - & 20 & - & 24 & - & -\\
4.4 & - & - & - & - & - & 5 & 7 & 2 & - & - & 3 & - & 2 & 2 & - & - & 1 & - & - & - & -\\
4.3 & - & - & 48 & - & - & 88 & 12 & 10 & 4 & 6 & 102 & - & 6 & 6 & - & - & 18 & - & 14 & 1 & -\\
4.2 & - & - & - & - & - & 30 & 12 & 4 & 4 & - & 10 & - & 2 & 2 & - & - & 2 & - & - & - & -\\
4.1 & - & - & 10 & - & - & 16 & 2 & 4 & 2 & 2 & 34 & - & 2 & 2 & - & - & 4 & - & 6 & - & -\\
2.3 & - & - & 24 & - & - & 27 & 9 & 6 & 4 & 2 & 71 & - & 2 & 2 & - & - & 11 & - & 12 & - & -\\
\end{tabular}
\caption{Possibly irreducible torsion-free $(6,6)$-groups. (Part 1/2)}\label{table:66}
\end{table}

\begin{table}
\scriptsize
\centering
\setlength{\tabcolsep}{0.4em}
\begin{tabular}{c|cccccccccccccccccccccccc}
 & 18.3 & 18.2 & 18.1 & 16.1 & 12.3 & 12.2 & 12.1 & 9.1 & 8.7 & 8.6 & 8.5 & 8.4 & 8.3 & 8.2 & 8.1 & 6.6 & 4.7 & 4.6 & 4.5 & 4.4 & 4.3 & 4.2 & 4.1 & 2.3\\
\hline
$\{0,1\}^*$ & - & - & 21 & - & - & - & - & 10 & - & - & - & - & - & - & - & 39 & - & - & - & - & - & - & - & -\Tstrut\\
$\{0,1\}$ & - & 4 & 18 & - & 6 & - & - & 10 & - & - & - & - & - & - & - & 40 & - & - & - & - & - & - & - & -\\
$\{1\}$ & 10 & 9 & 54 & 67 & 2 & - & 68 & 6 & 14 & 18 & 68 & - & 35 & 2 & 68 & 10 & - & 14 & 44 & - & 48 & - & 10 & 24\\
$\{0\}^{**}$ & 2 & - & 2 & - & - & - & - & 1 & - & - & - & - & - & - & - & 2 & - & - & - & - & - & - & - & -\\
$\{0\}^{\prime*}$ & 2 & 3 & 2 & - & 2 & - & - & 1 & - & - & - & - & - & - & - & 2 & - & - & - & - & - & - & - & -\\
$\{0\}^*$ & 1 & - & 25 & 283 & 1 & - & 28 & 11 & 81 & 36 & 429 & 54 & 96 & 19 & 55 & 40 & 4 & 46 & 72 & 5 & 88 & 30 & 16 & 27\\
$\{0\}$ & 1 & 9 & 21 & 78 & 12 & - & 12 & 9 & 24 & 6 & 139 & 57 & 30 & 8 & 22 & 38 & 4 & 14 & 16 & 7 & 12 & 12 & 2 & 9\\
\hline
120.2 & 1 & 2 & 6 & 35 & 5 & - & 2 & 3 & 7 & 11 & 55 & 19 & 12 & 3 & 11 & 17 & - & 10 & 12 & 2 & 10 & 4 & 4 & 6\Tstrut\\
72.1 & - & - & 5 & 11 & - & 4 & 8 & 4 & 5 & 4 & 33 & - & 2 & 2 & - & 10 & - & 8 & 8 & - & 4 & 4 & 2 & 4\\
60.2 & 1 & - & 2 & 20 & 6 & - & 2 & 3 & 7 & 4 & 19 & 15 & 6 & - & 11 & 7 & - & 2 & 4 & - & 6 & - & 2 & 2\\
48.2 & - & 7 & 32 & 145 & 11 & - & 130 & 26 & 39 & 44 & 254 & 7 & 43 & 5 & 93 & 66 & - & 57 & 136 & 3 & 102 & 10 & 34 & 71\\
48.1 & - & 4 & 16 & 29 & 6 & - & - & 10 & 9 & - & 30 & - & 7 & - & - & 30 & - & - & - & - & - & - & - & -\\
36.3 & 1 & - & - & 8 & - & - & 4 & - & 4 & 8 & 8 & 8 & 4 & 2 & 8 & - & - & 4 & 4 & 2 & 6 & 2 & 2 & 2\\
36.2 & - & 6 & 6 & 31 & 9 & 2 & 4 & 4 & 11 & 8 & 25 & 12 & 11 & 2 & 8 & 9 & - & 4 & 4 & 2 & 6 & 2 & 2 & 2\\
36.1 & - & - & - & 1 & - & - & 2 & - & - & - & 2 & - & - & - & 1 & 1 & - & - & - & - & - & - & - & -\\
24.6 & - & - & 4 & 5 & - & - & - & 2 & 4 & - & 18 & - & 3 & - & - & 10 & - & - & - & - & - & - & - & -\\
24.5 & - & 2 & 3 & 10 & 2 & - & 14 & 5 & 7 & 7 & 35 & 2 & 8 & - & 17 & 11 & - & 8 & 20 & 1 & 18 & 2 & 4 & 11\\
24.4 & - & 8 & 8 & 30 & 12 & - & - & 4 & 6 & - & 18 & - & 6 & - & - & 12 & - & - & - & - & - & - & - & -\\
24.3 & - & 2 & 6 & 25 & 5 & - & 16 & 7 & 5 & 6 & 41 & 4 & 5 & - & 14 & 16 & - & 9 & 24 & - & 14 & - & 6 & 12\\
24.2 & - & 3 & - & 8 & - & - & - & 9 & 2 & 3 & - & - & 2 & 1 & 1 & - & - & - & - & - & 1 & - & - & -\\
24.1 & - & - & 4 & 54 & 16 & - & - & - & 13 & - & 48 & 22 & 10 & - & - & 20 & - & - & - & - & - & - & - & -\\
18.3 & 1 & - & - & - & - & - & - & - & - & - & - & - & - & - & - & - & - & - & - & - & - & - & - & -\\
18.2 & - & - & - & 2 & - & 1 & - & - & 1 & - & 8 & - & 1 & - & - & 4 & - & - & - & - & - & - & - & -\\
18.1 & - & - & - & 6 & 3 & - & 2 & - & 2 & - & 6 & 5 & 2 & - & 1 & 3 & - & - & - & - & - & - & - & -\\
16.1 & - & 2 & 6 & - & 11 & - & 42 & 4 & - & - & - & - & - & - & - & 30 & - & - & - & - & - & - & - & -\\
12.3 & - & - & 3 & 11 & - & - & 14 & 2 & 3 & 2 & 26 & - & 4 & 1 & 9 & 6 & - & 4 & 12 & - & 8 & 2 & 4 & 6\\
12.2 & - & 1 & - & - & - & - & - & 1 & - & - & - & - & - & - & - & - & - & - & - & - & - & - & - & -\\
12.1 & - & - & 2 & 42 & 14 & - & - & - & 6 & - & 18 & 12 & 6 & - & - & 6 & - & - & - & - & - & - & - & -\\
9.1 & - & - & - & 4 & 2 & 1 & - & - & 2 & - & 4 & 4 & 2 & - & - & 2 & - & - & - & - & - & - & - & -\\
8.7 & - & 1 & 2 & - & 3 & - & 6 & 2 & - & - & - & - & - & - & - & 10 & - & - & - & - & - & - & - & -\\
8.6 & - & - & - & - & 2 & - & - & - & - & - & - & - & - & - & - & 8 & - & - & - & - & - & - & - & -\\
8.5 & - & 8 & 6 & - & 26 & - & 18 & 4 & - & - & - & - & - & - & - & 22 & - & - & - & - & - & - & - & -\\
8.4 & - & - & 5 & - & - & - & 12 & 4 & - & - & - & - & - & - & - & 3 & - & - & - & - & - & - & - & -\\
8.3 & - & 1 & 2 & - & 4 & - & 6 & 2 & - & - & - & - & - & - & - & 10 & - & - & - & - & - & - & - & -\\
8.2 & - & - & - & - & 1 & - & - & - & - & - & - & - & - & - & - & 2 & - & - & - & - & - & - & - & -\\
8.1 & - & - & 1 & - & 9 & - & - & - & - & - & - & - & - & - & - & 7 & - & - & - & - & - & - & - & -\\
6.6 & - & 4 & 3 & 30 & 6 & - & 6 & 2 & 10 & 8 & 22 & 3 & 10 & 2 & 7 & 4 & - & 4 & 4 & 2 & 6 & 2 & 2 & 2\\
4.7 & - & - & - & - & - & - & - & - & - & - & - & - & - & - & - & - & - & - & - & - & - & - & - & -\\
4.6 & - & - & - & - & 4 & - & - & - & - & - & - & - & - & - & - & 4 & - & - & - & - & - & - & - & -\\
4.5 & - & - & - & - & 12 & - & - & - & - & - & - & - & - & - & - & 4 & - & - & - & - & - & - & - & -\\
4.4 & - & - & - & - & - & - & - & - & - & - & - & - & - & - & - & 2 & - & - & - & - & - & - & - & -\\
4.3 & - & - & - & - & 8 & - & - & - & - & - & - & - & - & - & - & 6 & - & - & - & - & - & - & - & -\\
4.2 & - & - & - & - & 2 & - & - & - & - & - & - & - & - & - & - & 2 & - & - & - & - & - & - & - & -\\
4.1 & - & - & - & - & 4 & - & - & - & - & - & - & - & - & - & - & 2 & - & - & - & - & - & - & - & -\\
2.3 & - & - & - & - & 6 & - & - & - & - & - & - & - & - & - & - & 2 & - & - & - & - & - & - & - & -\\
\end{tabular}
\caption{Possibly irreducible torsion-free $(6,6)$-groups. (Part 2/2)}\label{table:66b}
\end{table}
\renewcommand{\arraystretch}{1}
\end{landscape}

\begin{landscape}
\renewcommand{\arraystretch}{0.87}
\begin{table}
\scriptsize
\centering
\setlength{\tabcolsep}{0.4em}
\begin{tabular}{c|ccccccccccc|cccccccccc}
 & $\{0,2\}^*$ & $\{0,2\}$ & $\{2\}$ & $\{0,1\}^*$ & $\{0,1\}$ & $\{1\}^{**}$ & $\{1\}^*$ & $\{1\}$ & $\{0\}^{**}$ & $\{0\}^*$ & $\{0\}$ & 120.2 & 72.1 & 60.2 & 48.2 & 48.1 & 36.3 & 36.1 & 24.6 & 24.5 & 24.4\\
\hline
$\{0,2\}^*$ & 190 & 330 & 764 & 267 & 361 & 69 & - & 603 & 61 & 465 & 538 & 179 & 579 & - & 150 & - & 22 & 47 & - & 2 & -\Tstrut\\
$\{0,2\}$ & 330 & 174 & 667 & 321 & 402 & 46 & - & 591 & 48 & 450 & 467 & 173 & 486 & - & 123 & - & 28 & 47 & - & 1 & -\\
$\{2\}$ & 764 & 667 & 641 & 708 & 742 & 57 & - & 1167 & 82 & 898 & 723 & 344 & 859 & - & 28 & - & 289 & 53 & - & 2 & -\\
$\{0,1\}^*$ & 267 & 321 & 708 & 191 & 377 & 62 & - & 596 & 70 & 564 & 479 & 250 & 507 & 29 & 352 & - & 76 & 45 & - & 26 & -\\
$\{0,1\}$ & 361 & 402 & 742 & 377 & 266 & 44 & - & 670 & 69 & 569 & 645 & 274 & 733 & 29 & 422 & - & 64 & 56 & - & 31 & -\\
$\{1\}^{**}$ & 69 & 46 & 57 & 62 & 44 & 5 & - & 47 & 4 & 48 & 71 & 13 & 61 & - & - & - & 22 & - & - & - & -\\
$\{1\}^*$ & - & - & - & - & - & - & 76 & 69 & 9 & 181 & 32 & 17 & 85 & - & 5 & 4179 & 23 & 6 & 998 & 1 & 1424\\
$\{1\}$ & 603 & 591 & 1167 & 596 & 670 & 47 & 69 & 447 & 76 & 1004 & 777 & 374 & 581 & 92 & 134 & 2558 & 115 & 26 & 805 & 20 & 834\\
$\{0\}^{**}$ & 61 & 48 & 82 & 70 & 69 & 4 & 9 & 76 & 5 & 70 & 90 & 30 & 71 & 5 & 15 & 102 & 14 & - & 19 & 2 & 22\\
$\{0\}^*$ & 465 & 450 & 898 & 564 & 569 & 48 & 181 & 1004 & 70 & 521 & 705 & 347 & 972 & 28 & 308 & 6322 & 167 & 80 & 1694 & 31 & 1986\\
$\{0\}$ & 538 & 467 & 723 & 479 & 645 & 71 & 32 & 777 & 90 & 705 & 314 & 300 & 529 & 25 & 229 & 1632 & 110 & 39 & 407 & 14 & 620\\
\hline
120.2 & 179 & 173 & 344 & 250 & 274 & 13 & 17 & 374 & 30 & 347 & 300 & 70 & 358 & 12 & 142 & 1134 & 78 & 18 & 363 & 14 & 345\Tstrut\\
72.1 & 579 & 486 & 859 & 507 & 733 & 61 & 85 & 581 & 71 & 972 & 529 & 358 & 575 & 29 & 263 & 1183 & 193 & 53 & 397 & 20 & 409\\
60.2 & - & - & - & 29 & 29 & - & - & 92 & 5 & 28 & 25 & 12 & 29 & 4 & 11 & 133 & 6 & 2 & 30 & 2 & 33\\
48.2 & 150 & 123 & 28 & 352 & 422 & - & 5 & 134 & 15 & 308 & 229 & 142 & 263 & 11 & 106 & 1199 & 36 & 17 & 390 & 22 & 435\\
48.1 & - & - & - & - & - & - & 4179 & 2558 & 102 & 6322 & 1632 & 1134 & 1183 & 133 & 1199 & 2662 & 310 & 185 & 1136 & 114 & 1009\\
36.3 & 22 & 28 & 289 & 76 & 64 & 22 & 23 & 115 & 14 & 167 & 110 & 78 & 193 & 6 & 36 & 310 & 30 & 11 & 93 & 3 & 89\\
36.1 & 47 & 47 & 53 & 45 & 56 & - & 6 & 26 & - & 80 & 39 & 18 & 53 & 2 & 17 & 185 & 11 & 3 & 49 & 1 & 75\\
24.6 & - & - & - & - & - & - & 998 & 805 & 19 & 1694 & 407 & 363 & 397 & 30 & 390 & 1136 & 93 & 49 & 127 & 33 & 231\\
24.5 & 2 & 1 & 2 & 26 & 31 & - & 1 & 20 & 2 & 31 & 14 & 14 & 20 & 2 & 22 & 114 & 3 & 1 & 33 & 3 & 43\\
24.4 & - & - & - & - & - & - & 1424 & 834 & 22 & 1986 & 620 & 345 & 409 & 33 & 435 & 1009 & 89 & 75 & 231 & 43 & 107\\
24.3 & 3 & 2 & 2 & 40 & 55 & - & 1 & 32 & 5 & 48 & 31 & 27 & 25 & 2 & 24 & 156 & 3 & 1 & 51 & 6 & 74\\
24.2 & - & - & - & - & - & - & - & - & - & - & - & - & 56 & - & - & 12 & 8 & - & 12 & - & 117\\
24.1 & - & - & - & - & - & - & - & 928 & 42 & 751 & 200 & 141 & 130 & - & 190 & 602 & 15 & 23 & 112 & 25 & 126\\
18.2 & 60 & 41 & - & 19 & 26 & - & 2 & 68 & 5 & 95 & 48 & 50 & 94 & 2 & 41 & 96 & 22 & 2 & 34 & 11 & 27\\
16.1 & - & - & - & - & - & - & 1122 & 1269 & 66 & 2805 & 625 & 525 & 465 & 53 & 553 & 1954 & 147 & 53 & 464 & 75 & 542\\
12.3 & 86 & 50 & - & 40 & 102 & - & 1 & 38 & 5 & 93 & 36 & 37 & 58 & 1 & 45 & 182 & 8 & 3 & 60 & 4 & 60\\
12.2 & - & - & - & - & - & - & - & - & - & - & - & - & 3 & - & - & 2 & 1 & - & 2 & - & 7\\
12.1 & - & - & - & - & - & - & - & 206 & 10 & 186 & 56 & 32 & 28 & - & 36 & 114 & - & 5 & 22 & 4 & 24\\
8.7 & - & - & - & - & - & - & 270 & 276 & 25 & 749 & 169 & 122 & 105 & 14 & 156 & 369 & 37 & 16 & 82 & 23 & 102\\
8.6 & - & - & - & - & - & - & 161 & 224 & 26 & 477 & 120 & 139 & 119 & 8 & 134 & 258 & 48 & 3 & 26 & 23 & 48\\
8.5 & - & - & - & - & - & - & 704 & 560 & 40 & 1312 & 319 & 232 & 229 & 31 & 338 & 429 & 47 & 36 & 101 & 51 & 124\\
8.4 & - & - & - & - & - & - & 822 & 484 & 20 & 1787 & 303 & 316 & 330 & 25 & 186 & 504 & 61 & 38 & 126 & 15 & 150\\
8.3 & - & - & - & - & - & - & 284 & 446 & 25 & 851 & 194 & 138 & 127 & 15 & 156 & 443 & 36 & 18 & 91 & 20 & 118\\
8.2 & - & - & - & - & - & - & - & 106 & 10 & 121 & 34 & 31 & 20 & 2 & 34 & 72 & 10 & - & 8 & 5 & 16\\
8.1 & - & - & - & - & - & - & - & 227 & 4 & 1680 & 206 & 204 & 138 & - & 123 & 18 & 53 & 4 & 2 & 14 & 4\\
4.7 & - & - & - & - & - & - & 62 & - & - & 132 & 28 & - & 12 & - & 4 & 44 & - & - & 4 & - & 8\\
4.6 & - & - & - & - & - & - & 48 & 74 & 8 & 142 & 42 & 32 & 30 & 4 & 67 & 44 & 6 & - & 4 & 12 & 8\\
4.5 & - & - & - & - & - & - & - & - & - & 1598 & 202 & 170 & 120 & - & - & - & 44 & - & - & - & -\\
4.4 & - & - & - & - & - & - & - & 102 & 12 & 128 & 22 & 30 & 23 & 2 & 24 & 36 & 5 & 2 & 4 & 3 & 8\\
4.3 & - & - & - & - & - & - & - & 142 & 8 & 1401 & 200 & 192 & 118 & - & 79 & 36 & 58 & 2 & 4 & 10 & 8\\
4.2 & - & - & - & - & - & - & - & 64 & 8 & 51 & 15 & 14 & 10 & 2 & 28 & - & 2 & - & - & 4 & -\\
4.1 & - & - & - & - & - & - & - & 32 & 4 & 416 & 58 & 53 & 36 & - & 22 & 18 & 17 & - & 2 & 4 & 4\\
2.3 & - & - & - & - & - & - & - & - & - & 157 & 23 & 20 & 14 & - & - & - & 6 & - & - & - & -\\
2.1 & - & - & - & - & - & - & - & - & - & 245 & 37 & 29 & 21 & - & - & - & 8 & - & - & - & -\\
\end{tabular}
\caption{Possibly irreducible $(6,6)$-groups with torsion and $\tau_1 = \tau_2 = 0$. (Part 1/2)}\label{table:6600T}
\end{table}

\begin{table}
\scriptsize
\centering
\setlength{\tabcolsep}{0.4em}
\begin{tabular}{c|cccccccccccccccccccccccc}
 & 24.3 & 24.2 & 24.1 & 18.2 & 16.1 & 12.3 & 12.2 & 12.1 & 8.7 & 8.6 & 8.5 & 8.4 & 8.3 & 8.2 & 8.1 & 4.7 & 4.6 & 4.5 & 4.4 & 4.3 & 4.2 & 4.1 & 2.3 & 2.1\\
\hline
$\{0,2\}^*$ & 3 & - & - & 60 & - & 86 & - & - & - & - & - & - & - & - & - & - & - & - & - & - & - & - & - & -\Tstrut\\
$\{0,2\}$ & 2 & - & - & 41 & - & 50 & - & - & - & - & - & - & - & - & - & - & - & - & - & - & - & - & - & -\\
$\{2\}$ & 2 & - & - & - & - & - & - & - & - & - & - & - & - & - & - & - & - & - & - & - & - & - & - & -\\
$\{0,1\}^*$ & 40 & - & - & 19 & - & 40 & - & - & - & - & - & - & - & - & - & - & - & - & - & - & - & - & - & -\\
$\{0,1\}$ & 55 & - & - & 26 & - & 102 & - & - & - & - & - & - & - & - & - & - & - & - & - & - & - & - & - & -\\
$\{1\}^{**}$ & - & - & - & - & - & - & - & - & - & - & - & - & - & - & - & - & - & - & - & - & - & - & - & -\\
$\{1\}^*$ & 1 & - & - & 2 & 1122 & 1 & - & - & 270 & 161 & 704 & 822 & 284 & - & - & 62 & 48 & - & - & - & - & - & - & -\\
$\{1\}$ & 32 & - & 928 & 68 & 1269 & 38 & - & 206 & 276 & 224 & 560 & 484 & 446 & 106 & 227 & - & 74 & - & 102 & 142 & 64 & 32 & - & -\\
$\{0\}^{**}$ & 5 & - & 42 & 5 & 66 & 5 & - & 10 & 25 & 26 & 40 & 20 & 25 & 10 & 4 & - & 8 & - & 12 & 8 & 8 & 4 & - & -\\
$\{0\}^*$ & 48 & - & 751 & 95 & 2805 & 93 & - & 186 & 749 & 477 & 1312 & 1787 & 851 & 121 & 1680 & 132 & 142 & 1598 & 128 & 1401 & 51 & 416 & 157 & 245\\
$\{0\}$ & 31 & - & 200 & 48 & 625 & 36 & - & 56 & 169 & 120 & 319 & 303 & 194 & 34 & 206 & 28 & 42 & 202 & 22 & 200 & 15 & 58 & 23 & 37\\
\hline
120.2 & 27 & - & 141 & 50 & 525 & 37 & - & 32 & 122 & 139 & 232 & 316 & 138 & 31 & 204 & - & 32 & 170 & 30 & 192 & 14 & 53 & 20 & 29\Tstrut\\
72.1 & 25 & 56 & 130 & 94 & 465 & 58 & 3 & 28 & 105 & 119 & 229 & 330 & 127 & 20 & 138 & 12 & 30 & 120 & 23 & 118 & 10 & 36 & 14 & 21\\
60.2 & 2 & - & - & 2 & 53 & 1 & - & - & 14 & 8 & 31 & 25 & 15 & 2 & - & - & 4 & - & 2 & - & 2 & - & - & -\\
48.2 & 24 & - & 190 & 41 & 553 & 45 & - & 36 & 156 & 134 & 338 & 186 & 156 & 34 & 123 & 4 & 67 & - & 24 & 79 & 28 & 22 & - & -\\
48.1 & 156 & 12 & 602 & 96 & 1954 & 182 & 2 & 114 & 369 & 258 & 429 & 504 & 443 & 72 & 18 & 44 & 44 & - & 36 & 36 & - & 18 & - & -\\
36.3 & 3 & 8 & 15 & 22 & 147 & 8 & 1 & - & 37 & 48 & 47 & 61 & 36 & 10 & 53 & - & 6 & 44 & 5 & 58 & 2 & 17 & 6 & 8\\
36.1 & 1 & - & 23 & 2 & 53 & 3 & - & 5 & 16 & 3 & 36 & 38 & 18 & - & 4 & - & - & - & 2 & 2 & - & - & - & -\\
24.6 & 51 & 12 & 112 & 34 & 464 & 60 & 2 & 22 & 82 & 26 & 101 & 126 & 91 & 8 & 2 & 4 & 4 & - & 4 & 4 & - & 2 & - & -\\
24.5 & 6 & - & 25 & 11 & 75 & 4 & - & 4 & 23 & 23 & 51 & 15 & 20 & 5 & 14 & - & 12 & - & 3 & 10 & 4 & 4 & - & -\\
24.4 & 74 & 117 & 126 & 27 & 542 & 60 & 7 & 24 & 102 & 48 & 124 & 150 & 118 & 16 & 4 & 8 & 8 & - & 8 & 8 & - & 4 & - & -\\
24.3 & 5 & - & 46 & 15 & 82 & 6 & - & 8 & 26 & 22 & 53 & 26 & 24 & 6 & 24 & - & 9 & - & 4 & 16 & 6 & 4 & - & -\\
24.2 & - & - & 6 & 10 & 54 & - & - & - & 11 & 10 & 44 & 22 & 13 & 3 & 1 & - & 15 & - & 2 & 1 & - & - & - & -\\
24.1 & 46 & 6 & - & 20 & 366 & 26 & - & - & 71 & - & 78 & 104 & 74 & - & - & - & - & - & - & - & - & - & - & -\\
18.2 & 15 & 10 & 20 & 10 & 38 & 8 & - & 5 & 15 & - & 26 & 44 & 15 & - & - & - & - & - & - & - & - & - & - & -\\
16.1 & 82 & 54 & 366 & 38 & 148 & 67 & 6 & 58 & 34 & 32 & 108 & 95 & 46 & - & - & - & 16 & - & - & - & - & - & - & -\\
12.3 & 6 & - & 26 & 8 & 67 & 5 & - & - & 17 & 14 & 28 & 19 & 24 & 5 & 12 & - & 4 & - & 6 & 10 & 4 & 2 & - & -\\
12.2 & - & - & - & - & 6 & - & - & - & 1 & 1 & 4 & 4 & 1 & - & - & - & 1 & - & - & - & - & - & - & -\\
12.1 & 8 & - & - & 5 & 58 & - & - & - & 14 & - & 12 & 18 & 14 & - & - & - & - & - & - & - & - & - & - & -\\
8.7 & 26 & 11 & 71 & 15 & 34 & 17 & 1 & 14 & 3 & 4 & 14 & 12 & 6 & - & - & - & 2 & - & - & - & - & - & - & -\\
8.6 & 22 & 10 & - & - & 32 & 14 & 1 & - & 4 & - & 8 & 8 & 4 & - & - & - & - & - & - & - & - & - & - & -\\
8.5 & 53 & 44 & 78 & 26 & 108 & 28 & 4 & 12 & 14 & 8 & 22 & 36 & 18 & - & - & - & 4 & - & - & - & - & - & - & -\\
8.4 & 26 & 22 & 104 & 44 & 95 & 19 & 4 & 18 & 12 & 8 & 36 & 19 & 16 & - & - & - & 4 & - & - & - & - & - & - & -\\
8.3 & 24 & 13 & 74 & 15 & 46 & 24 & 1 & 14 & 6 & 4 & 18 & 16 & 5 & - & - & - & 2 & - & - & - & - & - & - & -\\
8.2 & 6 & 3 & - & - & - & 5 & - & - & - & - & - & - & - & - & - & - & - & - & - & - & - & - & - & -\\
8.1 & 24 & 1 & - & - & - & 12 & - & - & - & - & - & - & - & - & - & - & - & - & - & - & - & - & - & -\\
4.7 & - & - & - & - & - & - & - & - & - & - & - & - & - & - & - & - & - & - & - & - & - & - & - & -\\
4.6 & 9 & 15 & - & - & 16 & 4 & 1 & - & 2 & - & 4 & 4 & 2 & - & - & - & - & - & - & - & - & - & - & -\\
4.5 & - & - & - & - & - & - & - & - & - & - & - & - & - & - & - & - & - & - & - & - & - & - & - & -\\
4.4 & 4 & 2 & - & - & - & 6 & - & - & - & - & - & - & - & - & - & - & - & - & - & - & - & - & - & -\\
4.3 & 16 & 1 & - & - & - & 10 & - & - & - & - & - & - & - & - & - & - & - & - & - & - & - & - & - & -\\
4.2 & 6 & - & - & - & - & 4 & - & - & - & - & - & - & - & - & - & - & - & - & - & - & - & - & - & -\\
4.1 & 4 & - & - & - & - & 2 & - & - & - & - & - & - & - & - & - & - & - & - & - & - & - & - & - & -\\
2.3 & - & - & - & - & - & - & - & - & - & - & - & - & - & - & - & - & - & - & - & - & - & - & - & -\\
2.1 & - & - & - & - & - & - & - & - & - & - & - & - & - & - & - & - & - & - & - & - & - & - & - & -\\
\end{tabular}
\caption{Possibly irreducible $(6,6)$-groups with torsion and $\tau_1 = \tau_2 = 0$. (Part 2/2)}\label{table:6600Tb}
\end{table}
\renewcommand{\arraystretch}{1}
\end{landscape}

\section{Virtually simple \texorpdfstring{$(d_1,d_2)$-groups}{(d_1,d_2)-groups}}\label{section:simple}

In~\cite{Burger2} and~\cite{Rattaggi}, the authors constructed virtually simple torsion-free $(d_1,d_2)$-groups for different values of $(d_1,d_2)$, for instance $(6,16)$ and $(10,10)$. More recently and with the same ideas, Bondarenko and Kivva constructed two virtually simple torsion-free $(8,8)$-groups in \cite{Bondarenko}.

In this section we find a list of virtually simple $(6,6)$-groups and $(4,5)$-groups. We also give a virtually simple $(2n,2n+1)$-group for each $n \geq 2$. We then end with virtually simple $(6,4n)$-groups with $n \geq 2$ so that the projection on the $6$-regular tree $T$ Chabauty converges to $\Aut(T)$ when $n$ goes to infinity.

\subsection{Virtually simple \texorpdfstring{$(6,6)$-groups}{(6,6)-groups}}
\label{subsection:66}

The idea for constructing virtually simple $(d_1,d_2)$-groups is to use the \textit{Normal Subgroup Theorem} \cite[Theorem~4.1]{Burger2} due to Burger and Mozes, stating that if $\Gamma$ is a $(d_1,d_2)$-group with $H_t$ being $2$-transitive on $\partial T_t$ and $[H_t : H_t^{(\infty)}] < \infty$ for each $t \in \{1,2\}$, then any non-trivial normal subgroup of $\Gamma$ has finite index (i.e.\ $\Gamma$ is \textbf{just-infinite}). Bader and Shalom later proved a generalization of that theorem in~\cite{Bader-Shalom}. We give below a statement which is a consequence of their result. We call it the \textit{Normal Subgroup Theorem} (NST) for future references. A tree is \textbf{thick} if each of its vertices has at least $3$ neighbors.

\begin{theorem*}[Normal Subgroup Theorem, Bader--Shalom]
Let $T_1$ and $T_2$ be two locally finite thick trees and let $\Gamma \leq \Aut(T_1) \times \Aut(T_2)$ be a cocompact lattice such that $\overline{\proj_t(\Gamma)}$ is $2$-transitive on $\partial T_t$ for each $t \in \{1,2\}$. Then $\Gamma$ and all its finite index subgroups are just-infinite. In particular, $\Gamma$ is either residually finite or virtually simple.
\end{theorem*}

\begin{proof}
By \cite[Proposition~3.1.2]{Burger}, all finite index subgroups of a closed subgroup of $\Aut(T_t)$ acting $2$-transitively on $\partial T_t$ also acts $2$-transitively on $\partial T_t$. Up to replacing $\Gamma$ by a finite index subgroup, we can therefore just show that $\Gamma$ is just-infinite.

This is a consequence of \cite[Theorem~1.1]{Bader-Shalom}, modulo the fact that if $H$ is a closed subgroup of $\Aut(T)$ acting $2$-transitively on $\partial T$ (with $T$ being a locally finite thick tree), then $H$ is just non-compact (i.e.\ it is non-compact and all its non-trivial normal subgroups are cocompact) and $H$ does not contain any non-trivial abelian normal subgroup. This is an easy exercise: it suffices to remember that a non-trivial normal subgroup of a $2$-transitive group is transitive, and to use the characterizations of the $2$-transitivity on $\partial T$ given in~\cite[Lemma~3.1.1]{Burger}.

Let us now show that $\Gamma$ is residually finite or virtually simple. As $\Gamma$ is just-infinite, $\Gamma^{(\infty)}$ is trivial or has finite index in $\Gamma$. If $\Gamma^{(\infty)} = 1$ then $\Gamma$ is residually finite. On the contrary, if $\Gamma^{(\infty)}$ has finite index in $\Gamma$ then it is also just-infinite. So any non-trivial normal subgroup $N$ of $\Gamma^{(\infty)}$ has finite index in $\Gamma^{(\infty)}$ and thus in $\Gamma$, hence $N = \Gamma^{(\infty)}$. This means that $\Gamma^{(\infty)}$ is simple.
\end{proof}

In this subsection, we first present a torsion-free $(4,4)$-group $\Gamma_{4,4}$ which is not residually finite. The NST does not directly apply to $\Gamma_{4,4}$, but the strategy is then to embed $\Gamma_{4,4}$ in some other $(d_1,d_2)$-group $\Gamma$ on which the NST can be used. Then $\Gamma$ cannot be residually finite because it contains $\Gamma_{4,4}$, and hence it must be virtually simple.

Let $\Gamma_{4,4}$ be the torsion-free $(4,4)$-group associated to the four squares in Figure~\ref{picture:44}. The local action of $\Gamma_{4,4}$ on $T_1$ (resp.\ $T_2$) is $\D_8$ (resp.\ $\Alt(4)$), and it is possibly irreducible: it appears in Table~\ref{table:44}.

We show that $\Gamma_{4,4}$ is non-residually finite. This was already proved in \cite[Theorem~15]{Bondarenko} and \cite[Corollary~6.4]{CapraceWesolek} but we here prove it by finding an explicit non-trivial element $\gamma \in \Gamma_{4,4}^{(\infty)}$. In other words $\gamma$ is a non-trivial element of $\Gamma_{4,4}$ such that $\varphi(\gamma) = 1$ for any finite quotient $\varphi \colon \Gamma_{4,4} \to Q$. The ideas of this proof are due to Caprace and are already written in \cite[Remark~4.19]{CapraceNote} but we give here some additional details.

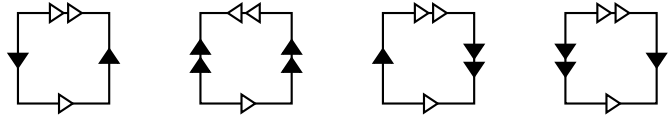
\begin{figure}
\centering
\begin{pspicture*}(-0.2,4.6)(8.6,6.2)
\fontsize{10pt}{10pt}\selectfont
\psset{unit=1.2cm}

\pspolygon(0,4)(1,4)(1,5)(0,5)

\pspolygon[fillstyle=solid,fillcolor=white](0.45,4.1)(0.45,3.9)(0.6,4)
\pspolygon[fillstyle=solid,fillcolor=white](0.35,5.1)(0.35,4.9)(0.5,5)
\pspolygon[fillstyle=solid,fillcolor=white](0.55,5.1)(0.55,4.9)(0.7,5)
\pspolygon[fillstyle=solid,fillcolor=black](0.9,4.45)(1.1,4.45)(1,4.6)
\pspolygon[fillstyle=solid,fillcolor=black](-0.1,4.55)(0.1,4.55)(0,4.4)

\pspolygon(2,4)(3,4)(3,5)(2,5)

\pspolygon[fillstyle=solid,fillcolor=white](2.45,4.1)(2.45,3.9)(2.6,4)
\pspolygon[fillstyle=solid,fillcolor=white](2.65,5.1)(2.65,4.9)(2.5,5)
\pspolygon[fillstyle=solid,fillcolor=white](2.45,5.1)(2.45,4.9)(2.3,5)
\pspolygon[fillstyle=solid,fillcolor=black](2.9,4.55)(3.1,4.55)(3,4.7)
\pspolygon[fillstyle=solid,fillcolor=black](2.9,4.35)(3.1,4.35)(3,4.5)
\pspolygon[fillstyle=solid,fillcolor=black](1.9,4.55)(2.1,4.55)(2,4.7)
\pspolygon[fillstyle=solid,fillcolor=black](1.9,4.35)(2.1,4.35)(2,4.5)

\pspolygon(4,4)(5,4)(5,5)(4,5)

\pspolygon[fillstyle=solid,fillcolor=white](4.45,4.1)(4.45,3.9)(4.6,4)
\pspolygon[fillstyle=solid,fillcolor=white](4.35,5.1)(4.35,4.9)(4.5,5)
\pspolygon[fillstyle=solid,fillcolor=white](4.55,5.1)(4.55,4.9)(4.7,5)
\pspolygon[fillstyle=solid,fillcolor=black](4.9,4.45)(5.1,4.45)(5,4.3)
\pspolygon[fillstyle=solid,fillcolor=black](4.9,4.65)(5.1,4.65)(5,4.5)
\pspolygon[fillstyle=solid,fillcolor=black](3.9,4.45)(4.1,4.45)(4,4.6)

\pspolygon(6,4)(7,4)(7,5)(6,5)

\pspolygon[fillstyle=solid,fillcolor=white](6.45,4.1)(6.45,3.9)(6.6,4)
\pspolygon[fillstyle=solid,fillcolor=white](6.35,5.1)(6.35,4.9)(6.5,5)
\pspolygon[fillstyle=solid,fillcolor=white](6.55,5.1)(6.55,4.9)(6.7,5)
\pspolygon[fillstyle=solid,fillcolor=black](6.9,4.55)(7.1,4.55)(7,4.4)
\pspolygon[fillstyle=solid,fillcolor=black](5.9,4.45)(6.1,4.45)(6,4.3)
\pspolygon[fillstyle=solid,fillcolor=black](5.9,4.65)(6.1,4.65)(6,4.5)
\end{pspicture*}
\caption{The torsion-free $(4,4)$-group $\Gamma_{4,4}$}\label{picture:44}
\end{figure}

\begin{proposition}\label{proposition:nonrf}
The group $\Gamma_{4,4}$ is irreducible and not residually finite. Moreover, $[a_1^3, a_2^4]$ and $[a_2^3, a_1^4]$ are non-trivial elements of $\Gamma_{4,4}^{(\infty)}$.
\end{proposition}

\begin{proof}
Irreducibility has been proven in \cite[Theorem~3]{JW}, but it can also be established by using \cite[Theorem~1.1]{Weiss}. Indeed, in our situation the result of Weiss implies that, if $\Gamma_{4,4}$ was reducible, then $\Fix_{\proj_2(\Gamma_{4,4})}(B(v_2,4))$ would be trivial. So for proving irreducibility we just need to find some element in $\Gamma_{4,4}$ fixing $B(v_2,4)$ but not $B(v_2,5)$. We claim that $(a_1a_2)^{81}$ is such an element. First, we can compute that $(a_1a_2)^3$ fixes $B(v_2,1)$. Then, for each vertex $w$ at distance~$1$ from $v_2$, $(a_1a_2)^3$ can only act trivially or as a $3$-cycle on the three neighbors of $w$ different from $v_2$ (because the local action is $\Alt(4)$). So $(a_1a_2)^9$ fixes $B(v_2,2)$. Continuing with the same argument, we obtain that $(a_1a_2)^{81}$ fixes $B(v_2,4)$. Finally, the fact that $(a_1a_2)^{81}$ does not fix $B(v_2,5)$ can be proved by drawing a $162 \times 5$ rectangle. This can be automatized with a computer, and we get for instance that $(a_1a_2)^{81}(b_1^5(v_2)) = b_1^4b_2^{-1}(v_2)$. 

The strategy to find a non-trivial element in $\Gamma_{4,4}^{(\infty)}$ is to use that for any group $G$ and any subgroup $H \leq G$, the inclusion $[C_G(H), \overline{H}] \subseteq G^{(\infty)}$ holds, where $\overline{H}$ is the profinite closure of $H$ (see~\cite[Lemma~4.13]{CapraceNote}). Here we take $G = \Gamma_{4,4}$. For $H$ we consider $B_{a_1} = \Gamma_{4,4}(v_1, a_1(v_1))$, i.e.\ the fixator of $a_1(v_1)$ in $B = \langle b_1, b_2 \rangle = \Gamma_{4,4}(v_1)$. We claim that $a_1^3 \in C_{\Gamma_{4,4}}(B_{a_1})$ and $a_2^4 \in \overline{B_{a_1}}$. This will thus show that $[a_1^3, a_2^4] \in \Gamma_{4,4}^{(\infty)}$.

Since $B$ acts transitively on the four vertices adjacent to $v_1$ in $T_1$, the subgroup $B_{a_1}$ has index~$4$ in $B$. Using the Reidemeister--Schreier method, we could find the following set of generators for $B_{a_1}$:
$$B_{a_1} = \langle b_1b_2, b_1^{-1}b_2, b_2b_1b_2^2, b_2b_1^{-1}b_2^2, b_2^4 \rangle.$$
From the geometric squares, it can be checked that $a_1^3$ centralizes $B_{a_1}$, i.e.\ $a_1^3 \in C_{\Gamma_{4,4}}(B_{a_1})$. This indeed directly follows from the equalities
\begin{align*}
b_1a_1^3b_1^{-1} &= a_2^3,\\
b_1^{-1}a_1^3b_1 &= a_2^3,\\
b_2a_1^3b_2^{-1} &= a_2^3,\\
b_2^{-1}a_1^3b_2 &= a_2^{-3}.
\end{align*}

Note also that $B_{a_1}$ is contained in $B^{(2)}$, the index~$2$ subgroup of $B$ consisting of elements whose length is even (with respect to the generators $b_1$ and $b_2$). Our next goal is to show that $a_2^2 \in \overline{B^{(2)}}$, and it will then follow that $a_2^4 \in \overline{B_{a_1}}$ as wanted.

\begin{figure}[b!]
\centering
\begin{pspicture*}(-0.8,-0.4)(5,3.4)
\fontsize{10pt}{10pt}\selectfont
\psset{unit=1cm}

\psline(0,0)(4,0)
\rput(1.5,0.2){$h$}
\rput(3.5,0.2){$x'$}
\psline(0,0)(0,3)
\rput(0.25,1.5){$\gamma$}
\psline(0,3)(4,3)
\rput(1.5,2.8){$h$}
\rput(3.5,2.8){$x$}
\psline(3,0)(3,3)
\psline(4,0)(4,3)
\rput(3.7,1.5){$\gamma'$}

\psdot(0,0)
\rput(0,-0.25){$(v_1,v_2)$}
\psdot(0,3)
\rput(0,3.25){$(v_1,\gamma(v_2))$}
\psdot(3,0)
\rput(2.9,-0.25){$(w,v_2)$}
\psdot(3,3)
\rput(2.75,3.25){$(w,\gamma(v_2))$}
\psdot(4,0)
\rput(4.1,-0.25){$(z',v_2)$}
\psdot(4,3)
\rput(4.25,3.25){$(z,\gamma(v_2))$}
\end{pspicture*}
\caption{Illustration of Proposition~\ref{proposition:nonrf}}\label{picture:nonrf}
\end{figure}
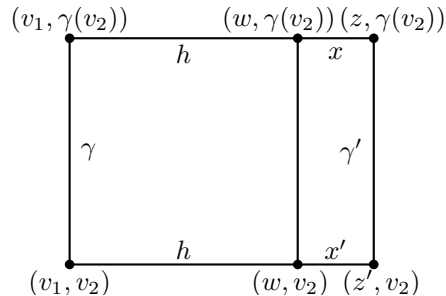

Consider a finite quotient $\varphi \colon \Gamma_{4,4} \to Q$. Since $\Gamma_{4,4}$ is irreducible, the projection $\proj_1(\Gamma_{4,4}(v_1))$ is infinite. Hence, the finite index subgroup $\proj_1(\Fix_{\Gamma_{4,4}}(B(v_1,1)) \cap  B^{(2)} \cap \ker \varphi)$ is also infinite. Let $\gamma$ be an element of $\Fix_{\Gamma_{4,4}}(B(v_1,1)) \cap  B^{(2)} \cap \ker \varphi$ such that $\proj_1(\gamma)$ is non-trivial. In $T_1$, there is a vertex $w \neq v_1$ such that $\gamma$ fixes the path from $v_1$ to $w$ but does not fix some neighbor $z$ of $w$: $\gamma(z) = z' \neq z$. Write $w = h(v_1)$ with $h \in \langle a_1, a_2 \rangle$, $z = hx(v_1)$ with $x \in \{a_1, a_1^{-1}, a_2, a_2^{-1}\}$ and $z' = hx'(v_1)$ with $x' \in \{a_1, a_1^{-1}, a_2, a_2^{-1}\}$, see Figure~\ref{picture:nonrf}. Recall that $\proj_1(\Gamma_{4,4}(v_1))$ acts on the four neighbors of $v_1$ as $\D_8$ acting on the four vertices of a square (where $a_1(v_1)$ and $a_1^{-1}(v_1)$ correspond to opposite vertices of the square). We thus have the same local action around $w$, and the fact that $\gamma$ fixes $w$ and some neighbor of $w$ while not fixing $z$ implies that $x' = x^{-1}$. On Figure~\ref{picture:nonrf} we see that $h x^{-1} \gamma' = \gamma h x$ for some $\gamma' \in B^{(2)}$. Using the fact that $\varphi(\gamma) = 1$, this implies that $\varphi(\gamma') = \varphi(x^2)$. We can summarize this by saying that, for each finite quotient $\varphi \colon \Gamma_{4,4} \to Q$, either $\varphi(a_1^2) \in \varphi(B^{(2)})$ or $\varphi(a_2^2) \in \varphi(B^{(2)})$ $(*)$. In fact, we can even say that there exists $k \in \{1,2\}$ such that $\varphi(a_k^2) \in \varphi(B^{(2)})$ for all finite quotients $\varphi \colon \Gamma_{4,4} \to Q$ $(**)$. Indeed, if $(**)$ was not true then we would have two finite quotients $\varphi_1 \colon \Gamma_{4,4} \to Q_1$ and $\varphi_2 \colon \Gamma_{4,4} \to Q_2$ with $\varphi_1(a_1^2) \not\in \varphi_1(B^{(2)})$ and $\varphi_2(a_2^2) \not\in \varphi_2(B^{(2)})$, and the new finite quotient $(\varphi_1 \times \varphi_2) \colon \Gamma_{4,4} \to Q_1 \times Q_2$ would give a contradiction with $(*)$. Now there suffices to remark that $\Gamma_{4,4}$ has an automorphism defined by $a_1 \mapsto a_2$, $a_2 \mapsto a_1$, $b_1 \mapsto b_1^{-1}$ and $b_2 \mapsto b_2^{-1}$. Therefore, $(**)$ even tells us that $\varphi(a_1^2)$ and $\varphi(a_2^2)$ both belong to $\varphi(B^{(2)})$ for all finite quotients $\varphi \colon \Gamma_{4,4} \to Q$. In particular, we have $a_2^2 \in \overline{B^{(2)}}$ as wanted.

Remark that, thanks to the automorphism of $\Gamma_{4,4}$ defined above, we also obtain $[a_2^3, a_1^4] \in \Gamma_{4,4}^{(\infty)}$.
\end{proof}

Using GAP, we could search for $(d_1,d_2)$-groups $\Gamma$ with $d_1, d_2 \geq 6$, containing $\Gamma_{4,4}$ (in the sense that the four geometric squares defining $\Gamma_{4,4}$ are part of the geometric squares defining $\Gamma$) and such that $\underline{H_1}(v_1) \geq \Alt(d_1)$ and $\underline{H_2}(v_2) \geq \Alt(d_2)$. We say that $\Gamma$ satisfies $(*)$ if the above conditions are true. Since $\Gamma_{4,4}$ is irreducible, a group $\Gamma$ satisfying $(*)$ is also irreducible and $H_t$ is $2$-transitive on $\partial T_t$ for each $t \in \{1,2\}$ (see \cite[Propositions~3.3.1 and~3.3.2]{Burger}). (We even know by Theorem~\ref{theorem:Raduclassification} that $H_t$ belongs to $\mathcal{G}'_{(i)}$ for some legal coloring $i$ of $T_t$.) Thus $\Gamma$ is virtually simple, by the NST.

We could find torsion-free $(6,8)$-groups and torsion-free $(8,6)$-groups satisfying $(*)$, by adding one (resp.\ two) horizontal generator(s), two (resp.\ one) vertical generator(s) and $8$ geometric squares to the ones of $\Gamma_{4,4}$. We could also show that there does not exist any torsion-free $(6,6)$-group satisfying $(*)$. However, there exists $(6,6)$-groups (with torsion) with $(*)$. In total, there are $160$ equivalence classes of such groups. We give all these groups in Tables~\ref{table:simple1}--\ref{table:simple5}, by giving the geometric squares that must be added to the four geometric squares $a_1b_1a_2^{-1}b_1$, $a_1b_2a_2b_2^{-1}$, $a_1b_2^{-1}a_2^{-1}b_1^{-1}$ and $a_1b_1^{-1}a_2^{-1}b_2$ defining $\Gamma_{4,4}$. We call them $\Gamma_{6,6,1}, \ldots, \Gamma_{6,6,160}$. Some remarks follow about these groups:

\begin{itemize}
\item The index of the simple subgroup $\Gamma_{6,6,k}^{(\infty)}$ of $\Gamma_{6,6,k}$ can be computed by using the fact that $[a_1^3, a_2^4] \in \Gamma_{4,4}^{(\infty)} \leq \Gamma_{6,6,k}^{(\infty)}$ (see Proposition~\ref{proposition:nonrf}). Indeed, let $Q$ be the group obtained by adding the relator $[a_1^3, a_2^4]$ to the relators of $\Gamma_{6,6,k}$. Then the kernel of the projection $\Gamma_{6,6,k} \to Q$ is the smallest normal subgroup of $\Gamma_{6,6,k}$ containing $[a_1^3, a_2^4]$, i.e.\ it must be $\Gamma_{6,6,k}^{(\infty)}$. Hence, we just need to compute (with GAP) the order of the finite group $Q$ obtained as above and this gives us the index of $\Gamma_{6,6,k}^{(\infty)}$ in $\Gamma_{6,6,k}$. As written in the tables, for all groups $\Gamma_{6,6,k}$ with $k \in \{1, \ldots, 160\} \setminus \{104,116\}$, we obtain that $|Q|\ = 4$. As $[\Gamma_{6,6,k} : \Gamma_{6,6,k}^+] = 4$, this implies that $\Gamma_{6,6,k}^{(\infty)} = \Gamma_{6,6,k}^+$ i.e.\ $\Gamma_{6,6,k}^+$ is the simple subgroup of finite index in $\Gamma_{6,6,k}$. For $\Gamma_{6,6,104}$ and $\Gamma_{6,6,116}$, we get $|Q|\ = 12$. More precisely, $Q \cong (\C_2)^2 \times \C_3$. So $\Gamma_{6,6,k}^{(\infty)}$ is a subgroup of index $3$ of $\Gamma_{6,6,k}^+$ when $k \in \{104,116\}$.
\item The groups $H_1 = \overline{\proj_1(\Gamma_{6,6,k})}$ and $H_2 = \overline{\proj_2(\Gamma_{6,6,k})}$ are given, using the notation of \S\ref{subsection:inventory}. As explained above, we could compute that $\bigslant{\Gamma_{6,6,k}}{\Gamma_{6,6,k}^{(\infty)}} \cong (\C_2)^2$ or $(\C_2)^2 \times \C_3$ for each $k \in \{1,\ldots,160\}$. This explains why $H_1$ and $H_2$ never take the form $X^{**}$ or $X^{\prime*}$ for some $X \subset_f \N$ : recall from Theorem~\ref{theorem:Radusimple} that $G_{(i)}^+(X,X)$ is simple and that $[G_{(i)}(X^*,X^*) : G_{(i)}^+(X,X)] = 8$ and $\bigslant{G'_{(i)}(X,X)^*}{G_{(i)}^+(X,X)} \cong \C_4$.
\item For the last column, recall that $X_{\Gamma_{6,6,k}^+}$ is a $(d_1,d_2)$-complex, as defined in Definition~\ref{definition:complex}. We write $\Aut(X_{\Gamma_{6,6,k}^+})$ for the set of automorphisms of that complex which do not exchange horizontal and vertical edges. We already know by hypothesis that $\Aut(X_{\Gamma_{6,6,k}^+}) \geq \C_2 \times \C_2$, and we computed the number of automorphisms of $X_{\Gamma_{6,6,k}^+}$ that fix the four vertices $v_{00}, v_{10}, v_{11}$ and $v_{01}$. As written in the tables, for each $k \in \{1, \ldots, 160\}$ we could observe that there is at most one non-trivial such automorphism, so that $\Aut(X_{\Gamma_{6,6,k}^+}) \cong \C_2 \times \C_2$ or $\C_2 \times \C_2 \times \C_2$.
\begin{itemize}
\item If $\Aut(X_{\Gamma_{6,6,k}^+}) \cong \C_2 \times \C_2$, then there is exactly one good $\C_2 \times \C_2$-action on $X_{\Gamma_{6,6,k}^+}$, so this means that $\Gamma_{6,6,k}$ is the only $(6,6)$-group whose type-preserving subgroup is $\Gamma_{6,6,k}^+$.
\item If $\Aut(X_{\Gamma_{6,6,k}^+}) \cong \C_2 \times \C_2 \times \C_2$, then there are four good $\C_2 \times \C_2$-actions on $X_{\Gamma_{6,6,k}^+}$. This means that there are four $(6,6)$-groups whose type-preserving subgroup is $\Gamma_{6,6,k}^+$. For each such $k$ we could compute the three new $(6,6)$-groups containing $\Gamma_{6,6,k}^+$, but it directly appears that they have much more torsion, i.e.\ their $\tau_1$ and $\tau_2$ satisfy $\tau_1+\tau_2 \geq 4$. In particular, none of the new groups obtained in that way is equivalent to some $\Gamma_{6,6,k'}$.
\end{itemize}
From this discussion it follows that all $\Gamma_{6,6,k}^+$ are pairwise non-conjugate in $\Aut(T_1 \times T_2)$. By~\cite[Corollary~1.1.22]{BMZ}, this also means that they are all pairwise non-isomorphic.
\end{itemize}

We summarize some of those results in the next theorem.

\begin{theorem}[Theorem~\ref{maintheorem:simple66-45}(i)]
Let $\Gamma_{6,6,k}$ ($k \in \{1,\ldots,160\}$) be one of the $(6,6)$-groups given by Tables~\ref{table:simple1}--\ref{table:simple5}.
\begin{itemize}
\item If $k \not \in \{104,116\}$, then $\Gamma_{6,6,k}^+$ is simple.
\item If $k \in \{104,116\}$, then $\Gamma_{6,6,k}^+$ has a simple subgroup of index~$3$.
\end{itemize}
Moreover, all simple groups $\Gamma_{6,6,k}^{(\infty)}$ are pairwise non-isomorphic. In particular, the groups $\Gamma_{6,6,k}$ are pairwise non commensurable.
\end{theorem}

\begin{proof}
See the discussion above. Note that $\Gamma_{6,6,k}^{(\infty)} \not \cong \Gamma_{6,6,k'}^{(\infty)}$ for each $k \in \{104,116\}$ and each $k' \in \{1, \ldots, 160\} \setminus \{104,116\}$, see~\cite[Theorem~1.4.1]{BMZ}. We also have $\Gamma_{6,6,104}^{(\infty)} \not \cong \Gamma_{6,6,116}^{(\infty)}$. Indeed, they are not conjugate in $\Aut(T_1 \times T_2)$ since $\overline{\proj_2(\Gamma_{6,6,104}^{(\infty)})} \not \cong \overline{\proj_2(\Gamma_{6,6,116}^{(\infty)})}$ (see Theorem~\ref{theorem:Radusimple}).
\end{proof}

\begin{proof}[Proof of Corollary~\ref{maincorollary:presentations}(i)]
This group is $\Gamma_{6,6,2}$, see Table~\ref{table:simple1}.
\end{proof}

\begin{landscape}
\begin{table}
\scriptsize
\centering
\begin{tabular}{|c|c|c|l|c|c|c|c|}
\hline
Name & $\tau_1$ & $\tau_2$ & Squares: $a_1 b_1 a_2^{-1} b_1$, $a_1 b_2 a_2 b_2^{-1}$, $a_1 b_2^{-1} a_2^{-1} b_1^{-1}$, $a_1 b_1^{-1} a_2^{-1} b_2$ + & $H_1$ & $H_2$
 & $[\Gamma : \Gamma^{(\infty)}]$ & $\Aut(X_{\Gamma^+})$\\
\hline
$\Gamma_{6,6,1}$ & 0 & 0 & $ a_1 b_3 a_1 b_3 $, $ a_1 b_3^{-1} a_1 b_3^{-1} $, $ a_2 b_3 a_2 b_3 $, $ a_2 b_3^{-1} a_3 b_3^{-1} $, $ a_3 b_1 a_3^{-1} b_1^{-1} $, $ a_3 b_2 a_3 b_3 $, $ a_3 b_2^{-1} a_3 b_2^{-1} $ & $\{0\}$ & $\{0\}^*$ & 4 & $\C_2 \times \C_2$\\
$\Gamma_{6,6,2}$ & 0 & 0 & $ a_1 b_3 a_1 b_3^{-1} $, $ a_2 b_3 a_2 b_3 $, $ a_2 b_3^{-1} a_3 b_3^{-1} $, $ a_3 b_1 a_3^{-1} b_1^{-1} $, $ a_3 b_2 a_3 b_3 $, $ a_3 b_2^{-1} a_3 b_2^{-1} $ & $\{0\}$ & $\{1\}$ & 4 & $\C_2 \times \C_2$\\
$\Gamma_{6,6,3}$ & 0 & 0 & $ a_1 b_3 a_1 b_3^{-1} $, $ a_2 b_3 a_2 b_3 $, $ a_2 b_3^{-1} a_3 b_3^{-1} $, $ a_3 b_1 a_3^{-1} b_1 $, $ a_3 b_2 a_3 b_3 $, $ a_3 b_2^{-1} a_3 b_2^{-1} $ & $\{0\}$ & $\{2\}$ & 4 & $\C_2 \times \C_2$\\
$\Gamma_{6,6,4}$ & 0 & 0 & $ a_1 b_3 a_1 b_3 $, $ a_1 b_3^{-1} a_1 b_3^{-1} $, $ a_2 b_3 a_2 b_3 $, $ a_2 b_3^{-1} a_3 b_3^{-1} $, $ a_3 b_1 a_3^{-1} b_1 $, $ a_3 b_2 a_3 b_3 $, $ a_3 b_2^{-1} a_3 b_2^{-1} $ & $\{0\}$ & $\{2\}$ & 4 & $\C_2 \times \C_2$\\
$\Gamma_{6,6,5}$ & 0 & 0 & $ a_1 b_3 a_1 b_3 $, $ a_1 b_3^{-1} a_2 b_3^{-1} $, $ a_2 b_3 a_3 b_3 $, $ a_3 b_1 a_3 b_3^{-1} $, $ a_3 b_2 a_3^{-1} b_2^{-1} $, $ a_3 b_1^{-1} a_3 b_1^{-1} $ & $\{0\}^*$ & $\{0\}^*$ & 4 & $\C_2 \times \C_2 \times \C_2$\\
$\Gamma_{6,6,6}$ & 0 & 0 & $ a_1 b_3 a_1 b_3 $, $ a_1 b_3^{-1} a_2 b_3^{-1} $, $ a_2 b_3 a_3 b_3 $, $ a_3 b_1 a_3 b_1 $, $ a_3 b_2 a_3^{-1} b_2^{-1} $, $ a_3 b_3^{-1} a_3 b_1^{-1} $ & $\{0\}^*$ & $\{0\}^*$ & 4 & $\C_2 \times \C_2 \times \C_2$\\
$\Gamma_{6,6,7}$ & 0 & 0 & $ a_1 b_3 a_1 b_3 $, $ a_1 b_3^{-1} a_3 b_3^{-1} $, $ a_2 b_3 a_2^{-1} b_3 $, $ a_3 b_1 a_3 b_3 $, $ a_3 b_2 a_3^{-1} b_2^{-1} $, $ a_3 b_1^{-1} a_3 b_1^{-1} $ & $\{0\}^*$ & $\{0\}^*$ & 4 & $\C_2 \times \C_2 \times \C_2$\\
$\Gamma_{6,6,8}$ & 0 & 0 & $ a_1 b_3 a_1 b_3 $, $ a_1 b_3^{-1} a_3 b_3^{-1} $, $ a_2 b_3 a_2^{-1} b_3 $, $ a_3 b_1 a_3 b_1 $, $ a_3 b_2 a_3^{-1} b_2^{-1} $, $ a_3 b_3 a_3 b_1^{-1} $ & $\{0\}^*$ & $\{0\}^*$ & 4 & $\C_2 \times \C_2 \times \C_2$\\
$\Gamma_{6,6,9}$ & 0 & 0 & $ a_1 b_3 a_1 b_3 $, $ a_1 b_3^{-1} a_2^{-1} b_3^{-1} $, $ a_2 b_3^{-1} a_3 b_3^{-1} $, $ a_3 b_1 a_3 b_3 $, $ a_3 b_2 a_3^{-1} b_2^{-1} $, $ a_3 b_1^{-1} a_3 b_1^{-1} $ & $\{0\}^*$ & $\{0\}^*$ & 4 & $\C_2 \times \C_2 \times \C_2$\\
$\Gamma_{6,6,10}$ & 0 & 0 & $ a_1 b_3 a_1 b_3 $, $ a_1 b_3^{-1} a_2^{-1} b_3^{-1} $, $ a_2 b_3^{-1} a_3 b_3^{-1} $, $ a_3 b_1 a_3 b_1 $, $ a_3 b_2 a_3^{-1} b_2^{-1} $, $ a_3 b_3 a_3 b_1^{-1} $ & $\{0\}^*$ & $\{0\}^*$ & 4 & $\C_2 \times \C_2 \times \C_2$\\
$\Gamma_{6,6,11}$ & 0 & 0 & $ a_1 b_3 a_1^{-1} b_3^{-1} $, $ a_2 b_3 a_2 b_3 $, $ a_2 b_3^{-1} a_3 b_3^{-1} $, $ a_3 b_1 a_3 b_3 $, $ a_3 b_2 a_3^{-1} b_2^{-1} $, $ a_3 b_1^{-1} a_3 b_1^{-1} $ & $\{0\}^*$ & $\{1\}$ & 4 & $\C_2 \times \C_2 \times \C_2$\\
$\Gamma_{6,6,12}$ & 0 & 0 & $ a_1 b_3 a_1^{-1} b_3^{-1} $, $ a_2 b_3 a_2 b_3 $, $ a_2 b_3^{-1} a_3 b_3^{-1} $, $ a_3 b_1 a_3 b_1 $, $ a_3 b_2 a_3^{-1} b_2^{-1} $, $ a_3 b_3 a_3 b_1^{-1} $ & $\{0\}^*$ & $\{1\}$ & 4 & $\C_2 \times \C_2 \times \C_2$\\
$\Gamma_{6,6,13}$ & 0 & 0 & $ a_1 b_3 a_1^{-1} b_3^{-1} $, $ a_2 b_3 a_2 b_3 $, $ a_2 b_3^{-1} a_3 b_3^{-1} $, $ a_3 b_1 a_3 b_3 $, $ a_3 b_2 a_3^{-1} b_2 $, $ a_3 b_1^{-1} a_3 b_1^{-1} $ & $\{0\}^*$ & $\{2\}$ & 4 & $\C_2 \times \C_2 \times \C_2$\\
$\Gamma_{6,6,14}$ & 0 & 0 & $ a_1 b_3 a_1^{-1} b_3^{-1} $, $ a_2 b_3 a_2 b_3 $, $ a_2 b_3^{-1} a_3 b_3^{-1} $, $ a_3 b_1 a_3 b_1 $, $ a_3 b_2 a_3^{-1} b_2 $, $ a_3 b_3 a_3 b_1^{-1} $ & $\{0\}^*$ & $\{2\}$ & 4 & $\C_2 \times \C_2 \times \C_2$\\
$\Gamma_{6,6,15}$ & 0 & 0 & $ a_1 b_3 a_1 b_3 $, $ a_1 b_3^{-1} a_3 b_3^{-1} $, $ a_2 b_3 a_2^{-1} b_3 $, $ a_3 b_1 a_3 b_3 $, $ a_3 b_2 a_3^{-1} b_2 $, $ a_3 b_1^{-1} a_3 b_1^{-1} $ & $\{0\}^*$ & $\{2\}$ & 4 & $\C_2 \times \C_2 \times \C_2$\\
$\Gamma_{6,6,16}$ & 0 & 0 & $ a_1 b_3 a_1 b_3 $, $ a_1 b_3^{-1} a_3 b_3^{-1} $, $ a_2 b_3 a_2^{-1} b_3 $, $ a_3 b_1 a_3 b_1 $, $ a_3 b_2 a_3^{-1} b_2 $, $ a_3 b_3 a_3 b_1^{-1} $ & $\{0\}^*$ & $\{2\}$ & 4 & $\C_2 \times \C_2 \times \C_2$\\
$\Gamma_{6,6,17}$ & 0 & 0 & $ a_1 b_3 a_1 b_3 $, $ a_1 b_3^{-1} a_2 b_3^{-1} $, $ a_2 b_3 a_3 b_3 $, $ a_3 b_1 a_3 b_3^{-1} $, $ a_3 b_2 a_3^{-1} b_2 $, $ a_3 b_1^{-1} a_3 b_1^{-1} $ & $\{0\}^*$ & $\{0,2\}$ & 4 & $\C_2 \times \C_2 \times \C_2$\\
$\Gamma_{6,6,18}$ & 0 & 0 & $ a_1 b_3 a_1 b_3 $, $ a_1 b_3^{-1} a_2 b_3^{-1} $, $ a_2 b_3 a_3 b_3 $, $ a_3 b_1 a_3 b_1 $, $ a_3 b_2 a_3^{-1} b_2 $, $ a_3 b_3^{-1} a_3 b_1^{-1} $ & $\{0\}^*$ & $\{0,2\}$ & 4 & $\C_2 \times \C_2 \times \C_2$\\
$\Gamma_{6,6,19}$ & 0 & 0 & $ a_1 b_3 a_1 b_3 $, $ a_1 b_3^{-1} a_2^{-1} b_3^{-1} $, $ a_2 b_3^{-1} a_3 b_3^{-1} $, $ a_3 b_1 a_3 b_3 $, $ a_3 b_2 a_3^{-1} b_2 $, $ a_3 b_1^{-1} a_3 b_1^{-1} $ & $\{0\}^*$ & $\{0,2\}$ & 4 & $\C_2 \times \C_2 \times \C_2$\\
$\Gamma_{6,6,20}$ & 0 & 0 & $ a_1 b_3 a_1 b_3 $, $ a_1 b_3^{-1} a_2^{-1} b_3^{-1} $, $ a_2 b_3^{-1} a_3 b_3^{-1} $, $ a_3 b_1 a_3 b_1 $, $ a_3 b_2 a_3^{-1} b_2 $, $ a_3 b_3 a_3 b_1^{-1} $ & $\{0\}^*$ & $\{0,2\}$ & 4 & $\C_2 \times \C_2 \times \C_2$\\
$\Gamma_{6,6,21}$ & 0 & 0 & $ a_1 b_3 a_1 b_3 $, $ a_1 b_3^{-1} a_1 b_3^{-1} $, $ a_2 b_3 a_2 b_3 $, $ a_2 b_3^{-1} a_3 b_3^{-1} $, $ a_3 b_1 a_3 b_1^{-1} $, $ a_3 b_2 a_3 b_3 $, $ a_3 b_2^{-1} a_3 b_2^{-1} $ & $\{1\}$ & $\{0\}^*$ & 4 & $\C_2 \times \C_2$\\
$\Gamma_{6,6,22}$ & 0 & 0 & $ a_1 b_3 a_1 b_3 $, $ a_1 b_3^{-1} a_2 b_3^{-1} $, $ a_2 b_3 a_3 b_3 $, $ a_3 b_1 a_3 b_3^{-1} $, $ a_3 b_2 a_3 b_2^{-1} $, $ a_3 b_1^{-1} a_3 b_1^{-1} $ & $\{1\}$ & $\{0\}^*$ & 4 & $\C_2 \times \C_2 \times \C_2$\\
$\Gamma_{6,6,23}$ & 0 & 0 & $ a_1 b_3 a_1 b_3 $, $ a_1 b_3^{-1} a_2 b_3^{-1} $, $ a_2 b_3 a_3 b_3 $, $ a_3 b_1 a_3 b_1 $, $ a_3 b_2 a_3 b_2^{-1} $, $ a_3 b_3^{-1} a_3 b_1^{-1} $ & $\{1\}$ & $\{0\}^*$ & 4 & $\C_2 \times \C_2 \times \C_2$\\
$\Gamma_{6,6,24}$ & 0 & 0 & $ a_1 b_3 a_1 b_3 $, $ a_1 b_3^{-1} a_3 b_3^{-1} $, $ a_2 b_3 a_2^{-1} b_3 $, $ a_3 b_1 a_3 b_3 $, $ a_3 b_2 a_3 b_2^{-1} $, $ a_3 b_1^{-1} a_3 b_1^{-1} $ & $\{1\}$ & $\{0\}^*$ & 4 & $\C_2 \times \C_2 \times \C_2$\\
$\Gamma_{6,6,25}$ & 0 & 0 & $ a_1 b_3 a_1 b_3 $, $ a_1 b_3^{-1} a_3 b_3^{-1} $, $ a_2 b_3 a_2^{-1} b_3 $, $ a_3 b_1 a_3 b_1 $, $ a_3 b_2 a_3 b_2^{-1} $, $ a_3 b_3 a_3 b_1^{-1} $ & $\{1\}$ & $\{0\}^*$ & 4 & $\C_2 \times \C_2 \times \C_2$\\
$\Gamma_{6,6,26}$ & 0 & 0 & $ a_1 b_3 a_1 b_3 $, $ a_1 b_3^{-1} a_2^{-1} b_3^{-1} $, $ a_2 b_3^{-1} a_3 b_3^{-1} $, $ a_3 b_1 a_3 b_3 $, $ a_3 b_2 a_3 b_2^{-1} $, $ a_3 b_1^{-1} a_3 b_1^{-1} $ & $\{1\}$ & $\{0\}^*$ & 4 & $\C_2 \times \C_2 \times \C_2$\\
$\Gamma_{6,6,27}$ & 0 & 0 & $ a_1 b_3 a_1 b_3 $, $ a_1 b_3^{-1} a_2^{-1} b_3^{-1} $, $ a_2 b_3^{-1} a_3 b_3^{-1} $, $ a_3 b_1 a_3 b_1 $, $ a_3 b_2 a_3 b_2^{-1} $, $ a_3 b_3 a_3 b_1^{-1} $ & $\{1\}$ & $\{0\}^*$ & 4 & $\C_2 \times \C_2 \times \C_2$\\
$\Gamma_{6,6,28}$ & 0 & 0 & $ a_1 b_3 a_1 b_3^{-1} $, $ a_2 b_3 a_2 b_3 $, $ a_2 b_3^{-1} a_3 b_3^{-1} $, $ a_3 b_1 a_3 b_1^{-1} $, $ a_3 b_2 a_3 b_3 $, $ a_3 b_2^{-1} a_3 b_2^{-1} $ & $\{1\}$ & $\{1\}$ & 4 & $\C_2 \times \C_2$\\
$\Gamma_{6,6,29}$ & 0 & 0 & $ a_1 b_3 a_1^{-1} b_3^{-1} $, $ a_2 b_3 a_2 b_3 $, $ a_2 b_3^{-1} a_3 b_3^{-1} $, $ a_3 b_1 a_3 b_3 $, $ a_3 b_2 a_3 b_2^{-1} $, $ a_3 b_1^{-1} a_3 b_1^{-1} $ & $\{1\}$ & $\{1\}$ & 4 & $\C_2 \times \C_2 \times \C_2$\\
$\Gamma_{6,6,30}$ & 0 & 0 & $ a_1 b_3 a_1^{-1} b_3^{-1} $, $ a_2 b_3 a_2 b_3 $, $ a_2 b_3^{-1} a_3 b_3^{-1} $, $ a_3 b_1 a_3 b_1 $, $ a_3 b_2 a_3 b_2^{-1} $, $ a_3 b_3 a_3 b_1^{-1} $ & $\{1\}$ & $\{1\}$ & 4 & $\C_2 \times \C_2 \times \C_2$\\
$\Gamma_{6,6,31}$ & 0 & 0 & $ a_1 b_3 a_1 b_3^{-1} $, $ a_2 b_3 a_2 b_3 $, $ a_2 b_3^{-1} a_3 b_3^{-1} $, $ a_3 b_1 a_3 b_1 $, $ a_3 b_2 a_3 b_3 $, $ a_3 b_2^{-1} a_3 b_2^{-1} $, $ a_3 b_1^{-1} a_3 b_1^{-1} $ & $\{1\}$ & $\{2\}$ & 4 & $\C_2 \times \C_2$\\
$\Gamma_{6,6,32}$ & 0 & 0 & $ a_1 b_3 a_1^{-1} b_3^{-1} $, $ a_2 b_3 a_2 b_3 $, $ a_2 b_3^{-1} a_3 b_3^{-1} $, $ a_3 b_1 a_3 b_3 $, $ a_3 b_2 a_3 b_2 $, $ a_3 b_2^{-1} a_3 b_2^{-1} $, $ a_3 b_1^{-1} a_3 b_1^{-1} $ & $\{1\}$ & $\{2\}$ & 4 & $\C_2 \times \C_2 \times \C_2$\\
\hline
\end{tabular}
\caption{Some virtually simple $(6,6)$-groups. (Part 1/5)}\label{table:simple1}
\end{table}

\begin{table}
\scriptsize
\centering
\begin{tabular}{|c|c|c|l|c|c|c|c|}
\hline
Name & $\tau_1$ & $\tau_2$ & Squares: $a_1 b_1 a_2^{-1} b_1$, $a_1 b_2 a_2 b_2^{-1}$, $a_1 b_2^{-1} a_2^{-1} b_1^{-1}$, $a_1 b_1^{-1} a_2^{-1} b_2$ + & $H_1$ & $H_2$
 & $[\Gamma : \Gamma^{(\infty)}]$ & $\Aut(X_{\Gamma^+})$\\
\hline
$\Gamma_{6,6,33}$ & 0 & 0 & $ a_1 b_3 a_1^{-1} b_3^{-1} $, $ a_2 b_3 a_2 b_3 $, $ a_2 b_3^{-1} a_3 b_3^{-1} $, $ a_3 b_1 a_3 b_1 $, $ a_3 b_2 a_3 b_2 $, $ a_3 b_3 a_3 b_1^{-1} $, $ a_3 b_2^{-1} a_3 b_2^{-1} $ & $\{1\}$ & $\{2\}$ & 4 & $\C_2 \times \C_2 \times \C_2$\\
$\Gamma_{6,6,34}$ & 0 & 0 & $ a_1 b_3 a_1 b_3 $, $ a_1 b_3^{-1} a_1 b_3^{-1} $, $ a_2 b_3 a_2 b_3 $, $ a_2 b_3^{-1} a_3 b_3^{-1} $, $ a_3 b_1 a_3 b_1 $, $ a_3 b_2 a_3 b_3 $, $ a_3 b_2^{-1} a_3 b_2^{-1} $, $ a_3 b_1^{-1} a_3 b_1^{-1} $ & $\{1\}$ & $\{2\}$ & 4 & $\C_2 \times \C_2$\\
$\Gamma_{6,6,35}$ & 0 & 0 & $ a_1 b_3 a_1 b_3 $, $ a_1 b_3^{-1} a_3 b_3^{-1} $, $ a_2 b_3 a_2^{-1} b_3 $, $ a_3 b_1 a_3 b_3 $, $ a_3 b_2 a_3 b_2 $, $ a_3 b_2^{-1} a_3 b_2^{-1} $, $ a_3 b_1^{-1} a_3 b_1^{-1} $ & $\{1\}$ & $\{2\}$ & 4 & $\C_2 \times \C_2 \times \C_2$\\
$\Gamma_{6,6,36}$ & 0 & 0 & $ a_1 b_3 a_1 b_3 $, $ a_1 b_3^{-1} a_3 b_3^{-1} $, $ a_2 b_3 a_2^{-1} b_3 $, $ a_3 b_1 a_3 b_1 $, $ a_3 b_2 a_3 b_2 $, $ a_3 b_3 a_3 b_1^{-1} $, $ a_3 b_2^{-1} a_3 b_2^{-1} $ & $\{1\}$ & $\{2\}$ & 4 & $\C_2 \times \C_2 \times \C_2$\\
$\Gamma_{6,6,37}$ & 0 & 0 & $ a_1 b_3 a_1 b_3 $, $ a_1 b_3^{-1} a_2 b_3^{-1} $, $ a_2 b_3 a_3 b_3 $, $ a_3 b_1 a_3 b_3^{-1} $, $ a_3 b_2 a_3 b_2 $, $ a_3 b_2^{-1} a_3 b_2^{-1} $, $ a_3 b_1^{-1} a_3 b_1^{-1} $ & $\{1\}$ & $\{0,2\}$ & 4 & $\C_2 \times \C_2 \times \C_2$\\
$\Gamma_{6,6,38}$ & 0 & 0 & $ a_1 b_3 a_1 b_3 $, $ a_1 b_3^{-1} a_2 b_3^{-1} $, $ a_2 b_3 a_3 b_3 $, $ a_3 b_1 a_3 b_1 $, $ a_3 b_2 a_3 b_2 $, $ a_3 b_3^{-1} a_3 b_1^{-1} $, $ a_3 b_2^{-1} a_3 b_2^{-1} $ & $\{1\}$ & $\{0,2\}$ & 4 & $\C_2 \times \C_2 \times \C_2$\\
$\Gamma_{6,6,39}$ & 0 & 0 & $ a_1 b_3 a_1 b_3 $, $ a_1 b_3^{-1} a_2^{-1} b_3^{-1} $, $ a_2 b_3^{-1} a_3 b_3^{-1} $, $ a_3 b_1 a_3 b_3 $, $ a_3 b_2 a_3 b_2 $, $ a_3 b_2^{-1} a_3 b_2^{-1} $, $ a_3 b_1^{-1} a_3 b_1^{-1} $ & $\{1\}$ & $\{0,2\}$ & 4 & $\C_2 \times \C_2 \times \C_2$\\
$\Gamma_{6,6,40}$ & 0 & 0 & $ a_1 b_3 a_1 b_3 $, $ a_1 b_3^{-1} a_2^{-1} b_3^{-1} $, $ a_2 b_3^{-1} a_3 b_3^{-1} $, $ a_3 b_1 a_3 b_1 $, $ a_3 b_2 a_3 b_2 $, $ a_3 b_3 a_3 b_1^{-1} $, $ a_3 b_2^{-1} a_3 b_2^{-1} $ & $\{1\}$ & $\{0,2\}$ & 4 & $\C_2 \times \C_2 \times \C_2$\\
$\Gamma_{6,6,41}$ & 0 & 0 & $ a_1 b_3 a_1 b_3 $, $ a_1 b_3^{-1} a_2 b_3^{-1} $, $ a_2 b_3 a_3 b_3 $, $ a_3 b_1 a_3 b_2^{-1} $, $ a_3 b_2 a_3 b_3^{-1} $, $ a_3 b_1^{-1} a_3 b_1^{-1} $ & $\{0,1\}$ & $\{0\}^*$ & 4 & $\C_2 \times \C_2$\\
$\Gamma_{6,6,42}$ & 0 & 0 & $ a_1 b_3 a_1 b_3 $, $ a_1 b_3^{-1} a_2 b_3^{-1} $, $ a_2 b_3 a_3 b_3 $, $ a_3 b_1 a_3 b_1 $, $ a_3 b_2 a_3 b_1^{-1} $, $ a_3 b_3^{-1} a_3 b_2^{-1} $ & $\{0,1\}$ & $\{0\}^*$ & 4 & $\C_2 \times \C_2$\\
$\Gamma_{6,6,43}$ & 0 & 0 & $ a_1 b_3 a_1 b_3 $, $ a_1 b_3^{-1} a_3 b_3^{-1} $, $ a_2 b_3 a_2^{-1} b_3 $, $ a_3 b_1 a_3 b_2^{-1} $, $ a_3 b_2 a_3 b_3 $, $ a_3 b_1^{-1} a_3 b_1^{-1} $ & $\{0,1\}$ & $\{0\}^*$ & 4 & $\C_2 \times \C_2$\\
$\Gamma_{6,6,44}$ & 0 & 0 & $ a_1 b_3 a_1 b_3 $, $ a_1 b_3^{-1} a_3 b_3^{-1} $, $ a_2 b_3 a_2^{-1} b_3 $, $ a_3 b_1 a_3 b_1 $, $ a_3 b_2 a_3 b_1^{-1} $, $ a_3 b_3 a_3 b_2^{-1} $ & $\{0,1\}$ & $\{0\}^*$ & 4 & $\C_2 \times \C_2$\\
$\Gamma_{6,6,45}$ & 0 & 0 & $ a_1 b_3 a_1 b_3 $, $ a_1 b_3^{-1} a_2^{-1} b_3^{-1} $, $ a_2 b_3^{-1} a_3 b_3^{-1} $, $ a_3 b_1 a_3 b_2^{-1} $, $ a_3 b_2 a_3 b_3 $, $ a_3 b_1^{-1} a_3 b_1^{-1} $ & $\{0,1\}$ & $\{0\}^*$ & 4 & $\C_2 \times \C_2$\\
$\Gamma_{6,6,46}$ & 0 & 0 & $ a_1 b_3 a_1 b_3 $, $ a_1 b_3^{-1} a_2^{-1} b_3^{-1} $, $ a_2 b_3^{-1} a_3 b_3^{-1} $, $ a_3 b_1 a_3 b_1 $, $ a_3 b_2 a_3 b_1^{-1} $, $ a_3 b_3 a_3 b_2^{-1} $ & $\{0,1\}$ & $\{0\}^*$ & 4 & $\C_2 \times \C_2$\\
$\Gamma_{6,6,47}$ & 0 & 0 & $ a_1 b_3 a_1^{-1} b_3^{-1} $, $ a_2 b_3 a_2 b_3 $, $ a_2 b_3^{-1} a_3 b_3^{-1} $, $ a_3 b_1 a_3 b_2^{-1} $, $ a_3 b_2 a_3 b_3 $, $ a_3 b_1^{-1} a_3 b_1^{-1} $ & $\{0,1\}$ & $\{1\}$ & 4 & $\C_2 \times \C_2$\\
$\Gamma_{6,6,48}$ & 0 & 0 & $ a_1 b_3 a_1^{-1} b_3^{-1} $, $ a_2 b_3 a_2 b_3 $, $ a_2 b_3^{-1} a_3 b_3^{-1} $, $ a_3 b_1 a_3 b_1 $, $ a_3 b_2 a_3 b_1^{-1} $, $ a_3 b_3 a_3 b_2^{-1} $ & $\{0,1\}$ & $\{1\}$ & 4 & $\C_2 \times \C_2$\\
$\Gamma_{6,6,49}$ & 0 & 0 & $ a_1 b_3 a_1 b_3 $, $ a_1 b_3^{-1} a_1 b_3^{-1} $, $ a_2 b_3 a_2 b_3 $, $ a_2 b_3^{-1} a_3 b_3^{-1} $, $ a_3 b_1 a_3 b_2^{-1} $, $ a_3 b_2 a_3 b_2 $, $ a_3 b_3 a_3 b_1^{-1} $ & $\{0,1\}^*$ & $\{0\}^*$ & 4 & $\C_2 \times \C_2$\\
$\Gamma_{6,6,50}$ & 0 & 0 & $ a_1 b_3 a_1 b_3 $, $ a_1 b_3^{-1} a_1 b_3^{-1} $, $ a_2 b_3 a_2 b_3 $, $ a_2 b_3^{-1} a_3 b_3^{-1} $, $ a_3 b_1 a_3 b_3 $, $ a_3 b_2 a_3 b_1^{-1} $, $ a_3 b_2^{-1} a_3 b_2^{-1} $ & $\{0,1\}^*$ & $\{0\}^*$ & 4 & $\C_2 \times \C_2$\\
$\Gamma_{6,6,51}$ & 0 & 0 & $ a_1 b_3 a_1 b_3^{-1} $, $ a_2 b_3 a_2 b_3 $, $ a_2 b_3^{-1} a_3 b_3^{-1} $, $ a_3 b_1 a_3 b_2^{-1} $, $ a_3 b_2 a_3 b_2 $, $ a_3 b_3 a_3 b_1^{-1} $ & $\{0,1\}^*$ & $\{1\}$ & 4 & $\C_2 \times \C_2$\\
$\Gamma_{6,6,52}$ & 0 & 0 & $ a_1 b_3 a_1 b_3^{-1} $, $ a_2 b_3 a_2 b_3 $, $ a_2 b_3^{-1} a_3 b_3^{-1} $, $ a_3 b_1 a_3 b_3 $, $ a_3 b_2 a_3 b_1^{-1} $, $ a_3 b_2^{-1} a_3 b_2^{-1} $ & $\{0,1\}^*$ & $\{1\}$ & 4 & $\C_2 \times \C_2$\\
$\Gamma_{6,6,53}$ & 0 & 0 & $ a_1 b_3 a_1 b_3 $, $ a_1 b_3^{-1} a_1 b_3^{-1} $, $ a_2 b_3 a_2 b_3 $, $ a_2 b_3^{-1} a_3 b_3^{-1} $, $ a_3 b_1 a_3 b_3 $, $ a_3 b_2 a_3 b_2^{-1} $, $ a_3 b_1^{-1} a_3 b_1^{-1} $ & $\{2\}$ & $\{0\}^*$ & 4 & $\C_2 \times \C_2 \times \C_2$\\
$\Gamma_{6,6,54}$ & 0 & 0 & $ a_1 b_3 a_1 b_3 $, $ a_1 b_3^{-1} a_1 b_3^{-1} $, $ a_2 b_3 a_2 b_3 $, $ a_2 b_3^{-1} a_3 b_3^{-1} $, $ a_3 b_1 a_3 b_3 $, $ a_3 b_2 a_3^{-1} b_2^{-1} $, $ a_3 b_1^{-1} a_3 b_1^{-1} $ & $\{2\}$ & $\{0\}^*$ & 4 & $\C_2 \times \C_2 \times \C_2$\\
$\Gamma_{6,6,55}$ & 0 & 0 & $ a_1 b_3 a_1 b_3 $, $ a_1 b_3^{-1} a_1 b_3^{-1} $, $ a_2 b_3 a_2 b_3 $, $ a_2 b_3^{-1} a_3 b_3^{-1} $, $ a_3 b_1 a_3 b_1 $, $ a_3 b_2 a_3 b_2^{-1} $, $ a_3 b_3 a_3 b_1^{-1} $ & $\{2\}$ & $\{0\}^*$ & 4 & $\C_2 \times \C_2 \times \C_2$\\
$\Gamma_{6,6,56}$ & 0 & 0 & $ a_1 b_3 a_1 b_3 $, $ a_1 b_3^{-1} a_1 b_3^{-1} $, $ a_2 b_3 a_2 b_3 $, $ a_2 b_3^{-1} a_3 b_3^{-1} $, $ a_3 b_1 a_3 b_1 $, $ a_3 b_2 a_3^{-1} b_2^{-1} $, $ a_3 b_3 a_3 b_1^{-1} $ & $\{2\}$ & $\{0\}^*$ & 4 & $\C_2 \times \C_2 \times \C_2$\\
$\Gamma_{6,6,57}$ & 0 & 0 & $ a_1 b_3 a_1 b_3 $, $ a_1 b_3^{-1} a_2 b_3^{-1} $, $ a_2 b_3 a_3 b_3 $, $ a_3 b_1 a_3 b_1^{-1} $, $ a_3 b_2 a_3 b_3^{-1} $, $ a_3 b_2^{-1} a_3 b_2^{-1} $ & $\{2\}$ & $\{0\}^*$ & 4 & $\C_2 \times \C_2$\\
$\Gamma_{6,6,58}$ & 0 & 0 & $ a_1 b_3 a_1 b_3 $, $ a_1 b_3^{-1} a_2 b_3^{-1} $, $ a_2 b_3 a_3 b_3 $, $ a_3 b_1 a_3^{-1} b_1^{-1} $, $ a_3 b_2 a_3 b_3^{-1} $, $ a_3 b_2^{-1} a_3 b_2^{-1} $ & $\{2\}$ & $\{0\}^*$ & 4 & $\C_2 \times \C_2$\\
$\Gamma_{6,6,59}$ & 0 & 0 & $ a_1 b_3 a_1 b_3 $, $ a_1 b_3^{-1} a_3 b_3^{-1} $, $ a_2 b_3 a_2^{-1} b_3 $, $ a_3 b_1 a_3 b_1^{-1} $, $ a_3 b_2 a_3 b_3 $, $ a_3 b_2^{-1} a_3 b_2^{-1} $ & $\{2\}$ & $\{0\}^*$ & 4 & $\C_2 \times \C_2$\\
$\Gamma_{6,6,60}$ & 0 & 0 & $ a_1 b_3 a_1 b_3 $, $ a_1 b_3^{-1} a_3 b_3^{-1} $, $ a_2 b_3 a_2^{-1} b_3 $, $ a_3 b_1 a_3^{-1} b_1^{-1} $, $ a_3 b_2 a_3 b_3 $, $ a_3 b_2^{-1} a_3 b_2^{-1} $ & $\{2\}$ & $\{0\}^*$ & 4 & $\C_2 \times \C_2$\\
$\Gamma_{6,6,61}$ & 0 & 0 & $ a_1 b_3 a_1 b_3 $, $ a_1 b_3^{-1} a_2^{-1} b_3^{-1} $, $ a_2 b_3^{-1} a_3 b_3^{-1} $, $ a_3 b_1 a_3 b_1^{-1} $, $ a_3 b_2 a_3 b_3 $, $ a_3 b_2^{-1} a_3 b_2^{-1} $ & $\{2\}$ & $\{0\}^*$ & 4 & $\C_2 \times \C_2$\\
$\Gamma_{6,6,62}$ & 0 & 0 & $ a_1 b_3 a_1 b_3 $, $ a_1 b_3^{-1} a_2^{-1} b_3^{-1} $, $ a_2 b_3^{-1} a_3 b_3^{-1} $, $ a_3 b_1 a_3^{-1} b_1^{-1} $, $ a_3 b_2 a_3 b_3 $, $ a_3 b_2^{-1} a_3 b_2^{-1} $ & $\{2\}$ & $\{0\}^*$ & 4 & $\C_2 \times \C_2$\\
$\Gamma_{6,6,63}$ & 0 & 0 & $ a_1 b_3 a_1 b_3^{-1} $, $ a_2 b_3 a_2 b_3 $, $ a_2 b_3^{-1} a_3 b_3^{-1} $, $ a_3 b_1 a_3 b_3 $, $ a_3 b_2 a_3 b_2^{-1} $, $ a_3 b_1^{-1} a_3 b_1^{-1} $ & $\{2\}$ & $\{1\}$ & 4 & $\C_2 \times \C_2 \times \C_2$\\
$\Gamma_{6,6,64}$ & 0 & 0 & $ a_1 b_3 a_1 b_3^{-1} $, $ a_2 b_3 a_2 b_3 $, $ a_2 b_3^{-1} a_3 b_3^{-1} $, $ a_3 b_1 a_3 b_3 $, $ a_3 b_2 a_3^{-1} b_2^{-1} $, $ a_3 b_1^{-1} a_3 b_1^{-1} $ & $\{2\}$ & $\{1\}$ & 4 & $\C_2 \times \C_2 \times \C_2$\\
\hline
\end{tabular}
\caption{Some virtually simple $(6,6)$-groups. (Part 2/5)}\label{table:simple2}
\end{table}

\begin{table}
\scriptsize
\centering
\begin{tabular}{|c|c|c|l|c|c|c|c|}
\hline
Name & $\tau_1$ & $\tau_2$ & Squares: $a_1 b_1 a_2^{-1} b_1$, $a_1 b_2 a_2 b_2^{-1}$, $a_1 b_2^{-1} a_2^{-1} b_1^{-1}$, $a_1 b_1^{-1} a_2^{-1} b_2$ + & $H_1$ & $H_2$
 & $[\Gamma : \Gamma^{(\infty)}]$ & $\Aut(X_{\Gamma^+})$\\
\hline
$\Gamma_{6,6,65}$ & 0 & 0 & $ a_1 b_3 a_1 b_3^{-1} $, $ a_2 b_3 a_2 b_3 $, $ a_2 b_3^{-1} a_3 b_3^{-1} $, $ a_3 b_1 a_3 b_1 $, $ a_3 b_2 a_3 b_2^{-1} $, $ a_3 b_3 a_3 b_1^{-1} $ & $\{2\}$ & $\{1\}$ & 4 & $\C_2 \times \C_2 \times \C_2$\\
$\Gamma_{6,6,66}$ & 0 & 0 & $ a_1 b_3 a_1 b_3^{-1} $, $ a_2 b_3 a_2 b_3 $, $ a_2 b_3^{-1} a_3 b_3^{-1} $, $ a_3 b_1 a_3 b_1 $, $ a_3 b_2 a_3^{-1} b_2^{-1} $, $ a_3 b_3 a_3 b_1^{-1} $ & $\{2\}$ & $\{1\}$ & 4 & $\C_2 \times \C_2 \times \C_2$\\
$\Gamma_{6,6,67}$ & 0 & 0 & $ a_1 b_3 a_1^{-1} b_3^{-1} $, $ a_2 b_3 a_2 b_3 $, $ a_2 b_3^{-1} a_3 b_3^{-1} $, $ a_3 b_1 a_3 b_1^{-1} $, $ a_3 b_2 a_3 b_3 $, $ a_3 b_2^{-1} a_3 b_2^{-1} $ & $\{2\}$ & $\{1\}$ & 4 & $\C_2 \times \C_2$\\
$\Gamma_{6,6,68}$ & 0 & 0 & $ a_1 b_3 a_1^{-1} b_3^{-1} $, $ a_2 b_3 a_2 b_3 $, $ a_2 b_3^{-1} a_3 b_3^{-1} $, $ a_3 b_1 a_3^{-1} b_1^{-1} $, $ a_3 b_2 a_3 b_3 $, $ a_3 b_2^{-1} a_3 b_2^{-1} $ & $\{2\}$ & $\{1\}$ & 4 & $\C_2 \times \C_2$\\
$\Gamma_{6,6,69}$ & 0 & 0 & $ a_1 b_3 a_1 b_3^{-1} $, $ a_2 b_3 a_2 b_3 $, $ a_2 b_3^{-1} a_3 b_3^{-1} $, $ a_3 b_1 a_3 b_3 $, $ a_3 b_2 a_3 b_2 $, $ a_3 b_2^{-1} a_3 b_2^{-1} $, $ a_3 b_1^{-1} a_3 b_1^{-1} $ & $\{2\}$ & $\{2\}$ & 4 & $\C_2 \times \C_2 \times \C_2$\\
$\Gamma_{6,6,70}$ & 0 & 0 & $ a_1 b_3 a_1 b_3^{-1} $, $ a_2 b_3 a_2 b_3 $, $ a_2 b_3^{-1} a_3 b_3^{-1} $, $ a_3 b_1 a_3 b_3 $, $ a_3 b_2 a_3^{-1} b_2 $, $ a_3 b_1^{-1} a_3 b_1^{-1} $ & $\{2\}$ & $\{2\}$ & 4 & $\C_2 \times \C_2 \times \C_2$\\
$\Gamma_{6,6,71}$ & 0 & 0 & $ a_1 b_3 a_1 b_3^{-1} $, $ a_2 b_3 a_2 b_3 $, $ a_2 b_3^{-1} a_3 b_3^{-1} $, $ a_3 b_1 a_3 b_1 $, $ a_3 b_2 a_3 b_2 $, $ a_3 b_3 a_3 b_1^{-1} $, $ a_3 b_2^{-1} a_3 b_2^{-1} $ & $\{2\}$ & $\{2\}$ & 4 & $\C_2 \times \C_2 \times \C_2$\\
$\Gamma_{6,6,72}$ & 0 & 0 & $ a_1 b_3 a_1 b_3^{-1} $, $ a_2 b_3 a_2 b_3 $, $ a_2 b_3^{-1} a_3 b_3^{-1} $, $ a_3 b_1 a_3 b_1 $, $ a_3 b_2 a_3^{-1} b_2 $, $ a_3 b_3 a_3 b_1^{-1} $ & $\{2\}$ & $\{2\}$ & 4 & $\C_2 \times \C_2 \times \C_2$\\
$\Gamma_{6,6,73}$ & 0 & 0 & $ a_1 b_3 a_1^{-1} b_3^{-1} $, $ a_2 b_3 a_2 b_3 $, $ a_2 b_3^{-1} a_3 b_3^{-1} $, $ a_3 b_1 a_3 b_1 $, $ a_3 b_2 a_3 b_3 $, $ a_3 b_2^{-1} a_3 b_2^{-1} $, $ a_3 b_1^{-1} a_3 b_1^{-1} $ & $\{2\}$ & $\{2\}$ & 4 & $\C_2 \times \C_2$\\
$\Gamma_{6,6,74}$ & 0 & 0 & $ a_1 b_3 a_1^{-1} b_3^{-1} $, $ a_2 b_3 a_2 b_3 $, $ a_2 b_3^{-1} a_3 b_3^{-1} $, $ a_3 b_1 a_3^{-1} b_1 $, $ a_3 b_2 a_3 b_3 $, $ a_3 b_2^{-1} a_3 b_2^{-1} $ & $\{2\}$ & $\{2\}$ & 4 & $\C_2 \times \C_2$\\
$\Gamma_{6,6,75}$ & 0 & 0 & $ a_1 b_3 a_1 b_3 $, $ a_1 b_3^{-1} a_1 b_3^{-1} $, $ a_2 b_3 a_2 b_3 $, $ a_2 b_3^{-1} a_3 b_3^{-1} $, $ a_3 b_1 a_3 b_3 $, $ a_3 b_2 a_3 b_2 $, $ a_3 b_2^{-1} a_3 b_2^{-1} $, $ a_3 b_1^{-1} a_3 b_1^{-1} $ & $\{2\}$ & $\{2\}$ & 4 & $\C_2 \times \C_2 \times \C_2$\\
$\Gamma_{6,6,76}$ & 0 & 0 & $ a_1 b_3 a_1 b_3 $, $ a_1 b_3^{-1} a_1 b_3^{-1} $, $ a_2 b_3 a_2 b_3 $, $ a_2 b_3^{-1} a_3 b_3^{-1} $, $ a_3 b_1 a_3 b_3 $, $ a_3 b_2 a_3^{-1} b_2 $, $ a_3 b_1^{-1} a_3 b_1^{-1} $ & $\{2\}$ & $\{2\}$ & 4 & $\C_2 \times \C_2 \times \C_2$\\
$\Gamma_{6,6,77}$ & 0 & 0 & $ a_1 b_3 a_1 b_3 $, $ a_1 b_3^{-1} a_1 b_3^{-1} $, $ a_2 b_3 a_2 b_3 $, $ a_2 b_3^{-1} a_3 b_3^{-1} $, $ a_3 b_1 a_3 b_1 $, $ a_3 b_2 a_3 b_2 $, $ a_3 b_3 a_3 b_1^{-1} $, $ a_3 b_2^{-1} a_3 b_2^{-1} $ & $\{2\}$ & $\{2\}$ & 4 & $\C_2 \times \C_2 \times \C_2$\\
$\Gamma_{6,6,78}$ & 0 & 0 & $ a_1 b_3 a_1 b_3 $, $ a_1 b_3^{-1} a_1 b_3^{-1} $, $ a_2 b_3 a_2 b_3 $, $ a_2 b_3^{-1} a_3 b_3^{-1} $, $ a_3 b_1 a_3 b_1 $, $ a_3 b_2 a_3^{-1} b_2 $, $ a_3 b_3 a_3 b_1^{-1} $ & $\{2\}$ & $\{2\}$ & 4 & $\C_2 \times \C_2 \times \C_2$\\
$\Gamma_{6,6,79}$ & 0 & 0 & $ a_1 b_3 a_1 b_3 $, $ a_1 b_3^{-1} a_3 b_3^{-1} $, $ a_2 b_3 a_2^{-1} b_3 $, $ a_3 b_1 a_3 b_1 $, $ a_3 b_2 a_3 b_3 $, $ a_3 b_2^{-1} a_3 b_2^{-1} $, $ a_3 b_1^{-1} a_3 b_1^{-1} $ & $\{2\}$ & $\{2\}$ & 4 & $\C_2 \times \C_2$\\
$\Gamma_{6,6,80}$ & 0 & 0 & $ a_1 b_3 a_1 b_3 $, $ a_1 b_3^{-1} a_3 b_3^{-1} $, $ a_2 b_3 a_2^{-1} b_3 $, $ a_3 b_1 a_3^{-1} b_1 $, $ a_3 b_2 a_3 b_3 $, $ a_3 b_2^{-1} a_3 b_2^{-1} $ & $\{2\}$ & $\{2\}$ & 4 & $\C_2 \times \C_2$\\
$\Gamma_{6,6,81}$ & 0 & 0 & $ a_1 b_3 a_1 b_3 $, $ a_1 b_3^{-1} a_2 b_3^{-1} $, $ a_2 b_3 a_3 b_3 $, $ a_3 b_1 a_3 b_1 $, $ a_3 b_2 a_3 b_3^{-1} $, $ a_3 b_2^{-1} a_3 b_2^{-1} $, $ a_3 b_1^{-1} a_3 b_1^{-1} $ & $\{2\}$ & $\{0,2\}$ & 4 & $\C_2 \times \C_2$\\
$\Gamma_{6,6,82}$ & 0 & 0 & $ a_1 b_3 a_1 b_3 $, $ a_1 b_3^{-1} a_2 b_3^{-1} $, $ a_2 b_3 a_3 b_3 $, $ a_3 b_1 a_3^{-1} b_1 $, $ a_3 b_2 a_3 b_3^{-1} $, $ a_3 b_2^{-1} a_3 b_2^{-1} $ & $\{2\}$ & $\{0,2\}$ & 4 & $\C_2 \times \C_2$\\
$\Gamma_{6,6,83}$ & 0 & 0 & $ a_1 b_3 a_1 b_3 $, $ a_1 b_3^{-1} a_2^{-1} b_3^{-1} $, $ a_2 b_3^{-1} a_3 b_3^{-1} $, $ a_3 b_1 a_3 b_1 $, $ a_3 b_2 a_3 b_3 $, $ a_3 b_2^{-1} a_3 b_2^{-1} $, $ a_3 b_1^{-1} a_3 b_1^{-1} $ & $\{2\}$ & $\{0,2\}$ & 4 & $\C_2 \times \C_2$\\
$\Gamma_{6,6,84}$ & 0 & 0 & $ a_1 b_3 a_1 b_3 $, $ a_1 b_3^{-1} a_2^{-1} b_3^{-1} $, $ a_2 b_3^{-1} a_3 b_3^{-1} $, $ a_3 b_1 a_3^{-1} b_1 $, $ a_3 b_2 a_3 b_3 $, $ a_3 b_2^{-1} a_3 b_2^{-1} $ & $\{2\}$ & $\{0,2\}$ & 4 & $\C_2 \times \C_2$\\
$\Gamma_{6,6,85}$ & 0 & 0 & $ a_1 b_3 a_1 b_3 $, $ a_1 b_3^{-1} a_2 b_3^{-1} $, $ a_2 b_3 a_3 b_3 $, $ a_3 b_1 a_3 b_2^{-1} $, $ a_3 b_2 a_3 b_2 $, $ a_3 b_3^{-1} a_3 b_1^{-1} $ & $\{0,2\}$ & $\{0\}^*$ & 4 & $\C_2 \times \C_2$\\
$\Gamma_{6,6,86}$ & 0 & 0 & $ a_1 b_3 a_1 b_3 $, $ a_1 b_3^{-1} a_2 b_3^{-1} $, $ a_2 b_3 a_3 b_3 $, $ a_3 b_1 a_3 b_3^{-1} $, $ a_3 b_2 a_3 b_1^{-1} $, $ a_3 b_2^{-1} a_3 b_2^{-1} $ & $\{0,2\}$ & $\{0\}^*$ & 4 & $\C_2 \times \C_2$\\
$\Gamma_{6,6,87}$ & 0 & 0 & $ a_1 b_3 a_1 b_3 $, $ a_1 b_3^{-1} a_3 b_3^{-1} $, $ a_2 b_3 a_2^{-1} b_3 $, $ a_3 b_1 a_3 b_2^{-1} $, $ a_3 b_2 a_3 b_2 $, $ a_3 b_3 a_3 b_1^{-1} $ & $\{0,2\}$ & $\{0\}^*$ & 4 & $\C_2 \times \C_2$\\
$\Gamma_{6,6,88}$ & 0 & 0 & $ a_1 b_3 a_1 b_3 $, $ a_1 b_3^{-1} a_3 b_3^{-1} $, $ a_2 b_3 a_2^{-1} b_3 $, $ a_3 b_1 a_3 b_3 $, $ a_3 b_2 a_3 b_1^{-1} $, $ a_3 b_2^{-1} a_3 b_2^{-1} $ & $\{0,2\}$ & $\{0\}^*$ & 4 & $\C_2 \times \C_2$\\
$\Gamma_{6,6,89}$ & 0 & 0 & $ a_1 b_3 a_1 b_3 $, $ a_1 b_3^{-1} a_2^{-1} b_3^{-1} $, $ a_2 b_3^{-1} a_3 b_3^{-1} $, $ a_3 b_1 a_3 b_2^{-1} $, $ a_3 b_2 a_3 b_2 $, $ a_3 b_3 a_3 b_1^{-1} $ & $\{0,2\}$ & $\{0\}^*$ & 4 & $\C_2 \times \C_2$\\
$\Gamma_{6,6,90}$ & 0 & 0 & $ a_1 b_3 a_1 b_3 $, $ a_1 b_3^{-1} a_2^{-1} b_3^{-1} $, $ a_2 b_3^{-1} a_3 b_3^{-1} $, $ a_3 b_1 a_3 b_3 $, $ a_3 b_2 a_3 b_1^{-1} $, $ a_3 b_2^{-1} a_3 b_2^{-1} $ & $\{0,2\}$ & $\{0\}^*$ & 4 & $\C_2 \times \C_2$\\
$\Gamma_{6,6,91}$ & 0 & 0 & $ a_1 b_3 a_1^{-1} b_3^{-1} $, $ a_2 b_3 a_2 b_3 $, $ a_2 b_3^{-1} a_3 b_3^{-1} $, $ a_3 b_1 a_3 b_2^{-1} $, $ a_3 b_2 a_3 b_2 $, $ a_3 b_3 a_3 b_1^{-1} $ & $\{0,2\}$ & $\{1\}$ & 4 & $\C_2 \times \C_2$\\
$\Gamma_{6,6,92}$ & 0 & 0 & $ a_1 b_3 a_1^{-1} b_3^{-1} $, $ a_2 b_3 a_2 b_3 $, $ a_2 b_3^{-1} a_3 b_3^{-1} $, $ a_3 b_1 a_3 b_3 $, $ a_3 b_2 a_3 b_1^{-1} $, $ a_3 b_2^{-1} a_3 b_2^{-1} $ & $\{0,2\}$ & $\{1\}$ & 4 & $\C_2 \times \C_2$\\
$\Gamma_{6,6,93}$ & 0 & 0 & $ a_1 b_3 a_1 b_3 $, $ a_1 b_3^{-1} a_1 b_3^{-1} $, $ a_2 b_3 a_2 b_3 $, $ a_2 b_3^{-1} a_3 b_3^{-1} $, $ a_3 b_1 a_3 b_2^{-1} $, $ a_3 b_2 a_3 b_3 $, $ a_3 b_1^{-1} a_3 b_1^{-1} $ & $\{0,2\}^*$ & $\{0\}^*$ & 4 & $\C_2 \times \C_2$\\
$\Gamma_{6,6,94}$ & 0 & 0 & $ a_1 b_3 a_1 b_3 $, $ a_1 b_3^{-1} a_1 b_3^{-1} $, $ a_2 b_3 a_2 b_3 $, $ a_2 b_3^{-1} a_3 b_3^{-1} $, $ a_3 b_1 a_3 b_1 $, $ a_3 b_2 a_3 b_1^{-1} $, $ a_3 b_3 a_3 b_2^{-1} $ & $\{0,2\}^*$ & $\{0\}^*$ & 4 & $\C_2 \times \C_2$\\
$\Gamma_{6,6,95}$ & 0 & 0 & $ a_1 b_3 a_1 b_3^{-1} $, $ a_2 b_3 a_2 b_3 $, $ a_2 b_3^{-1} a_3 b_3^{-1} $, $ a_3 b_1 a_3 b_2^{-1} $, $ a_3 b_2 a_3 b_3 $, $ a_3 b_1^{-1} a_3 b_1^{-1} $ & $\{0,2\}^*$ & $\{1\}$ & 4 & $\C_2 \times \C_2$\\
$\Gamma_{6,6,96}$ & 0 & 0 & $ a_1 b_3 a_1 b_3^{-1} $, $ a_2 b_3 a_2 b_3 $, $ a_2 b_3^{-1} a_3 b_3^{-1} $, $ a_3 b_1 a_3 b_1 $, $ a_3 b_2 a_3 b_1^{-1} $, $ a_3 b_3 a_3 b_2^{-1} $ & $\{0,2\}^*$ & $\{1\}$ & 4 & $\C_2 \times \C_2$\\
\hline
\end{tabular}
\caption{Some virtually simple $(6,6)$-groups. (Part 3/5)}\label{table:simple3}
\end{table}

\begin{table}
\scriptsize
\centering
\begin{tabular}{|c|c|c|l|c|c|c|c|}
\hline
Name & $\tau_1$ & $\tau_2$ & Squares: $a_1 b_1 a_2^{-1} b_1$, $a_1 b_2 a_2 b_2^{-1}$, $a_1 b_2^{-1} a_2^{-1} b_1^{-1}$, $a_1 b_1^{-1} a_2^{-1} b_2$ + & $H_1$ & $H_2$
 & $[\Gamma : \Gamma^{(\infty)}]$ & $\Aut(X_{\Gamma^+})$\\
\hline
$\Gamma_{6,6,97}$ & 2 & 0 & $ a_1 b_3 a_1 b_3 $, $ a_1 b_3^{-1} a_2 b_3^{-1} $, $ a_2 b_3 A_3 b_3 $, $ A_3 b_1 A_3 b_1 $, $ A_3 b_2 A_4 b_2 $, $ A_4 b_1 A_4 b_3^{-1} $ & $\{0\}$ & $\{0\}^*$ & 4 & $\C_2 \times \C_2 \times \C_2$\\
$\Gamma_{6,6,98}$ & 2 & 0 & $ a_1 b_3 a_1 b_3 $, $ a_1 b_3^{-1} a_2 b_3^{-1} $, $ a_2 b_3 A_3 b_3 $, $ A_3 b_1 A_3 b_1 $, $ A_3 b_2 A_4 b_2 $, $ A_4 b_1 A_4 b_3 $ & $\{0\}$ & $\{0\}^*$ & 4 & $\C_2 \times \C_2 \times \C_2$\\
$\Gamma_{6,6,99}$ & 2 & 0 & $ a_1 b_3 a_1 b_3 $, $ a_1 b_3^{-1} A_3 b_3^{-1} $, $ a_2 b_3 a_2^{-1} b_3 $, $ A_3 b_1 A_3 b_1 $, $ A_3 b_2 A_4 b_2 $, $ A_4 b_1 A_4 b_3^{-1} $ & $\{0\}$ & $\{0\}^*$ & 4 & $\C_2 \times \C_2 \times \C_2$\\
$\Gamma_{6,6,100}$ & 2 & 0 & $ a_1 b_3 a_1 b_3 $, $ a_1 b_3^{-1} A_3 b_3^{-1} $, $ a_2 b_3 a_2^{-1} b_3 $, $ A_3 b_1 A_3 b_1 $, $ A_3 b_2 A_4 b_2 $, $ A_4 b_1 A_4 b_3 $ & $\{0\}$ & $\{0\}^*$ & 4 & $\C_2 \times \C_2 \times \C_2$\\
$\Gamma_{6,6,101}$ & 2 & 0 & $ a_1 b_3 a_1 b_3 $, $ a_1 b_3^{-1} a_2^{-1} b_3^{-1} $, $ a_2 b_3^{-1} A_3 b_3^{-1} $, $ A_3 b_1 A_3 b_1 $, $ A_3 b_2 A_4 b_2 $, $ A_4 b_1 A_4 b_3^{-1} $ & $\{0\}$ & $\{0\}^*$ & 4 & $\C_2 \times \C_2 \times \C_2$\\
$\Gamma_{6,6,102}$ & 2 & 0 & $ a_1 b_3 a_1 b_3 $, $ a_1 b_3^{-1} a_2^{-1} b_3^{-1} $, $ a_2 b_3^{-1} A_3 b_3^{-1} $, $ A_3 b_1 A_3 b_1 $, $ A_3 b_2 A_4 b_2 $, $ A_4 b_1 A_4 b_3 $ & $\{0\}$ & $\{0\}^*$ & 4 & $\C_2 \times \C_2 \times \C_2$\\
$\Gamma_{6,6,103}$ & 2 & 0 & $ a_1 b_3 a_1^{-1} b_3^{-1} $, $ a_2 b_3 a_2 b_3 $, $ a_2 b_3^{-1} A_3 b_3^{-1} $, $ A_3 b_1 A_3 b_1^{-1} $, $ A_3 b_2 A_4 b_2^{-1} $, $ A_3 b_2^{-1} A_4 b_2 $, $ A_4 b_1 A_4 b_3^{-1} $ & $\{0\}$ & $\{1\}$ & 4 & $\C_2 \times \C_2 \times \C_2$\\
$\Gamma_{6,6,104}$ & 2 & 0 & $ a_1 b_3 a_1^{-1} b_3^{-1} $, $ a_2 b_3 a_2 b_3 $, $ a_2 b_3^{-1} A_3 b_3^{-1} $, $ A_3 b_1 A_3 b_1^{-1} $, $ A_3 b_2 A_4 b_2^{-1} $, $ A_3 b_2^{-1} A_4 b_2 $, $ A_4 b_1 A_4 b_3 $ & $\{0\}$ & $\{1\}$ & 12 & $\C_2 \times \C_2 \times \C_2$\\
$\Gamma_{6,6,105}$ & 2 & 0 & $ a_1 b_3 a_1^{-1} b_3^{-1} $, $ a_2 b_3 a_2 b_3 $, $ a_2 b_3^{-1} A_3 b_3^{-1} $, $ A_3 b_1 A_3 b_1 $, $ A_3 b_2 A_4 b_2 $, $ A_4 b_1 A_4 b_3^{-1} $ & $\{0\}$ & $\{1\}$ & 4 & $\C_2 \times \C_2 \times \C_2$\\
$\Gamma_{6,6,106}$ & 2 & 0 & $ a_1 b_3 a_1^{-1} b_3^{-1} $, $ a_2 b_3 a_2 b_3 $, $ a_2 b_3^{-1} A_3 b_3^{-1} $, $ A_3 b_1 A_3 b_1 $, $ A_3 b_2 A_4 b_2 $, $ A_4 b_1 A_4 b_3 $ & $\{0\}$ & $\{1\}$ & 4 & $\C_2 \times \C_2 \times \C_2$\\
$\Gamma_{6,6,107}$ & 2 & 0 & $ a_1 b_3 a_1 b_3 $, $ a_1 b_3^{-1} a_2 b_3^{-1} $, $ a_2 b_3 A_3 b_3 $, $ A_3 b_1 A_3 b_1^{-1} $, $ A_3 b_2 A_4 b_2^{-1} $, $ A_3 b_2^{-1} A_4 b_2 $, $ A_4 b_1 A_4 b_3^{-1} $ & $\{0\}$ & $\{1\}$ & 4 & $\C_2 \times \C_2 \times \C_2$\\
$\Gamma_{6,6,108}$ & 2 & 0 & $ a_1 b_3 a_1 b_3 $, $ a_1 b_3^{-1} a_2 b_3^{-1} $, $ a_2 b_3 A_3 b_3 $, $ A_3 b_1 A_3 b_1^{-1} $, $ A_3 b_2 A_4 b_2^{-1} $, $ A_3 b_2^{-1} A_4 b_2 $, $ A_4 b_1 A_4 b_3 $ & $\{0\}$ & $\{1\}$ & 4 & $\C_2 \times \C_2 \times \C_2$\\
$\Gamma_{6,6,109}$ & 2 & 0 & $ a_1 b_3 a_1 b_3 $, $ a_1 b_3^{-1} A_3 b_3^{-1} $, $ a_2 b_3 a_2^{-1} b_3 $, $ A_3 b_1 A_3 b_1^{-1} $, $ A_3 b_2 A_4 b_2^{-1} $, $ A_3 b_2^{-1} A_4 b_2 $, $ A_4 b_1 A_4 b_3^{-1} $ & $\{0\}$ & $\{1\}$ & 4 & $\C_2 \times \C_2 \times \C_2$\\
$\Gamma_{6,6,110}$ & 2 & 0 & $ a_1 b_3 a_1 b_3 $, $ a_1 b_3^{-1} A_3 b_3^{-1} $, $ a_2 b_3 a_2^{-1} b_3 $, $ A_3 b_1 A_3 b_1^{-1} $, $ A_3 b_2 A_4 b_2^{-1} $, $ A_3 b_2^{-1} A_4 b_2 $, $ A_4 b_1 A_4 b_3 $ & $\{0\}$ & $\{1\}$ & 4 & $\C_2 \times \C_2 \times \C_2$\\
$\Gamma_{6,6,111}$ & 2 & 0 & $ a_1 b_3 a_1 b_3 $, $ a_1 b_3^{-1} a_2^{-1} b_3^{-1} $, $ a_2 b_3^{-1} A_3 b_3^{-1} $, $ A_3 b_1 A_3 b_1^{-1} $, $ A_3 b_2 A_4 b_2^{-1} $, $ A_3 b_2^{-1} A_4 b_2 $, $ A_4 b_1 A_4 b_3^{-1} $ & $\{0\}$ & $\{1\}$ & 4 & $\C_2 \times \C_2 \times \C_2$\\
$\Gamma_{6,6,112}$ & 2 & 0 & $ a_1 b_3 a_1 b_3 $, $ a_1 b_3^{-1} a_2^{-1} b_3^{-1} $, $ a_2 b_3^{-1} A_3 b_3^{-1} $, $ A_3 b_1 A_3 b_1^{-1} $, $ A_3 b_2 A_4 b_2^{-1} $, $ A_3 b_2^{-1} A_4 b_2 $, $ A_4 b_1 A_4 b_3 $ & $\{0\}$ & $\{1\}$ & 4 & $\C_2 \times \C_2 \times \C_2$\\
$\Gamma_{6,6,113}$ & 2 & 0 & $ a_1 b_3 a_1^{-1} b_3^{-1} $, $ a_2 b_3 a_2 b_3 $, $ a_2 b_3^{-1} A_3 b_3^{-1} $, $ A_3 b_1 A_3 b_1 $, $ A_3 b_2 A_4 b_2^{-1} $, $ A_3 b_2^{-1} A_4 b_2 $, $ A_4 b_1 A_4 b_3^{-1} $ & $\{0\}$ & $\{0,1\}$ & 4 & $\C_2 \times \C_2 \times \C_2$\\
$\Gamma_{6,6,114}$ & 2 & 0 & $ a_1 b_3 a_1^{-1} b_3^{-1} $, $ a_2 b_3 a_2 b_3 $, $ a_2 b_3^{-1} A_3 b_3^{-1} $, $ A_3 b_1 A_3 b_1 $, $ A_3 b_2 A_4 b_2^{-1} $, $ A_3 b_2^{-1} A_4 b_2 $, $ A_4 b_1 A_4 b_3 $ & $\{0\}$ & $\{0,1\}$ & 4 & $\C_2 \times \C_2 \times \C_2$\\
$\Gamma_{6,6,115}$ & 2 & 0 & $ a_1 b_3 a_1^{-1} b_3^{-1} $, $ a_2 b_3 a_2 b_3 $, $ a_2 b_3^{-1} A_3 b_3^{-1} $, $ A_3 b_1 A_3 b_1^{-1} $, $ A_3 b_2 A_4 b_2 $, $ A_4 b_1 A_4 b_3^{-1} $ & $\{0\}$ & $\{1,2\}$ & 4 & $\C_2 \times \C_2 \times \C_2$\\
$\Gamma_{6,6,116}$ & 2 & 0 & $ a_1 b_3 a_1^{-1} b_3^{-1} $, $ a_2 b_3 a_2 b_3 $, $ a_2 b_3^{-1} A_3 b_3^{-1} $, $ A_3 b_1 A_3 b_1^{-1} $, $ A_3 b_2 A_4 b_2 $, $ A_4 b_1 A_4 b_3 $ & $\{0\}$ & $\{1,2\}$ & 12 & $\C_2 \times \C_2 \times \C_2$\\
$\Gamma_{6,6,117}$ & 2 & 0 & $ a_1 b_3 a_1 b_3 $, $ a_1 b_3^{-1} A_3 b_3^{-1} $, $ a_2 b_3 a_2^{-1} b_3 $, $ A_3 b_1 A_3 b_1^{-1} $, $ A_3 b_2 A_4 b_2 $, $ A_4 b_1 A_4 b_3^{-1} $ & $\{0\}$ & $\{1,2\}$ & 4 & $\C_2 \times \C_2 \times \C_2$\\
$\Gamma_{6,6,118}$ & 2 & 0 & $ a_1 b_3 a_1 b_3 $, $ a_1 b_3^{-1} A_3 b_3^{-1} $, $ a_2 b_3 a_2^{-1} b_3 $, $ A_3 b_1 A_3 b_1^{-1} $, $ A_3 b_2 A_4 b_2 $, $ A_4 b_1 A_4 b_3 $ & $\{0\}$ & $\{1,2\}$ & 4 & $\C_2 \times \C_2 \times \C_2$\\
$\Gamma_{6,6,119}$ & 2 & 0 & $ a_1 b_3 a_1 b_3 $, $ a_1 b_3^{-1} A_3 b_3^{-1} $, $ a_2 b_3 a_2^{-1} b_3 $, $ A_3 b_1 A_3 b_1 $, $ A_3 b_2 A_4 b_2^{-1} $, $ A_3 b_2^{-1} A_4 b_2 $, $ A_4 b_1 A_4 b_3^{-1} $ & $\{0\}$ & $\{1,2\}$ & 4 & $\C_2 \times \C_2 \times \C_2$\\
$\Gamma_{6,6,120}$ & 2 & 0 & $ a_1 b_3 a_1 b_3 $, $ a_1 b_3^{-1} A_3 b_3^{-1} $, $ a_2 b_3 a_2^{-1} b_3 $, $ A_3 b_1 A_3 b_1 $, $ A_3 b_2 A_4 b_2^{-1} $, $ A_3 b_2^{-1} A_4 b_2 $, $ A_4 b_1 A_4 b_3 $ & $\{0\}$ & $\{1,2\}$ & 4 & $\C_2 \times \C_2 \times \C_2$\\
$\Gamma_{6,6,121}$ & 2 & 0 & $ a_1 b_3 a_1 b_3 $, $ a_1 b_3^{-1} a_2 b_3^{-1} $, $ a_2 b_3 A_3 b_3 $, $ A_3 b_1 A_3 b_1 $, $ A_3 b_2 A_4 b_2^{-1} $, $ A_3 b_2^{-1} A_4 b_2 $, $ A_4 b_1 A_4 b_3^{-1} $ & $\{0\}$ & $\{0,1,2\}$ & 4 & $\C_2 \times \C_2 \times \C_2$\\
$\Gamma_{6,6,122}$ & 2 & 0 & $ a_1 b_3 a_1 b_3 $, $ a_1 b_3^{-1} a_2 b_3^{-1} $, $ a_2 b_3 A_3 b_3 $, $ A_3 b_1 A_3 b_1 $, $ A_3 b_2 A_4 b_2^{-1} $, $ A_3 b_2^{-1} A_4 b_2 $, $ A_4 b_1 A_4 b_3 $ & $\{0\}$ & $\{0,1,2\}$ & 4 & $\C_2 \times \C_2 \times \C_2$\\
$\Gamma_{6,6,123}$ & 2 & 0 & $ a_1 b_3 a_1 b_3 $, $ a_1 b_3^{-1} a_2^{-1} b_3^{-1} $, $ a_2 b_3^{-1} A_3 b_3^{-1} $, $ A_3 b_1 A_3 b_1 $, $ A_3 b_2 A_4 b_2^{-1} $, $ A_3 b_2^{-1} A_4 b_2 $, $ A_4 b_1 A_4 b_3^{-1} $ & $\{0\}$ & $\{0,1,2\}$ & 4 & $\C_2 \times \C_2 \times \C_2$\\
$\Gamma_{6,6,124}$ & 2 & 0 & $ a_1 b_3 a_1 b_3 $, $ a_1 b_3^{-1} a_2^{-1} b_3^{-1} $, $ a_2 b_3^{-1} A_3 b_3^{-1} $, $ A_3 b_1 A_3 b_1 $, $ A_3 b_2 A_4 b_2^{-1} $, $ A_3 b_2^{-1} A_4 b_2 $, $ A_4 b_1 A_4 b_3 $ & $\{0\}$ & $\{0,1,2\}$ & 4 & $\C_2 \times \C_2 \times \C_2$\\
$\Gamma_{6,6,125}$ & 2 & 0 & $ a_1 b_3 a_1 b_3 $, $ a_1 b_3^{-1} a_2 b_3^{-1} $, $ a_2 b_3 A_3 b_3 $, $ A_3 b_1 A_3 b_1^{-1} $, $ A_3 b_2 A_4 b_2 $, $ A_4 b_1 A_4 b_3^{-1} $ & $\{0\}$ & $\{2,3\}$ & 4 & $\C_2 \times \C_2 \times \C_2$\\
$\Gamma_{6,6,126}$ & 2 & 0 & $ a_1 b_3 a_1 b_3 $, $ a_1 b_3^{-1} a_2 b_3^{-1} $, $ a_2 b_3 A_3 b_3 $, $ A_3 b_1 A_3 b_1^{-1} $, $ A_3 b_2 A_4 b_2 $, $ A_4 b_1 A_4 b_3 $ & $\{0\}$ & $\{2,3\}$ & 4 & $\C_2 \times \C_2 \times \C_2$\\
$\Gamma_{6,6,127}$ & 2 & 0 & $ a_1 b_3 a_1 b_3 $, $ a_1 b_3^{-1} a_2^{-1} b_3^{-1} $, $ a_2 b_3^{-1} A_3 b_3^{-1} $, $ A_3 b_1 A_3 b_1^{-1} $, $ A_3 b_2 A_4 b_2 $, $ A_4 b_1 A_4 b_3^{-1} $ & $\{0\}$ &  $\{2,3\}$ & 4 & $\C_2 \times \C_2 \times \C_2$\\
$\Gamma_{6,6,128}$ & 2 & 0 & $ a_1 b_3 a_1 b_3 $, $ a_1 b_3^{-1} a_2^{-1} b_3^{-1} $, $ a_2 b_3^{-1} A_3 b_3^{-1} $, $ A_3 b_1 A_3 b_1^{-1} $, $ A_3 b_2 A_4 b_2 $, $ A_4 b_1 A_4 b_3 $ & $\{0\}$ & $\{2,3\}$ & 4 & $\C_2 \times \C_2 \times \C_2$\\
\hline
\end{tabular}
\caption{Some virtually simple $(6,6)$-groups. (Part 4/5)}\label{table:simple4}
\end{table}

\begin{table}
\scriptsize
\centering
\begin{tabular}{|c|c|c|l|c|c|c|c|}
\hline
Name & $\tau_1$ & $\tau_2$ & Squares: $a_1 b_1 a_2^{-1} b_1$, $a_1 b_2 a_2 b_2^{-1}$, $a_1 b_2^{-1} a_2^{-1} b_1^{-1}$, $a_1 b_1^{-1} a_2^{-1} b_2$ + & $H_1$ & $H_2$
 & $[\Gamma : \Gamma^{(\infty)}]$ & $\Aut(X_{\Gamma^+})$\\
\hline
$\Gamma_{6,6,129}$ & 2 & 0 & $ a_1 b_3 a_1 b_3 $, $ a_1 b_3^{-1} a_1 b_3^{-1} $, $ a_2 b_3 a_2 b_3 $, $ a_2 b_3^{-1} A_3 b_3^{-1} $, $ A_3 b_1 A_4 b_1 $, $ A_3 b_2 A_3 b_2 $, $ A_4 b_2 A_4 b_3^{-1} $ & $\{0\}^*$ & $\{0\}^*$ & 4 & $\C_2 \times \C_2$\\
$\Gamma_{6,6,130}$ & 2 & 0 & $ a_1 b_3 a_1 b_3^{-1} $, $ a_2 b_3 a_2 b_3 $, $ a_2 b_3^{-1} A_3 b_3^{-1} $, $ A_3 b_1 A_4 b_1^{-1} $, $ A_3 b_2 A_3 b_2^{-1} $, $ A_3 b_1^{-1} A_4 b_1 $, $ A_4 b_2 A_4 b_3^{-1} $ & $\{0\}^*$ & $\{1\}$ & 4 & $\C_2 \times \C_2$\\
$\Gamma_{6,6,131}$ & 2 & 0 & $ a_1 b_3 a_1 b_3^{-1} $, $ a_2 b_3 a_2 b_3 $, $ a_2 b_3^{-1} A_3 b_3^{-1} $, $ A_3 b_1 A_4 b_1 $, $ A_3 b_2 A_3 b_2 $, $ A_4 b_2 A_4 b_3^{-1} $ & $\{0\}^*$ & $\{1\}$ & 4 & $\C_2 \times \C_2$\\
$\Gamma_{6,6,132}$ & 2 & 0 & $ a_1 b_3 a_1 b_3 $, $ a_1 b_3^{-1} a_1 b_3^{-1} $, $ a_2 b_3 a_2 b_3 $, $ a_2 b_3^{-1} A_3 b_3^{-1} $, $ A_3 b_1 A_4 b_1^{-1} $, $ A_3 b_2 A_3 b_2^{-1} $, $ A_3 b_1^{-1} A_4 b_1 $, $ A_4 b_2 A_4 b_3^{-1} $ & $\{0\}^*$ & $\{1\}$ & 4 & $\C_2 \times \C_2$\\
$\Gamma_{6,6,133}$ & 2 & 0 & $ a_1 b_3 a_1 b_3^{-1} $, $ a_2 b_3 a_2 b_3 $, $ a_2 b_3^{-1} A_3 b_3^{-1} $, $ A_3 b_1 A_4 b_1^{-1} $, $ A_3 b_2 A_3 b_2 $, $ A_3 b_1^{-1} A_4 b_1 $, $ A_4 b_2 A_4 b_3^{-1} $ & $\{0\}^*$ & $\{0,1\}$ & 4 & $\C_2 \times \C_2$\\
$\Gamma_{6,6,134}$ & 2 & 0 & $ a_1 b_3 a_1 b_3^{-1} $, $ a_2 b_3 a_2 b_3 $, $ a_2 b_3^{-1} A_3 b_3^{-1} $, $ A_3 b_1 A_4 b_1 $, $ A_3 b_2 A_3 b_2^{-1} $, $ A_4 b_2 A_4 b_3^{-1} $ & $\{0\}^*$ & $\{1,2\}$ & 4 & $\C_2 \times \C_2$\\
$\Gamma_{6,6,135}$ & 2 & 0 & $ a_1 b_3 a_1 b_3 $, $ a_1 b_3^{-1} a_1 b_3^{-1} $, $ a_2 b_3 a_2 b_3 $, $ a_2 b_3^{-1} A_3 b_3^{-1} $, $ A_3 b_1 A_4 b_1^{-1} $, $ A_3 b_2 A_3 b_2 $, $ A_3 b_1^{-1} A_4 b_1 $, $ A_4 b_2 A_4 b_3^{-1} $ & $\{0\}^*$ & $\{1,2\}$ & 4 & $\C_2 \times \C_2$\\
$\Gamma_{6,6,136}$ & 2 & 0 & $ a_1 b_3 a_1 b_3 $, $ a_1 b_3^{-1} a_1 b_3^{-1} $, $ a_2 b_3 a_2 b_3 $, $ a_2 b_3^{-1} A_3 b_3^{-1} $, $ A_3 b_1 A_4 b_1 $, $ A_3 b_2 A_3 b_2^{-1} $, $ A_4 b_2 A_4 b_3^{-1} $ & $\{0\}^*$ & $\{1,2\}$ & 4 & $\C_2 \times \C_2$\\
$\Gamma_{6,6,137}$ & 2 & 0 & $ a_1 b_3 a_1 b_3 $, $ a_1 b_3^{-1} a_1 b_3^{-1} $, $ a_2 b_3 a_2 b_3 $, $ a_2 b_3^{-1} A_3 b_3^{-1} $, $ A_3 b_1 A_3 b_1 $, $ A_3 b_2 A_4 b_2 $, $ A_4 b_1 A_4 b_3^{-1} $ & $\{2\}$ & $\{0\}^*$ & 4 & $\C_2 \times \C_2 \times \C_2$\\
$\Gamma_{6,6,138}$ & 2 & 0 & $ a_1 b_3 a_1 b_3 $, $ a_1 b_3^{-1} a_1 b_3^{-1} $, $ a_2 b_3 a_2 b_3 $, $ a_2 b_3^{-1} A_3 b_3^{-1} $, $ A_3 b_1 A_3 b_1 $, $ A_3 b_2 A_4 b_2 $, $ A_4 b_1 A_4 b_3 $ & $\{2\}$ & $\{0\}^*$ & 4 & $\C_2 \times \C_2 \times \C_2$\\
$\Gamma_{6,6,139}$ & 2 & 0 & $ a_1 b_3 a_1 b_3 $, $ a_1 b_3^{-1} A_3 b_3^{-1} $, $ a_2 b_3 a_2^{-1} b_3 $, $ A_3 b_1 A_4 b_1 $, $ A_3 b_2 A_3 b_2 $, $ A_4 b_2 A_4 b_3^{-1} $ & $\{2\}$ & $\{0\}^*$ & 4 & $\C_2 \times \C_2$\\
$\Gamma_{6,6,140}$ & 2 & 0 & $ a_1 b_3 a_1 b_3^{-1} $, $ a_2 b_3 a_2 b_3 $, $ a_2 b_3^{-1} A_3 b_3^{-1} $, $ A_3 b_1 A_3 b_1^{-1} $, $ A_3 b_2 A_4 b_2^{-1} $, $ A_3 b_2^{-1} A_4 b_2 $, $ A_4 b_1 A_4 b_3^{-1} $ & $\{2\}$ & $\{1\}$ & 4 & $\C_2 \times \C_2 \times \C_2$\\
$\Gamma_{6,6,141}$ & 2 & 0 & $ a_1 b_3 a_1 b_3^{-1} $, $ a_2 b_3 a_2 b_3 $, $ a_2 b_3^{-1} A_3 b_3^{-1} $, $ A_3 b_1 A_3 b_1^{-1} $, $ A_3 b_2 A_4 b_2^{-1} $, $ A_3 b_2^{-1} A_4 b_2 $, $ A_4 b_1 A_4 b_3 $ & $\{2\}$ & $\{1\}$ & 4 & $\C_2 \times \C_2 \times \C_2$\\
$\Gamma_{6,6,142}$ & 2 & 0 & $ a_1 b_3 a_1 b_3^{-1} $, $ a_2 b_3 a_2 b_3 $, $ a_2 b_3^{-1} A_3 b_3^{-1} $, $ A_3 b_1 A_3 b_1 $, $ A_3 b_2 A_4 b_2 $, $ A_4 b_1 A_4 b_3^{-1} $ & $\{2\}$ & $\{1\}$ & 4 & $\C_2 \times \C_2 \times \C_2$\\
$\Gamma_{6,6,143}$ & 2 & 0 & $ a_1 b_3 a_1 b_3^{-1} $, $ a_2 b_3 a_2 b_3 $, $ a_2 b_3^{-1} A_3 b_3^{-1} $, $ A_3 b_1 A_3 b_1 $, $ A_3 b_2 A_4 b_2 $, $ A_4 b_1 A_4 b_3 $ & $\{2\}$ & $\{1\}$ & 4 & $\C_2 \times \C_2 \times \C_2$\\
$\Gamma_{6,6,144}$ & 2 & 0 & $ a_1 b_3 a_1^{-1} b_3^{-1} $, $ a_2 b_3 a_2 b_3 $, $ a_2 b_3^{-1} A_3 b_3^{-1} $, $ A_3 b_1 A_4 b_1^{-1} $, $ A_3 b_2 A_3 b_2^{-1} $, $ A_3 b_1^{-1} A_4 b_1 $, $ A_4 b_2 A_4 b_3^{-1} $ & $\{2\}$ & $\{1\}$ & 4 & $\C_2 \times \C_2$\\
$\Gamma_{6,6,145}$ & 2 & 0 & $ a_1 b_3 a_1^{-1} b_3^{-1} $, $ a_2 b_3 a_2 b_3 $, $ a_2 b_3^{-1} A_3 b_3^{-1} $, $ A_3 b_1 A_4 b_1 $, $ A_3 b_2 A_3 b_2 $, $ A_4 b_2 A_4 b_3^{-1} $ & $\{2\}$ & $\{1\}$ & 4 & $\C_2 \times \C_2$\\
$\Gamma_{6,6,146}$ & 2 & 0 & $ a_1 b_3 a_1 b_3 $, $ a_1 b_3^{-1} a_1 b_3^{-1} $, $ a_2 b_3 a_2 b_3 $, $ a_2 b_3^{-1} A_3 b_3^{-1} $, $ A_3 b_1 A_3 b_1^{-1} $, $ A_3 b_2 A_4 b_2^{-1} $, $ A_3 b_2^{-1} A_4 b_2 $, $ A_4 b_1 A_4 b_3^{-1} $ & $\{2\}$ & $\{1\}$ & 4 & $\C_2 \times \C_2 \times \C_2$\\
$\Gamma_{6,6,147}$ & 2 & 0 & $ a_1 b_3 a_1 b_3 $, $ a_1 b_3^{-1} a_1 b_3^{-1} $, $ a_2 b_3 a_2 b_3 $, $ a_2 b_3^{-1} A_3 b_3^{-1} $, $ A_3 b_1 A_3 b_1^{-1} $, $ A_3 b_2 A_4 b_2^{-1} $, $ A_3 b_2^{-1} A_4 b_2 $, $ A_4 b_1 A_4 b_3 $ & $\{2\}$ & $\{1\}$ & 4 & $\C_2 \times \C_2 \times \C_2$\\
$\Gamma_{6,6,148}$ & 2 & 0 & $ a_1 b_3 a_1 b_3 $, $ a_1 b_3^{-1} A_3 b_3^{-1} $, $ a_2 b_3 a_2^{-1} b_3 $, $ A_3 b_1 A_4 b_1^{-1} $, $ A_3 b_2 A_3 b_2^{-1} $, $ A_3 b_1^{-1} A_4 b_1 $, $ A_4 b_2 A_4 b_3^{-1} $ & $\{2\}$ & $\{1\}$ & 4 & $\C_2 \times \C_2$\\
$\Gamma_{6,6,149}$ & 2 & 0 & $ a_1 b_3 a_1 b_3^{-1} $, $ a_2 b_3 a_2 b_3 $, $ a_2 b_3^{-1} A_3 b_3^{-1} $, $ A_3 b_1 A_3 b_1 $, $ A_3 b_2 A_4 b_2^{-1} $, $ A_3 b_2^{-1} A_4 b_2 $, $ A_4 b_1 A_4 b_3^{-1} $ & $\{2\}$ & $\{0,1\}$ & 4 & $\C_2 \times \C_2 \times \C_2$\\
$\Gamma_{6,6,150}$ & 2 & 0 & $ a_1 b_3 a_1 b_3^{-1} $, $ a_2 b_3 a_2 b_3 $, $ a_2 b_3^{-1} A_3 b_3^{-1} $, $ A_3 b_1 A_3 b_1 $, $ A_3 b_2 A_4 b_2^{-1} $, $ A_3 b_2^{-1} A_4 b_2 $, $ A_4 b_1 A_4 b_3 $ & $\{2\}$ & $\{0,1\}$ & 4 & $\C_2 \times \C_2 \times \C_2$\\
$\Gamma_{6,6,151}$ & 2 & 0 & $ a_1 b_3 a_1^{-1} b_3^{-1} $, $ a_2 b_3 a_2 b_3 $, $ a_2 b_3^{-1} A_3 b_3^{-1} $, $ A_3 b_1 A_4 b_1^{-1} $, $ A_3 b_2 A_3 b_2 $, $ A_3 b_1^{-1} A_4 b_1 $, $ A_4 b_2 A_4 b_3^{-1} $ & $\{2\}$ & $\{0,1\}$ & 4 & $\C_2 \times \C_2$\\
$\Gamma_{6,6,152}$ & 2 & 0 & $ a_1 b_3 a_1 b_3^{-1} $, $ a_2 b_3 a_2 b_3 $, $ a_2 b_3^{-1} A_3 b_3^{-1} $, $ A_3 b_1 A_3 b_1^{-1} $, $ A_3 b_2 A_4 b_2 $, $ A_4 b_1 A_4 b_3^{-1} $ & $\{2\}$ & $\{1,2\}$ & 4 & $\C_2 \times \C_2 \times \C_2$\\
$\Gamma_{6,6,153}$ & 2 & 0 & $ a_1 b_3 a_1 b_3^{-1} $, $ a_2 b_3 a_2 b_3 $, $ a_2 b_3^{-1} A_3 b_3^{-1} $, $ A_3 b_1 A_3 b_1^{-1} $, $ A_3 b_2 A_4 b_2 $, $ A_4 b_1 A_4 b_3 $ & $\{2\}$ & $\{1,2\}$ & 4 & $\C_2 \times \C_2 \times \C_2$\\
$\Gamma_{6,6,154}$ & 2 & 0 & $ a_1 b_3 a_1^{-1} b_3^{-1} $, $ a_2 b_3 a_2 b_3 $, $ a_2 b_3^{-1} A_3 b_3^{-1} $, $ A_3 b_1 A_4 b_1 $, $ A_3 b_2 A_3 b_2^{-1} $, $ A_4 b_2 A_4 b_3^{-1} $ & $\{2\}$ & $\{1,2\}$ & 4 & $\C_2 \times \C_2$\\
$\Gamma_{6,6,155}$ & 2 & 0 & $ a_1 b_3 a_1 b_3 $, $ a_1 b_3^{-1} a_1 b_3^{-1} $, $ a_2 b_3 a_2 b_3 $, $ a_2 b_3^{-1} A_3 b_3^{-1} $, $ A_3 b_1 A_3 b_1^{-1} $, $ A_3 b_2 A_4 b_2 $, $ A_4 b_1 A_4 b_3^{-1} $ & $\{2\}$ & $\{1,2\}$ & 4 & $\C_2 \times \C_2 \times \C_2$\\
$\Gamma_{6,6,156}$ & 2 & 0 & $ a_1 b_3 a_1 b_3 $, $ a_1 b_3^{-1} a_1 b_3^{-1} $, $ a_2 b_3 a_2 b_3 $, $ a_2 b_3^{-1} A_3 b_3^{-1} $, $ A_3 b_1 A_3 b_1^{-1} $, $ A_3 b_2 A_4 b_2 $, $ A_4 b_1 A_4 b_3 $ & $\{2\}$ & $\{1,2\}$ & 4 & $\C_2 \times \C_2 \times \C_2$\\
$\Gamma_{6,6,157}$ & 2 & 0 & $ a_1 b_3 a_1 b_3 $, $ a_1 b_3^{-1} a_1 b_3^{-1} $, $ a_2 b_3 a_2 b_3 $, $ a_2 b_3^{-1} A_3 b_3^{-1} $, $ A_3 b_1 A_3 b_1 $, $ A_3 b_2 A_4 b_2^{-1} $, $ A_3 b_2^{-1} A_4 b_2 $, $ A_4 b_1 A_4 b_3^{-1} $ & $\{2\}$ & $\{1,2\}$ & 4 & $\C_2 \times \C_2 \times \C_2$\\
$\Gamma_{6,6,158}$ & 2 & 0 & $ a_1 b_3 a_1 b_3 $, $ a_1 b_3^{-1} a_1 b_3^{-1} $, $ a_2 b_3 a_2 b_3 $, $ a_2 b_3^{-1} A_3 b_3^{-1} $, $ A_3 b_1 A_3 b_1 $, $ A_3 b_2 A_4 b_2^{-1} $, $ A_3 b_2^{-1} A_4 b_2 $, $ A_4 b_1 A_4 b_3 $ & $\{2\}$ & $\{1,2\}$ & 4 & $\C_2 \times \C_2 \times \C_2$\\
$\Gamma_{6,6,159}$ & 2 & 0 & $ a_1 b_3 a_1 b_3 $, $ a_1 b_3^{-1} A_3 b_3^{-1} $, $ a_2 b_3 a_2^{-1} b_3 $, $ A_3 b_1 A_4 b_1^{-1} $, $ A_3 b_2 A_3 b_2 $, $ A_3 b_1^{-1} A_4 b_1 $, $ A_4 b_2 A_4 b_3^{-1} $ & $\{2\}$ & $\{1,2\}$ & 4 & $\C_2 \times \C_2$\\
$\Gamma_{6,6,160}$ & 2 & 0 & $ a_1 b_3 a_1 b_3 $, $ a_1 b_3^{-1} A_3 b_3^{-1} $, $ a_2 b_3 a_2^{-1} b_3 $, $ A_3 b_1 A_4 b_1 $, $ A_3 b_2 A_3 b_2^{-1} $, $ A_4 b_2 A_4 b_3^{-1} $ & $\{2\}$ & $\{1,2\}$ & 4 & $\C_2 \times \C_2$\\
\hline
\end{tabular}
\caption{Some virtually simple $(6,6)$-groups. (Part 5/5)}\label{table:simple5}
\end{table}
\end{landscape}

\subsection{Virtually simple \texorpdfstring{$(4,5)$-groups}{(4,5)-groups}}
\label{subsection:45}

In this subsection we use the same strategy as above so as to discover virtually simple $(4,5)$-groups. The NST however requires the closures of the projections to be boundary-$2$-transitive. In the previous section we were dealing with $6$-regular tree, so \cite[Propositions~3.3.1 and~3.3.2]{Burger} could be used to ensure the $2$-transitivity on the boundary. For $4$-regular and $5$-regular trees, those results do not apply. We will therefore need the following theorem, due to Trofimov.

\begin{theorem}[Trofimov]\label{theorem:Trofimov}
Let $X$ be a connected $(q+1)$-regular graph with $q \geq 2$, and let $G \leq \Aut(X)$ be vertex-transitive. Let $v \in V(X)$ and suppose that $\underline{G}(v)$ contains $\mathrm{PSL}(2,q)$ (acting on the projective line). If $G$ is non-discrete, then $X$ is the $(q+1)$-regular tree and the closure $\overline{G} \leq \Aut(X)$ of $G$ is $2$-transitive on $\partial X$.
\end{theorem}

\begin{proof}
See~\cite[Proposition~3.1 and Example~3.2]{Trofimov}. Note that the original statement only mentions that $X$ is the $(q+1)$-regular tree. However, the proof consists in showing that $G$ is transitive on paths of length $\ell$ of $X$ for each $\ell \geq 1$. This assertion implies that $X$ is a tree, but also that $\overline{G}$ is $2$-transitive on $\partial X$.
\end{proof}

In \S\ref{subsection:66} we started with a non-residually finite torsion-free $(4,4)$-group. This time we start with a non-residually finite $(3,3)$-group. Let $\Gamma_{3,3}$ be the $(3,3)$-group associated to the six squares in Figure~\ref{picture:33}. The local action of $\Gamma_{3,3}$ on $T_1$ (resp.\ $T_2$) is $\Sym(3)$ (resp.\ $\C_2$). In the next result, with the same ideas as for Proposition~\ref{proposition:nonrf}, we show that $\Gamma_{3,3}$ is irreducible and not residually finite. Note that we could have used \cite[Corollary~6.4]{CapraceWesolek} for the non-residual finiteness, but once again we wanted an explicit non-trivial element of $\Gamma_{3,3}^{(\infty)}$.

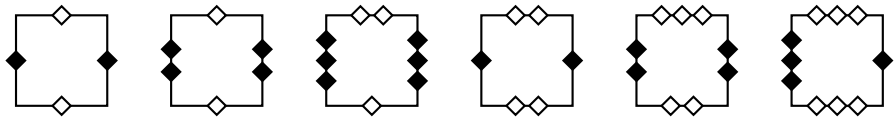
\begin{figure}[b!]
\centering
\begin{pspicture*}(-0.2,-0.2)(11.6,1.4)
\fontsize{10pt}{10pt}\selectfont
\psset{unit=1.2cm}

\pspolygon(0,0)(1,0)(1,1)(0,1)

\pspolygon[fillstyle=solid,fillcolor=white](0.4,0)(0.5,0.1)(0.6,0)(0.5,-0.1)
\pspolygon[fillstyle=solid,fillcolor=white](0.4,1)(0.5,1.1)(0.6,1)(0.5,0.9)
\pspolygon[fillstyle=solid,fillcolor=black](0,0.4)(0.1,0.5)(0,0.6)(-0.1,0.5)
\pspolygon[fillstyle=solid,fillcolor=black](1,0.4)(1.1,0.5)(1,0.6)(0.9,0.5)

\pspolygon(1.7,0)(2.7,0)(2.7,1)(1.7,1)

\pspolygon[fillstyle=solid,fillcolor=white](2.1,0)(2.2,0.1)(2.3,0)(2.2,-0.1)
\pspolygon[fillstyle=solid,fillcolor=white](2.1,1)(2.2,1.1)(2.3,1)(2.2,0.9)
\pspolygon[fillstyle=solid,fillcolor=black](1.7,0.525)(1.8,0.625)(1.7,0.725)(1.6,0.625)
\pspolygon[fillstyle=solid,fillcolor=black](1.7,0.475)(1.8,0.375)(1.7,0.275)(1.6,0.375)
\pspolygon[fillstyle=solid,fillcolor=black](2.7,0.525)(2.8,0.625)(2.7,0.725)(2.6,0.625)
\pspolygon[fillstyle=solid,fillcolor=black](2.7,0.475)(2.8,0.375)(2.7,0.275)(2.6,0.375)

\pspolygon(3.4,0)(4.4,0)(4.4,1)(3.4,1)

\pspolygon[fillstyle=solid,fillcolor=white](3.8,0)(3.9,0.1)(4,0)(3.9,-0.1)
\pspolygon[fillstyle=solid,fillcolor=white](3.925,1)(4.025,1.1)(4.125,1)(4.025,0.9)
\pspolygon[fillstyle=solid,fillcolor=white](3.875,1)(3.775,1.1)(3.675,1)(3.775,0.9)
\pspolygon[fillstyle=solid,fillcolor=black](4.4,0.4)(4.5,0.5)(4.4,0.6)(4.3,0.5)
\pspolygon[fillstyle=solid,fillcolor=black](4.4,0.625)(4.5,0.725)(4.4,0.825)(4.3,0.725)
\pspolygon[fillstyle=solid,fillcolor=black](4.4,0.375)(4.5,0.275)(4.4,0.175)(4.3,0.275)
\pspolygon[fillstyle=solid,fillcolor=black](3.4,0.4)(3.5,0.5)(3.4,0.6)(3.3,0.5)
\pspolygon[fillstyle=solid,fillcolor=black](3.4,0.625)(3.5,0.725)(3.4,0.825)(3.3,0.725)
\pspolygon[fillstyle=solid,fillcolor=black](3.4,0.375)(3.5,0.275)(3.4,0.175)(3.3,0.275)

\pspolygon(5.1,0)(6.1,0)(6.1,1)(5.1,1)

\pspolygon[fillstyle=solid,fillcolor=white](5.625,0)(5.725,0.1)(5.825,0)(5.725,-0.1)
\pspolygon[fillstyle=solid,fillcolor=white](5.575,0)(5.475,0.1)(5.375,0)(5.475,-0.1)
\pspolygon[fillstyle=solid,fillcolor=white](5.625,1)(5.725,1.1)(5.825,1)(5.725,0.9)
\pspolygon[fillstyle=solid,fillcolor=white](5.575,1)(5.475,1.1)(5.375,1)(5.475,0.9)
\pspolygon[fillstyle=solid,fillcolor=black](5.1,0.4)(5.2,0.5)(5.1,0.6)(5,0.5)
\pspolygon[fillstyle=solid,fillcolor=black](6.1,0.4)(6.2,0.5)(6.1,0.6)(6,0.5)

\pspolygon(6.8,0)(7.8,0)(7.8,1)(6.8,1)

\pspolygon[fillstyle=solid,fillcolor=white](7.325,0)(7.425,0.1)(7.525,0)(7.425,-0.1)
\pspolygon[fillstyle=solid,fillcolor=white](7.275,0)(7.175,0.1)(7.075,0)(7.175,-0.1)
\pspolygon[fillstyle=solid,fillcolor=white](7.2,1)(7.3,1.1)(7.4,1)(7.3,0.9)
\pspolygon[fillstyle=solid,fillcolor=white](7.175,1)(7.075,1.1)(6.975,1)(7.075,0.9)
\pspolygon[fillstyle=solid,fillcolor=white](7.425,1)(7.525,1.1)(7.625,1)(7.525,0.9)
\pspolygon[fillstyle=solid,fillcolor=black](6.8,0.525)(6.9,0.625)(6.8,0.725)(6.7,0.625)
\pspolygon[fillstyle=solid,fillcolor=black](6.8,0.475)(6.9,0.375)(6.8,0.275)(6.7,0.375)
\pspolygon[fillstyle=solid,fillcolor=black](7.8,0.525)(7.9,0.625)(7.8,0.725)(7.7,0.625)
\pspolygon[fillstyle=solid,fillcolor=black](7.8,0.475)(7.9,0.375)(7.8,0.275)(7.7,0.375)

\pspolygon(8.5,0)(9.5,0)(9.5,1)(8.5,1)

\pspolygon[fillstyle=solid,fillcolor=white](8.9,0)(9,0.1)(9.1,0)(9,-0.1)
\pspolygon[fillstyle=solid,fillcolor=white](8.875,0)(8.775,0.1)(8.675,0)(8.775,-0.1)
\pspolygon[fillstyle=solid,fillcolor=white](9.125,0)(9.225,0.1)(9.325,0)(9.225,-0.1)
\pspolygon[fillstyle=solid,fillcolor=white](8.9,1)(9,1.1)(9.1,1)(9,0.9)
\pspolygon[fillstyle=solid,fillcolor=white](8.875,1)(8.775,1.1)(8.675,1)(8.775,0.9)
\pspolygon[fillstyle=solid,fillcolor=white](9.125,1)(9.225,1.1)(9.325,1)(9.225,0.9)
\pspolygon[fillstyle=solid,fillcolor=black](9.5,0.4)(9.6,0.5)(9.5,0.6)(9.4,0.5)
\pspolygon[fillstyle=solid,fillcolor=black](8.5,0.4)(8.6,0.5)(8.5,0.6)(8.4,0.5)
\pspolygon[fillstyle=solid,fillcolor=black](8.5,0.625)(8.6,0.725)(8.5,0.825)(8.4,0.725)
\pspolygon[fillstyle=solid,fillcolor=black](8.5,0.375)(8.6,0.275)(8.5,0.175)(8.4,0.275)
\end{pspicture*}
\caption{The $(3,3)$-group $\Gamma_{3,3}$}\label{picture:33}
\end{figure}

\begin{proposition}\label{proposition:33nonrf}
The group $\Gamma_{3,3}$ is irreducible and not residually finite. Moreover, $[B_2(B_1B_3)^2B_2, B_1B_3]$ or $[B_2(B_1B_3)^2B_2, B_1B_3A_2]$ is a non-trivial element of $\Gamma_{3,3}^{(\infty)}$.
\end{proposition}

\begin{proof}
By \cite[Theorems~1.1 and~1.4]{Weiss}, $\Gamma_{3,3}$ is irreducible if and only if the fixator of $B(v_1,3)$ in $\proj_1(\Gamma_{3,3})$ is non-trivial. We first remark that $(B_1 B_2)^2$ fixes $B(v_1,1)$. Hence, $(B_1 B_2)^4$ fixes $B(v_1,2)$ and $(B_1 B_2)^8$ fixes $B(v_1,3)$. Moreover, $(B_1 B_2)^8$ does not fix $B(v_1,4)$ as $(B_1 B_2)^8(A_2 A_3 A_1 A_3(v_1)) = A_2 A_3 A_1 A_2 (v_1)$ (this can be seen by drawing a $4 \times 16$ rectangle). So $\Gamma_{3,3}$ is irreducible.

We now want a non-trivial element in $\Gamma_{3,3}^{(\infty)}$. Recall that $[C_G(H),\overline{H}] \subseteq G^{(\infty)}]$ for any subgroup $H$ of a group $G$ (see \cite[Lemma~4.13]{CapraceNote}). We take $G = \Gamma_{3,3}$ and $H = \langle A_1, A_3, A_2 A_1 A_2, A_2 A_3 A_2 \rangle$. Note that $H$ is a subgroup of $\langle A_1, A_2, A_3 \rangle = \Gamma_{3,3}(v_2)$. Actually, $\Gamma_{3,3}(v_2)$ acts simply transitively on the vertices of $T_1$ and $H$ has two orbits of vertices in $T_1$, so its index in $\Gamma_{3,3}(v_2)$ is~$2$. It is also quick to check that $B_2(B_1B_3)^2B_2 \in C_{\Gamma_{3,3}}(H)$. We now claim that $B_1B_3 \in \overline{\Gamma_{3,3}(v_2)}$. As $A_2 \not \in H$ and $H$ is an index~$2$ subgroup of $\Gamma_{3,3}(v_2)$, it will follow that $B_1B_3 \in \overline{H}$ or $B_1B_3A_2 \in \overline{H}$.

We show that $B_1B_3 \in \overline{\Gamma_{3,3}(v_2)}$ by mimicking the proof of Proposition~\ref{proposition:nonrf}, which was illustrated on Figure~\ref{picture:nonrf}. Consider a finite quotient $\varphi \colon \Gamma_{3,3} \to Q$. Since $\Gamma_{3,3}$ is irreducible, the projection $\proj_2(\Gamma_{3,3}(v_2))$ is infinite. Hence, its finite index subgroup $\proj_2(\Fix_{\Gamma_{3,3}}(B(v_2,1)) \cap \ker \varphi)$ is also infinite. Let $\gamma$ be an element of $\Fix_{\Gamma_{3,3}}(B(v_2,1)) \cap \ker \varphi$ such that $\proj_2(\gamma)$ is non-trivial. In $T_2$, there is a vertex $w \neq v_2$ such that $\gamma$ fixes the path from $v_2$ to $w$ but does not fix some neighbor $z$ of $w$: $\gamma(z) = z' \neq z$. Write $w = h(v_2)$ with $h \in \langle B_1, B_2, B_3 \rangle$, $z = hx(v_2)$ with $x \in \{B_1, B_2, B_3\}$ and $z' = hx'(v_2)$ with $x' \in \{B_1, B_2, B_3\}$. Recall that $\proj_2(\Gamma_{3,3}(v_2))$ acts on the three neighbors of $v_2$ as $\C_2$: the only non-trivial permutation induced on these three vertices is the transposition $(B_1(v_2)\ B_3(v_2))$. We thus have the same local action around $w$, and the fact that $\gamma$ fixes $w$ and some neighbor of $w$ while not fixing $z$ implies that $\{x, x'\} = \{B_1, B_3\}$. Replacing $\gamma$ by $\gamma^{-1}$ if necessary, we can assume that $x = B_3$ and $x' = B_1$. Then we get that $\gamma h B_3 = h B_1 \gamma'$ for some $\gamma' \in \Gamma_{3,3}(v_2)$. As $\varphi(\gamma) = 1$, this implies that $\varphi(B_1 B_3) \in \varphi(\Gamma_{3,3}(v_2))$. This is true for all finite quotients $\varphi \colon \Gamma_{3,3} \to Q$, so $B_1B_3 \in \overline{\Gamma_{3,3}(v_2)}$.
\end{proof}

The same proof can actually show that $[B_2(B_1B_3)^2B_2, (B_1B_3)^2]$ is always a non-trivial element of $\Gamma_{3,3}^{(\infty)}$. But this element has length~$20$ and our computer could hardly deal with it. Instead, the elements $[B_2(B_1B_3)^2B_2, B_1B_3]$ and
$$[B_2(B_1B_3)^2B_2, B_1B_3A_2] = B_2(B_1B_3)^2B_2B_1B_3B_2(B_1B_3)^2B_2B_3B_1$$
given by Proposition~\ref{proposition:33nonrf} have length~$16$ which is slightly better.

Using GAP we could search for $(4,5)$-groups $\Gamma$ containing $\Gamma_{3,3}$ or the mirror of $\Gamma_{3,3}$ (i.e.\ $\{(g_1,g_2) \in \Aut(T_1) \times \Aut(T_2) \mid (g_2,g_1) \in \Gamma_{3,3}\}$), and such that $\underline{H_1}(v_1) \geq \mathrm{PSL}(2,3)$ and $\underline{H_2}(v_2) \geq \mathrm{PSL}(2,4)$. We say that $\Gamma$ satisfies $(**)$ if the above conditions are true. Since $\Gamma_{3,3}$ is irreducible, a group $\Gamma$ satisfying $(**)$ is also irreducible and $H_t$ is $2$-transitive on $\partial T_t$ for each $t \in \{1,2\}$ by Theorem~\ref{theorem:Trofimov}. Thus the NST applies and $\Gamma$ is virtually simple.

There are $60$ equivalence classes of $(4,5)$-groups satisfying $(**)$, all with $\tau_1 = 4$ and $\tau_2 = 5$ (i.e.\ with nine generators, each of order~$2$). We give in Tables~\ref{table:simple45-1}--\ref{table:simple45-2} a group in each class: call them $\Gamma_{4,5,1}, \ldots, \Gamma_{4,5,60}$. Note that $\Gamma_{4,5,1}, \ldots, \Gamma_{4,5,28}$ contain $\Gamma_{3,3}$ while $\Gamma_{4,5,29}, \ldots, \Gamma_{4,5,60}$ contain its mirror. We can make some remarks similar to those in~\S\ref{subsection:66}:

\begin{itemize}
\item The index of the simple subgroup $\Gamma_{4,5,k}^{(\infty)}$ of $\Gamma_{4,5,k}$ can be computed by using the fact that $r_1 = [B_2(B_1B_3)^2B_2, B_1B_3]$ or $r_2 = [B_2(B_1B_3)^2B_2, B_1B_3A_2]$ belongs to $\Gamma_{3,3}^{(\infty)}$. Indeed, if $Q_1$ (resp.\ $Q_2$) is the group obtained by adding the relator $r_1$ (resp.\ $r_2$) to the relators of $\Gamma_{4,5,k}$, then $[\Gamma_{4,5,k} : \Gamma_{4,5,k}^{(\infty)}] = \max(|Q_1|,|Q_2|)$. (For $k \geq 29$ we must actually consider the mirrors of $r_1$ and $r_2$). The indices that we obtain are written in the tables. When the index is~$4$ we have that $\Gamma_{4,5,k}^{(\infty)} = \Gamma_{4,5,k}^+$.
\item For each $k \in \{1,\ldots,60\}$ we get $\Aut(X_{\Gamma_{4,5,k}^+}) \cong \C_2 \times \C_2$, so $\Gamma_{4,5,k}$ is the only $(4,5)$-group whose type-preserving subgroup is $\Gamma_{4,5,k}^+$. Therefore all $\Gamma_{4,5,k}^+$ are pairwise non-conjugate in $\Aut(T_1 \times T_2)$ (and thus pairwise non-isomorphic by \cite[Corollary~1.1.22]{BMZ}).
\end{itemize}

\begin{theorem}[Theorem~\ref{maintheorem:simple66-45}(ii)]
Let $\Gamma_{4,5,k}$ ($k \in \{1,\ldots,60\}$) be one of the $(4,5)$-groups given by Tables~\ref{table:simple45-1}--\ref{table:simple45-2}.
\begin{itemize}
\item If $k \in \{1,\ldots,32\} \cup \{39,\ldots,54\}$, then $\Gamma_{4,5,k}^+$ is simple.
\item If $k \in \{33,\ldots,38\} \cup \{55,\ldots,60\}$, then $\Gamma_{4,5,k}^+$ has a simple subgroup of index~$2$.
\end{itemize}
Moreover, all groups $\Gamma_{4,5,k}^+$ are pairwise non-isomorphic.
\end{theorem}

\begin{proof}
See the discussion above.
\end{proof}

\begin{corollary}
For each $k \in \{1,\ldots,32\} \cup \{39,\ldots,54\}$, there exist two injections $F_{11}\hookrightarrow F_3$ of free groups such that the simple group $\Gamma_{4,5,k}^+$ is isomorphic to the amalgamated free product $F_3 \ast_{F_{11}} F_3$.
\end{corollary}

\begin{proof}
Recall that $\Gamma_{4,5,k}(v_2) = \langle A_1, A_2, A_3, A_4 \rangle$, where $A_1$, $A_2$, $A_3$ and $A_4$ are order~$2$ elements sending $v_1$ to its four neighbours in $T_1$. Hence, the index~$2$ subgroup $G = \Gamma_{4,5,k}^+(v_2)$ is a free group of rank~$3$, on the $3$ generators $A_1A_2$, $A_1A_3$ and $A_1A_4$. If $v'_2$ is a vertex adjacent to $v_2$ in $T_2$, then $G' = \Gamma_{4,5,k}^+(v'_2)$ is also isomorphic to $F_3$. Moreover, these two point stabilizers $G$ and $G'$ generate $\Gamma_{4,5,k}^+$ so that $\Gamma_{4,5,k}^+ = G \ast_{G \cap G'} G'$. The subgroup $G \cap G'$ has index~$5$ in both $G$ and $G'$, so $G \cap G'$ is free of rank $1 + 5(3-1) = 11$ by the Nielsen-Schreier formula.
\end{proof}

\begin{proof}[Proof of Corollary~\ref{maincorollary:presentations}(ii)]
\begin{figure}[b!]
\centering
\begin{pspicture*}(-0.2,-2.1)(11.6,1.4)
\fontsize{10pt}{10pt}\selectfont
\psset{unit=1.2cm}

\pspolygon(0,0)(1,0)(1,1)(0,1)

\pspolygon[fillstyle=solid,fillcolor=white](0.4,0)(0.5,0.1)(0.6,0)(0.5,-0.1)
\pspolygon[fillstyle=solid,fillcolor=white](0.4,1)(0.5,1.1)(0.6,1)(0.5,0.9)
\pspolygon[fillstyle=solid,fillcolor=black](0,0.4)(0.1,0.5)(0,0.6)(-0.1,0.5)
\pspolygon[fillstyle=solid,fillcolor=black](1,0.4)(1.1,0.5)(1,0.6)(0.9,0.5)

\pspolygon(1.7,0)(2.7,0)(2.7,1)(1.7,1)

\pspolygon[fillstyle=solid,fillcolor=white](2.1,0)(2.2,0.1)(2.3,0)(2.2,-0.1)
\pspolygon[fillstyle=solid,fillcolor=white](2.1,1)(2.2,1.1)(2.3,1)(2.2,0.9)
\pspolygon[fillstyle=solid,fillcolor=black](1.7,0.525)(1.8,0.625)(1.7,0.725)(1.6,0.625)
\pspolygon[fillstyle=solid,fillcolor=black](1.7,0.475)(1.8,0.375)(1.7,0.275)(1.6,0.375)
\pspolygon[fillstyle=solid,fillcolor=black](2.7,0.525)(2.8,0.625)(2.7,0.725)(2.6,0.625)
\pspolygon[fillstyle=solid,fillcolor=black](2.7,0.475)(2.8,0.375)(2.7,0.275)(2.6,0.375)

\pspolygon(3.4,0)(4.4,0)(4.4,1)(3.4,1)

\pspolygon[fillstyle=solid,fillcolor=white](3.8,0)(3.9,0.1)(4,0)(3.9,-0.1)
\pspolygon[fillstyle=solid,fillcolor=white](3.925,1)(4.025,1.1)(4.125,1)(4.025,0.9)
\pspolygon[fillstyle=solid,fillcolor=white](3.875,1)(3.775,1.1)(3.675,1)(3.775,0.9)
\pspolygon[fillstyle=solid,fillcolor=black](4.4,0.4)(4.5,0.5)(4.4,0.6)(4.3,0.5)
\pspolygon[fillstyle=solid,fillcolor=black](4.4,0.625)(4.5,0.725)(4.4,0.825)(4.3,0.725)
\pspolygon[fillstyle=solid,fillcolor=black](4.4,0.375)(4.5,0.275)(4.4,0.175)(4.3,0.275)
\pspolygon[fillstyle=solid,fillcolor=black](3.4,0.4)(3.5,0.5)(3.4,0.6)(3.3,0.5)
\pspolygon[fillstyle=solid,fillcolor=black](3.4,0.625)(3.5,0.725)(3.4,0.825)(3.3,0.725)
\pspolygon[fillstyle=solid,fillcolor=black](3.4,0.375)(3.5,0.275)(3.4,0.175)(3.3,0.275)

\pspolygon(5.1,0)(6.1,0)(6.1,1)(5.1,1)

\pspolygon[fillstyle=solid,fillcolor=white](5.625,0)(5.725,0.1)(5.825,0)(5.725,-0.1)
\pspolygon[fillstyle=solid,fillcolor=white](5.575,0)(5.475,0.1)(5.375,0)(5.475,-0.1)
\pspolygon[fillstyle=solid,fillcolor=white](5.625,1)(5.725,1.1)(5.825,1)(5.725,0.9)
\pspolygon[fillstyle=solid,fillcolor=white](5.575,1)(5.475,1.1)(5.375,1)(5.475,0.9)
\pspolygon[fillstyle=solid,fillcolor=black](5.1,0.4)(5.2,0.5)(5.1,0.6)(5,0.5)
\pspolygon[fillstyle=solid,fillcolor=black](6.1,0.4)(6.2,0.5)(6.1,0.6)(6,0.5)

\pspolygon(6.8,0)(7.8,0)(7.8,1)(6.8,1)

\pspolygon[fillstyle=solid,fillcolor=white](7.325,0)(7.425,0.1)(7.525,0)(7.425,-0.1)
\pspolygon[fillstyle=solid,fillcolor=white](7.275,0)(7.175,0.1)(7.075,0)(7.175,-0.1)
\pspolygon[fillstyle=solid,fillcolor=white](7.2,1)(7.3,1.1)(7.4,1)(7.3,0.9)
\pspolygon[fillstyle=solid,fillcolor=white](7.175,1)(7.075,1.1)(6.975,1)(7.075,0.9)
\pspolygon[fillstyle=solid,fillcolor=white](7.425,1)(7.525,1.1)(7.625,1)(7.525,0.9)
\pspolygon[fillstyle=solid,fillcolor=black](6.8,0.525)(6.9,0.625)(6.8,0.725)(6.7,0.625)
\pspolygon[fillstyle=solid,fillcolor=black](6.8,0.475)(6.9,0.375)(6.8,0.275)(6.7,0.375)
\pspolygon[fillstyle=solid,fillcolor=black](7.8,0.525)(7.9,0.625)(7.8,0.725)(7.7,0.625)
\pspolygon[fillstyle=solid,fillcolor=black](7.8,0.475)(7.9,0.375)(7.8,0.275)(7.7,0.375)

\pspolygon(8.5,0)(9.5,0)(9.5,1)(8.5,1)

\pspolygon[fillstyle=solid,fillcolor=white](8.9,0)(9,0.1)(9.1,0)(9,-0.1)
\pspolygon[fillstyle=solid,fillcolor=white](8.875,0)(8.775,0.1)(8.675,0)(8.775,-0.1)
\pspolygon[fillstyle=solid,fillcolor=white](9.125,0)(9.225,0.1)(9.325,0)(9.225,-0.1)
\pspolygon[fillstyle=solid,fillcolor=white](8.9,1)(9,1.1)(9.1,1)(9,0.9)
\pspolygon[fillstyle=solid,fillcolor=white](8.875,1)(8.775,1.1)(8.675,1)(8.775,0.9)
\pspolygon[fillstyle=solid,fillcolor=white](9.125,1)(9.225,1.1)(9.325,1)(9.225,0.9)
\pspolygon[fillstyle=solid,fillcolor=black](9.5,0.4)(9.6,0.5)(9.5,0.6)(9.4,0.5)
\pspolygon[fillstyle=solid,fillcolor=black](8.5,0.4)(8.6,0.5)(8.5,0.6)(8.4,0.5)
\pspolygon[fillstyle=solid,fillcolor=black](8.5,0.625)(8.6,0.725)(8.5,0.825)(8.4,0.725)
\pspolygon[fillstyle=solid,fillcolor=black](8.5,0.375)(8.6,0.275)(8.5,0.175)(8.4,0.275)


\pspolygon(0.85,-1.5)(1.85,-1.5)(1.85,-0.5)(0.85,-0.5)

\pspolygon[fillstyle=solid,fillcolor=white](1.25,-1.5)(1.35,-1.4)(1.45,-1.5)(1.35,-1.6)
\pspolygon[fillstyle=solid,fillcolor=white](1.25,-0.5)(1.35,-0.4)(1.45,-0.5)(1.35,-0.6)
\pspolygon[fillstyle=solid,fillcolor=black](1.85,-1.325)(1.95,-1.225)(1.85,-1.125)(1.75,-1.225)
\pspolygon[fillstyle=solid,fillcolor=black](1.85,-0.675)(1.95,-0.775)(1.85,-0.875)(1.75,-0.775)
\rput(1.99,-1){$4$}
\pspolygon[fillstyle=solid,fillcolor=black](0.85,-1.325)(0.95,-1.225)(0.85,-1.125)(0.75,-1.225)
\pspolygon[fillstyle=solid,fillcolor=black](0.85,-0.675)(0.95,-0.775)(0.85,-0.875)(0.75,-0.775)
\rput(0.7,-1){$4$}

\pspolygon(2.55,-1.5)(3.55,-1.5)(3.55,-0.5)(2.55,-0.5)

\pspolygon[fillstyle=solid,fillcolor=white](2.95,-1.5)(3.05,-1.4)(3.15,-1.5)(3.05,-1.6)
\pspolygon[fillstyle=solid,fillcolor=white](3.175,-0.5)(3.275,-0.4)(3.375,-0.5)(3.275,-0.6)
\pspolygon[fillstyle=solid,fillcolor=white](2.925,-0.5)(2.825,-0.4)(2.725,-0.5)(2.825,-0.6)
\rput(3.05,-0.35){$4$}
\pspolygon[fillstyle=solid,fillcolor=black](3.55,-1.325)(3.65,-1.225)(3.55,-1.125)(3.45,-1.225)
\pspolygon[fillstyle=solid,fillcolor=black](3.55,-0.675)(3.65,-0.775)(3.55,-0.875)(3.45,-0.775)
\rput(3.69,-1){$5$}
\pspolygon[fillstyle=solid,fillcolor=black](2.55,-1.325)(2.65,-1.225)(2.55,-1.125)(2.45,-1.225)
\pspolygon[fillstyle=solid,fillcolor=black](2.55,-0.675)(2.65,-0.775)(2.55,-0.875)(2.45,-0.775)
\rput(2.4,-1){$5$}

\pspolygon(4.25,-1.5)(5.25,-1.5)(5.25,-0.5)(4.25,-0.5)

\pspolygon[fillstyle=solid,fillcolor=white](4.525,-1.5)(4.625,-1.4)(4.725,-1.5)(4.625,-1.6)
\pspolygon[fillstyle=solid,fillcolor=white](4.775,-1.5)(4.875,-1.4)(4.975,-1.5)(4.875,-1.6)
\pspolygon[fillstyle=solid,fillcolor=white](4.65,-0.5)(4.75,-0.4)(4.85,-0.5)(4.75,-0.6)
\pspolygon[fillstyle=solid,fillcolor=white](4.875,-0.5)(4.975,-0.4)(5.075,-0.5)(4.975,-0.6)
\pspolygon[fillstyle=solid,fillcolor=white](4.625,-0.5)(4.525,-0.4)(4.425,-0.5)(4.525,-0.6)
\pspolygon[fillstyle=solid,fillcolor=black](5.25,-1.325)(5.35,-1.225)(5.25,-1.125)(5.15,-1.225)
\pspolygon[fillstyle=solid,fillcolor=black](5.25,-0.675)(5.35,-0.775)(5.25,-0.875)(5.15,-0.775)
\rput(5.39,-1){$4$}
\pspolygon[fillstyle=solid,fillcolor=black](4.25,-1.325)(4.35,-1.225)(4.25,-1.125)(4.15,-1.225)
\pspolygon[fillstyle=solid,fillcolor=black](4.25,-0.675)(4.35,-0.775)(4.25,-0.875)(4.15,-0.775)
\rput(4.1,-1){$5$}

\pspolygon(5.95,-1.5)(6.95,-1.5)(6.95,-0.5)(5.95,-0.5)

\pspolygon[fillstyle=solid,fillcolor=white](6.575,-1.5)(6.675,-1.4)(6.775,-1.5)(6.675,-1.6)
\pspolygon[fillstyle=solid,fillcolor=white](6.325,-1.5)(6.225,-1.4)(6.125,-1.5)(6.225,-1.6)
\rput(6.45,-1.65){$4$}
\pspolygon[fillstyle=solid,fillcolor=white](6.575,-0.5)(6.675,-0.4)(6.775,-0.5)(6.675,-0.6)
\pspolygon[fillstyle=solid,fillcolor=white](6.325,-0.5)(6.225,-0.4)(6.125,-0.5)(6.225,-0.6)
\rput(6.45,-0.35){$4$}
\pspolygon[fillstyle=solid,fillcolor=black](6.95,-0.9)(7.05,-1)(6.95,-1.1)(6.85,-1)
\pspolygon[fillstyle=solid,fillcolor=black](5.95,-1.225)(6.05,-1.125)(5.95,-1.025)(5.85,-1.125)
\pspolygon[fillstyle=solid,fillcolor=black](5.95,-0.775)(6.05,-0.875)(5.95,-0.975)(5.85,-0.875)

\pspolygon(7.65,-1.5)(8.65,-1.5)(8.65,-0.5)(7.65,-0.5)

\pspolygon[fillstyle=solid,fillcolor=white](8.275,-1.5)(8.375,-1.4)(8.475,-1.5)(8.375,-1.6)
\pspolygon[fillstyle=solid,fillcolor=white](8.025,-1.5)(7.925,-1.4)(7.825,-1.5)(7.925,-1.6)
\rput(8.15,-1.65){$4$}
\pspolygon[fillstyle=solid,fillcolor=white](8.275,-0.5)(8.375,-0.4)(8.475,-0.5)(8.375,-0.6)
\pspolygon[fillstyle=solid,fillcolor=white](8.025,-0.5)(7.925,-0.4)(7.825,-0.5)(7.925,-0.6)
\rput(8.15,-0.35){$4$}
\pspolygon[fillstyle=solid,fillcolor=black](8.65,-0.675)(8.75,-0.775)(8.65,-0.875)(8.55,-0.775)
\pspolygon[fillstyle=solid,fillcolor=black](8.65,-0.9)(8.75,-1)(8.65,-1.1)(8.55,-1)
\pspolygon[fillstyle=solid,fillcolor=black](8.65,-1.125)(8.75,-1.225)(8.65,-1.325)(8.55,-1.225)
\pspolygon[fillstyle=solid,fillcolor=black](7.65,-0.675)(7.75,-0.775)(7.65,-0.875)(7.55,-0.775)
\pspolygon[fillstyle=solid,fillcolor=black](7.65,-1.125)(7.75,-1.225)(7.65,-1.325)(7.55,-1.225)
\rput(7.5,-1){$4$}
\end{pspicture*}
\caption{The $(4,5)$-group $\Gamma_{4,5,9}$}\label{picture:459}
\end{figure}
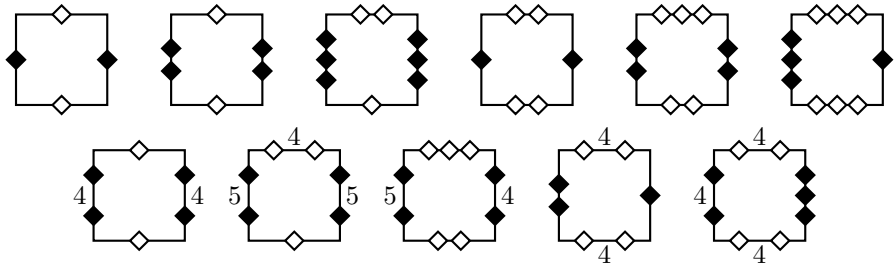
This is the presentation of $\Gamma_{4,5,9}^+$ (see Figure~\ref{picture:459}). In order to find this presentation, we write $x_1 = A_1A_2$, $x_2 = A_1A_3$ and $x_3 = A_1A_4$ so that $G = \Gamma_{4,5,9}^+(v_2)$ is freely generated by $x_1$, $x_2$ and $x_3$. In the same manner, the stabilizer $G' = \Gamma_{4,5,9}^+(B_1(v_2))$ of the vertex $B_1(v_2)$ (which is adjacent to $v_2$ in $T_2$) is freely generated by $y_1 = B_1A_1A_2B_1$, $y_2 = B_1A_1A_3B_1$ and $y_3 = B_1A_1A_4B_1$. The subgroup $G \cap G'$ of $G$ has index~$5$ in $G$, and the Reidemeister--Schreier method can be used to find $11$ generators of this subgroup (which is isomorphic to $F_{11}$). After some computations we came up with the $11$ elements of $G = \langle x_1, x_2, x_3\rangle$ written on the left-hand sides of the $11$ relations in the presentation. Those $11$ elements also belong to $G' = \langle y_1, y_2, y_3 \rangle$, and it then suffices to write them in terms of the $y_i$'s. For instance, for $x_2^2 \in G$ we have 
\begin{align*}
x_2^2 &= (A_1A_3)^2\\
&= B_1(B_1(A_1A_3)^2B_1)B_1\\
&= B_1(A_1A_3A_2A_3)B_1\\
&= B_1(A_1A_3A_2A_1A_1A_3)B_1\\
&= B_1(X_2X_1^{-1}X_2)B_1\\
&= y_2y_1^{-1}y_2
\end{align*}
We used the geometric squares defining $\Gamma_{4,5,9}$ to find the equality $B_1(A_1A_3)^2B_1 = A_1A_3A_2A_3$.
\end{proof}

\begin{landscape}
\begin{table}
\scriptsize
\centering
\begin{tabular}{|c|l|c|c|c|c|}
\hline
Name &
 Squares: $A_1B_1A_1B_1$, $A_1B_2A_1B_2$, $A_1B_3A_2B_3$, $A_2B_1A_2B_1$, $A_2B_2A_3B_2$, $A_3B_1A_3B_3$ +
 & $\underline{H_1}(v_1)$ & $\underline{H_2}(v_2)$ & $[\Gamma : \Gamma^{(\infty)}]$ & $\Aut(X_{\Gamma^+})$\\
\hline
$\Gamma_{4,5,1}$ & $ A_1 B_4 A_1 B_4 $, $ A_1 B_5 A_1 B_5 $, $ A_2 B_4 A_2 B_5 $, $ A_3 B_4 A_3 B_4 $, $ A_3 B_5 A_4 B_5 $, $ A_4 B_1 A_4 B_2 $, $ A_4 B_3 A_4 B_4 $ & $\Sym(4)$ & $\Sym(5)$ & 4 & $\C_2 \times \C_2$ \\
$\Gamma_{4,5,2}$ & $ A_1 B_4 A_1 B_4 $, $ A_1 B_5 A_1 B_5 $, $ A_2 B_4 A_2 B_5 $, $ A_3 B_4 A_3 B_4 $, $ A_3 B_5 A_4 B_5 $, $ A_4 B_1 A_4 B_4 $, $ A_4 B_2 A_4 B_3 $ & $\Sym(4)$ & $\Sym(5)$ & 4 & $\C_2 \times \C_2$ \\
$\Gamma_{4,5,3}$ & $ A_1 B_4 A_1 B_4 $, $ A_1 B_5 A_3 B_5 $, $ A_2 B_4 A_2 B_5 $, $ A_3 B_4 A_4 B_4 $, $ A_4 B_1 A_4 B_2 $, $ A_4 B_3 A_4 B_5 $ & $\Sym(4)$ & $\Sym(5)$ & 4 & $\C_2 \times \C_2$ \\
$\Gamma_{4,5,4}$ & $ A_1 B_4 A_1 B_4 $, $ A_1 B_5 A_3 B_5 $, $ A_2 B_4 A_2 B_5 $, $ A_3 B_4 A_4 B_4 $, $ A_4 B_1 A_4 B_5 $, $ A_4 B_2 A_4 B_3 $ & $\Sym(4)$ & $\Sym(5)$ & 4 & $\C_2 \times \C_2$ \\
$\Gamma_{4,5,5}$ & $ A_1 B_4 A_1 B_4 $, $ A_1 B_5 A_4 B_5 $, $ A_2 B_4 A_2 B_5 $, $ A_3 B_4 A_3 B_4 $, $ A_3 B_5 A_3 B_5 $, $ A_4 B_1 A_4 B_2 $, $ A_4 B_3 A_4 B_4 $ & $\Sym(4)$ & $\Sym(5)$ & 4 & $\C_2 \times \C_2$ \\
$\Gamma_{4,5,6}$ & $ A_1 B_4 A_1 B_4 $, $ A_1 B_5 A_4 B_5 $, $ A_2 B_4 A_2 B_5 $, $ A_3 B_4 A_3 B_4 $, $ A_3 B_5 A_3 B_5 $, $ A_4 B_1 A_4 B_4 $, $ A_4 B_2 A_4 B_3 $ & $\Sym(4)$ & $\Sym(5)$ & 4 & $\C_2 \times \C_2$ \\
$\Gamma_{4,5,7}$ & $ A_1 B_4 A_1 B_4 $, $ A_1 B_5 A_4 B_5 $, $ A_2 B_4 A_2 B_5 $, $ A_3 B_4 A_3 B_5 $, $ A_4 B_1 A_4 B_2 $, $ A_4 B_3 A_4 B_4 $ & $\Sym(4)$ & $\Sym(5)$ & 4 & $\C_2 \times \C_2$ \\
$\Gamma_{4,5,8}$ & $ A_1 B_4 A_1 B_4 $, $ A_1 B_5 A_4 B_5 $, $ A_2 B_4 A_2 B_5 $, $ A_3 B_4 A_3 B_5 $, $ A_4 B_1 A_4 B_4 $, $ A_4 B_2 A_4 B_3 $ & $\Sym(4)$ & $\Sym(5)$ & 4 & $\C_2 \times \C_2$ \\
$\Gamma_{4,5,9}$ & $ A_1 B_4 A_1 B_4 $, $ A_1 B_5 A_4 B_5 $, $ A_2 B_4 A_3 B_5 $, $ A_4 B_1 A_4 B_2 $, $ A_4 B_3 A_4 B_4 $ & $\Sym(4)$ & $\Sym(5)$ & 4 & $\C_2 \times \C_2$ \\
$\Gamma_{4,5,10}$ & $ A_1 B_4 A_1 B_4 $, $ A_1 B_5 A_4 B_5 $, $ A_2 B_4 A_3 B_5 $, $ A_4 B_1 A_4 B_4 $, $ A_4 B_2 A_4 B_3 $ & $\Sym(4)$ & $\Sym(5)$ & 4 & $\C_2 \times \C_2$ \\
$\Gamma_{4,5,11}$ & $ A_1 B_4 A_3 B_4 $, $ A_1 B_5 A_4 B_5 $, $ A_2 B_4 A_2 B_5 $, $ A_3 B_5 A_3 B_5 $, $ A_4 B_1 A_4 B_2 $, $ A_4 B_3 A_4 B_4 $ & $\Sym(4)$ & $\Sym(5)$ & 4 & $\C_2 \times \C_2$ \\
$\Gamma_{4,5,12}$ & $ A_1 B_4 A_3 B_4 $, $ A_1 B_5 A_4 B_5 $, $ A_2 B_4 A_2 B_5 $, $ A_3 B_5 A_3 B_5 $, $ A_4 B_1 A_4 B_4 $, $ A_4 B_2 A_4 B_3 $ & $\Sym(4)$ & $\Sym(5)$ & 4 & $\C_2 \times \C_2$ \\
$\Gamma_{4,5,13}$ & $ A_1 B_4 A_1 B_5 $, $ A_2 B_4 A_2 B_4 $, $ A_2 B_5 A_2 B_5 $, $ A_3 B_4 A_3 B_4 $, $ A_3 B_5 A_4 B_5 $, $ A_4 B_1 A_4 B_2 $, $ A_4 B_3 A_4 B_4 $ & $\Sym(4)$ & $\Sym(5)$ & 4 & $\C_2 \times \C_2$ \\
$\Gamma_{4,5,14}$ & $ A_1 B_4 A_1 B_5 $, $ A_2 B_4 A_2 B_4 $, $ A_2 B_5 A_2 B_5 $, $ A_3 B_4 A_3 B_4 $, $ A_3 B_5 A_4 B_5 $, $ A_4 B_1 A_4 B_4 $, $ A_4 B_2 A_4 B_3 $ & $\Sym(4)$ & $\Sym(5)$ & 4 & $\C_2 \times \C_2$ \\
$\Gamma_{4,5,15}$ & $ A_1 B_4 A_1 B_5 $, $ A_2 B_4 A_2 B_4 $, $ A_2 B_5 A_3 B_5 $, $ A_3 B_4 A_4 B_4 $, $ A_4 B_1 A_4 B_2 $, $ A_4 B_3 A_4 B_5 $ & $\Sym(4)$ & $\Sym(5)$ & 4 & $\C_2 \times \C_2$ \\
$\Gamma_{4,5,16}$ & $ A_1 B_4 A_1 B_5 $, $ A_2 B_4 A_2 B_4 $, $ A_2 B_5 A_3 B_5 $, $ A_3 B_4 A_4 B_4 $, $ A_4 B_1 A_4 B_5 $, $ A_4 B_2 A_4 B_3 $ & $\Sym(4)$ & $\Sym(5)$ & 4 & $\C_2 \times \C_2$ \\
$\Gamma_{4,5,17}$ & $ A_1 B_4 A_1 B_5 $, $ A_2 B_4 A_2 B_4 $, $ A_2 B_5 A_4 B_5 $, $ A_3 B_4 A_3 B_4 $, $ A_3 B_5 A_3 B_5 $, $ A_4 B_1 A_4 B_2 $, $ A_4 B_3 A_4 B_4 $ & $\Sym(4)$ & $\Sym(5)$ & 4 & $\C_2 \times \C_2$ \\
$\Gamma_{4,5,18}$ & $ A_1 B_4 A_1 B_5 $, $ A_2 B_4 A_2 B_4 $, $ A_2 B_5 A_4 B_5 $, $ A_3 B_4 A_3 B_4 $, $ A_3 B_5 A_3 B_5 $, $ A_4 B_1 A_4 B_4 $, $ A_4 B_2 A_4 B_3 $ & $\Sym(4)$ & $\Sym(5)$ & 4 & $\C_2 \times \C_2$ \\
$\Gamma_{4,5,19}$ & $ A_1 B_4 A_1 B_5 $, $ A_2 B_4 A_2 B_4 $, $ A_2 B_5 A_4 B_5 $, $ A_3 B_4 A_3 B_5 $, $ A_4 B_1 A_4 B_2 $, $ A_4 B_3 A_4 B_4 $ & $\Sym(4)$ & $\Sym(5)$ & 4 & $\C_2 \times \C_2$ \\
$\Gamma_{4,5,20}$ & $ A_1 B_4 A_1 B_5 $, $ A_2 B_4 A_2 B_4 $, $ A_2 B_5 A_4 B_5 $, $ A_3 B_4 A_3 B_5 $, $ A_4 B_1 A_4 B_4 $, $ A_4 B_2 A_4 B_3 $ & $\Sym(4)$ & $\Sym(5)$ & 4 & $\C_2 \times \C_2$ \\
$\Gamma_{4,5,21}$ & $ A_1 B_4 A_1 B_5 $, $ A_2 B_4 A_3 B_4 $, $ A_2 B_5 A_4 B_5 $, $ A_3 B_5 A_3 B_5 $, $ A_4 B_1 A_4 B_2 $, $ A_4 B_3 A_4 B_4 $ & $\Sym(4)$ & $\Sym(5)$ & 4 & $\C_2 \times \C_2$ \\
$\Gamma_{4,5,22}$ & $ A_1 B_4 A_1 B_5 $, $ A_2 B_4 A_3 B_4 $, $ A_2 B_5 A_4 B_5 $, $ A_3 B_5 A_3 B_5 $, $ A_4 B_1 A_4 B_4 $, $ A_4 B_2 A_4 B_3 $ & $\Sym(4)$ & $\Sym(5)$ & 4 & $\C_2 \times \C_2$ \\
$\Gamma_{4,5,23}$ & $ A_1 B_4 A_1 B_5 $, $ A_2 B_4 A_2 B_5 $, $ A_3 B_4 A_3 B_4 $, $ A_3 B_5 A_4 B_5 $, $ A_4 B_1 A_4 B_2 $, $ A_4 B_3 A_4 B_4 $ & $\Sym(4)$ & $\Sym(5)$ & 4 & $\C_2 \times \C_2$ \\
$\Gamma_{4,5,24}$ & $ A_1 B_4 A_1 B_5 $, $ A_2 B_4 A_2 B_5 $, $ A_3 B_4 A_3 B_4 $, $ A_3 B_5 A_4 B_5 $, $ A_4 B_1 A_4 B_4 $, $ A_4 B_2 A_4 B_3 $ & $\Sym(4)$ & $\Sym(5)$ & 4 & $\C_2 \times \C_2$ \\
$\Gamma_{4,5,25}$ & $ A_1 B_4 A_2 B_5 $, $ A_3 B_4 A_3 B_4 $, $ A_3 B_5 A_4 B_5 $, $ A_4 B_1 A_4 B_2 $, $ A_4 B_3 A_4 B_4 $ & $\Sym(4)$ & $\Sym(5)$ & 4 & $\C_2 \times \C_2$ \\
$\Gamma_{4,5,26}$ & $ A_1 B_4 A_2 B_5 $, $ A_3 B_4 A_3 B_4 $, $ A_3 B_5 A_4 B_5 $, $ A_4 B_1 A_4 B_4 $, $ A_4 B_2 A_4 B_3 $ & $\Sym(4)$ & $\Sym(5)$ & 4 & $\C_2 \times \C_2$ \\
$\Gamma_{4,5,27}$ & $ A_1 B_4 A_3 B_5 $, $ A_2 B_4 A_2 B_4 $, $ A_2 B_5 A_4 B_5 $, $ A_4 B_1 A_4 B_2 $, $ A_4 B_3 A_4 B_4 $ & $\Sym(4)$ & $\Sym(5)$ & 4 & $\C_2 \times \C_2$ \\
$\Gamma_{4,5,28}$ & $ A_1 B_4 A_3 B_5 $, $ A_2 B_4 A_2 B_4 $, $ A_2 B_5 A_4 B_5 $, $ A_4 B_1 A_4 B_4 $, $ A_4 B_2 A_4 B_3 $ & $\Sym(4)$ & $\Sym(5)$ & 4 & $\C_2 \times \C_2$ \\
\hline
\end{tabular}
\caption{Some virtually simple $(4,5)$-groups containing $\Gamma_{3,3}$.}\label{table:simple45-1}
\end{table}

\begin{table}
\scriptsize
\centering
\begin{tabular}{|c|l|c|c|c|c|}
\hline
Name &
 Squares: $A_1B_1A_1B_1$, $A_1B_2A_1B_2$, $A_1B_3A_3B_3$, $A_2B_1A_2B_1$, $A_2B_2A_2B_3$, $A_3B_1A_3B_2$ +
 & $\underline{H_1}(v_1)$ & $\underline{H_2}(v_2)$ & $[\Gamma : \Gamma^{(\infty)}]$ & $\Aut(X_{\Gamma^+})$\\
\hline
$\Gamma_{4,5,29}$ & $ A_1 B_4 A_1 B_4 $, $ A_1 B_5 A_2 B_5 $, $ A_2 B_4 A_4 B_4 $, $ A_3 B_4 A_3 B_5 $, $ A_4 B_1 A_4 B_1 $, $ A_4 B_2 A_4 B_2 $, $ A_4 B_3 A_4 B_5 $ & $\Sym(4)$ & $\Sym(5)$ & 4 & $\C_2 \times \C_2$ \\
$\Gamma_{4,5,30}$ & $ A_1 B_4 A_1 B_4 $, $ A_1 B_5 A_2 B_5 $, $ A_2 B_4 A_4 B_4 $, $ A_3 B_4 A_3 B_5 $, $ A_4 B_1 A_4 B_1 $, $ A_4 B_2 A_4 B_5 $, $ A_4 B_3 A_4 B_3 $ & $\Sym(4)$ & $\Sym(5)$ & 4 & $\C_2 \times \C_2$ \\
$\Gamma_{4,5,31}$ & $ A_1 B_4 A_1 B_4 $, $ A_1 B_5 A_2 B_5 $, $ A_2 B_4 A_4 B_4 $, $ A_3 B_4 A_3 B_5 $, $ A_4 B_1 A_4 B_2 $, $ A_4 B_3 A_4 B_5 $ & $\Sym(4)$ & $\Sym(5)$ & 4 & $\C_2 \times \C_2$ \\
$\Gamma_{4,5,32}$ & $ A_1 B_4 A_1 B_4 $, $ A_1 B_5 A_2 B_5 $, $ A_2 B_4 A_4 B_4 $, $ A_3 B_4 A_3 B_5 $, $ A_4 B_1 A_4 B_3 $, $ A_4 B_2 A_4 B_5 $ & $\Sym(4)$ & $\Sym(5)$ & 4 & $\C_2 \times \C_2$ \\
$\Gamma_{4,5,33}$ & $ A_1 B_4 A_1 B_4 $, $ A_1 B_5 A_4 B_5 $, $ A_2 B_4 A_3 B_5 $, $ A_4 B_1 A_4 B_1 $, $ A_4 B_2 A_4 B_2 $, $ A_4 B_3 A_4 B_4 $ & $\Sym(4)$ & $\Sym(5)$ & 8 & $\C_2 \times \C_2$ \\
$\Gamma_{4,5,34}$ & $ A_1 B_4 A_1 B_4 $, $ A_1 B_5 A_4 B_5 $, $ A_2 B_4 A_3 B_5 $, $ A_4 B_1 A_4 B_1 $, $ A_4 B_2 A_4 B_4 $, $ A_4 B_3 A_4 B_3 $ & $\Sym(4)$ & $\Sym(5)$ & 8 & $\C_2 \times \C_2$ \\
$\Gamma_{4,5,35}$ & $ A_1 B_4 A_1 B_4 $, $ A_1 B_5 A_4 B_5 $, $ A_2 B_4 A_3 B_5 $, $ A_4 B_1 A_4 B_2 $, $ A_4 B_3 A_4 B_4 $ & $\Sym(4)$ & $\Alt(5)$ & 8 & $\C_2 \times \C_2$ \\
$\Gamma_{4,5,36}$ & $ A_1 B_4 A_1 B_4 $, $ A_1 B_5 A_4 B_5 $, $ A_2 B_4 A_3 B_5 $, $ A_4 B_1 A_4 B_3 $, $ A_4 B_2 A_4 B_4 $ & $\Sym(4)$ & $\Alt(5)$ & 8 & $\C_2 \times \C_2$ \\
$\Gamma_{4,5,37}$ & $ A_1 B_4 A_1 B_4 $, $ A_1 B_5 A_4 B_5 $, $ A_2 B_4 A_3 B_5 $, $ A_4 B_1 A_4 B_4 $, $ A_4 B_2 A_4 B_2 $, $ A_4 B_3 A_4 B_3 $ & $\Sym(4)$ & $\Sym(5)$ & 8 & $\C_2 \times \C_2$ \\
$\Gamma_{4,5,38}$ & $ A_1 B_4 A_1 B_4 $, $ A_1 B_5 A_4 B_5 $, $ A_2 B_4 A_3 B_5 $, $ A_4 B_1 A_4 B_4 $, $ A_4 B_2 A_4 B_3 $ & $\Sym(4)$ & $\Alt(5)$ & 8 & $\C_2 \times \C_2$ \\
$\Gamma_{4,5,39}$ & $ A_1 B_4 A_2 B_4 $, $ A_1 B_5 A_4 B_5 $, $ A_2 B_5 A_2 B_5 $, $ A_3 B_4 A_3 B_5 $, $ A_4 B_1 A_4 B_1 $, $ A_4 B_2 A_4 B_2 $, $ A_4 B_3 A_4 B_4 $ & $\Sym(4)$ & $\Sym(5)$ & 4 & $\C_2 \times \C_2$ \\
$\Gamma_{4,5,40}$ & $ A_1 B_4 A_2 B_4 $, $ A_1 B_5 A_4 B_5 $, $ A_2 B_5 A_2 B_5 $, $ A_3 B_4 A_3 B_5 $, $ A_4 B_1 A_4 B_1 $, $ A_4 B_2 A_4 B_4 $, $ A_4 B_3 A_4 B_3 $ & $\Sym(4)$ & $\Sym(5)$ & 4 & $\C_2 \times \C_2$ \\
$\Gamma_{4,5,41}$ & $ A_1 B_4 A_2 B_4 $, $ A_1 B_5 A_4 B_5 $, $ A_2 B_5 A_2 B_5 $, $ A_3 B_4 A_3 B_5 $, $ A_4 B_1 A_4 B_2 $, $ A_4 B_3 A_4 B_4 $ & $\Sym(4)$ & $\Sym(5)$ & 4 & $\C_2 \times \C_2$ \\
$\Gamma_{4,5,42}$ & $ A_1 B_4 A_2 B_4 $, $ A_1 B_5 A_4 B_5 $, $ A_2 B_5 A_2 B_5 $, $ A_3 B_4 A_3 B_5 $, $ A_4 B_1 A_4 B_3 $, $ A_4 B_2 A_4 B_4 $ & $\Sym(4)$ & $\Sym(5)$ & 4 & $\C_2 \times \C_2$ \\
$\Gamma_{4,5,43}$ & $ A_1 B_4 A_1 B_5 $, $ A_2 B_4 A_2 B_4 $, $ A_2 B_5 A_3 B_5 $, $ A_3 B_4 A_4 B_4 $, $ A_4 B_1 A_4 B_1 $, $ A_4 B_2 A_4 B_2 $, $ A_4 B_3 A_4 B_5 $ & $\Sym(4)$ & $\Sym(5)$ & 4 & $\C_2 \times \C_2$ \\
$\Gamma_{4,5,44}$ & $ A_1 B_4 A_1 B_5 $, $ A_2 B_4 A_2 B_4 $, $ A_2 B_5 A_3 B_5 $, $ A_3 B_4 A_4 B_4 $, $ A_4 B_1 A_4 B_1 $, $ A_4 B_2 A_4 B_5 $, $ A_4 B_3 A_4 B_3 $ & $\Sym(4)$ & $\Sym(5)$ & 4 & $\C_2 \times \C_2$ \\
$\Gamma_{4,5,45}$ & $ A_1 B_4 A_1 B_5 $, $ A_2 B_4 A_2 B_4 $, $ A_2 B_5 A_3 B_5 $, $ A_3 B_4 A_4 B_4 $, $ A_4 B_1 A_4 B_2 $, $ A_4 B_3 A_4 B_5 $ & $\Sym(4)$ & $\Sym(5)$ & 4 & $\C_2 \times \C_2$ \\
$\Gamma_{4,5,46}$ & $ A_1 B_4 A_1 B_5 $, $ A_2 B_4 A_2 B_4 $, $ A_2 B_5 A_3 B_5 $, $ A_3 B_4 A_4 B_4 $, $ A_4 B_1 A_4 B_3 $, $ A_4 B_2 A_4 B_5 $ & $\Sym(4)$ & $\Sym(5)$ & 4 & $\C_2 \times \C_2$ \\
$\Gamma_{4,5,47}$ & $ A_1 B_4 A_1 B_5 $, $ A_2 B_4 A_2 B_4 $, $ A_2 B_5 A_3 B_5 $, $ A_3 B_4 A_4 B_4 $, $ A_4 B_1 A_4 B_5 $, $ A_4 B_2 A_4 B_2 $, $ A_4 B_3 A_4 B_3 $ & $\Sym(4)$ & $\Sym(5)$ & 4 & $\C_2 \times \C_2$ \\
$\Gamma_{4,5,48}$ & $ A_1 B_4 A_1 B_5 $, $ A_2 B_4 A_2 B_4 $, $ A_2 B_5 A_3 B_5 $, $ A_3 B_4 A_4 B_4 $, $ A_4 B_1 A_4 B_5 $, $ A_4 B_2 A_4 B_3 $ & $\Sym(4)$ & $\Sym(5)$ & 4 & $\C_2 \times \C_2$ \\
$\Gamma_{4,5,49}$ & $ A_1 B_4 A_1 B_5 $, $ A_2 B_4 A_3 B_4 $, $ A_2 B_5 A_4 B_5 $, $ A_3 B_5 A_3 B_5 $, $ A_4 B_1 A_4 B_1 $, $ A_4 B_2 A_4 B_2 $, $ A_4 B_3 A_4 B_4 $ & $\Sym(4)$ & $\Sym(5)$ & 4 & $\C_2 \times \C_2$ \\
$\Gamma_{4,5,50}$ & $ A_1 B_4 A_1 B_5 $, $ A_2 B_4 A_3 B_4 $, $ A_2 B_5 A_4 B_5 $, $ A_3 B_5 A_3 B_5 $, $ A_4 B_1 A_4 B_1 $, $ A_4 B_2 A_4 B_4 $, $ A_4 B_3 A_4 B_3 $ & $\Sym(4)$ & $\Sym(5)$ & 4 & $\C_2 \times \C_2$ \\
$\Gamma_{4,5,51}$ & $ A_1 B_4 A_1 B_5 $, $ A_2 B_4 A_3 B_4 $, $ A_2 B_5 A_4 B_5 $, $ A_3 B_5 A_3 B_5 $, $ A_4 B_1 A_4 B_2 $, $ A_4 B_3 A_4 B_4 $ & $\Sym(4)$ & $\Sym(5)$ & 4 & $\C_2 \times \C_2$ \\
$\Gamma_{4,5,52}$ & $ A_1 B_4 A_1 B_5 $, $ A_2 B_4 A_3 B_4 $, $ A_2 B_5 A_4 B_5 $, $ A_3 B_5 A_3 B_5 $, $ A_4 B_1 A_4 B_3 $, $ A_4 B_2 A_4 B_4 $ & $\Sym(4)$ & $\Sym(5)$ & 4 & $\C_2 \times \C_2$ \\
$\Gamma_{4,5,53}$ & $ A_1 B_4 A_1 B_5 $, $ A_2 B_4 A_3 B_4 $, $ A_2 B_5 A_4 B_5 $, $ A_3 B_5 A_3 B_5 $, $ A_4 B_1 A_4 B_4 $, $ A_4 B_2 A_4 B_2 $, $ A_4 B_3 A_4 B_3 $ & $\Sym(4)$ & $\Sym(5)$ & 4 & $\C_2 \times \C_2$ \\
$\Gamma_{4,5,54}$ & $ A_1 B_4 A_1 B_5 $, $ A_2 B_4 A_3 B_4 $, $ A_2 B_5 A_4 B_5 $, $ A_3 B_5 A_3 B_5 $, $ A_4 B_1 A_4 B_4 $, $ A_4 B_2 A_4 B_3 $ & $\Sym(4)$ & $\Sym(5)$ & 4 & $\C_2 \times \C_2$ \\
$\Gamma_{4,5,55}$ & $ A_1 B_4 A_2 B_5 $, $ A_3 B_4 A_3 B_4 $, $ A_3 B_5 A_4 B_5 $, $ A_4 B_1 A_4 B_1 $, $ A_4 B_2 A_4 B_2 $, $ A_4 B_3 A_4 B_4 $ & $\Sym(4)$ & $\Sym(5)$ & 8 & $\C_2 \times \C_2$ \\
$\Gamma_{4,5,56}$ & $ A_1 B_4 A_2 B_5 $, $ A_3 B_4 A_3 B_4 $, $ A_3 B_5 A_4 B_5 $, $ A_4 B_1 A_4 B_1 $, $ A_4 B_2 A_4 B_4 $, $ A_4 B_3 A_4 B_3 $ & $\Sym(4)$ & $\Sym(5)$ & 8 & $\C_2 \times \C_2$ \\
$\Gamma_{4,5,57}$ & $ A_1 B_4 A_2 B_5 $, $ A_3 B_4 A_3 B_4 $, $ A_3 B_5 A_4 B_5 $, $ A_4 B_1 A_4 B_2 $, $ A_4 B_3 A_4 B_4 $ & $\Sym(4)$ & $\Sym(5)$ & 8 & $\C_2 \times \C_2$ \\
$\Gamma_{4,5,58}$ & $ A_1 B_4 A_2 B_5 $, $ A_3 B_4 A_3 B_4 $, $ A_3 B_5 A_4 B_5 $, $ A_4 B_1 A_4 B_3 $, $ A_4 B_2 A_4 B_4 $ & $\Sym(4)$ & $\Sym(5)$ & 8 & $\C_2 \times \C_2$ \\
$\Gamma_{4,5,59}$ & $ A_1 B_4 A_2 B_5 $, $ A_3 B_4 A_3 B_4 $, $ A_3 B_5 A_4 B_5 $, $ A_4 B_1 A_4 B_4 $, $ A_4 B_2 A_4 B_2 $, $ A_4 B_3 A_4 B_3 $ & $\Sym(4)$ & $\Sym(5)$ & 8 & $\C_2 \times \C_2$ \\
$\Gamma_{4,5,60}$ & $ A_1 B_4 A_2 B_5 $, $ A_3 B_4 A_3 B_4 $, $ A_3 B_5 A_4 B_5 $, $ A_4 B_1 A_4 B_4 $, $ A_4 B_2 A_4 B_3 $ & $\Sym(4)$ & $\Sym(5)$ & 8 & $\C_2 \times \C_2$ \\
\hline
\end{tabular}
\caption{Some virtually simple $(4,5)$-groups containing the mirror of $\Gamma_{3,3}$.}\label{table:simple45-2}
\end{table}
\end{landscape}

\subsection{Virtually simple \texorpdfstring{$(2n,2n+1)$-groups ($n \geq 2$)}{(2n,2n+1)-groups (n >= 2)}}

In \S\ref{subsection:45} we gave a list of virtually simple $(4,5)$-groups $\Gamma_{4,5,k}$. We now construct for each $n \geq 2$ a virtually simple $(2n,2n+1)$-group. For $n = 2$ we take $\Gamma_{4,5} = \Gamma_{4,5,9}$, see Table~\ref{table:simple45-1} and Figure~\ref{picture:459}. For $n \geq 3$ we define $\Gamma_{2n,2n+1}$ as the $(2n,2n+1)$-group whose geometric squares are:
\begin{enumerate}[(1)]
\item the $11$ geometric squares of $\Gamma_{4,5}$;
\item for each $3 \leq k \leq n$, the $3$ geometric squares $A_{2k} B_{2k+1} A_1 B_{2k}$, $A_{2k-1} B_{2k} A_{2k-1} B_1$ and $A_{2k-1} B_{2k+1} A_2 B_{2k+1}$, see Figure~\ref{picture:3more};
\item all geometric squares $A_j B_k A_j B_k$ with $j \in \{1,\ldots,2n\}$ and $k \in \{1,\ldots,2n+1\}$ such that the corner $(A_j,B_k)$ does not already appear in a geometric square of (1) or (2).
\end{enumerate}

\begin{figure}[h!]
\centering
\begin{pspicture*}(-0.5,-0.5)(9.5,1.7)
\fontsize{10pt}{10pt}\selectfont
\psset{unit=1.2cm}

\pspolygon(0,0)(1,0)(1,1)(0,1)

\pspolygon[fillstyle=solid,fillcolor=white](0.175,0)(0.275,0.1)(0.375,0)(0.275,-0.1)
\pspolygon[fillstyle=solid,fillcolor=white](0.625,0)(0.725,0.1)(0.825,0)(0.725,-0.1)
\rput(0.5,-0.24){$2k$}
\pspolygon[fillstyle=solid,fillcolor=white](0.4,1)(0.5,1.1)(0.6,1)(0.5,0.9)
\pspolygon[fillstyle=solid,fillcolor=black](1,0.625)(1.1,0.725)(1,0.825)(0.9,0.725)
\pspolygon[fillstyle=solid,fillcolor=black](1,0.175)(1.1,0.275)(1,0.375)(0.9,0.275)
\rput(1.48,0.5){$2k+1$}
\pspolygon[fillstyle=solid,fillcolor=black](0,0.625)(0.1,0.725)(0,0.825)(-0.1,0.725)
\pspolygon[fillstyle=solid,fillcolor=black](0,0.175)(0.1,0.275)(0,0.375)(-0.1,0.275)
\rput(-0.21,0.5){$2k$}

\pspolygon(3,0)(4,0)(4,1)(3,1)

\pspolygon[fillstyle=solid,fillcolor=white](3.175,0)(3.275,0.1)(3.375,0)(3.275,-0.1)
\pspolygon[fillstyle=solid,fillcolor=white](3.625,0)(3.725,0.1)(3.825,0)(3.725,-0.1)
\rput(3.5,-0.26){$2k-1$}
\pspolygon[fillstyle=solid,fillcolor=white](3.175,1)(3.275,1.1)(3.375,1)(3.275,0.9)
\pspolygon[fillstyle=solid,fillcolor=white](3.625,1)(3.725,1.1)(3.825,1)(3.725,0.9)
\rput(3.5,1.25){$2k-1$}
\pspolygon[fillstyle=solid,fillcolor=black](4,0.625)(4.1,0.725)(4,0.825)(3.9,0.725)
\pspolygon[fillstyle=solid,fillcolor=black](4,0.175)(4.1,0.275)(4,0.375)(3.9,0.275)
\rput(4.23,0.5){$2k$}
\pspolygon[fillstyle=solid,fillcolor=black](3,0.4)(3.1,0.5)(3,0.6)(2.9,0.5)

\pspolygon(6,0)(7,0)(7,1)(6,1)

\pspolygon[fillstyle=solid,fillcolor=white](6.175,0)(6.275,0.1)(6.375,0)(6.275,-0.1)
\pspolygon[fillstyle=solid,fillcolor=white](6.625,0)(6.725,0.1)(6.825,0)(6.725,-0.1)
\rput(6.5,-0.24){$2k-1$}
\pspolygon[fillstyle=solid,fillcolor=white](6.275,1)(6.375,1.1)(6.475,1)(6.375,0.9)
\pspolygon[fillstyle=solid,fillcolor=white](6.725,1)(6.625,1.1)(6.525,1)(6.625,0.9)
\pspolygon[fillstyle=solid,fillcolor=black](7,0.625)(7.1,0.725)(7,0.825)(6.9,0.725)
\pspolygon[fillstyle=solid,fillcolor=black](7,0.175)(7.1,0.275)(7,0.375)(6.9,0.275)
\rput(7.48,0.5){$2k+1$}
\pspolygon[fillstyle=solid,fillcolor=black](6,0.625)(6.1,0.725)(6,0.825)(5.9,0.725)
\pspolygon[fillstyle=solid,fillcolor=black](6,0.175)(6.1,0.275)(6,0.375)(5.9,0.275)
\rput(5.54,0.5){$2k+1$}

\end{pspicture*}
\caption{Additional squares in $\Gamma_{2n,2n+1}$.}\label{picture:3more}
\end{figure}
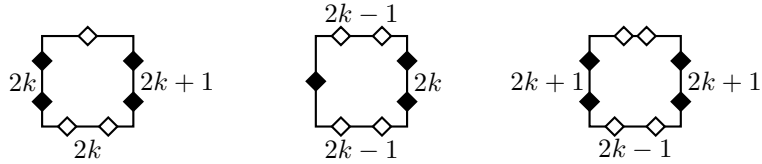

\begin{theorem}[Theorem~\ref{maintheorem:2n2n+1}]\label{theorem:2n2n+1}
For each $n \geq 2$, $\Gamma_{2n,2n+1}$ is a virtually simple $(2n,2n+1)$-group with $\underline{H_1}(v_1) \cong \Sym(2n)$ and $\underline{H_2}(v_2) \cong \Sym(2n+1)$. Moreover, if $n \geq 3$ then there is a legal coloring $i$ of $T_1$ such that
\begin{itemize}
\item $H_1 = G_{(i)}(\{4\}, \{4\})$ if $n$ is even;
\item $H_1 = G_{(i)}(\{0,2,3\}, \{0,2,3\})$ if $n$ is odd.
\end{itemize}
\end{theorem}

\begin{proof}
The group $\underline{H_1}(v_1)$ is generated by the following $2n+1$ permutations, which clearly generate $\Sym(2n)$:
\begin{align*}
B_1\ &:\ ()\\
B_2\ &:\ (A_2\ A_3)\\
B_3\ &:\ (A_1\ A_2)\\
B_4\ &:\ (A_2\ A_3)\\
B_5\ &:\ (A_1\ A_4)(A_2\ A_3)\\
B_{2k} \ \text{($3 \leq k \leq n$)}\ &:\ (A_1\ A_{2k})\\
B_{2k+1} \ \text{($3 \leq k \leq n$)}\ &:\ (A_1\ A_{2k})(A_2\ A_{2k-1})
\end{align*}
For $\underline{H_2}(v_2)$ we have the $2n$ permutations
\begin{align*}
A_1\ &:\ (B_6\ B_7)(B_8\ B_9)\ldots (B_{2n}\ B_{2n+1})\\
A_2\ &:\ (B_4\ B_5)\\
A_3\ &:\ (B_1\ B_3)(B_4\ B_5)\\
A_4\ &:\ (B_1\ B_2)(B_3\ B_4)\\
A_{2k-1} \ \text{($3 \leq k \leq n$)}\ &:\ (B_1\ B_{2k})\\
A_{2k} \ \text{($3 \leq k \leq n$)}\ &:\ (B_{2k}\ B_{2k+1})
\end{align*}
The permutations $A_2, A_3, A_4$ generate $\Sym(5)$, and we then get by induction that the permutations $A_2, \ldots, A_{2k}$ generate $\Sym(2k+1)$ for each $3 \leq k \leq n$. In particular, we have $\underline{H_2}(v_2) \cong \Sym(2k+1)$.

We already know that $\Gamma_{4,5}$ is virtually simple. For $n \geq 3$, by \cite[Proposition~3.3.2]{Burger} the groups $H_1$ and $H_2$ are boundary-$2$-transitive. The NST then implies that $\Gamma_{2n,2n+1}$ is virtually simple.

When $n \geq 3$, the group $H_1$ can be computed thanks to the algorithms developed in \S\ref{section:projections}. We do not give the details here, but it can be seen when computing the graph $G^{(1)}_{\Gamma_{2n,2n+1}}$ that the result will only depend on the parity of $n$. So it suffices to proceed for $n = 3$ and $n = 4$.
\end{proof}

\subsection{Virtually simple \texorpdfstring{$(6,4n)$-groups ($n \geq 2$)}{(6,4n)-groups (n >= 2)}}

Let $T_1$ be the $6$-regular tree and $T_2^{(n)}$ be the $4n$-regular tree for $n \geq 2$. The set of closed subgroups of $\Aut(T_1)$ carries the Chabauty topology. In this section we describe a sequence $(\Gamma_{6,4n})_{n \geq 2}$ of groups, with $\Gamma_{6,4n}$ being a $(6,4n)$-group, such that $\overline{\proj_1(\Gamma_{6,4n})} \to \Aut(T_1)$ in the Chabauty topology when $n \to \infty$.

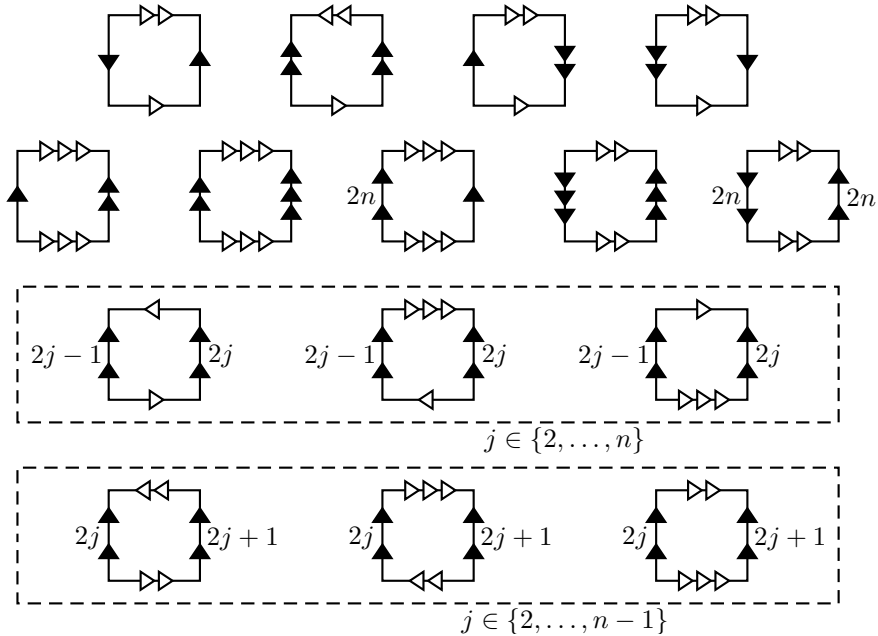
\begin{figure}[b!]
\centering
\begin{pspicture*}(-1.4,-2.2)(10.1,6.2)
\fontsize{10pt}{10pt}\selectfont
\psset{unit=1.2cm}


\pspolygon(0,4)(1,4)(1,5)(0,5)

\pspolygon[fillstyle=solid,fillcolor=white](0.45,4.1)(0.45,3.9)(0.6,4)
\pspolygon[fillstyle=solid,fillcolor=white](0.35,5.1)(0.35,4.9)(0.5,5)
\pspolygon[fillstyle=solid,fillcolor=white](0.55,5.1)(0.55,4.9)(0.7,5)
\pspolygon[fillstyle=solid,fillcolor=black](0.9,4.45)(1.1,4.45)(1,4.6)
\pspolygon[fillstyle=solid,fillcolor=black](-0.1,4.55)(0.1,4.55)(0,4.4)

\pspolygon(2,4)(3,4)(3,5)(2,5)

\pspolygon[fillstyle=solid,fillcolor=white](2.45,4.1)(2.45,3.9)(2.6,4)
\pspolygon[fillstyle=solid,fillcolor=white](2.65,5.1)(2.65,4.9)(2.5,5)
\pspolygon[fillstyle=solid,fillcolor=white](2.45,5.1)(2.45,4.9)(2.3,5)
\pspolygon[fillstyle=solid,fillcolor=black](2.9,4.55)(3.1,4.55)(3,4.7)
\pspolygon[fillstyle=solid,fillcolor=black](2.9,4.35)(3.1,4.35)(3,4.5)
\pspolygon[fillstyle=solid,fillcolor=black](1.9,4.55)(2.1,4.55)(2,4.7)
\pspolygon[fillstyle=solid,fillcolor=black](1.9,4.35)(2.1,4.35)(2,4.5)

\pspolygon(4,4)(5,4)(5,5)(4,5)

\pspolygon[fillstyle=solid,fillcolor=white](4.45,4.1)(4.45,3.9)(4.6,4)
\pspolygon[fillstyle=solid,fillcolor=white](4.35,5.1)(4.35,4.9)(4.5,5)
\pspolygon[fillstyle=solid,fillcolor=white](4.55,5.1)(4.55,4.9)(4.7,5)
\pspolygon[fillstyle=solid,fillcolor=black](4.9,4.45)(5.1,4.45)(5,4.3)
\pspolygon[fillstyle=solid,fillcolor=black](4.9,4.65)(5.1,4.65)(5,4.5)
\pspolygon[fillstyle=solid,fillcolor=black](3.9,4.45)(4.1,4.45)(4,4.6)

\pspolygon(6,4)(7,4)(7,5)(6,5)

\pspolygon[fillstyle=solid,fillcolor=white](6.45,4.1)(6.45,3.9)(6.6,4)
\pspolygon[fillstyle=solid,fillcolor=white](6.35,5.1)(6.35,4.9)(6.5,5)
\pspolygon[fillstyle=solid,fillcolor=white](6.55,5.1)(6.55,4.9)(6.7,5)
\pspolygon[fillstyle=solid,fillcolor=black](6.9,4.55)(7.1,4.55)(7,4.4)
\pspolygon[fillstyle=solid,fillcolor=black](5.9,4.45)(6.1,4.45)(6,4.3)
\pspolygon[fillstyle=solid,fillcolor=black](5.9,4.65)(6.1,4.65)(6,4.5)


\pspolygon(-1,2.5)(0,2.5)(0,3.5)(-1,3.5)

\pspolygon[fillstyle=solid,fillcolor=white](-0.75,2.6)(-0.75,2.4)(-0.6,2.5)
\pspolygon[fillstyle=solid,fillcolor=white](-0.55,2.6)(-0.55,2.4)(-0.4,2.5)
\pspolygon[fillstyle=solid,fillcolor=white](-0.35,2.6)(-0.35,2.4)(-0.2,2.5)
\pspolygon[fillstyle=solid,fillcolor=white](-0.75,3.6)(-0.75,3.4)(-0.6,3.5)
\pspolygon[fillstyle=solid,fillcolor=white](-0.55,3.6)(-0.55,3.4)(-0.4,3.5)
\pspolygon[fillstyle=solid,fillcolor=white](-0.35,3.6)(-0.35,3.4)(-0.2,3.5)
\pspolygon[fillstyle=solid,fillcolor=black](-0.1,2.85)(0.1,2.85)(0,3)
\pspolygon[fillstyle=solid,fillcolor=black](-0.1,3.05)(0.1,3.05)(0,3.2)
\pspolygon[fillstyle=solid,fillcolor=black](-1.1,2.95)(-0.9,2.95)(-1,3.1)

\pspolygon(1,2.5)(2,2.5)(2,3.5)(1,3.5)

\pspolygon[fillstyle=solid,fillcolor=white](1.25,2.6)(1.25,2.4)(1.4,2.5)
\pspolygon[fillstyle=solid,fillcolor=white](1.45,2.6)(1.45,2.4)(1.6,2.5)
\pspolygon[fillstyle=solid,fillcolor=white](1.65,2.6)(1.65,2.4)(1.8,2.5)
\pspolygon[fillstyle=solid,fillcolor=white](1.25,3.6)(1.25,3.4)(1.4,3.5)
\pspolygon[fillstyle=solid,fillcolor=white](1.45,3.6)(1.45,3.4)(1.6,3.5)
\pspolygon[fillstyle=solid,fillcolor=white](1.65,3.6)(1.65,3.4)(1.8,3.5)
\pspolygon[fillstyle=solid,fillcolor=black](1.9,2.75)(2.1,2.75)(2,2.9)
\pspolygon[fillstyle=solid,fillcolor=black](1.9,2.95)(2.1,2.95)(2,3.1)
\pspolygon[fillstyle=solid,fillcolor=black](1.9,3.15)(2.1,3.15)(2,3.3)
\pspolygon[fillstyle=solid,fillcolor=black](0.9,2.85)(1.1,2.85)(1,3)
\pspolygon[fillstyle=solid,fillcolor=black](0.9,3.05)(1.1,3.05)(1,3.2)

\pspolygon(3,2.5)(4,2.5)(4,3.5)(3,3.5)

\pspolygon[fillstyle=solid,fillcolor=white](3.25,2.6)(3.25,2.4)(3.4,2.5)
\pspolygon[fillstyle=solid,fillcolor=white](3.45,2.6)(3.45,2.4)(3.6,2.5)
\pspolygon[fillstyle=solid,fillcolor=white](3.65,2.6)(3.65,2.4)(3.8,2.5)
\pspolygon[fillstyle=solid,fillcolor=white](3.25,3.6)(3.25,3.4)(3.4,3.5)
\pspolygon[fillstyle=solid,fillcolor=white](3.45,3.6)(3.45,3.4)(3.6,3.5)
\pspolygon[fillstyle=solid,fillcolor=white](3.65,3.6)(3.65,3.4)(3.8,3.5)
\pspolygon[fillstyle=solid,fillcolor=black](3.9,2.95)(4.1,2.95)(4,3.1)
\pspolygon[fillstyle=solid,fillcolor=black](2.9,2.75)(3.1,2.75)(3,2.9)
\rput(2.77,3.01){$2n$}
\pspolygon[fillstyle=solid,fillcolor=black](2.9,3.15)(3.1,3.15)(3,3.3)

\pspolygon(5,2.5)(6,2.5)(6,3.5)(5,3.5)

\pspolygon[fillstyle=solid,fillcolor=white](5.35,2.6)(5.35,2.4)(5.5,2.5)
\pspolygon[fillstyle=solid,fillcolor=white](5.55,2.6)(5.55,2.4)(5.7,2.5)
\pspolygon[fillstyle=solid,fillcolor=white](5.35,3.6)(5.35,3.4)(5.5,3.5)
\pspolygon[fillstyle=solid,fillcolor=white](5.55,3.6)(5.55,3.4)(5.7,3.5)
\pspolygon[fillstyle=solid,fillcolor=black](5.9,2.75)(6.1,2.75)(6,2.9)
\pspolygon[fillstyle=solid,fillcolor=black](5.9,2.95)(6.1,2.95)(6,3.1)
\pspolygon[fillstyle=solid,fillcolor=black](5.9,3.15)(6.1,3.15)(6,3.3)
\pspolygon[fillstyle=solid,fillcolor=black](4.9,2.85)(5.1,2.85)(5,2.7)
\pspolygon[fillstyle=solid,fillcolor=black](4.9,3.05)(5.1,3.05)(5,2.9)
\pspolygon[fillstyle=solid,fillcolor=black](4.9,3.25)(5.1,3.25)(5,3.1)

\pspolygon(7,2.5)(8,2.5)(8,3.5)(7,3.5)

\pspolygon[fillstyle=solid,fillcolor=white](7.35,2.6)(7.35,2.4)(7.5,2.5)
\pspolygon[fillstyle=solid,fillcolor=white](7.55,2.6)(7.55,2.4)(7.7,2.5)
\pspolygon[fillstyle=solid,fillcolor=white](7.35,3.6)(7.35,3.4)(7.5,3.5)
\pspolygon[fillstyle=solid,fillcolor=white](7.55,3.6)(7.55,3.4)(7.7,3.5)
\pspolygon[fillstyle=solid,fillcolor=black](7.9,2.75)(8.1,2.75)(8,2.9)
\rput(8.25,2.98){$2n$}
\pspolygon[fillstyle=solid,fillcolor=black](7.9,3.15)(8.1,3.15)(8,3.3)
\pspolygon[fillstyle=solid,fillcolor=black](6.9,2.85)(7.1,2.85)(7,2.7)
\rput(6.77,3.01){$2n$}
\pspolygon[fillstyle=solid,fillcolor=black](6.9,3.25)(7.1,3.25)(7,3.1)


\pspolygon(0,0.75)(1,0.75)(1,1.75)(0,1.75)

\pspolygon[fillstyle=solid,fillcolor=white](0.45,0.85)(0.45,0.65)(0.6,0.75)
\pspolygon[fillstyle=solid,fillcolor=white](0.55,1.85)(0.55,1.65)(0.4,1.75)
\pspolygon[fillstyle=solid,fillcolor=black](0.9,1)(1.1,1)(1,1.15)
\rput(1.23,1.23){$2j$}
\pspolygon[fillstyle=solid,fillcolor=black](0.9,1.4)(1.1,1.4)(1,1.55)
\pspolygon[fillstyle=solid,fillcolor=black](-0.1,1)(0.1,1)(0,1.15)
\rput(-0.46,1.22){$2j-1$}
\pspolygon[fillstyle=solid,fillcolor=black](-0.1,1.4)(0.1,1.4)(0,1.55)

\pspolygon(3,0.75)(4,0.75)(4,1.75)(3,1.75)

\pspolygon[fillstyle=solid,fillcolor=white](3.55,0.85)(3.55,0.65)(3.4,0.75)
\pspolygon[fillstyle=solid,fillcolor=white](3.25,1.85)(3.25,1.65)(3.4,1.75)
\pspolygon[fillstyle=solid,fillcolor=white](3.45,1.85)(3.45,1.65)(3.6,1.75)
\pspolygon[fillstyle=solid,fillcolor=white](3.65,1.85)(3.65,1.65)(3.8,1.75)
\pspolygon[fillstyle=solid,fillcolor=black](3.9,1)(4.1,1)(4,1.15)
\rput(4.23,1.23){$2j$}
\pspolygon[fillstyle=solid,fillcolor=black](3.9,1.4)(4.1,1.4)(4,1.55)
\pspolygon[fillstyle=solid,fillcolor=black](2.9,1)(3.1,1)(3,1.15)
\rput(2.54,1.22){$2j-1$}
\pspolygon[fillstyle=solid,fillcolor=black](2.9,1.4)(3.1,1.4)(3,1.55)

\pspolygon(6,0.75)(7,0.75)(7,1.75)(6,1.75)

\pspolygon[fillstyle=solid,fillcolor=white](6.25,0.85)(6.25,0.65)(6.4,0.75)
\pspolygon[fillstyle=solid,fillcolor=white](6.45,0.85)(6.45,0.65)(6.6,0.75)
\pspolygon[fillstyle=solid,fillcolor=white](6.65,0.85)(6.65,0.65)(6.8,0.75)
\pspolygon[fillstyle=solid,fillcolor=white](6.45,1.85)(6.45,1.65)(6.6,1.75)
\pspolygon[fillstyle=solid,fillcolor=black](6.9,1)(7.1,1)(7,1.15)
\rput(7.23,1.23){$2j$}
\pspolygon[fillstyle=solid,fillcolor=black](6.9,1.4)(7.1,1.4)(7,1.55)
\pspolygon[fillstyle=solid,fillcolor=black](5.9,1)(6.1,1)(6,1.15)
\rput(5.54,1.22){$2j-1$}
\pspolygon[fillstyle=solid,fillcolor=black](5.9,1.4)(6.1,1.4)(6,1.55)

\pspolygon[linestyle=dashed](-1,0.5)(8,0.5)(8,2)(-1,2)
\rput(5,0.3){$j \in \{2, \ldots, n\}$}


\pspolygon(0,-1.25)(1,-1.25)(1,-0.25)(0,-0.25)

\pspolygon[fillstyle=solid,fillcolor=white](0.35,-1.15)(0.35,-1.35)(0.5,-1.25)
\pspolygon[fillstyle=solid,fillcolor=white](0.55,-1.15)(0.55,-1.35)(0.7,-1.25)
\pspolygon[fillstyle=solid,fillcolor=white](0.45,-0.15)(0.45,-0.35)(0.3,-0.25)
\pspolygon[fillstyle=solid,fillcolor=white](0.65,-0.15)(0.65,-0.35)(0.5,-0.25)
\pspolygon[fillstyle=solid,fillcolor=black](0.9,-1)(1.1,-1)(1,-0.85)
\rput(1.47,-0.79){$2j+1$}
\pspolygon[fillstyle=solid,fillcolor=black](0.9,-0.6)(1.1,-0.6)(1,-0.45)
\pspolygon[fillstyle=solid,fillcolor=black](-0.1,-1)(0.1,-1)(0,-0.85)
\rput(-0.23,-0.78){$2j$}
\pspolygon[fillstyle=solid,fillcolor=black](-0.1,-0.6)(0.1,-0.6)(0,-0.45)

\pspolygon(3,-1.25)(4,-1.25)(4,-0.25)(3,-0.25)

\pspolygon[fillstyle=solid,fillcolor=white](3.65,-1.15)(3.65,-1.35)(3.5,-1.25)
\pspolygon[fillstyle=solid,fillcolor=white](3.45,-1.15)(3.45,-1.35)(3.3,-1.25)
\pspolygon[fillstyle=solid,fillcolor=white](3.25,-0.15)(3.25,-0.35)(3.4,-0.25)
\pspolygon[fillstyle=solid,fillcolor=white](3.45,-0.15)(3.45,-0.35)(3.6,-0.25)
\pspolygon[fillstyle=solid,fillcolor=white](3.65,-0.15)(3.65,-0.35)(3.8,-0.25)
\pspolygon[fillstyle=solid,fillcolor=black](3.9,-1)(4.1,-1)(4,-0.85)
\rput(4.47,-0.79){$2j+1$}
\pspolygon[fillstyle=solid,fillcolor=black](3.9,-0.6)(4.1,-0.6)(4,-0.45)
\pspolygon[fillstyle=solid,fillcolor=black](2.9,-1)(3.1,-1)(3,-0.85)
\rput(2.77,-0.78){$2j$}
\pspolygon[fillstyle=solid,fillcolor=black](2.9,-0.6)(3.1,-0.6)(3,-0.45)

\pspolygon(6,-1.25)(7,-1.25)(7,-0.25)(6,-0.25)

\pspolygon[fillstyle=solid,fillcolor=white](6.25,-1.15)(6.25,-1.35)(6.4,-1.25)
\pspolygon[fillstyle=solid,fillcolor=white](6.45,-1.15)(6.45,-1.35)(6.6,-1.25)
\pspolygon[fillstyle=solid,fillcolor=white](6.65,-1.15)(6.65,-1.35)(6.8,-1.25)
\pspolygon[fillstyle=solid,fillcolor=white](6.35,-0.15)(6.35,-0.35)(6.5,-0.25)
\pspolygon[fillstyle=solid,fillcolor=white](6.55,-0.15)(6.55,-0.35)(6.7,-0.25)
\pspolygon[fillstyle=solid,fillcolor=black](6.9,-1)(7.1,-1)(7,-0.85)
\rput(7.47,-0.79){$2j+1$}
\pspolygon[fillstyle=solid,fillcolor=black](6.9,-0.6)(7.1,-0.6)(7,-0.45)
\pspolygon[fillstyle=solid,fillcolor=black](5.9,-1)(6.1,-1)(6,-0.85)
\rput(5.77,-0.78){$2j$}
\pspolygon[fillstyle=solid,fillcolor=black](5.9,-0.6)(6.1,-0.6)(6,-0.45)

\pspolygon[linestyle=dashed](-1,-1.5)(8,-1.5)(8,0)(-1,0)
\rput(5,-1.7){$j \in \{2, \ldots, n-1\}$}

\end{pspicture*}
\caption{The torsion-free $(6,4n)$-group $\Gamma_{6,4n}$}\label{picture:64n}
\end{figure}

\begin{theorem}[Theorem~\ref{maintheorem:64n}]\label{theorem:64n}
Let $n \geq 2$ be an integer and let $\Gamma_{6,4n}$ be the torsion-free $(6,4n)$-group associated to the geometric squares in Figure~\ref{picture:64n}. Then $\Gamma_{6,4n}$ is virtually simple,  $\overline{\proj_1(\Gamma_{6,4n})} = G_{(i_1)}(\{n\},\{n\})$ for some legal coloring $i_1$ of $T_1$ and $\overline{\proj_2(\Gamma_{6,4n})} = G_{(i_2)}(\{0\},\{0\})$ for some legal coloring $i_2$ of $T_2^{(n)}$.
\end{theorem}

\begin{proof}
One easily checks that the geometric squares given in~Figure~\ref{picture:64n} indeed define a torsion-free $(6,4n)$-group. The first four squares correspond to $\Gamma_{4,4}$ (see \S\ref{subsection:66}), so that $\Gamma_{4,4} \leq \Gamma_{6,4n}$. In particular, $\Gamma_{6,4n}$ is irreducible. If we show that $\underline{H_1}(v_1) \geq \Alt(6)$ and $\underline{H_2}(v_2) \geq \Alt(4n)$ (where $H_t = \overline{\proj_t(\Gamma_{6,4n})}$), then it will follow from the NST and \cite[Propositions~3.3.1 and~3.3.2]{Burger} that $\Gamma_{6,4n}$ is virtually simple.

The group $\underline{H_1}(v_1)$ is generated by the following permutations.
\begin{align*}
b_1\ &:\ (a_1\ a_2)(a_1^{-1}\ a_2^{-1})(a_3)(a_3^{-1})\\
b_2\ &:\ (a_1\ a_2\ a_1^{-1}\ a_2^{-1})(a_3)(a_3^{-1})\\
b_3\ &:\ (a_1\ a_3\ a_1^{-1})(a_2\ a_2^{-1}\ a_3^{-1})\\
b_{2j}\ &:\ (a_1\ a_1^{-1}\ a_3^{-1})(a_2\ a_3\ a_2^{-1}) \qquad (j\in \{2, \ldots, n-1\})\\
b_{2j+1}\ &:\ (a_1\ a_3\ a_1^{-1})(a_2\ a_2^{-1}\ a_3^{-1}) \qquad (j \in \{2,\ldots, n-1\})\\
b_{2n}\ &: (a_1\ a_1^{-1}\ a_3^{-1})(a_2)(a_2^{-1})(a_3)
\end{align*}
The permutations induced by $b_1$, $b_2$, $b_3$ and $b_{2n}$ generate $\Sym(6)$, so $\underline{H_1}(v_1) = \Sym(6)$. For $\underline{H_2}(v_2)$ we get:
\begin{align*}
a_1\ &:\ (b_1\ b_1^{-1}\ b_2^{-1}) (b_2) (b_3\ b_4) (b_3^{-1}\ b_4^{-1}) \ldots (b_{2n-1}\ b_{2n}) (b_{2n-1}^{-1}\ b_{2n}^{-1})\\
a_2\ &:\ (b_1\ b_2\ b_1^{-1})(b_2^{-1}) (b_3\ b_3^{-1}) (b_{2n}\ b_{2n}^{-1})\\
& \qquad\qquad (b_4\ b_5) (b_4^{-1}\ b_5^{-1}) \ldots (b_{2n-2}\ b_{2n-1}) (b_{2n-2}^{-1}\ b_{2n-1}^{-1})\\
a_3\ &:\ (b_{2n}\ b_{2n-1}\ \ldots \ b_2\ b_1) (b_{2n}^{-1}\ b_{2n-1}^{-1}\ \ldots \ b_2^{-1}\ b_1^{-1}) 
\end{align*}
We observe that $a_1^2$ and $a_2^2$ induce the permutations $(b_1\ b_2^{-1}\ b_1^{-1})$ and $(b_1\ b_1^{-1}\ b_2)$ respectively, which generate $\Alt(\{b_1, b_1^{-1}, b_2, b_2^{-1}\})$. Conjugating this alternating group by several powers of $a_3$, we obtain all $\Alt(\{b_i, b_i^{-1}, b_{i+1}, b_{i+1}^{-1}\})$ with $i \in \{1,\ldots, 2n-1\}$. These alternating groups together generate $\Alt(2n)$. As the permutations induced by $a_1$, $a_2$ and $a_3$ are all even, we get $\underline{H_2}(v_2) = \Alt(2n)$. This already implies that $H_2 = G_{(i_2)}(\{0\},\{0\})$ for some legal coloring $i_2$ of $T_2^{(n)}$.

There remains to compute $H_1$, using the algorithms developed in \S\ref{section:projections}. The simplified labelled graph $\tilde{G}_{\Gamma_{6,4n}}^{(1)}$ is a cycle with only one label $-1$, see Figure~\ref{picture:64n-graph}. From this graph and via Proposition~\ref{proposition:algo2}, we can compute the values of $s_k^{(1)}(b_j)$ for $j \in \{1,\ldots,2n\}$ and $k \in \N$:
$$\begin{array}{c|cccccc}
 & s^{(1)}_0 & s^{(1)}_1 & s^{(1)}_2 & \ldots & s^{(1)}_{n-1} & s^{(1)}_{n}\\
 \hline
 b_2 & -1 & +1 & +1 & \ldots & +1 & +1\\
 b_3 & +1 & -1 & +1 & \ldots & +1 & +1\\
 b_4 & +1 & +1 & -1 & \ldots & +1 & +1\\
 \vdots & & & & & & \\
 b_{n+1} & +1 & +1 & +1 & \ldots & -1 & +1\\
 b_{n+2} & +1 & +1 & +1 & \ldots & +1 & +1\\
 b_{n+3} & +1 & +1 & +1 & \ldots & -1 & +1\\
 \vdots & & & & & & \\
 b_{2n} & +1 & +1 & -1 & \ldots & +1 & +1\\
 b_1 & +1 & -1 & +1 & \ldots & +1 & +1
\end{array}$$
From these values we deduce that $K^{(1)} = n$ and
$$s^{(1)}(H_1(v_1)) = \{(s_0,\ldots,s_n) \mid s_n = 1\}.$$
The groups that can be isomorphic to $H_1$ are thus $G_{(i_1)}(\{n\},\{n\})$, $G_{(i_1)}(Y,Y)^*$, $G'_{(i_1)}(Y,Y)^*$ and $G_{(i_1)}(Y^*,Y^*)$ where $\alpha(Y) = \{n\}$. We have $Y = \{1,3,\ldots,n-1\}$ if $n$ is even and $Y = \{0,2,\ldots,n-1\}$ if $n$ is odd.

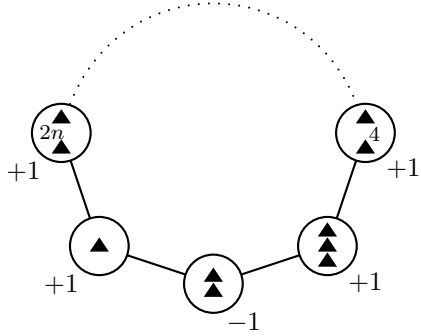
\begin{figure}[t!]
\centering
\begin{pspicture*}(-1.2,-0.6)(4.2,3.8)
\fontsize{10pt}{10pt}\selectfont
\psset{unit=1cm}

\psline(-0.5,2.01)(0,0.51)

\psline(0,0.51)(1.5,0.01)

\psline(1.5,0.01)(3,0.51)

\psline(3,0.51)(3.5,2.01)

\psbezier[linestyle=dotted](-0.5,2.01)(0,4.3)(3,4.3)(3.5,2.01)

\pscircle[fillstyle=solid,fillcolor=white](-0.5,2.01){0.4}
\pspolygon[fillstyle=solid,fillcolor=black](-0.4,1.75)(-0.6,1.75)(-0.5,1.9)
\rput(-0.62,2.02){{\fontsize{8pt}{8pt}$2n$}}
\pspolygon[fillstyle=solid,fillcolor=black](-0.4,2.15)(-0.6,2.15)(-0.5,2.3)
\rput(-1,1.5){$+1$}

\pscircle[fillstyle=solid,fillcolor=white](0,0.51){0.4}
\pspolygon[fillstyle=solid,fillcolor=black](0.1,0.45)(-0.1,0.45)(0,0.6)
\rput(-0.5,0){$+1$}

\pscircle[fillstyle=solid,fillcolor=white](1.5,0.01){0.4}
\pspolygon[fillstyle=solid,fillcolor=black](1.6,-0.15)(1.4,-0.15)(1.5,0)
\pspolygon[fillstyle=solid,fillcolor=black](1.6,0.05)(1.4,0.05)(1.5,0.2)
\rput(1.9,-0.5){$-1$}

\pscircle[fillstyle=solid,fillcolor=white](3,0.51){0.4}
\pspolygon[fillstyle=solid,fillcolor=black](3.1,0.25)(2.9,0.25)(3,0.4)
\pspolygon[fillstyle=solid,fillcolor=black](3.1,0.45)(2.9,0.45)(3,0.6)
\pspolygon[fillstyle=solid,fillcolor=black](3.1,0.65)(2.9,0.65)(3,0.8)
\rput(3.5,0.03){$+1$}

\pscircle[fillstyle=solid,fillcolor=white](3.5,2.01){0.4}
\pspolygon[fillstyle=solid,fillcolor=black](3.6,1.75)(3.4,1.75)(3.5,1.9)
\rput(3.62,1.99){{\fontsize{8pt}{8pt}$4$}}
\pspolygon[fillstyle=solid,fillcolor=black](3.6,2.15)(3.4,2.15)(3.5,2.3)
\rput(4,1.53){$+1$}

\end{pspicture*}
\caption{The simplified labelled graph $\tilde{G}^{(1)}_{\Gamma_{6,4n}}$.}\label{picture:64n-graph}
\end{figure}

Now let us see which of the four groups is the good one, thanks to Proposition~\ref{proposition:choose4}. The very first equality in both systems $(*)$ and $(**)$ comes from the first and third geometric squares defining $\Gamma_{6,4n}$ and is $x_1 x_2 \Sigma_{4n} = x_2 x_1 \Sigma_{4n-1}$, where $\Sigma_{4n} = \prod_{r \in Y} s_r^{(1)}(b_1^{-1})$ and $\Sigma_{4n-1} = \prod_{r \in Y} s_r^{(1)}(b_2^{-1})$. But from the table above we can compute that $\Sigma_{4n} \neq \Sigma_{4n-1}$ in any case, so $(*)$ and $(**)$ have no solution. Hence $H_1 = G_{(i_1)}(\{n\},\{n\})$ for some legal coloring $i_1$ of $T_1$.
\end{proof}

\begin{proof}[Proof of Corollary~\ref{maincorollary:commensurator}]
Let $n \geq 2$ and define $\Gamma_{6,4n} \leq \Aut(T) \times \Aut(T_2^{(n)})$ as in Proposition~\ref{theorem:64n}. For $v_2 \in V(T_2^{(n)})$, the group $F = \proj_1(\Gamma_{6,4n}(v_2)) \leq \Aut(T)$ is torsion-free and acts simply transitively on the vertices of $\Aut(T)$: it is thus conjugate to $F_3$ in $\Aut(T)$. Moreover, the full projection $\proj_1(\Gamma_{6,4n}) \leq \Aut(T)$ commensurates $F$. Indeed, if $\gamma \in \Gamma_{6,4n}$ then $\gamma \Gamma_{6,4n}(v_2) \gamma^{-1} = \Gamma_{6,4n}(\gamma(v_2))$ so $\Gamma_{6,4n}(v_2) \cap \gamma \Gamma_{6,4n}(v_2) \gamma^{-1}$ is nothing else than the fixator of $v_2$ and $\gamma(v_2)$ in $\Gamma_{6,4n}$. This is a finite index subgroup of $\Gamma_{6,4n}(v_2)$ as wanted. Hence, the closure of the commensurator of $F_3$ in $\Aut(T)$ contains $G_{(i^{(n)})}(\{n\},\{n\})$ for some legal coloring $i^{(n)}$ of $T$ (see Proposition~\ref{theorem:64n}). The conclusion follows from the fact that $\bigcup_{n \geq 2}G_{(i^{(n)})}(\{n\},\{n\})$ is dense in $\Aut(T)$.
\end{proof}

\section{About products of three trees}

We finally prove Theorem~\ref{maintheorem:3trees}, which deals with products of three trees.


\begin{proof}[Proof of Theorem~\ref{maintheorem:3trees}]
Suppose that such a group $\Gamma$ exists. Let us consider $\Gamma(v_3)$, the fixator of $v_3$ in $T_3$. The group $\Gamma' = \proj_{1,2}(\Gamma(v_3)) \leq \Aut(T_1) \times \Aut(T_2)$ acts simply transitively on the vertices of $T_1 \times T_2$, i.e.\ it is a $(6,6)$-group. By hypothesis, $\proj_{1,3}(\Gamma)$ is dense in $H_1 \times H_3$, so $\proj_{1,3}(\Gamma(v_3))$ is dense in $H_1 \times H_3(v_3)$ (because $H_3(v_3)$ is open in $H_3$). Taking images under the continuous map $\proj_1$, we get that $\proj_1(\Gamma(v_3))$ is dense in $H_1$, i.e.\ $\overline{\proj_1(\Gamma')} = H_1$. Similarly, we have $\overline{\proj_2(\Gamma')} = H_2$. We deduce in particular that $\Gamma'$ is an irreducible $(6,6)$-group whose local actions on $T_1$ and $T_2$ contain $\Alt(6)$. The last hypothesis also implies that the values $\tau_1$ and $\tau_2$ associated to $\Gamma'$ are both equal to zero.

As can be read from Tables~\ref{table:66} and~\ref{table:6600T}, there are $23225$ equivalence classes of irreducible $(6,6)$-groups with $\tau_1 = \tau_2 = 0$ and $\underline{H_1}(v_1), \underline{H_2}(v_2) \geq \Alt(6)$. We indeed have $2240$ such groups that are torsion-free and $20985$ such groups with torsion.

There remains to prove that none of those $23225$ groups can be equal to $\Gamma'$. Let $\gamma$ be an element of $\Gamma'$. It induces a permutation of the six neighbors of $v_2$. Since all elements of $\Sym(6)$ have order $\leq 6$, there exists $o \in \{4,5,6\}$ such that $\gamma^o$ fixes $B(v_2,1)$ in $T_2$. If $Q_{\gamma^o}$ is the group obtained by adding the relation $\gamma^o = 1$ to the presentation of $\Gamma'$, then we have a natural surjection $Q_{\gamma^o} \to \underline{H_3}(v_3)$. Now recall that, if $\gamma^o$ is non-trivial, then $Q_{\gamma^o}$ is a finite group by the Normal Subgroup Theorem. Also, $\underline{H_3}(v_3)$ is isomorphic to $\Alt(6)$ or $\Sym(6)$ by hypothesis. Hence, for each of the $23225$ groups mentioned above and for each generator $g \in \{a_1,a_2,a_3,b_1,b_2,b_3\}$ we can compute (with GAP) the groups $Q_{g^o}$ for each $o \in \{4,5,6\}$ and check if one of these three finite groups surjects onto $\Alt(6)$ or $\Sym(6)$. If the answer is no for one of the six generators, then that group can be excluded.

We could check this condition on all $23225$ groups, and the answer is clear: none of them satisfies the condition.
\end{proof}

\newpage


\appendix

\section[Almost simplicity of commensurators of \texorpdfstring{$F_n$}{F_n} - by P.-E. Caprace]{Almost simplicity of commensurators of free groups -- by Pierre-Emmanuel Caprace}\label{appendix:A}

\begin{flushright}
\begin{minipage}[t]{0.45\linewidth}\itshape\small
Parfois le vacarme des espaces infinis\\
Effrayait un peu les humains\\
Qui se trainaient péniblement\\
A toute allure\\
Sur la voie lactée du progrès\\

Quand dans leur champ visuel\\
Un arbre surgissait encore

\hfill\upshape (Jacques Pr\'evert, \emph{Arbres}, 1976)
\end{minipage}
\end{flushright}


\subsection{Introduction}

Let $G$ be a group and $\Gamma \leq G$ be a subgroup. The \textbf{group of relative commensurators} (or simply the \textbf{relative commensurator}) of $\Gamma$ in $G$ is defined as the set 
$$\Comm_G(\Gamma) = \{g \in G \mid [\Gamma : \Gamma \cap g \Gamma g^{-1}] < \infty\}.$$
It is a subgroup of $G$ containing the normalizer $\mathrm N_G(\Gamma)$. 

Let $m \geq 3$, let $W_m$ be the \textbf{free Coxeter group of rank~$m$}, i.e.\ the free product of $m$ copies of the cyclic group of order~$2$, and $T$ be the $m$-regular tree, viewed as a Cayley tree of $W_m$. Let also $\Aut(T)^+$ be the index~$2$ subgroup of the full automorphism group $\Aut(T)$ consisting of the automorphisms preserving the canonical bipartition of $T$. In \cite[Remark~2.12(i)]{LMZ}, Lubotzky--Mozes--Zimmer emphasize the open problem asking whether the relative commensurator $\Comm_{\Aut(T)^+}(W_m)$ is a simple group. That problem is notably motivated by the analogy between the properties of tree lattices and lattices in rank one simple Lie groups, illustrated by numerous results (see \cite{Bass-Lubotzky}). The relative commensurator group $\Comm_{\Aut(T)^+}(W_m)$ could be compared with $\Comm_{\PSL_2(\RR)}(\PSL_2(\ZZ))$, which coincides with the simple group $\PSL_2(\QQ)$. 
Our first goal is to show that the relative commensurator group $\Comm_{\Aut(T)^+}(W_m)$ enjoys the weaker property of being \emph{almost simple}. In order to define that notion, we recall that a group $G$ is called \textbf{monolithic} if the intersection of all non-identity normal subgroups of $G$ is non-trivial. That intersection is then called the \textbf{monolith} of $G$; it is denoted by $\Mon(G)$. The group $G$ is called \textbf{almost simple} if $G$ is monolithic and if its monolith is a non-abelian simple group. Thus, if $G$ is almost simple, then the natural conjugation action of $G$ on its monolith $S = \Mon(G)$ is injective, and after identifying $G$ with its image in $\Aut(S)$, we obtain $S \cong \Inn(S) \leq G \leq \Aut(S)$. Conversely, given a non-abelian simple group $S$, any group $G$ with $S \cong \Inn(S) \leq G \leq \Aut(S)$ is almost simple. The term \emph{almost simple} is standard in finite group theory; its use in the study of infinite groups is less frequent.


\begin{thm}\label{thm:RelCom}
	Let $m \geq 3$ and $C_m= \Comm_{\Aut(T)}(W_m)$ be the relative commensurator of the free Coxeter group $W_m$ of rank~$m$ in the automorphism group of its Cayley tree $T$. Then:
\begin{enumerate}[(i)]
	\item $C_m$ is almost simple. 
	\item The simple monolith $\Mon(C_m)$ contains a finite index subgroup of $W_m$. 
	\item The closure of $\Mon(C_m)$ in $\Aut(T)$ coincides with $\Aut(T)^+$. 
	\item  $\Mon(C_m)$ is not finitely generated. 	
\end{enumerate}
\end{thm}

This readily implies that $\Comm_{\Aut(T)^+}(W_m)$, which is an index~$2$ subgroup of $C_m$, is almost simple. Theorem~\ref{thm:RelCom} also implies that for any $d \geq 2$, the relative commensurator group $\Comm_{\Aut(T_{2d})}(F_d)$ of the free group $F_d$ in the automorphism group of its Cayley tree $T_{2d}$ is almost simple, since $F_d$ and $W_{2d}$ both contain $F_{2d-1}$ as a subgroup of index~$2$, and are thus commensurate in $\Aut(T_{2d})$.  

Denoting the monolith of $C_m =  \Comm_{\Aut(T)}(W_m)$ by $S_m$, we emphasize that  $\{S_m \}_{m \geq 3}$ constitutes an infinite family of pairwise non-isomorphic simple groups. Indeed, this follows from the Commensurator Rigidity Theorem in \cite[Theorem~A.1]{Monod_JAMS} or \cite[Theorem~8.1]{GKM}. The fact that the full commensurator groups $C_m$ are pairwise non-isomorphic for distinct values of $m$ was established in \cite[Corollary~2]{LMZ}. Moreover, Commensurator Rigidity implies that $\Aut(S_m)$ coincides with $\mathrm N_{\Aut(T)}(S_m)$. I do not know whether the latter group coincides with $C_m$; similarly,  \cite[Corollary~2]{LMZ} ensures that $\Aut(C_m)  = \mathrm N_{\Aut(T)}(C_m)$ and, as pointed out in \cite{LMZ}, it is an open problem to determine whether $\mathrm N_{\Aut(T)}(C_m) = C_m$. Note that $S_m$ is a characteristic subgroup of $C_m$, so that $\mathrm N_{\Aut(T)}(C_m) \leq \mathrm N_{\Aut(T)}(S_m)$.

Another open problem, popularized by A.\ Luboztky, asks whether the group of \emph{abstract commensurators} of the free group $F_d$ of rank~$d \geq 2$ is simple. We recall that for any group $\Gamma$, the group $\Comm(\Gamma)$ of \textbf{abstract commensurators} of $\Gamma$ is defined as the quotient of the set of all pairs $(\Gamma_0, \alpha_0)$ consisting of a finite index subgroup $\Gamma_0 \leq \Gamma$ and an injective homorphism $\alpha_0 \colon \Gamma_0 \to \Gamma$ whose image is of finite index, by the equivalence relation identifying pairs $(\Gamma_0, \alpha_0)$ and $(\Gamma_1, \alpha_1)$ such that $\alpha_0$ and $\alpha _1$ coincide on a finite index subgroup of $\Gamma_0\cap \Gamma_1$. The group $\Comm(\Gamma)$  was introduced in \cite[Section~6, Appendix~B]{Bass-Kulkarni}, where the idea of the concept is attributed to J.-P.\ Serre and W.\ Neumann. We contribute to Luboztky's question by establishing the following abstract companion to Theorem~\ref{thm:RelCom}. 

\begin{thm}\label{thm:AbsCom}
Let  $d \geq 2$ and $A=\Comm(F_d)$ be the group of abstract commensurators of the free group $F_d$. We identify $F_d$ with its natural image in $A$. Then:
\begin{enumerate}[(i)]
	\item $A$ is almost simple. 
	\item  $\Mon(A)$ contains a finite index subgroup of $F_d$. 
	\item  $\Mon(A)$ is not finitely generated. 
\end{enumerate}
\end{thm}

Since $F_2$ and $F_d$ are abstractly commensurable,  we have in fact $\Comm(F_2) = \Comm(F_d)$ for all $d \geq 2$. Hence Theorem~\ref{thm:AbsCom} affords a single simple group up to isomorphism, which contrasts with Theorem~\ref{thm:RelCom}. In fact the simple group $\Mon(A)$ is the largest one, in the sense that the natural homomorphism $C_m \to A$ induces an embedding $S_m \to \Mon(A)$  for every $m \geq 3$ (see Remark~\ref{rem:embedding} below). 

The main ingredient in the proofs of Theorems~\ref{thm:RelCom} and~\ref{thm:AbsCom} is the existence of simple subgroups of $A$ and $C_m$ containing a finite index subgroup of $F_d$ and $W_m$ respectively. In the case of $A$, such a simple group is provided by $\PSL_2(\mathbf Q)$. In the case of $C_m$, we rely on Radu's Theorem~\ref{maintheorem:2n2n+1}.

An important result on the normal subgroup structure of relative commensurators of lattices in general locally compact groups is established by D.~Creutz and Y.~Shalom~\cite[Theorem~1.1]{CreutzShalom}; it implies notably that any non-trivial normal subgroup of $C_m$ contains a finite index subgroup of $W_m$.  However, the proof of Theorem~\ref{thm:RelCom} does not require to invoke that result.

\subsection{A criterion of almost simplicity}

Let $G$ be a group and $\Gamma \leq G$ be a subgroup. The relative commensurator group of $\Gamma$ in $G$ was defined in the introduction. We also define the \textbf{FC-centralizer} of $\Gamma$ in $G$ as the set $\FC_G(\Gamma)$ of those $g \in G$ which centralize a finite index subgroup of $\Gamma$. It is easy to check  that $\FC_G(\Gamma)$ is a normal subgroup of $\Comm_G(\Gamma)$. It coincides with the kernel of the natural homomorphism of $\Comm_G(\Gamma)$ to the group $\Comm(\Gamma)$ of abstract commensurators of $\Gamma$. 

\begin{lem}\label{lem:NormalFC}
Let $G$ be a group, $\Gamma \leq G$ be a subgroup and let $N \leq \Comm_G(\Gamma)$ be such that $\Gamma \leq \mathrm N_G(N)$. 
If $N \cap \Gamma = \{1\}$, then $N \leq \FC_G(\Gamma)$. 
\end{lem}
\begin{proof}
Let $x \in N$. Since $x$ commensutares $\Gamma$, there is a finite index subgroup $\Gamma_0 \leq \Gamma$ such that $x \Gamma_0 x^{-1} \leq \Gamma$. Since $\Gamma_0$ normalizes $N$, it follows that for all $\gamma \in \Gamma_0$, the commutator $[x, \gamma]$ is contained in the intersection $N \cap \Gamma$. The latter being trivial by hypothesis, we infer that $x$ commutes with $\Gamma_0$. Thus $x \in \FC_G(\Gamma)$. 
\end{proof}

We say that a subgroup $\Gamma$ of a group $G$  is a \textbf{commensurated} subgroup of $G$ if $\Comm_G(\Gamma) = G$. 

\begin{lem}\label{lem:AlmostSimple}
Let $G$ be a non-trivial group and $\Gamma$ be a commensurated subgroup such that $\FC_G(\Gamma)=\{1\}$. Suppose that $G$ has a   simple subgroup $B$ containing a finite index subgroup of $\Gamma$. Then $G$ is almost simple,  and the monolith $\Mon(G)$ coincides with the normal closure of $B$ in $G$. In particular $\Mon(G)$ contains a finite index subgroup of $\Gamma$. 
\end{lem}
\begin{proof}
Let $N$ be a non-identity normal subgroup of $G$. By hypothesis, we have $[\Gamma : B \cap \Gamma] < \infty$. Thus $\FC_G(B \cap \Gamma) = \FC_G(\Gamma)= \{1\}$.  By Lemma~\ref{lem:NormalFC}, this implies that $N \cap B \cap \Gamma$ is non-trivial. In particular $N \cap B$ is non-trivial. Since $B$ is simple, we obtain $N \geq B$. Therefore $G$ is monolithic and $\Mon(G) \geq B$. 

The group $B \cap \Gamma$ is a commensurated subgroup of $\Mon(G)$, and its FC-centralizer is trivial. By the first part of the proof, the monolith of $\Mon(G)$ is thus monolithic, and its monolith contains $B$. Since $\Mon(\Mon(G))$ is a  characteristic subgroup of $\Mon(G)$, it is normal in $G$. Thus $\Mon(\Mon(G))= \Mon(G)$. This proves that $\Mon(G)$ is simple. 

It remains to show that $\Mon(G)$ is not abelian. If $\Mon(G)$ were abelian, then $B$, hence $\Gamma$, would be finite. This would imply that  $G = \FC_G(\Gamma)=\{1\}$, contradicting the hypothesis that $G$ is non-trivial. 
\end{proof}

\subsection{Proofs of Theorems~\ref{thm:RelCom} and~\ref{thm:AbsCom}}


%

We will rely on Theorem~\ref{maintheorem:2n2n+1} to verify the hypotheses of Lemma~\ref{lem:AlmostSimple}. In order to cover the trivalent tree, we need the following result relying on arithmetic groups.

\begin{prop}\label{prop:trivalent}
Let $W_3 = \mathbf C_2 * \mathbf C_2 * \mathbf C_2$ be the free Coxeter group of rank~$3$ and $T_3$ be its Cayley tree. Then $\Comm_{\Aut(T_3)}(W_3)$ contains a simple subgroup $B$ such that $[W_3 : B \cap W_3] < \infty$. 
\end{prop}	
\begin{proof}
Let $K = \mathbf Q(\sqrt 2)$ and let $H = (-1, -1)_K$ be the Hamilton quaternion algebra over $K$. The set $\mathscr R_\infty$ of the infinite places of the field $K$ contains exactly two elements   (coming from the two embeddings $K \to \mathbf R$), and for each $v \in \mathscr R_\infty$, the algebra $H \otimes_K K_v$ is a division algebra (it is a subalgebra of the Hamilton quaterions over $\mathbf R$). Moreover, for any finite place $v$ of $K$, the algebra $H \otimes_K K_v$ splits, i.e. it is isomorphic to the matrix algebra $M_2(K_v)$. This follows from \cite[Propositions~9.13 and~9.14]{Jacobson}. We may thus invoke  the main result of \cite{Margulis80}, ensuring that   the quotient $B = \SL_1(H)/Z$ of the group of reduced norm one elements $\SL_1(H)$ by its center  is simple. 

On the other hand, considering the $2$-adic place, we obtain an embedding 
$$\varphi \colon \SL_1(H)/Z \to \PGL_2(\mathbf Q_2(\sqrt 2)).$$ 
Since $\mathbf Q_2(\sqrt 2)$ is a totally ramified extension of $\mathbf Q_2$, its residue field has order~$2$, so that the Bruhat--Tits tree of $\PGL_2(\mathbf Q_2(\sqrt 2))$ is the trivalent tree.  Moreover, by \cite[Lemma I.3.1.1(v) and Theorem 3.2.7(b)]{Margulis}, the group $\varphi(\SL_1(H)/Z)$ contains a cocompact lattice of $\PGL_2(\mathbf Q_2(\sqrt 2))$ as a commensurated subgroup. The latter lattice must   be commensurable with $W_3$ by \cite[Commensurability Theorem]{Bass-Kulkarni}. The desired result follows. 
\end{proof}

We now collect the relevant information from  Theorem~\ref{maintheorem:2n2n+1} and Proposition~\ref{prop:trivalent} that we shall need for the proof of Theorem~\ref{thm:RelCom}. 

\begin{cor}\label{cor:Radu}
Let $m \geq 3$ and $C_m= \Comm_{\Aut(T)}(W_m)$. Then $C_m$ contains a simple subgroup $B$ such that $[W_m : B \cap W_m] < \infty$.

Furthermore, for $m \geq 6$, the simple group $B$  can be chosen finitely generated, and such that the Schlichting completion $G = B/\!\! /B \cap W_m$ is a  simple group containing an infinite pro-$p$ subgroup for every prime $p < m-1$. 
\end{cor}
\begin{proof}
For $m = 3$, we invoke   Proposition~\ref{prop:trivalent} and the conclusion follows. For $m \geq 4$ we set $n = \lfloor \frac m 2 \rfloor$ and consider the virtually simple lattice $\Gamma_{2n, 2n+1} \leq \Aut(T_{2n}) \times \Aut(T_{2n+1})$ afforded by  Theorem~\ref{maintheorem:2n2n+1}, where $T_d$ is the $d$-regular tree. Let also $\Gamma_{2n, 2n+1}^{(\infty)} $ be its simple subgroup of finite index. Let $m' = m+1$ if $m$ is even and $m' = m-1$ if $m$ is odd. For every vertex $v$ of $T_{m'}$, the group $\Gamma_{2n, 2n+1}^{(\infty)} $ commensurates the stabilizer $\Gamma_{2n, 2n+1}^{(\infty)}(v)$.
Moreover, denoting by $B$ and $\Gamma$ the respective projections of $\Gamma_{2n, 2n+1}^{(\infty)} $ and $\Gamma_{2n, 2n+1}^{(\infty)}(v)$ to $\Aut(T_m)$,
we deduce that $B$ is a simple subgroup of $\Aut(T_m)$ which commensurates the cocompact lattice $\Gamma < \Aut(T_m)$. Since $\Gamma$ and $W_m$ are commensurable by \cite[Commensurability Theorem]{Bass-Kulkarni}, we may assume, upon replacing $B$ by a conjugate, that $B \leq C_m$. 

Let us now assume that $m \geq 6$ and  consider the Schlichting completion $G = B/\!\! / B \cap W_m$. By   \cite[Lemma~5.16]{CM_discrete} (or \cite[Lemma~3.6]{ShalomWillis}), the locally compact group  $G$ is isomorphic to the closure of the projection of $B \leq \Aut(T_{m}) \times \Aut(T_{m'})$ to $\Aut(T_{m'})$. Hence, we deduce from Theorems~\ref{theorem:2n2n+1} and~\ref{theorem:Raduclassification} that $G$ belongs to the collection $\mathcal{G}'_{(i)}$. Therefore it contains an infinite pro-$p$ subgroup for every prime $p < m'$. Moreover every member of $\mathcal{G}'_{(i)}$ is virtually simple (see Theorem~\ref{theorem:Radusimple}), and since $B$ is simple and dense in $G$, it follows that every finite quotient of $G$ must be trivial. Thus $G$ is  simple.
\end{proof}

\begin{proof}[Proof of Theorem~\ref{thm:RelCom}]
Every finite index subgroup $\Gamma$ of $W_m$ is a cocompact lattice in $\Aut(T)$, acting minimally. Therefore the centralizer $\mathrm C_{\Aut(T)}(\Gamma)$ is trivial by \cite[Proposition~6.5]{Bass-Lubotzky}. This shows that $\FC_{\Aut(T)}(W_m)=\{1\}$. By Corollary~\ref{cor:Radu}, the hypotheses of Lemma~\ref{lem:AlmostSimple} are satisfied. This proves (i) and (ii). 

By \cite[Corollary 4.25]{Bass-Kulkarni}, the group $C_m$ is dense in $\Aut(T)$. Thus the closure $\overline{\Mon(C_m)}$ is a non-trivial closed normal subgroup of $\Aut(T)$. Tits' simplicity theorem  \cite[Theorem~4.5]{Titsarbres} then implies that  $\overline{\Mon(C_m)} \geq \Aut(T)^+$, hence   $\overline{\Mon(C_m)} = \Aut(T)^+$ since otherwise $\Mon(C_m)$ would have a non-trivial quotient of order~$2$. This proves (iii).

It remains to prove that the simple group $S = \Mon(C_m)$ is not finitely generated. Suppose for a contradiction that it is. Let then $G = S/\!\!/\Gamma \cap S$ be the Schlichting completion of the pair $(S, \Gamma \cap S)$ (see \cite[\S3]{ShalomWillis}). Then $G$ is a compactly generated totally disconnected locally compact group. Moreover, the diagonal embedding 
$$S \to \Aut(T)^+ \times G$$
maps $S$ to a cocompact lattice with dense projections in the product group \sloppy $\Aut(T)^+ \times G$ (see \cite[Lemma~5.15]{CM_discrete}). Now the fact that $G$ is compactly generated   and that $S$ is cocompact in the product $\Aut(T)^+ \times G$ implies that the compact open subgroups of $\Aut(T)^+$ are topologically finitely generated profinite groups (see the proof of \cite[Proposition~1.1.2]{BMZ}). This is absurd, since a vertex stabilizer in  $\Aut(T)^+$ maps continuously onto the infinite Cartesian product $\prod_{\mathbf N} \mathbf C_2$, which is not finitely generated. This proves (iv). 
\end{proof}

\begin{rmk}\label{rem:embedding}
As noticed in the proof above, we have $\FC_{\Aut(T)}(W_m)=\{1\}$, so that $\FC_{C_m}(W_m)=\{1\}$. Thus the natural map $C_m \to \Comm(W_m)$ is injective. By Lemma~\ref{lem:AlmostSimple}, it follows that $\Mon(\Comm(W_m))$ is the normal closure of the image of the simple group $S_m$ in $\Comm(W_m)$. As noticed in the introduction, for all $d \geq 2$ we have $\Comm(F_2)= \Comm(F_d) = \Comm(W_{2d})$. Thus the simple group $S_m$ embeds in  $\Mon(\Comm(F_2))$ for every $m \geq 3$ by Lemma~\ref{lem:AlmostSimple}. 
\end{rmk}

To complete the proof of Theorem~\ref{thm:AbsCom}, we  need the following basic property of Schlichting completions. 

\begin{lem}\label{lem:Schlichting}
Let $\Lambda$ be a group and $\Gamma \leq \Lambda$ be a commensurated subgroup.  Let also $H$ be a totally disconnected locally compact group and $\varphi \colon  \Lambda \to H$ be a    homomorphism such that    $\overline{\varphi(\Gamma)}$ is compact. Then  $\overline{\varphi(\Lambda)}$  has a compact normal subgroup $K$ such that the natural map $\Lambda \to \overline{\varphi(\Lambda)}/K$ extends to a continuous homomorphism with dense image  $  \Lambda /\!\! / \Gamma \to \overline{\varphi(\Lambda)}/K$. 
\end{lem}
\begin{proof}
By \cite[Lemma~3.1]{Cap_CMH}, the compact group $\overline{\varphi(\Gamma)}$  is contained in some compact open subgroup $L$ of $H$. Let $\Gamma_1= \varphi^{-1}(L)$. Then $\Gamma_1$ is a commensurated subgroup of $\Lambda$, and we have  $\Gamma \leq \Gamma_1$. By \cite[Lemmas~3.5 and 3.6]{ShalomWillis}, the group  $\overline{\varphi(\Lambda)}$  has a compact normal subgroup $K$ such that the quotient $\overline{\varphi(\Lambda)}/K$ is isomorphic to the Schlichting completion $\Lambda  /\!\! / \Gamma_1$. Moreover, by   \cite[Lemma~3.8(1)]{ShalomWillis}, the natural map $\Lambda \to \Lambda  /\!\! / \Gamma_1$ extends to a continuous homomorphism with dense image  $  \Lambda /\!\! / \Gamma \to \Lambda  /\!\! / \Gamma_1 \cong \overline{\varphi(\Lambda)}/K$. 
\end{proof}

\begin{proof}[Proof of Theorem~\ref{thm:AbsCom}]
	We first recall that, as pointed out in the introduction, we have $\Comm(F_d)=\Comm(F_2)$ for all $d \geq 2$. 
	
	The simple group $B = \PSL_2(\QQ)$ commensurates its finitely generated virtually free subgroup $\Gamma = \PSL_2(\ZZ)$. Thus we have a natural homomorphism $\varphi \colon B \to \Comm(\Gamma)$. 	The kernel of $\varphi$ is the FC-centralizer $\FC_B(\Gamma)$. Notice that $\FC_{\PSL_2(\RR)}(\Gamma)$ is trivial, since every finite index subgroup of $\Gamma$ is a lattice in $\PSL_2(\RR)$, and thus has a trivial centralizer by the Borel density theorem. In particular $\varphi$ is injective. This implies that  $\Comm(\Gamma) = \Comm(F_2)$ has a simple subgroup isomorphic to $\PSL_2(\QQ)$ containing a finite index subgroup of $F_2$. By the very definition, the FC-centralizer of $\Gamma$ in $\Comm(\Gamma)$ is trivial. The assertions (i) and (ii) thus follow from Lemma~\ref{lem:AlmostSimple}. Alternatively, we could have verified the hypotheses of Lemma~\ref{lem:AlmostSimple} using Theorem~\ref{thm:RelCom}, see Remark~\ref{rem:embedding}.

	In order to prove (iii), we set $A = \Comm(\Gamma)$ and $S = \Mon(A)$. 
Suppose for a contradiction that the simple group $S$ is finitely generated. Let $G = S/\!\!/\Gamma \cap S$ be the Schlichting completion and $\varphi \colon S \to G$ be the natural homomorphism. By \cite[Proposition~3.6(iii)]{CRW}, the fact that $S$ is a finitely generated simple group implies that there is  compactly generated topologically simple totally disconnected locally compact group $H$ and a continuous surjective homomorphism $\pi \colon G \to H$. 

Since $F_d$ is abstractly commensurable with the free Coxeter group $W_m$ of rank $m$ for every $m$, we have $A = \Comm(W_m)$. We consider the finitely generated simple group $B_{m}$ afforded by Corollary~\ref{cor:Radu} and  let $G_{m} = B_m /\! \! / B_m \cap W_m$ be the corresponding Schlichting completion. By Corollary~\ref{cor:Radu}, the group $G_{m}$ is simple and contains an infinite pro-$p$ subgroup for every prime $p < m-1$, provided $m \geq 6$.

We next deduce from Lemma~\ref{lem:AlmostSimple} and Remark~\ref{rem:embedding} that $S$ contains a copy of the simple group $B_m$ for all $m \geq 6$. We shall thus view $B_m$ as a subgroup of $S$. 

In view of Lemma~\ref{lem:Schlichting}, the closure $\overline{\pi ( \varphi(B_m))}$ of the image of $B_m$ in $H$ has a compact normal subgroup $K$, and there is a  continuous homomorphism with dense image 
$$G_{m} \to \overline{\pi ( \varphi(B_m))}/K.$$ 
Notice that the latter quotient cannot be trivial, since otherwise $\overline{\pi ( \varphi(B_m))}$ would be compact, which is absurd since $B_m$ is an infinite simple group, hence not residually finite. Since $G_{m}$  is  simple, the continuous homomorphism $G_{m} \to \overline{\pi ( \varphi(B_m))}/K$ is injective. Therefore $H$ contains an infinite pro-$p$ subgroup for every prime $p < m-1$. Since this argument is valid for every $m \geq 6$, we infer that $H$ contains an infinite pro-$p$ subgroup for every prime $p$. However, by \cite[Theorems~4.14 and~5.3]{CRW}, a compactly generated topologically simple totally disconnected locally compact group can contain an infinite pro-$p$ subgroup for only finitely many primes. This is a contradiction. 
\end{proof}

\clearpage

\phantomsection

\addcontentsline{toc}{section}{References}

\bibliographystyle{amsalpha}
\bibliography{biblio}

\end{document}